\documentclass[11pt]{amsart}

\usepackage{bm}
\usepackage{amssymb, amscd}
\usepackage{epsfig, mathtools, tensor}
\usepackage[vmargin=1in, hmargin=1.25in]{geometry}
\usepackage[font=small,format=plain,labelfont=bf,up,textfont=it,up]{caption}
\usepackage{microtype}
\usepackage[shortlabels]{enumitem}
\usepackage[backref=page, bookmarks, bookmarksdepth=2, colorlinks=true, linkcolor=blue, citecolor=blue, urlcolor=blue]{hyperref}
\usepackage{etoolbox}
\usepackage{tikz-cd}
\usepackage{tabularray}
\apptocmd{\thebibliography}{\raggedright}{}{}
\usepackage[only,llbracket,rrbracket,llparenthesis, rrparenthesis]{stmaryrd}
\usepackage{xcolor}

\usepackage[only,llparenthesis, rrparenthesis, llbracket,rrbracket]{stmaryrd}

\definecolor{colorblind_blue}{RGB}{0,114,178}
\definecolor{colorblind_orange}{RGB}{213,94,0}
\definecolor{colorblind_green}{RGB}{0,158,115}
\definecolor{colorblind_purple}{RGB}{204,121,167}

\newcommand{\orange}[1]{\ensuremath{{\color{colorblind_orange} #1}}}
\newcommand{\blue}[1]{\ensuremath{{\color{colorblind_blue} #1}}}
\newcommand{\green}[1]{\ensuremath{{\color{colorblind_green} #1}}}
\newcommand{\purple}[1]{\ensuremath{{\color{colorblind_purple} #1}}}

\newcommand{\QGen}[1]{\ensuremath{[#1]}}
\newcommand{\PQGen}[1]{\ensuremath{[#1]'}}
\newcommand{\HGen}[1]{\ensuremath{[#1]_{\Sp}}}

\newcommand{\AGen}[1]{\ensuremath{[#1]_0}}

\newcommand{\BAGen}[1]{\blue{\ensuremath{[#1]_0}}}
\newcommand{\GAGen}[1]{\green{\ensuremath{[#1]_0}}}
\newcommand{\UGAGen}[1]{\green{\underline{\ensuremath{[#1]_0}}}}

\newcommand{\PAGen}[1]{\ensuremath{[#1]'_0}}

\newcommand{\BPAGen}[1]{\blue{\ensuremath{[#1]'_0}}}
\newcommand{\GPAGen}[1]{\green{\ensuremath{[#1]'_0}}}

\newcommand{\ZGen}[1]{\ensuremath{\llparenthesis #1 \rrparenthesis}}
\newcommand{\ZGenS}[1]{\ensuremath{\llparenthesis #1 \rrparenthesis}_s}

\newcommand{\ZGenA}[1]{\ensuremath{\llparenthesis #1 \rrparenthesis}_a}

\newcommand{\Pres}[1]{\ensuremath{\llbracket #1 \rrbracket}}
\newcommand{\OPres}[1]{\ensuremath{\orange{\llbracket #1 \rrbracket}}}
\newcommand{\LPres}[1]{\ensuremath{\orange{\llbracket #1 \rrbracket_L}}}
\newcommand{\RPres}[1]{\ensuremath{\orange{\llbracket #1 \rrbracket_R}}}

\newcommand{\PresS}[1]{\ensuremath{\llbracket #1 \rrbracket}_s}
\newcommand{\ThetaS}[1]{\ensuremath{\Theta[#1]_s}}
\newcommand{\LambdaS}[1]{\ensuremath{\Lambda[#1]_s}}
\newcommand{\LambdaIS}[1]{\ensuremath{\Lambda_1[#1]_s}}
\newcommand{\LambdaIIS}[1]{\ensuremath{\Lambda_2[#1]_s}}
\newcommand{\OmegaS}[1]{\ensuremath{\Omega[#1]_s}}
\newcommand{\OmegaIS}[1]{\ensuremath{\Omega_1[#1]_s}}
\newcommand{\OmegaIIS}[1]{\ensuremath{\Omega_2[#1]_s}}
\newcommand{\OmegaIIIS}[1]{\ensuremath{\Omega_3[#1]_s}}
\newcommand{\OmegaIVS}[1]{\ensuremath{\Omega_4[#1]_s}}
\newcommand{\OmegaiS}[1]{\ensuremath{\Omega_i[#1]_s}}
\newcommand{\SPresS}[1]{\ensuremath{\blue{\llbracket #1 \rrbracket}_s}}
\newcommand{\SThetaS}[1]{\ensuremath{\purple{\Theta[#1]_s}}}
\newcommand{\SLambdaS}[1]{\ensuremath{\orange{\Lambda[#1]_s}}}
\newcommand{\SOmegaS}[1]{\ensuremath{\green{\Omega[#1]_s}}}

\newcommand{\PresA}[1]{\ensuremath{\llbracket #1 \rrbracket}_a}
\newcommand{\BPresA}[1]{\ensuremath{\blue{\llbracket #1 \rrbracket_a}}}
\newcommand{\PPresA}[1]{\ensuremath{\purple{\llbracket #1 \rrbracket_a}}}
\newcommand{\OPresA}[1]{\ensuremath{\orange{\llbracket #1 \rrbracket_a}}}

\newcommand{\Fix}[1]{\ensuremath{\fF[#1]}}
\newcommand{\FixIm}[1]{\ensuremath{\cF[#1]}}
\newcommand{\FixBigIm}[1]{\ensuremath{\widehat{\cF}[#1]}}
\newcommand{\pFix}[1]{\ensuremath{\purple{\Fix{#1}}}}
\newcommand{\oFix}[1]{\ensuremath{\orange{\Fix{#1}}}}

\newcommand{\SymSp}{\ensuremath{\operatorname{SymSp}}}

\mathtoolsset{showonlyrefs}

\makeatletter
\patchcmd{\@maketitle}{\global\topskip42\p@\relax}
  {\global\topskip42\p@\relax \vspace*{-38pt}}
  {}{}
\makeatother

\setenumerate[0]{label=(\alph*)}

\renewcommand*{\backref}[1]{}
\renewcommand*{\backrefalt}[4]{%
    \ifcase #1 (Not cited.)%
    \or        (Cited on page~#2.)%
    \else      (Cited on pages~#2.)%
    \fi}

\newcommand*{\Cdot}[1][1.25]{%
  \mathpalette{\CdotAux{#1}}\cdot%
}
\newdimen\CdotAxis
\newcommand*{\CdotAux}[3]{%
  {%
    \settoheight\CdotAxis{$#2\vcenter{}$}%
    \sbox0{%
      \raisebox\CdotAxis{%
        \scalebox{#1}{%
          \raisebox{-\CdotAxis}{%
            $\mathsurround=0pt #2#3$%
          }%
        }%
      }%
    }%
    \dp0=0pt %
    \sbox2{$#2\bullet$}%
    \ifdim\ht2<\ht0 %
      \ht0=\ht2 %
    \fi
    \sbox2{$\mathsurround=0pt #2#3$}%
    \hbox to \wd2{\hss\usebox{0}\hss}%
  }%
}

\numberwithin{equation}{section}

\theoremstyle{plain}
\newtheorem{theorem}{Theorem}[section]
\newtheorem{maintheorem}{Theorem}

\newtheorem{maintheoremprime}{Theorem}
 
\newtheorem{proposition}[theorem]{Proposition}
\newtheorem{lemma}[theorem]{Lemma}

\newtheorem{corollary}[theorem]{Corollary}

\newtheorem{question}[theorem]{Question}

\newtheorem*{unnumberedclaim}{Claim}

\newtheorem{stepx}{Step}
\newenvironment{step}[1]
 {\renewcommand\thestepx{#1}\stepx}
 {\endstepx}

\newtheorem{casex}{Case}
\newenvironment{case}[1]
 {\renewcommand\thecasex{#1}\casex}
 {\endcasex}

\newtheorem{claimx}{Claim}
\newenvironment{claim}[1]
 {\renewcommand\theclaimx{#1}\claimx}
 {\endclaimx}

\theoremstyle{definition}
\newtheorem{asm}[theorem]{Assumption}
\newenvironment{assumption}[1][]{\begin{asm}[#1]\pushQED{\qed}}{\popQED \end{asm}}
\newtheorem{defn}[theorem]{Definition}
\newenvironment{definition}[1][]{\begin{defn}[#1]\pushQED{\qed}}{\popQED \end{defn}}
\newtheorem{notn}[theorem]{Notation}

\newtheorem{warn}[theorem]{Warning}
\newenvironment{warning}[1][]{\begin{warn}[#1]\pushQED{\qed}}{\popQED \end{warn}}

\theoremstyle{remark}
\newtheorem{rmk}[theorem]{Remark}
\newenvironment{remark}[1][]{\begin{rmk}[#1] \pushQED{\qed}}{\popQED \end{rmk}}
\newtheorem{eg}[theorem]{Example}

\newtheorem{cvn}[theorem]{Convention}
\newenvironment{convention}[1][]{\begin{cvn}[#1] \pushQED{\qed}}{\popQED \end{cvn}}

\theoremstyle{plain}

\DeclareMathOperator{\Hom}{Hom}

\DeclareMathOperator{\Image}{Im}

\DeclareMathOperator{\GL}{GL}

\DeclareMathOperator{\SL}{SL}

\DeclareMathOperator{\Sp}{Sp}

\DeclareMathOperator{\fsl}{\mathfrak{sl}}

\newcommand\Z{\ensuremath{\mathbb{Z}}}
\newcommand\Q{\ensuremath{\mathbb{Q}}}

\DeclareMathOperator{\HH}{H}

\DeclareMathOperator{\Aut}{Aut}

\DeclareMathOperator{\Sym}{Sym}

\newcommand\notsubset{\ensuremath{\not\subset}}

\newcommand\Span[1]{\ensuremath{\langle #1 \rangle}}
\newcommand\Set[2]{\ensuremath{\left\{\text{#1 $|$ #2}\right\}}}
\newcommand\SetLong[4]{\ensuremath{\{\text{#1 $|$ \parbox[t]{\widthof{#4\quad}}{#2\}#3}}}}

\newcommand\cB{\ensuremath{\mathcal{B}}}

\newcommand\cF{\ensuremath{\mathcal{F}}}

\newcommand\cK{\ensuremath{\mathcal{K}}}

\newcommand\cO{\ensuremath{\mathcal{O}}}

\newcommand\cT{\ensuremath{\mathcal{T}}}
\newcommand\cU{\ensuremath{\mathcal{U}}}
\newcommand\cV{\ensuremath{\mathcal{V}}}
\newcommand\cW{\ensuremath{\mathcal{W}}}

\newcommand\cZ{\ensuremath{\mathcal{Z}}}

\newcommand\fA{\ensuremath{\mathfrak{A}}}

\newcommand\fF{\ensuremath{\mathfrak{F}}}

\newcommand\fH{\ensuremath{\mathfrak{H}}}

\newcommand\fK{\ensuremath{\mathfrak{K}}}

\newcommand\fQ{\ensuremath{\mathfrak{Q}}}

\newcommand\fS{\ensuremath{\mathfrak{S}}}
\newcommand\fT{\ensuremath{\mathfrak{T}}}

\newcommand\fV{\ensuremath{\mathfrak{V}}}

\newcommand\fZ{\ensuremath{\mathfrak{Z}}}

\newcommand\fc{\ensuremath{\mathfrak{c}}}

\newcommand\tkappa{\ensuremath{\widetilde{\kappa}}}

\newcommand\tPhi{\ensuremath{\widetilde{\Phi}}}

\newcommand\oI{\ensuremath{\overline{I}}}

\newcommand\oU{\ensuremath{\overline{U}}}

\newcommand\okappa{\ensuremath{\overline{\kappa}}}

\newcommand\oPhi{\ensuremath{\overline{\Phi}}}

\newcommand\ofc{\ensuremath{\overline{\fc}}}

\newcommand\oeta{\ensuremath{\overline{\eta}}}

\newcommand\oomega{\ensuremath{\overline{\omega}}}

\newcommand\hfc{\ensuremath{\widehat{\fc}}}

\newcommand\trace{\ensuremath{\operatorname{tr}}}
\newcommand\Mod{\ensuremath{\operatorname{Mod}}}
\newcommand\prim{\ensuremath{\operatorname{prim}}}

\title{Presentations of representations}

\author{Daniel Minahan}
\address{Dept of Mathematics; University of Chicago; Chicago, IL 60637}
\email{dminahan@uchicago.edu}

\author{Andrew Putman}
\address{Dept of Mathematics; University of Notre Dame; 255 Hurley Hall; Notre Dame, IN 46556}
\email{andyp@nd.edu}

\thanks{AP was supported by NSF grant DMS-2305183.  DM was supported by NSF grant DMS-2402060.}

\begin{document}

\newpage

\begin{abstract}
We give a new technique for constructing presentations by generators and relations for representations of groups like $\SL_n(\Z)$ and
$\Sp_{2g}(\Z)$.  Our results play 
an important role in recent work of the authors calculating $\HH_2$ 
of the Torelli group.
\end{abstract}

\maketitle
\thispagestyle{empty}

\section{Introduction}
\label{section:introduction}

In this paper, we give a new approach to constructing presentations by generators and relations
for representations\footnote{In this paper, representations are always defined over the field $\Q$, so
a representation of a group $G$ is a $\Q$-vector space $V$ equipped with a linear action of $G$.} of groups like $\SL_n(\Z)$ and $\Sp_{2g}(\Z)$.  The representations we have
in mind are finite-dimensional.  However, their presentations have infinitely many generators and relations,
so this finite-dimensionality is not obvious.  Our main goal is to identify
a representation we constructed in our work on the second homology group
of the Torelli group in \cite{MinahanPutmanAbelian, MinahanPutmanH2}.  These papers make essential
use of Theorems \ref{maintheorem:presentationalt} -- \ref{maintheorem:presentationsym} below.

\subsection{Special linear group, standard representation}

We start with an easy example.
A set $\{v_1,\ldots,v_k\}$ of vectors in $\Z^n$ is a {\em partial basis} if
it can be extended to a basis $\{v_1,\ldots,x_n\}$.  For $v \in \Z^n$, the set
$\{v\}$ is partial basis precisely when $v$ is a {\em primitive vector}, i.e.,
is not divisible by any integer $d \geq 2$.  

\begin{definition}
\label{definition:slstd}
Define $\fQ_n$ to be the $\Q$-vector space with the following presentation:
\begin{itemize}
\item {\bf Generators}. A generator $\QGen{v}$ for all primitive vectors $v \in \Z^n$.  Here $\QGen{v}$ should be interpreted
as a formal symbol associated to $v$.
\item {\bf Relations}.  For a partial basis $\{v_1,v_2\}$ of $\Z^n$, the relation
$\QGen{v_1}+\QGen{v_2} = \QGen{v_1+v_2}$.\qedhere
\end{itemize}
\end{definition}

The group $\SL_n(\Z)$ acts on the set of primitive vectors in $\Z^n$.  This induces
an action of $\SL_n(\Z)$ on $\fQ_n$, so $\fQ_n$ is a representation of $\SL_n(\Z)$.
Since $\fQ_n$ has infinitely many generators and relations, it
is not a priori clear if it is finite-dimensional.

Define $\Phi\colon \fQ_n \rightarrow \Q^n$ via the formula
$\Phi(\QGen{v}) = v$.  This takes relations to relations, and thus gives a well-defined
map that we call the {\em linearization map}.  Similar maps we will define in other
contexts will also be called linearization maps.  We will prove:

\begin{maintheorem}
\label{maintheorem:slstd}
For $n \geq 2$, the linearization map $\Phi\colon \fQ_n \rightarrow \Q^n$ is an isomorphism.
\end{maintheorem}

For the proof, let $\cB = \{e_1,\ldots,e_n\}$ be the standard basis
for $\Z^n$ and let $S = \{\QGen{e_1},\ldots,\QGen{e_n}\}$.  The map
$\Phi$ takes $S$ bijectively to the basis $\cB$ for $\Q^n$, so the restriction
of $\Phi$ to $\Span{S}$ is an isomorphism.  To prove Theorem \ref{maintheorem:slstd}, we must
prove that $\Span{S} = \fQ_n$.  For this, let 
$v \in \Z^n$ be a primitive vector.  Write
$v = \lambda_1 e_1 + \cdots + \lambda_n e_n$ with $\lambda_1,\ldots,\lambda_n \in \Z$.
We must prove that
\[\QGen{v} = \lambda_1 \QGen{e_1} + \cdots + \lambda_n \QGen{e_n}.\]
We will prove this by studying the action of $\SL_n(\Z)$ on $\fQ_n$.

\begin{remark}
That $\Span{S} = \fQ_n$ 
can be also be proved directly, and to help the reader appreciate the efficiency of our proof
we encourage them to work this out.  Our approach is the only one
we are aware of that can be adapted to prove the other results in this paper.
\end{remark}

\begin{remark}
Theorem \ref{maintheorem:slstd} implies not only that $\fQ_n$ is finite-dimensional, but also that
the $\SL_n(\Z)$-action on it extends to $\GL_n(\Q)$.  This is not obvious from the definition of $\fQ_n$.
\end{remark}

\begin{remark}
Theorem \ref{maintheorem:slstd} is false for $n=1$.  Indeed, $\Z^1$ has two
primitive vectors $\pm 1$, so $\fQ_1$ has two generators
$[1]$ and $[-1]$ and no relations.  It follows that $\fQ_1 \cong \Q^2$.
\end{remark}

\begin{remark}
The technique we use to prove Theorem \ref{maintheorem:slstd} is very flexible, and for instance
can also prove appropriate versions of Theorem \ref{maintheorem:slstd} with $\Z$ replaced by
a field.\footnote{Though for general fields the relations would need to be expanded slightly.}  Similar remarks apply to our other theorems.  Since this paper is already long
and technical, we chose to not attempt to state our results in maximal generality.
\end{remark}

\subsection{Adjoint representation}

The proof technique we use for Theorem \ref{maintheorem:slstd} can also be used
to construct presentations of things like tensor powers, symmetric powers, and exterior powers
of $\Q^n$.  There are numerous possibilities for the exact form of the relations, so
rather than try to prove a general theorem we will give one interesting variant.
Recall that the adjoint representation of $\SL_n(\Q)$ is the kernel $\fsl_n(\Q)$
of the trace map
\[\trace\colon (\Q^n)^{\ast} \otimes \Q^n \longrightarrow \Q\]
defined by $\trace(f,v) = f(v)$.  The dual space
$(\Q^n)^{\ast} = \Hom(\Q^n,\Q)$ contains the lattice
$(\Z^n)^{\ast} = \Hom(\Z^n,\Z)$.
Define the following:

\begin{definition}
\label{definition:sladjoint}
Define $\fA_n$ to be the $\Q$-vector space with the following presentation:
\begin{itemize}
\item {\bf Generators}.  A generator $\AGen{f,v}$ for all primitive vectors
$f \in (\Z^n)^{\ast}$ and $v \in \Z^n$ such that $f(v) = 0$.
\item {\bf Relations}.  The following two families of relations:
\begin{itemize}
\item For all primitive vectors $f \in (\Z^n)^{\ast}$ and
all partial bases $\{v_1,v_2\}$ of $\ker(f)$, the relation
$\AGen{f,v_1+v_2} = \AGen{f,v_1} + \AGen{f,v_2}$.
\item For all primitive vectors $v \in \Z^n$ and all partial bases
$\{f_1,f_2\}$ of 
\[\ker(v) = \Set{$f \in (\Z^n)^{\ast}$}{$f(v) = 0$},\]
the relation $\AGen{f_1+f_2,v} = \AGen{f_1,v}+\AGen{f_2,v}$.\qedhere
\end{itemize}
\end{itemize}
\end{definition}

Define $\Phi\colon \fA_n \rightarrow (\Q^n)^{\ast} \otimes \Q^n$ via the formula
$\Phi(\AGen{f,v}) = f \otimes v$.  This takes relations to relations, and thus
gives a well-defined linearization map with $\Image(\Phi) \subset \fsl_n(\Q)$. 
We will prove:

\begin{maintheorem}
\label{maintheorem:sladjoint}
For $n \geq 3$, the linearization map $\Phi\colon \fA_n \rightarrow \fsl_n(\Q)$ is an isomorphism.
\end{maintheorem}

\begin{remark}
Theorem \ref{maintheorem:sladjoint} is trivial for $n=1$ since $\fA_1 = \fsl_1(\Q) = 0$.  It
is false for $n=2$ since for primitive $f \in (\Z^2)^{\ast}$ and $v \in \Z^2$ we have $\ker(f),\ker(v) \cong \Z^1$.
This implies that $\fA_2$ has no relations, and thus is an infinite-dimensional vector
space with basis the set of all its generators $\AGen{f,v}$.
\end{remark}

\subsection{Symplectic group, standard representation}

We next turn to the symplectic group $\Sp_{2g}(\Z)$.  Set $H = \Q^{2g}$
and $H_{\Z} = \Z^{2g}$.  Let
$\omega\colon H \times H \longrightarrow \Q$
be the standard symplectic form, so $\Sp_{2g}(\Z)$ consists of all $M \in \GL_{2g}(\Z)$ such that
$\omega(M \Cdot v, M \Cdot w) = \omega(v,w)$ for all $v,w \in H$.
The following is an $\Sp_{2g}(\Z)$-analogue of $\fQ_n$:

\begin{definition}
\label{definition:spstd}
Define $\fH_g$ to be the $\Q$-vector space with the following presentation:
\begin{itemize}
\item {\bf Generators}. A generator $\HGen{v}$ for all primitive vectors $v \in H_{\Z}$.  
\item {\bf Relations}.  For a partial basis $\{v_1,v_2\}$ of $H_{\Z}$ with $\omega(v_1,v_2)=0$, the relation
$\HGen{v_1}+\HGen{v_2} = \HGen{v_1+v_2}$.\qedhere
\end{itemize}
\end{definition}

The action of $\Sp_{2g}(\Z)$ on $H_{\Z}$ induces an action of $\Sp_{2g}(\Z)$ on $\fH_g$.
Given Theorem \ref{maintheorem:slstd}, it is natural to expect that $\fH_g \cong H$.  However,
identifying $H$ with $\Q^{2g}$ this would imply that $\fH_g \cong \fQ_{2g}$.  The vector
spaces $\fH_g$ and $\fQ_{2g}$ have the same generators, but $\fH_g$ has fewer relations.
It seems hard to directly write each relation in $\fQ_{2g}$ in terms of the relations
in $\fH_g$.
Nevertheless, define $\Phi\colon \fH_g \rightarrow H$ via the formula
$\Phi(\HGen{v}) = v$.  This takes relations to relations, and thus gives a well-defined
linearization map.  We will prove:

\begin{maintheorem}
\label{maintheorem:spstd}
For $g \geq 2$, the linearization map $\Phi\colon \fH_g \rightarrow H$ is an isomorphism.
\end{maintheorem}

The proof is similar to that of Theorem \ref{maintheorem:slstd}, though the details
are harder since the group theory of $\Sp_{2g}(\Z)$ is less 
uniform than $\SL_n(\Z)$.

\begin{remark}
Theorem \ref{maintheorem:spstd} implies that the action of $\Sp_{2g}(\Z)$ on $\fH_g$ extends to an action of $\Sp_{2g}(\Q)$.  In
fact, it even extends to an action of $\GL_{2g}(\Q)$.  This seems hard to see directly
from the presentation.
\end{remark}

\begin{remark}
Theorem \ref{maintheorem:spstd} is false for $g =1$.  Indeed, $\fH_1$ has infinitely
many generators but no relations, so $\fH_1$ is infinite-dimensional.
\end{remark}

\subsection{Symplectic kernel}

We now discuss another representation of $\Sp_{2g}(\Z)$ that is similar to the adjoint representation
$\fsl_n(\Q)$.  The symplectic form $\omega$ induces a map $\wedge^2 H \rightarrow \Q$.  Let
$\cZ_g^a$ be its kernel.\footnote{The ``a'' in $\cZ_g^a$ stands for ``alternating''.}  
Say that $v_1,v_2 \in H$ are {\em orthogonal} if $\omega(v_1,v_2) = 0$.  For $v \in H$,
let $v^{\perp}$ be the set of all elements of $H$ that are orthogonal to $v$.  For $v \in H_{\Z}$,
let $v^{\perp}_{\Z}$ be the set of all element of $H_{\Z}$ that are orthogonal to $v$.

\begin{definition}
\label{definition:spkernelalt}
Define $\fZ_g^a$ to be the $\Q$-vector space with the following presentation:
\begin{itemize}
\item {\bf Generators}.  A generator $\ZGenA{v_1,v_2}$ for all orthogonal primitive vectors $v_1,v_2 \in H_{\Z}$.
\item {\bf Relations}.  The following two families of relations:
\begin{itemize}
\item For all generators $\ZGenA{v_1,v_2}$, the relation $\ZGenA{v_2,v_1} = -\ZGenA{v_1,v_2}$.
\item For all primitive vectors $v \in H_{\Z}$ and all partial bases $\{w_1,w_2\}$ of $v^{\perp}_{\Z}$,
the relation $\ZGenA{v,w_1 + w_2} = \ZGenA{v,w_1} + \ZGenA{v,w_2}$.\qedhere
\end{itemize}
\end{itemize}
\end{definition}

The group $\Sp_{2g}(\Z)$ acts on $\fZ_g^a$ via its action on $H_{\Z}$.
Define $\Phi\colon \fZ_g^a \rightarrow \wedge^2 H$ via the formula 
$\Phi(\ZGenA{w_1,w_2}) = w_1 \wedge w_2$.  This takes relations to relations, and thus gives a well-defined linearization map
with $\Image(\Phi) \subset \cZ_g^a$.  We will prove:

\begin{maintheorem}
\label{maintheorem:spkernelalt}
For $g \geq 1$, the linearization map $\Phi\colon \fZ_g^a \rightarrow \cZ_g^a$ is an isomorphism.
\end{maintheorem}

\subsection{Symmetric square}
\label{section:symmetricsquare}

It is also interesting to replace the anti-symmetric relation in $\fZ_g^a$ with
the corresponding symmetric relation:\footnote{The ``s'' in $\fZ_g^s$ stands for ``symmetric''.}

\begin{definition}
\label{definition:spsym}
Define $\fZ_g^s$ to be the $\Q$-vector space with the following presentation:
\begin{itemize}
\item {\bf Generators}.  A generator $\ZGenS{v_1,v_2}$ for all orthogonal primitive vectors $v_1,v_2 \in H_{\Z}$.
\item {\bf Relations}.  The following two families of relations:
\begin{itemize}
\item For all generators $\ZGenS{v_1,v_2}$, the relation $\ZGenS{v_2,v_1} = \ZGenS{v_1,v_2}$.
\item For all primitive vectors $v \in H_{\Z}$ and all partial bases $\{w_1,w_2\}$ of $v^{\perp}_{\Z}$,
the relation $\ZGenS{v,w_1 + w_2} = \ZGenS{v,w_1} + \ZGenS{v,w_2}$.\qedhere
\end{itemize}
\end{itemize}
\end{definition}

Again, $\Sp_{2g}(\Z)$ acts on $\fZ_g^s$.  Define $\Phi\colon \fZ_g^s \rightarrow \Sym^2(H)$ via the formula
$\Phi(\ZGenS{w_1,w_2}) = w_1 \Cdot w_2$. This takes relations to relations, and thus gives a well-defined linearization map.
Since $\Sym^2(H)$ is an irreducible representation of $\Sp_{2g}(\Z)$, it is surjective.  We will prove:

\begin{maintheorem}
\label{maintheorem:spsym}
For $g \geq 2$, the linearization map $\Phi\colon \fZ_g^s \rightarrow \Sym^2(H)$ is an isomorphism.
\end{maintheorem}

\subsection{Quotient representation}

Our final theorems\footnote{These are the theorems that are needed for our work on the Torelli group
in \cite{MinahanPutmanAbelian, MinahanPutmanH2}, and thus in some sense are the main point of this paper.} 
are about a subrepresentation of the $\Sp_{2g}(\Q)$-representation
$(\cZ_g^a)^{\otimes 2}$.  Above we defined $\cZ_g^a$ as 
the kernel of the map $\wedge^2 H \rightarrow \Q$ given by $\omega$.  For our final
theorems, it is more natural to view it as a quotient of $\wedge^2 H$.
The symplectic form $\omega$ on $H$ identifies $H$ with its dual.  Using this, we can
identify alternating forms on $H$ with elements of $\wedge^2 H$.  If
$\{a_1,b_1,\ldots,a_g,b_g\}$ is a symplectic basis for $H$, then
\[\omega = a_1 \wedge b_1 + \cdots a_g \wedge b_g.\]
The span of $\omega$ in $\wedge^2 H$ is a copy of $\Q$.  The
quotient $(\wedge^2 H)/\Q$ is isomorphic to $\cZ_g^a$.  Since $\omega$ lies
in $\wedge^2 H_{\Z}$, the quotient $(\wedge^2 H)/\Q$ has a lattice
$(\wedge^2 H_{\Z})/\Z$.

\begin{remark}
Except for a few places where clarity will demand we be more careful, our
notation will not distinguish elements of $\wedge^2 H$ from their
images in $(\wedge^2 H)/\Q$.  For instance, for $x,y \in H$
we will often write $x \wedge y$ for the corresponding element of
$(\wedge^2 H)/\Q$.
\end{remark}

\subsection{Symmetric contraction}
\label{section:symmetriccontractionintro}

The {\em symmetric contraction} is the alternating bilinear map
\begin{equation}
\label{eqn:definesymmetric}
\fc\colon ((\wedge^2 H)/\Q) \times ((\wedge^2 H)/\Q) \longrightarrow \Sym^2(H)
\end{equation}
defined as follows.  Start by letting
\[\hfc\colon (\wedge^2 H) \times (\wedge^2 H) \rightarrow \Sym^2(H)\]
be the alternating bilinear map defined by the formula
\[\text{$\hfc(x \wedge y,z \wedge w) = \omega(x,z) y \Cdot w - \omega(x,w) y \Cdot z - \omega(y,z) x \Cdot w + \omega(y,w) x \Cdot z$ for $x,y,z,w \in H$}.\]
This makes sense since the right hand side is alternating in $x$ and $y$ and also 
alternating in $z$ and $w$.  Regarding $\omega$ as an element of $\wedge^2 H$,
we have $\hfc(\omega,-)=0$ and $\hfc(-,\omega)=0$.  Indeed:
\begin{itemize}
\item Both $\hfc(\omega,-)$ and $\hfc(-,\omega)$ are maps
$\wedge^2 H \rightarrow \Sym^2(H)$.  The representation $\Sym^2(H)$ of $\Sp_{2g}(\Q)$
is irreducible and is not isomorphic to either of the two irreducible factors
$\Q$ and $(\wedge^2 H)/\Q$ of $\wedge^2 H$.  Thus the only
map $\wedge^2 H \rightarrow \Sym^2(H)$ is the zero map.
\item Alternatively, this 
can be seen directly 
using the fact that for a symplectic basis $\{a_1,b_1,\ldots,a_g,b_g\}$ of $H$ we have
$\omega = a_1 \wedge b_1 + \cdots + a_g \wedge b_g$.
\end{itemize}
Either way, this implies that $\hfc$ induces a map $\fc$ as in \eqref{eqn:definesymmetric}.

\subsection{Symmetric kernel}
\label{section:symmetrickernel}

The {\em symmetric kernel}, denoted $\cK_g^a$, is the kernel of the map
\[\wedge^2 ((\wedge^2 H)/\Q) \longrightarrow \Sym^2(H)\]
associated to $\fc$.  
Say that $\kappa_1,\kappa_2 \in (\wedge^2 H)/\Q$ are {\em sym-orthogonal}
if $\fc(\kappa_1,\kappa_2) = 0$, in which case $\kappa_1 \wedge \kappa_2 \in \cK_g^a$.
For $\kappa \in (\wedge^2 H)/\Q$, the {\em symmetric orthogonal complement}
of $\kappa$, denoted $\kappa^{\perp}$, consists of all $\kappa' \in (\wedge^2 H)/\Q$
that are sym-orthogonal to $\kappa$.

\subsection{Symplectic pairs}

A {\em symplectic pair} is an element of $(\wedge^2 H_{\Z})/\Z$ of the form
$a \wedge b$, where $a,b \in H_{\Z}$ are such that $\omega(a,b) = 1$.
Equivalently, there exists a symplectic basis $\{a_1,b_1,\ldots,a_g,b_g\}$
for $H_{\Z}$ with $a_1 = a$ and $b_1 = b$.  For $X \subset \wedge^2 H$, let $\overline{X}$
be its image in $(\wedge^2 H)/\Q$.  Also, for $V \subset H_{\Z}$ let $V_{\Q} = V \otimes \Q \subset H$.
We will later prove that for a symplectic pair $a \wedge b$ we have 
$(a \wedge b)^{\perp} = \overline{\wedge^2 \Span{a,b}_{\Q}^{\perp}}$.  
See Lemma \ref{lemma:symplecticorthogonal}.

\begin{remark}
A symplectic pair is an element of $(\wedge^2 H_{\Z})/\Z$, and can be expressed
in many ways as $a \wedge b$ with $a,b \in H_{\Z}$ satisfying $\omega(a,b) = 1$.  For instance,
if $a \wedge b$ is a symplectic pair, then
$a \wedge b = (2a+b) \wedge (a+b)$.
\end{remark}  

\subsection{Symmetric kernel presentation}
We now make the following definition:

\begin{definition}
\label{definition:kgalt}
Define $\fK_g^a$ to be the $\Q$-vector space with the following presentation:
\begin{itemize}
\item {\bf Generators}. 
A generator $\PresA{\kappa_1,\kappa_2}$ for all sym-orthogonal
$\kappa_1,\kappa_2 \in (\wedge^2 H)/\Q$ such that either
$\kappa_1$ or $\kappa_2$ (or both) is a symplectic pair in $(\wedge^2 H_{\Z})/\Z$.
\item {\bf Relations}.  The following two families of relations:
\begin{itemize}
\item For all generators $\PresA{\kappa_1,\kappa_2}$, the relation $\PresA{\kappa_2,\kappa_1}=-\PresA{\kappa_1,\kappa_2}$.
\item For all symplectic pairs $a \wedge b \in (\wedge^2 H_{\Z})/\Z$ and all
$\kappa_1,\kappa_2 \in (\wedge^2 H)/\Q$ that are sym-orthogonal to $a \wedge b$ 
and all $\lambda_1,\lambda_2 \in \Q$, the relation
\[\PresA{a \wedge b,\lambda_1 \kappa_1 + \lambda_2 \kappa_2} = \lambda_1 \PresA{a \wedge b,\kappa_1} + \lambda_2 \PresA{a \wedge b,\kappa_2}.\qedhere\]
\end{itemize}
\end{itemize}
\end{definition}

The group $\Sp_{2g}(\Z)$ acts on $\fK_g^a$ via its action on $H_{\Z}$.
Define $\Phi\colon \fK_g^a \rightarrow \wedge^2 (\wedge^2 H)/\Q$ via the formula 
$\Phi(\PresA{\kappa_1,\kappa_2}) = \kappa_1 \wedge \kappa_2$.
This takes relations to relations, and thus gives a well-defined linearization map.  Since
the $\kappa_i$ are sym-orthogonal, the image of $\Phi$ lies
in the symmetric kernel $\cK_g^a$.
We will prove:

\begin{maintheorem}
\label{maintheorem:presentationalt}
For $g \geq 4$, the linearization map $\Phi\colon \fK_g^a \rightarrow \cK_g^a$ is an isomorphism.
\end{maintheorem}

Theorem \ref{maintheorem:presentationalt} plays a key role in our work on $\HH_2$
of the Torelli group in \cite{MinahanPutmanAbelian, MinahanPutmanH2}.

\begin{remark}
Just like in our previous theorems, it is not obvious from the definitions that $\fK_g^a$ is finite-dimensional
or that the $\Sp_{2g}(\Z)$-action on it extends to $\Sp_{2g}(\Q)$.  We will only see these two facts
at the very end of our proof.  It is unclear if either fact holds for $g \leq 3$.
\end{remark}

\subsection{Symmetric square, II}

Just like we did for $\fZ_g$ in \S \ref{section:symmetricsquare}, it is also interesting to replace the anti-symmetric relation in $\fK_g^a$ with
the corresponding symmetric relation: 

\begin{definition}
\label{definition:kgsym}
Define $\fK_g^s$ to be the $\Q$-vector space with the following presentation:
\begin{itemize}
\item {\bf Generators}.  
A generator $\PresS{\kappa_1,\kappa_2}$ for all sym-orthogonal
$\kappa_1,\kappa_2 \in (\wedge^2 H)/\Q$ such that either
$\kappa_1$ or $\kappa_2$ (or both) is a symplectic pair in $(\wedge^2 H_{\Z})/\Z$.
\item {\bf Relations}.  The following two families of relations:
\begin{itemize}
\item For all generators $\PresS{\kappa_1,\kappa_2}$, the relation $\PresS{\kappa_2,\kappa_1}=\PresS{\kappa_1,\kappa_2}$.
\item For all symplectic pairs $a \wedge b \in (\wedge^2 H_{\Z})/\Z$ and all
$\kappa_1,\kappa_2 \in (\wedge^2 H)/\Q$ that are sym-orthogonal to $a \wedge b$  
and all $\lambda_1,\lambda_2 \in \Q$, the relation
\[\PresS{a \wedge b,\lambda_1 \kappa_1 + \lambda_2 \kappa_2} = \lambda_1 \PresS{a \wedge b,\kappa_1} + \lambda_2 \PresS{a \wedge b,\kappa_2}.\qedhere\]
\end{itemize}
\end{itemize}
\end{definition}

Define a linearization map $\Phi\colon \fK_g^s \rightarrow \Sym^2((\wedge^2 H)/\Q)$ via the formula
$\Phi(\PresS{\kappa_1,\kappa_2}) = \kappa_1 \Cdot \kappa_2$.  Unlike $\Sym^2(H)$, the representation $\Sym^2((\wedge^2 H)/\Q)$ of
$\Sp_{2g}(\Z)$ is not irreducible.  However, it turns out that $\Phi$ is surjective.  In fact:

\begin{maintheorem}
\label{maintheorem:presentationsym}
For $g \geq 4$, the linearization map $\Phi\colon \fK_g^s \rightarrow \Sym^2((\wedge^2 H)/\Q)$ is an isomorphism.
\end{maintheorem}

Theorem \ref{maintheorem:presentationsym} is also important for our work on the Torelli group.

\subsection{Final remarks}

Innumerable variants and generalizations of Theorems \ref{maintheorem:slstd} -- \ref{maintheorem:presentationsym} 
can be proved using our techniques.  While these theorems are proved via a common core proof technique, applying
this technique requires calculations that seem special to each theorem.  Theorems \ref{maintheorem:presentationalt} and \ref{maintheorem:presentationsym}
in particular require very elaborate calculations.  There should be a common generalization of all these
results:

\begin{question}
Does there exist a general abstract theorem that specializes to Theorems \ref{maintheorem:slstd}
-- \ref{maintheorem:presentationsym}, as well as their natural generalizations?
\end{question}

\subsection{Notation and conventions}
\label{section:notation}

Throughout this paper, we will let $H = \Q^{2g}$ and $H_{\Z} = \Z^{2g}$, and also
let $\omega\colon H \times H \rightarrow \Q$ be the standard symplectic form on
$H$.

\subsection{Outline}
We prove Theorems \ref{maintheorem:slstd} -- \ref{maintheorem:spsym} in Part \ref{part:1}.
Next, in Part \ref{part:2} we expand the presentations for $\fK_g^a$ and $\fK_g^s$ to ones with a larger
set of generators.  We remark that this expansion uses the same proof technique
as our other theorems.  We use this expanded presentation to prove Theorems \ref{maintheorem:presentationalt} -- \ref{maintheorem:presentationsym}
in Parts \ref{part:alt} and \ref{part:sym}.  We close with Appendix \ref{appendix:presentation}, which
adjusts the presentations given in Theorem \ref{maintheorem:presentationalt} -- \ref{maintheorem:presentationsym}
to make them match up better with our companion papers \cite{MinahanPutmanAbelian, MinahanPutmanH2} on
the Torelli group.

\part{Five easy examples}
\label{part:1}

In this part of the paper, we prove Theorems \ref{maintheorem:slstd} -- \ref{maintheorem:spsym}.
The proof of Theorem \ref{maintheorem:slstd} is in \S \ref{section:slstd}.  We then extract from that proof
an outline of our proof method in \S \ref{section:prooftechnique}.  We then use this proof technique
to prove Theorem \ref{maintheorem:sladjoint} in \S \ref{section:sladjoint}.  This is followed by 
\S \ref{section:slstdvar} and \S \ref{section:sladjointvar}, which prove variants of Theorems \ref{maintheorem:slstd} and \ref{maintheorem:sladjoint}
that will be needed for our work on symmetric kernel.
After the preliminary
\S \ref{section:spgen} on generators for $\Sp_{2g}(\Z)$, we then prove
Theorems \ref{maintheorem:spstd} -- \ref{maintheorem:spsym} in \S \ref{section:spstd} -- \S \ref{section:spkernel}.

\section{Special linear group I: standard representation}
\label{section:slstd}

Recall that $\fQ_n$ is the $\Q$-vector space with the following presentation:
\begin{itemize}
\item {\bf Generators}. A generator $\QGen{v}$ for all primitive vectors $v \in \Z^n$.
\item {\bf Relations}.  For a partial basis $\{v_1,v_2\}$, the relation
$\QGen{v_1}+\QGen{v_2} = \QGen{v_1+v_2}$.
\end{itemize}
Define a linearization map $\Phi\colon \fQ_n \rightarrow \Q^n$ via the formula
$\Phi(\QGen{v}) = v$.  This takes relations to relations, and thus gives a well-defined
map.  Our goal is to prove:

\newtheorem*{maintheorem:slstd}{Theorem \ref{maintheorem:slstd}}
\begin{maintheorem:slstd}
For $n \geq 2$, the linearization map $\Phi\colon \fQ_n \rightarrow \Q^n$ is an isomorphism.
\end{maintheorem:slstd}
\begin{proof}
Let $\cB = \{e_1,\ldots,e_n\}$ be the standard basis for $\Z^n$.
Set $S = \{\QGen{e_1},\ldots,\QGen{e_n}\}$.  The map $\Phi$ takes $S$ bijectively
to $\cB$.  This implies that the restriction of $\Phi$ to $\Span{S}$ is an isomorphism.
To prove that $\Phi$ is an isomorphism, we must prove that $\Span{S} = \fQ_n$.

The group $\SL_n(\Z)$ acts on $\fQ_n$.  Since $\SL_n(\Z)$ acts transitively on primitive vectors,\footnote{This
is where we use the fact that $n \geq 2$.}
it acts transitively on the generators for $\fQ_n$.  It follows that the $\SL_n(\Z)$-orbit
of $S$ spans $\fQ_n$.  To prove that $\Span{S} = \fQ_n$, it is therefore enough to prove
that $\SL_n(\Z)$ takes $\Span{S}$ to itself.

For distinct $1 \leq i,j \leq n$, let $E_{ij} \in \SL_n(\Z)$ be the elementary matrix
obtained from the identity by placing a $1$ at position $(i,j)$.  These 
generate
$\SL_n(\Z)$.  Fixing some distinct $1 \leq i,j \leq n$ and some $\epsilon = \pm 1$, 
it is enough to prove that $E_{ij}^{\epsilon}$ takes $\Span{S}$ to itself.
Consider some $\QGen{e_k} \in S$.  We must prove that $\QGen{E_{ij}^{\epsilon}(e_k)}$
can be written as a linear combination of elements of $S$.  If $k \neq j$, then
$E_{ij}^{\epsilon}(e_k) = e_k$ and there is nothing to prove.  If $k=j$, there are two cases:
\begin{itemize}
\item $\epsilon = 1$.  In this case, $\QGen{E_{ij}(e_j)} = \QGen{e_j + e_i} = \QGen{e_j} + \QGen{e_i} \in \Span{S}$.
\item $\epsilon = -1$.  In this case, $\QGen{E_{ij}^{-1}(e_j)} = \QGen{e_j-e_i}$.  We would like to prove
that this equals $\QGen{e_j} - \QGen{e_i} \in \Span{S}$.  
For this, since $\{e_j-e_i,e_i\}$ is a partial basis we have
\[\QGen{e_j-e_i} + \QGen{e_i}= \QGen{(e_j-e_i) + e_i} = \QGen{e_j}.\qedhere\]
\end{itemize}
\end{proof}

\begin{remark}
\label{remark:negative}
Let $n \geq 2$ and let $v \in \Z^n$ be a primitive vector.  The above implies that $\QGen{-v} = - \QGen{v}$.  Here is how
to prove directly that this holds.  Pick $w \in \Z^n$ such that $\{v,w\}$ is a partial
basis for $\Z^n$.  For $a,b,c,d \in \Z$, the pair
$\{av+bw,cv+dw\}$ forms a partial basis for $\Z^n$ precisely when
\[\det\left(\begin{matrix} a & b \\ c & d \end{matrix}\right) = \pm 1.\]
Using this, we prove that $\QGen{v}+\QGen{-v} = 0$ as follows:
\begin{align*}
\QGen{v} + \QGen{-v} &= \QGen{v} + \left(\QGen{v+w} - \QGen{v+w}\right) + \QGen{-v} + \left(\QGen{-w}-\QGen{-w}\right) \\
                     &= \left(\QGen{v}+\QGen{v+w}\right) + \left(\QGen{-v}+\QGen{-w}\right) - \left(\QGen{v+w}+\QGen{-w}\right) \\
                     &= \QGen{2v+w} + \QGen{-v-w} - \QGen{v} \\
                     &= \QGen{v} - \QGen{v} = 0.\qedhere
\end{align*}
\end{remark}

\section{Outline of proof technique}
\label{section:prooftechnique}

We now abstract a general proof technique from the proof of Theorem \ref{maintheorem:slstd}.
Let $G$ be a group and let $\cV$ be a representation of $G$ that we understand well.  Let
$\fV$ be a representation of $G$ given by generators and relations that we suspect
is isomorphic to $\cV$ and let $\Phi\colon \fV \rightarrow \cV$ be a $G$-equivariant map.  The
following steps will prove that $\Phi$ is an isomorphism:

\begin{step}{1}
Construct a subset $S$ of $\cV$ such that the restriction of $\Phi$ to $\Span{S}$ is
an isomorphism.
\end{step}

One way for this to hold is for $\Phi$ to take $S$ bijectively to a basis for $\cV$.
However, sometimes it is more natural to use a larger $S$ whose image is a generating
set satisfying some relations.

\begin{step}{2}
Prove that the $G$-orbit of $S$ spans $\fV$.  
\end{step}

Since $\fV$ is given by generators and relations,
this is done by making sure that this $G$-orbit contains all the generators.

\begin{step}{3}
Prove that $G$ takes $\Span{S}$ to itself.  By Step 2, this will imply that $\Span{S} = \fV$,
and thus by Step 1 that $\Phi$ is an isomorphism.  
\end{step}

We do this as follows.  Let $\Lambda$ be a generating set for $G$.  Then
it is enough to check that for $f \in \Lambda$ and $s \in S$ the elements
$f(x) \in \fV$ and $f^{-1}(x) \in \fV$ can be written as linear combinations of elements of $S$.
When we proved Theorem \ref{maintheorem:slstd} this step only required the easy identities
\begin{align*}
\QGen{e_1+e_2} &= \QGen{e_1} + \QGen{e_2},\\
\QGen{e_1-e_2} &= \QGen{e_1} - \QGen{e_2}.
\end{align*}
However, for our other theorems this will be the most calculation heavy step, and
the key will be verifying that a large number of explicit elements of $\fV$ lie in $\Span{S}$.

\section{Special linear group II: adjoint representation}
\label{section:sladjoint}

Recall that the adjoint representation of $\SL_n(\Q)$ is the kernel $\fsl_n(\Q)$
of the trace map
\[\trace\colon (\Q^n)^{\ast} \otimes \Q^n \longrightarrow \Q\]
defined by $\trace(f,v) = f(v)$.  Also, recall that $\fA_n$ is the $\Q$-vector
space with the following presentation:
\begin{itemize}
\item {\bf Generators}.  A generator $\AGen{f,v}$ for all primitive vectors
$f \in (\Z^n)^{\ast}$ and $v \in \Z^n$ such that $f(v) = 0$.
\item {\bf Relations}.  The following two families of relations:
\begin{itemize}
\item For all primitive vectors $f \in (\Z^n)^{\ast}$ and
all partial bases $\{v_1,v_2\}$ of $\ker(f)$, the relation
$\AGen{f,v_1+v_2} = \AGen{f,v_1} + \AGen{f,v_2}$.
\item For all primitive vectors $v \in \Z^n$ and all partial bases
$\{f_1,f_2\}$ of
\[\ker(v) = \Set{$f \in (\Z^n)^{\ast}$}{$f(v) = 0$},\]
the relation $\AGen{f_1+f_2,v} = \AGen{f_1,v}+\AGen{f_2,v}$.
\end{itemize}
\end{itemize}
Define $\Phi\colon \fA_n \rightarrow (\Q^n)^{\ast} \otimes \Q^n$ via the formula
$\Phi(\AGen{f,v}) = f \otimes v$.  This takes relations to relations, and thus
gives a well-defined linearization map with $\Image(\Phi) \subset \fsl_n(\Q)$.  Our goal is to prove: 

\newtheorem*{maintheorem:sladjoint}{Theorem \ref{maintheorem:sladjoint}}
\begin{maintheorem:sladjoint}
For $n \geq 3$, the linearization map $\Phi\colon \fA_n \rightarrow \fsl_n(\Q)$ is an isomorphism.
\end{maintheorem:sladjoint}
\begin{proof}
We start with the following, which we will use freely throughout the proof:

\begin{unnumberedclaim}
In $\fA_n$, the following relations hold for all $m \geq 1$:
\begin{itemize}
\item[(i)] For all primitive vectors $f \in (\Z^n)^{\ast}$ and
all primitive vectors $v_1,\ldots,v_m \in \ker(f)$ and all $\lambda_1,\ldots,\lambda_m \in \Z$
such that $\sum_{i=1}^m \lambda_i v_i$ is primitive, we have
\[\AGen{f,\sum_{i=1}^m \lambda_i v_i} = \sum_{i=1}^m \lambda_i \AGen{f,v_i}.\]
\item[(ii)] For all primitive vectors $v \in \Z^n$ and
all primitive vectors $f_1,\ldots,f_m \in \ker(v)$ and all $\lambda_1,\ldots,\lambda_m \in \Z$
such that $\sum_{i=1}^m \lambda_i f_i$ is primitive, we have 
\[\AGen{\sum_{i=1}^m \lambda_i f_i,v} = \sum_{i=1}^m \lambda_i \AGen{f_i,v}.\]
\end{itemize}
\end{unnumberedclaim}
\begin{proof}
Both are proved the same way, so we will give the details for (i).  Let $f \in (\Z^n)^{\ast}$
be primitive and let $v_1,\ldots,v_m \in \ker(f)$ and $\lambda_1,\ldots,\lambda_m \in \Z$ be as
in the claim.  Choose an isomorphism $\mu\colon \Z^{n-1} \rightarrow \ker(f)$, and let
$w_i = \mu^{-1}(v_i)$.  Recall that we defined $\fQ_{n-1}$ in Definition \ref{definition:slstd}.
Define a map $\psi\colon \fQ_{n-1} \rightarrow \fA_n$ via the formula
\[\psi(\QGen{x}) = \AGen{f,\mu(x)} \quad \text{for a primitive $x \in \Z^{n-1}$}.\]
This takes relations to relations, and thus gives a well-defined map.  Recall that
Theorem \ref{maintheorem:slstd} says that $\fQ_{n-1} \cong \Q^{n-1}$.  It follows from this
theorem that
\[\QGen{\sum_{i=1}^m \lambda_i w_i} = \sum_{i=1}^m \lambda_i \QGen{w_i}.\]
Plugging this into $\psi$, we see that
\[\AGen{f,\sum_{i=1}^m \lambda_i v_i} = \psi\left(\QGen{\sum_{i=1}^m \lambda_i w_i}\right) = \sum_{i=1}^m \lambda_i \psi\left(\QGen{w_i}\right) = \sum_{i=1}^m \lambda_i \AGen{f,v_i}.\qedhere\]
\end{proof}

Let $\cB = \{e_1,\ldots,e_n\}$ be the standard basis for $\Z^n$ and let
$\cB^{\ast} = \{e_1^{\ast},\ldots,e_n^{\ast}\}$ be the corresponding dual basis for $(\Z^n)^{\ast}$.
We follow the outline from \S \ref{section:prooftechnique}, though for readability
we divide Step 3 into Steps 3.A and 3.B.    

\begin{step}{1}
\label{step:sladjoint1}
Let $S = S_1 \cup S_2$, where the $S_i$ are:
\begin{align*}
S_1 &= \Set{$\BAGen{e_i^{\ast},e_j}$}{$1 \leq i,j \leq n$ distinct}, \\
S_2 &= \Set{$\BAGen{e_i^{\ast}+e_{i+1}^{\ast},e_i-e_{i+1}}$}{$1 \leq i < n$}.
\end{align*}
Like we did here, we will write elements of $S$ in blue.
Then the restriction of $\Phi$ to $\Span{S}$ is an isomorphism.
\end{step}

Let $T = T_1 \cup T_2$, where the $T_i$ are:
\begin{align*}
T_1 &= \Set{$e_i^{\ast} \otimes e_j$}{$1 \leq i,j \leq n$ distinct}, \\
T_2 &= \Set{$e_{i}^{\ast} \otimes e_{i} - e_{i+1}^{\ast} \otimes e_{i+1}$}{$1 \leq i < n$}.
\end{align*}
The set $T$ is a basis for the codimension-$1$ subspace $\fsl_n(\Q)$ of $(\Q^n)^{\ast} \otimes \Q^n$.
The map $\Phi$ takes $S_1$ bijectively to $T_1$.  As for $T_2$, observe that for $1 \leq i < n$ we have
\begin{align*}
\Phi(\BAGen{e_i^{\ast}+e_{i+1}^{\ast},e_i-e_{i+1}}) &= (e_i^{\ast}+e_{i+1}^{\ast}) \otimes (e_i - e_{i+1}) \\
                                      &= e_i^{\ast} \otimes e_i - e_{i+1}^{\ast} \otimes e_{i+1} - \orange{e_i^{\ast} \otimes e_{i+1}} + \orange{e_{i+1}^{\ast} \otimes e_i}.
\end{align*}
Here we have written elements of $T_1$ in orange.  This calculation implies that modulo $T_1$, the
map $\Phi$ takes $S_2$ bijectively to $T_2$.  Since $T$ is a basis for $\fsl_n(\Q)$, we deduce that
$\Phi$ takes $S$ bijectively to a basis for $\fsl_n(\Q)$.  This implies that the restriction
of $\Phi$ to $\Span{S}$ is an isomorphism.

\begin{step}{2}
\label{step:sladjoint2}
We prove that the $\SL_n(\Z)$-orbit of $S$ spans $\fA_n$.
\end{step}

Immediate from the fact that $\SL_n(\Z)$ acts transitively on the basis elements of $\fA_n$.

\begin{step}{3.A}
\label{step:sladjoint3A}
In preparation for proving that $\SL_n(\Z)$ takes $\Span{S}$ to $\Span{S}$, we prove
that all elements of\hspace{2pt}\footnote{The letter $E$ stands for ``extra elements''.}
\begin{align*}
E = &\Set{$\GAGen{e_i^{\ast}+e_j^{\ast},e_i-e_j}$}{$1 \leq i,j \leq n$ distinct} \\
    &\cup \Set{$\GAGen{e_i^{\ast}-e_j^{\ast},e_i+e_j}$}{$1 \leq i,j \leq n$ distinct} \\
    &\cup \Set{$\GAGen{e_i^{\ast}+2e_j^{\ast},2e_i-e_j}$}{$1 \leq i,j \leq n$ distinct}.
\end{align*}
lie in $\Span{S}$.  Like we did here, we will write elements of $E$ in green.
\end{step}

Above we wrote elements of $S_1$ in blue.  We extend this to certain elements that
``obviously'' lie in $\Span{S_1}$ as follows:
\begin{itemize}
\item Consider a generator $\AGen{f,v}$.  Assume there exist
$\cB^{\ast}_1 \subset \cB^{\ast}$ and $\cB_2 \subset \cB$ such that
$f \in \Span{\cB^{\ast}_1}$ and $v \in \Span{\cB_2}$ and such that
$g(w) = 0$ for all $g \in \cB^{\ast}_1$ and $w \in \cB_2$.  It
is then immediate that $\AGen{f,v} \in \Span{S_1}$.  This is most easily
seen by example:
\begin{align*}
\AGen{7e_1^{\ast}+3e_3^{\ast},2e_2-5e_4} &= 2\AGen{7e_1^{\ast}+3e_3^{\ast},e_2}-5\AGen{7e_1^{\ast}+3e_3^{\ast},e_4} \\
                                         &= 14\BAGen{e_1^{\ast},e_2} + 21\BAGen{e_3^{\ast},e_1} -35\BAGen{e_1^{\ast},e_4}-15\BAGen{e_3^{\ast},e_4}.
\end{align*}
Previously we were only writing elements of $S_1$ in blue, but now we will write these
elements in blue as well.  For instance, we will write
$\BAGen{7e_1^{\ast}+3e_3^{\ast},2e_2-5e_4}$.
\end{itemize}
Let $\equiv$ denote equality modulo $\Span{S}$.  During the proof, we will underline
elements of $E$ that we have not yet proven lie in $\Span{S}$.  We divide the proof into three claims.

\begin{claim}{3.A.1}
\label{claim:sladjoint3a1}
For distinct $1 \leq i,j \leq n$, we have $\UGAGen{e_i^{\ast}+e_j^{\ast},e_i-e_j} \equiv \UGAGen{e_i^{\ast}-e_j^{\ast},e_i+e_j}$.
\end{claim}

\noindent
Since $n \geq 3$, we can pick some $1 \leq k \leq n$ that is distinct from $i$ and $j$.
For $\epsilon \in \{\pm 1\}$, let $x_{\epsilon} = \UGAGen{e_i^{\ast}+\epsilon e_k^{\ast},e_i-\epsilon e_k}$.  For $c \in \{\pm 1\}$, we have
\begin{align*}
x_{\epsilon} \equiv &\UGAGen{e_i^{\ast}+\epsilon e_k^{\ast},e_i-\epsilon e_k} +\BAGen{c e_j^{\ast},e_i-\epsilon e_k} = \AGen{e_i^{\ast}+c e_j^{\ast}+\epsilon e_k^{\ast},e_i-\epsilon e_k} \\
    =          &\AGen{e_i^{\ast}+c e_j^{\ast}+\epsilon e_k^{\ast},(e_i-c e_j)+(c e_j-\epsilon e_k)} \\
    =      &\AGen{e_i^{\ast}+c e_j^{\ast}+\epsilon e_k^{\ast},e_i-c e_j} + \AGen{e_i^{\ast}+c e_j^{\ast}+\epsilon e_k^{\ast},c e_j-\epsilon e_k} \\
    =      &\UGAGen{e_i^{\ast}+c e_j^{\ast},e_i-c e_j} + \BAGen{\epsilon e_k^{\ast},e_i-c e_j}
            +\BAGen{e_i^{\ast},c e_j-\epsilon e_k} +c \AGen{e_j^{\ast}+\epsilon c e_k^{\ast},c e_j-\epsilon e_k} \\
    \equiv &\UGAGen{e_i^{\ast}+c e_j^{\ast},e_i-c e_j} + \UGAGen{e_j^{\ast}+\epsilon c e_k^{\ast},e_j-\epsilon c e_k}.
\end{align*}
We can equate these expressions for $x_1$ for $c = 1$ and $c = -1$ to get
\begin{equation}
\label{eqn:sladjoint3a1.1}
\UGAGen{e_i^{\ast}+e_j^{\ast},e_i-e_j} + \UGAGen{e_j^{\ast}+e_k^{\ast},e_j-e_k} \equiv \UGAGen{e_i^{\ast}-e_j^{\ast},e_i+e_j} + \UGAGen{e_j^{\ast}-e_k^{\ast},e_j+e_k}.
\end{equation}
Similarly, equating these expressions for $x_{-1}$ for $c=1$ and $c=-1$ we get
\begin{equation}
\label{eqn:sladjoint3a1.2}
\UGAGen{e_i^{\ast}+e_j^{\ast},e_i-e_j} + \UGAGen{e_j^{\ast}-e_k^{\ast},e_j+e_k} \equiv \UGAGen{e_i^{\ast}-e_j^{\ast},e_i+e_j} + \UGAGen{e_j^{\ast}+e_k^{\ast},e_j-e_k}.
\end{equation}
Combining \eqref{eqn:sladjoint3a1.1} and \eqref{eqn:sladjoint3a1.2}, we conclude that $\GAGen{e_i^{\ast}+e_j^{\ast},e_i-e_j} \equiv \GAGen{e_i^{\ast}-e_j^{\ast},e_i+e_j}$.

\begin{claim}{3.A.2}
\label{claim:sladjoint3a2}
For distinct $1 \leq i,j \leq n$, we have $\UGAGen{e_i^{\ast}+e_j^{\ast},e_i-e_j} \in \Span{S}$.  In light of
Claim \ref{claim:sladjoint3a1}, this will imply that $\UGAGen{e_i^{\ast}-e_j^{\ast},e_i+e_j} \in \Span{S}$ as well.  
\end{claim}

\noindent
Swapping $i$ and $j$ multiplies $\UGAGen{e_i^{\ast}+e_j^{\ast},e_i-e_j}$ by $-1$, so we can assume without loss
of generality that $i < j$.  The proof will be by induction on $j-i$.  The base case is when $j-i=1$, in which
case $\UGAGen{e_i^{\ast}+e_j^{\ast},e_i-e_j}$ lies in $S$ and there is nothing to prove.  Assume, therefore, that
$j-i>1$ and that the claim is true whenever $j-i$ is smaller.  Pick $k$ with $i < k < j$.  The element
$\AGen{e_i^{\ast}+e_{k}^{\ast}+e_j^{\ast},e_i-e_j}$ equals
\[\UGAGen{e_i^{\ast}+e_j^{\ast},e_i-e_j} + \BAGen{e_{k}^{\ast},e_i-e_j} \equiv \UGAGen{e_i^{\ast}+e_j^{\ast},e_i-e_j}.\]
On the other hand, it also equals
\begin{align*}
&\AGen{e_i^{\ast}+e_{k}^{\ast}+e_j^{\ast},(e_i-e_{k}) + (e_{k}-e_j)} \\
&\quad = \AGen{e_i^{\ast}+e_{k}^{\ast}+e_j^{\ast},e_i-e_{k}} + \AGen{e_i^{\ast}+e_{k}^{\ast}+e_j^{\ast},e_{k}-e_j} \\
&\quad = \GAGen{e_i^{\ast}+e_{k}^{\ast},e_i-e_{k}} + \BAGen{e_j^{\ast},e_i-e_{k}} 
+ \GAGen{e_{k}^{\ast}+e_j^{\ast},e_{k}-e_j} + \BAGen{e_i^{\ast},e_{k}-e_j} \equiv 0.
\end{align*}
Here the non-underlined green terms lie in $\Span{S}$ by our inductive hypothesis.
Combining these, we conclude that $\UGAGen{e_i^{\ast}+e_j^{\ast},e_i-e_j} \equiv 0$.

\begin{claim}{3.A.3}
\label{claim:sladjoint3a3}
For distinct $1 \leq i,j \leq n$, we have $\UGAGen{e_i^{\ast}+2e_j^{\ast},2e_i-e_j} \in \Span{S}$.  
\end{claim}

\noindent
Since $n \geq 3$, we can pick some $1 \leq k \leq n$ that is distinct from $i$ and $j$.
The element $\AGen{e_i^{\ast}+2e_j^{\ast}+e_k^{\ast},2e_i-e_j}$ equals
\[\UGAGen{e_i^{\ast}+2e_j^{\ast},2e_i-e_j} + \BAGen{e_k^{\ast},2e_i-e_j} \equiv \UGAGen{e_i^{\ast}+2e_j^{\ast},2e_i-e_j}.\]
On the other hand, it also equals
\begin{align*}
&\AGen{e_i^{\ast}+2e_j^{\ast}+e_k^{\ast},(e_i-e_k)+(e_i+e_k-e_j)} \\
&\quad = \AGen{e_i^{\ast}+2e_j^{\ast}+e_k^{\ast},e_i-e_k} + \AGen{e_i^{\ast}+2e_j^{\ast}+e_k^{\ast},e_i+e_k-e_j} \\
&\quad = \GAGen{e_i^{\ast}+e_k^{\ast},e_i-e_k} + 2 \BAGen{e_j^{\ast},e_i-e_k} 
+ \AGen{e_i^{\ast}+e_j^{\ast},e_i+e_k-e_j} + \AGen{e_j^{\ast}+e_k^{\ast},e_i+e_k-e_j} \\
&\quad \equiv \GAGen{e_i^{\ast}+e_j^{\ast},e_i-e_j} + \BAGen{e_i^{\ast}+e_j^{\ast},e_k} 
+ \GAGen{e_j^{\ast}+e_k^{\ast},e_k-e_j} + \BAGen{e_j^{\ast}+e_k^{\ast},e_i} \equiv 0.
\end{align*}
Combining these, we conclude that $\UGAGen{e_i^{\ast}+2e_j^{\ast},2e_i-e_j} \equiv 0$.

\begin{step}{3.B}
\label{step:sladjoint3B}
We prove that $\SL_n(\Z)$ takes $\Span{S}$ to itself.
By Step \ref{step:sladjoint2} this will imply that $\Span{S} = \fA_n$,
and thus by Step \ref{step:sladjoint1} that $\Phi$ is an isomorphism.
\end{step}

For distinct $1 \leq i,j \leq n$, let $E_{ij} \in \SL_n(\Z)$ be the elementary matrix
obtained from the identity by placing a $1$ at position $(i,j)$.  These
generate
$\SL_n(\Z)$.  Fixing some distinct $1 \leq i,j \leq n$ and some $\epsilon = \pm 1$,
it is enough to prove that $E_{ij}^{\epsilon}$ takes $\Span{S}$ to itself.  To do this,
we must prove that for all $s \in S$ the image $E_{ij}^{\epsilon}(s)$ can be written
as a linear combination of elements of $S$.
The matrix $E_{ij}$ satisfies
\begin{alignat*}{6}
&&E_{ij}(e_i^{\ast})      &= &e_i^{\ast} - e_j^{\ast} \quad & E_{ij}(e_j)        &= &e_j + e_i, \\
&&E_{ij}^{-1}(e_i^{\ast}) &= &e_i^{\ast} + e_j^{\ast} \quad & E_{ij}^{\ast}(e_j) &= &e_j - e_i.
\end{alignat*}
It fixes all other elements of $\cB$ and $\cB^{\ast}$.  Consider $s \in S$.  If
$E_{ij}^{\epsilon}(s) = s$, there is nothing to prove.  We can therefore assume
that $E_{ij}{\epsilon}(s) \neq s$.  There are two cases.

\begin{case}{3.B.1}
$s \in S_1 = \Set{$\BAGen{e_i^{\ast},e_j}$}{$1 \leq i,j \leq n$ distinct}$.
\end{case}

\noindent
The matrix $E_{ij}^{\epsilon}$
fixes all elements of $S_1$
except for the following:
\begin{itemize}
\item $\BAGen{e_i^{\ast},e_k}$ with $k \neq i,j$.  For these, we have
\begin{align*}
E_{ij}(\BAGen{e_i^{\ast},e_k})      &= \AGen{e_i^{\ast}-e_j^{\ast},e_k} = \BAGen{e_i^{\ast},e_k} - \BAGen{e_j^{\ast},e_k} \in \Span{S_1}, \\
E_{ij}^{-1}(\BAGen{e_i^{\ast},e_k}) &= \AGen{e_i^{\ast}+e_j^{\ast},e_k} = \BAGen{e_i^{\ast},e_k} + \BAGen{e_j^{\ast},e_k} \in \Span{S_1}.
\end{align*}
\item $\BAGen{e_k^{\ast},e_j}$ with $k \neq i,j$.  For these, we have
\begin{align*}
E_{ij}(\BAGen{e_k^{\ast},e_j})      &= \AGen{e_k^{\ast},e_j+e_i} = \BAGen{e_k^{\ast},e_j} + \BAGen{e_k^{\ast},e_i} \in \Span{S_1}, \\
E_{ij}^{-1}(\BAGen{e_k^{\ast},e_j}) &= \AGen{e_k^{\ast},e_j-e_i} = \BAGen{e_k^{\ast},e_j} - \BAGen{e_k^{\ast},e_i} \in \Span{S_1}.
\end{align*}
\item $\BAGen{e_j^{\ast},e_i}$.  For this, we have
\begin{align*}
E_{ij}(\BAGen{e_j^{\ast},e_i})      &= \GAGen{e_j^{\ast}-e_i^{\ast},e_i+e_j} \in E \subset \Span{S}, \\
E_{ij}^{-1}(\BAGen{e_j^{\ast},e_i}) &= \GAGen{e_j^{\ast}+e_i^{\ast},e_i-e_j} \in E \subset \Span{S}.
\end{align*}
\end{itemize}

\begin{case}{3.B.2}
$s \in S_2 = \Set{$\BAGen{e_i^{\ast}+e_{i+1}^{\ast},e_i-e_{i+1}}$}{$1 \leq i < n$}$.
\end{case}

\noindent
To decrease the number of special cases, we will deal more generally with elements of the form
$\GAGen{e_a^{\ast}+e_b^{\ast},e_a-e_b}$ for distinct $1 \leq a,b \leq n$.  The matrix $E_{ij}^{\epsilon}$
fixes these except when:
\begin{itemize}
\item $(a,b) = (i,k)$ or $(a,b) = (k,i)$ for some $1 \leq k \leq n$ with $k \neq i,j$.  Swapping $a$ and $b$ multiplies
$\GAGen{e_a^{\ast}+e_b^{\ast},e_a-e_b}$ by a sign, so it is enough to deal with $(a,b) = (i,k)$:
\begin{align*}
E_{ij}(\GAGen{e_i^{\ast}+e_k^{\ast},e_i-e_k})      &= \AGen{e_i^{\ast}-e_j^{\ast}+e_k^{\ast},e_i-e_k} = \GAGen{e_i^{\ast}+e_k^{\ast},e_i-e_k} - \BAGen{e_j^{\ast},e_i-e_k} \in \Span{S},\\
E_{ij}^{-1}(\GAGen{e_i^{\ast}+e_k^{\ast},e_i-e_k}) &= \AGen{e_i^{\ast}+e_j^{\ast}+e_k^{\ast},e_i-e_k} = \GAGen{e_i^{\ast}+e_k^{\ast},e_i-e_k} + \BAGen{e_j^{\ast},e_i-e_k} \in \Span{S}.
\end{align*}
\item $(a,b) = (j,k)$ or $(a,b) = (k,j)$ for some $1 \leq k \leq n$ with $k \neq i,j$.  Again, it is
enough to deal with $(a,b) = (j,k)$:
\begin{align*}
E_{ij}(\GAGen{e_j^{\ast}+e_k^{\ast},e_j-e_k})      &= \GAGen{e_j^{\ast}+e_k^{\ast},e_j+e_i-e_k} = \GAGen{e_j^{\ast}+e_k^{\ast},e_j-e_k} + \GAGen{e_j^{\ast}+e_k^{\ast},e_i} \in \Span{S},\\
E_{ij}^{-1}(\GAGen{e_j^{\ast}+e_k^{\ast},e_j-e_k}) &= \GAGen{e_j^{\ast}+e_k^{\ast},e_j-e_i-e_k} = \GAGen{e_j^{\ast}+e_k^{\ast},e_j-e_k} - \GAGen{e_j^{\ast}+e_k^{\ast},e_i} \in \Span{S}.
\end{align*}
\item $(a,b) = (i,j)$ or $(a,b) = (j,i)$.  Again it is enough to deal with $(a,b) = (i,j)$:
\begin{align*}
E_{ij}(\GAGen{e_i^{\ast}+e_j^{\ast},e_i-e_j})      &= \GAGen{e_i^{\ast}+2e_j^{\ast},2e_i-e_j} \in E,\\
E_{ij}^{-1}(\GAGen{e_i^{\ast}+e_j^{\ast},e_i-e_j}) &= \BAGen{e_i^{\ast},-e_j} = \in \Span{S_1}.\qedhere
\end{align*}
\end{itemize}
\end{proof}

\section{Special linear group I\texorpdfstring{$'$}{'}: variant presentation of standard representation}
\label{section:slstdvar}

Before proceeding with proofs of our main results, this section and the next give alternate presentations of the standard and adjoint representations
of $\SL_n(\Q)$ that are needed in Part \ref{part:alt} for our work on the symmetric kernel.  We
start with the standard representation.  Let $\cB = \{e_1,\ldots,e_n\}$ be the standard basis for
$\Z^n$.  Say that $v \in \Z^n$ is {\em $e_i$-standard} (resp.\ {\em $e_i$-vanishing}) if the $e_i$-coordinate of $v$ lies in $\{-1,0,1\}$ (resp.\ is $0$).
If $v_1 \in \Z^n$ is $e_i$-standard and $v_2 \in \Z^n$ is $e_i$-vanishing, then $v_1+v_2$ is $e_i$-standard.
Define the following:

\begin{definition}
\label{definition:slstdvar}
Define $\fQ'_n$ to be the $\Q$-vector space with the following presentation:
\begin{itemize}
\item {\bf Generators}. A generator $\PQGen{v}$ for all primitive vectors $v \in \Z^n$ that are $e_1$-standard.  
\item {\bf Relations}.  For a partial basis $\{v_1,v_2\}$ of $\Z^n$ such that $v_1$ is $e_1$-standard and $e_2$ is
$e_1$-vanishing, the relation $\PQGen{v_1}+\PQGen{v_2} = \PQGen{v_1+v_2}$.\qedhere
\end{itemize}
\end{definition}

Define a linearization map $\Phi\colon \fQ'_n \rightarrow \Q^n$ via the formula
$\Phi(\PQGen{v}) = v$.  This takes relations to relations, and thus gives a well-defined
map.  Our goal is to prove: 

\setcounter{maintheoremprime}{0}
\begin{maintheoremprime}
\label{maintheorem:slstdvar}
For $n \geq 3$, the linearization map $\Phi\colon \fQ'_n \rightarrow \Q^n$ is an isomorphism.
\end{maintheoremprime}
\begin{proof}
Recall that $\cB = \{e_1,\ldots,e_n\}$ is the standard basis for $\Z^n$.  Let
$\cB^{\ast} = \{e_1^{\ast},\ldots,e_n^{\ast}\}$ be the corresponding dual basis.  
Let $\SL_n(\Z,e_1^{\ast})$ be the stabilizer of $e_1^{\ast}$ in $\SL_n(\Z)$.  The action of $\SL_n(\Z,e_1^{\ast})$ on $\Z^n$ takes
$e_1$-standard vectors to $e_1$-standard vectors.  It follows that $\SL_n(\Z,e_1^{\ast})$ acts
on $\fQ'_n$, and we will use this action to prove our theorem.
We follow the outline from \S \ref{section:prooftechnique}.

\begin{step}{1}
\label{step:slstdvar1}
Let $S = \{\PQGen{e_1},\ldots,\PQGen{e_n}\}$.  Like we did here,
we will write elements of $S$ in blue.  
Then the restriction of $\Phi$ to $\Span{S}$ is an isomorphism.
\end{step}

Immediate from the fact that $\Phi$ takes $S$ to a basis for $\Q^n$.

\begin{step}{2}
\label{step:slstdvar2}
We prove that the $\SL_n(\Z,e_1^{\ast})$-orbit of $S$ spans $\fQ'_n$.
\end{step}

Let $W$ be the span of the $\SL_n(\Z,e_1^{\ast})$-orbit of $S$.
The action of $\SL_n(\Z,e_1^{\ast})$ on the set of generators for $\fQ'_n$
has three orbits:\footnote{For instance, to see that $\SL_n(\Z,e_1^{\ast})$ acts transitively on $\cO_1$, consider
some primitive $v \in \Z^n$ with $e_1$-coordinate $1$.  We must find $M \in \SL_n(\Z,e_1^{\ast})$ with
$M(\PQGen{e_1}) = \PQGen{v}$.  Since $v$ is primitive, we can find a basis
$\{v_1,\ldots,v_n\}$ for $\Z^n$ with $v_1 = v$.  Adding multiples of $v_1$ to the other $v_i$, we can
assume that the $e_1$-coordinate of $v_i$ is $0$ for $2 \leq i \leq n$.  Let $M \in \GL_n(\Z)$ be the matrix
whose columns are $\{v_1,\ldots,v_n\}$.  
Multiplying $v_2$ by $-1$ if necessary, we can assume that $\det(M) = 1$.  Since the $e_1$-coordinate
of $v_1$ is $1$ and the $e_1$-coordinate of $v_i$ is $0$ for $2 \leq i \leq n$, we have $M \in \SL_n(\Z,e_1^{\ast})$.
By construction, $M(\PQGen{e_1}) = \PQGen{v_1} = \PQGen{v}$, as desired.}
\begin{align*}
\cO_{-1} &= \Set{$\PQGen{v}$}{$v \in \Z^n$ is primitive and the $e_1$-coordinate of $v$ is $-1$},\\
\cO_{0}  &= \Set{$\PQGen{v}$}{$v \in \Z^n$ is primitive and the $e_1$-coordinate of $v$ is $0$},\\
\cO_{1}  &= \Set{$\PQGen{v}$}{$v \in \Z^n$ is primitive and the $e_1$-coordinate of $v$ is $1$}.
\end{align*}
It is enough to prove that $W$ contains elements from all three of these orbits.  We have
$\PQGen{e_1} \in \cO_1$ and $\PQGen{e_2} \in \cO_0$, so the only nontrivial case is $\cO_{-1}$.
Since $W$ contains $\PQGen{e_1} \in \cO_1$ and $\PQGen{e_2} \in \cO_0$, it follows that $W$ contains
all elements of $\cO_0$ and $\cO_1$.  In $\fQ'_n$, we have the relation
\[\PQGen{e_1 + e_2} + \PQGen{-e_1} = \PQGen{e_2}.\]
Since $W$ contains $\PQGen{e_1+e_2} \in \cO_1$ and $\PQGen{e_2} \in \cO_0$, it also contains
$\PQGen{-e_1} \in \cO_{-1}$.  The step follows.

\begin{step}{3}
\label{step:slstdvar3}
We prove that $\SL_n(\Z,e_1^{\ast})$ takes $\Span{S}$ to itself.
By Step \ref{step:slstdvar2} this will imply that $\Span{S} = \fQ'_n$,
and thus by Step \ref{step:slstdvar1} that $\Phi$ is an isomorphism.
\end{step}

For distinct $1 \leq i,j \leq n$, let $E_{ij} \in \SL_n(\Z)$ be the elementary matrix
obtained from the identity by placing a $1$ at position $(i,j)$.  The matrix $E_{ij}$ lies
in $\SL_n(\Z,e_1^{\ast})$ if $i \neq 1$, and the set
\[\Lambda = \Set{$E_{ij}$}{$1 \leq i,j \leq n$ distinct, $i \neq 1$}\]
generates $\SL_n(\Z,e_1^{\ast})$.  Fixing some $E_{ij} \in \Lambda$ and some $\epsilon = \pm 1$,
it is enough to prove that $E_{ij}^{\epsilon}$ takes $\Span{S}$ to itself.
Consider some $\PQGen{e_k} \in S$.  We must prove that $\PQGen{E_{ij}^{\epsilon}(e_k)}$
can be written as a linear combination of elements of $S$.  If $k \neq j$, then
$E_{ij}^{\epsilon}(e_k) = e_k$ and there is nothing to prove.  If $k=j$, there are two cases:
\begin{itemize}
\item $\epsilon = 1$.  In this case, since $i \neq 1$ it follows that $e_i$ is $e_1$-vanishing, so
$\PQGen{E_{ij}(e_j)} = \PQGen{e_j + e_i} = \PQGen{e_j} + \PQGen{e_i} \in \Span{S}$.
\item $\epsilon = -1$.  In this case, $\PQGen{E_{ij}^{-1}(e_j)} = \PQGen{e_j-e_i}$.  We would like to prove
that this equals $\QGen{e_j} - \QGen{e_i} \in \Span{S}$.  For this, since $i \neq 1$ the set $\{e_j-e_i,e_i\}$ is a partial basis 
such that $e_j - e_i$ is $e_1$-standard and $e_i$ is $e_1$-vanishing.  We have 
\[\QGen{e_j-e_i} + \QGen{e_i}= \QGen{(e_j-e_i) + e_i} = \QGen{e_j}.\qedhere\]
\end{itemize}
\end{proof}

\section{Special linear group II\texorpdfstring{$'$}{'}: variant presentation of adjoint representation}
\label{section:sladjointvar}

We now give a variant of our presentation for the adjoint representation.  
Let $\cB = \{e_1,\ldots,e_n\}$ be the standard basis for $\Z^n$ and let
$\cB^{\ast} = \{e_1^{\ast},\ldots,e_n^{\ast}\}$ be the corresponding dual basis.
As in \S \ref{section:slstdvar}, an element $f \in (\Z^n)^{\ast}$ is {\em $e_i^{\ast}$-standard} (resp.\ $e_i^{\ast}$-vanishing)
if the $e_i^{\ast}$-coordinate of $f$ lies in $\{-1,0,1\}$ (resp.\ is $0$).
We define $v \in \Z^n$ being $e_i$-standard and $e_i$-vanishing similarly.  Define:

\begin{definition}
\label{definition:sladjointvar}
Define $\fA'_n$ to be the $\Q$-vector space with the following presentation:
\begin{itemize}
\item {\bf Generators}.  A generator $\PAGen{f,v}$ for all primitive vectors
$f \in (\Z^n)^{\ast}$ and $v \in \Z^n$ such that $f(v) = 0$ and such that
$f$ is $e_1^{\ast}$-standard and $v$ is $e_2$-standard.
\item {\bf Relations}.  The following two families of relations:
\begin{itemize}
\item For all $e_1^{\ast}$-primitive vectors $f \in (\Z^n)^{\ast}$ and 
all partial bases $\{v_1,v_2\}$ of $\ker(f)$ such that $v_1$ is $e_2$-standard and $v_2$ is $e_2$-vanishing, the relation
$\PAGen{f,v_1+v_2} = \PAGen{f,v_1} + \PAGen{f,v_2}$.
\item For all $e_2$-standard primitive vectors $v \in \Z^n$ and all partial bases $\{f_1,f_2\}$ of
\[\ker(v) = \Set{$f \in (\Z^n)^{\ast}$}{$f(v) = 0$},\]
such that $f_1$ is $e_1^{\ast}$-standard and $f_2$ is $e_1^{\ast}$-vanishing,
the relation $\AGen{f_1+f_2,v} = \AGen{f_1,v}+\AGen{f_2,v}$.\qedhere
\end{itemize}
\end{itemize}
\end{definition}

Define $\Phi\colon \fA'_n \rightarrow (\Q^n)^{\ast} \otimes \Q^n$ via the formula
$\Phi(\PAGen{f,v}) = f \otimes v$.  This takes relations to relations, and thus
gives a well-defined linearization map with $\Image(\Phi) \subset \fsl_n(\Q)$.  Our goal is to prove: 

\setcounter{maintheoremprime}{1}
\begin{maintheoremprime}
\label{maintheorem:sladjointvar}
For $n \geq 4$, the linearization map $\Phi\colon \fA'_n \rightarrow \fsl_n(\Q)$ is an isomorphism.
\end{maintheoremprime}
\begin{proof}
We start with the following, which we will use freely throughout the proof:

\begin{unnumberedclaim}
In $\fA'_n$, the following relations hold for all $m \geq 1$:
\begin{itemize}
\item[(i)] For all $e_1^{\ast}$-standard primitive vectors $f \in (\Z^n)^{\ast}$ 
and all $e_2$-standard primitive vectors $v_1,\ldots,v_m \in \ker(f)$ and all $\lambda_1,\ldots,\lambda_m \in \Z$
such that $\sum_{i=1}^m \lambda_i v_i$ is primitive and $e_2$-standard, we have
\[\PAGen{f,\sum_{i=1}^m \lambda_i v_i} = \sum_{i=1}^m \lambda_i \PAGen{f,v_i}.\]
\item[(ii)] For all $e_2$-standard primitive vectors $v \in \Z^n$
and all $e_1^{\ast}$-standard primitive vectors $f_1,\ldots,f_m \in \ker(v)$ and all $\lambda_1,\ldots,\lambda_m \in \Z$
such that $\sum_{i=1}^m \lambda_i f_i$ is primitive and $e_1^{\ast}$-standard, we have 
\[\PAGen{\sum_{i=1}^m \lambda_i f_i,v} = \sum_{i=1}^m \lambda_i \PAGen{f_i,v}.\]
\end{itemize}
\end{unnumberedclaim}
\begin{proof}
Both are proved the same way, so we will give the details for (i).  Let $f \in (\Z^n)^{\ast}$
and $v_1,\ldots,v_m \in \ker(f)$ and $\lambda_1,\ldots,\lambda_m \in \Z$ be as
in the claim.  There are now two cases.  

The first is that the restriction of $e_2^{\ast}$ to $\ker(f)$ is surjective.
This implies that $\ker(f)$ contains $e_2$-standard vectors that are not $e_2$-vanishing.  We can then choose
an isomorphism $\mu\colon \Z^{n-1} \rightarrow \ker(f)$ with the following property:
\begin{itemize}
\item Let $\{x_1,\ldots,x_{n-1}\}$ be the standard basis for $\Z^{n-1}$ and let $\{x_1^{\ast},\ldots,x_{n-1}^{\ast}\}$
be the corresponding dual basis.  Then the induced map $\mu^{\ast}\colon \ker(f)^{\ast} \rightarrow (\Z^{n-1})^{\ast}$
takes the restriction of $e_2^{\ast}$ to $x_1^{\ast}$.  This condition ensures that $\mu$ gives a bijection between
$x_1$-standard vectors in $\Z^{n-1}$ and $e_2$-standard vectors in $\ker(f)$.
\end{itemize}
Let $u_i = \mu^{-1}(v_i)$.  Recall that we defined $\fQ'_{n-1}$ in Definition \ref{definition:slstdvar}.
Define a map $\psi\colon \fQ'_{n-1} \rightarrow \fA'_n$ via the formula
\[\psi(\PQGen{w}) = \PAGen{f,\mu(w)} \quad \text{for an $x_1$-standard primitive $w \in \Z^{n-1}$}.\]
This takes relations to relations, and thus gives a well-defined map.  Since $n \geq 3$, 
Theorem \ref{maintheorem:slstdvar} says that $\fQ'_{n-1} \cong \Q^{n-1}$.  It follows from this theorem that
\[\PQGen{\sum_{i=1}^m \lambda_i u_i} = \sum_{i=1}^m \lambda_i \PQGen{u_i}.\]
Plugging this into $\psi$, we see that
\[\PAGen{f,\sum_{i=1}^m \lambda_i v_i} = \psi\left(\PQGen{\sum_{i=1}^m \lambda_i u_i}\right) = \sum_{i=1}^m \lambda_i \psi\left(\PQGen{u_i}\right) = \sum_{i=1}^m \lambda_i \PAGen{f,v_i},\]
as desired.

The second case is that the restriction of $e_2^{\ast}$ to $\ker(f)$ is not surjective,
so all $e_2$-standard vectors in $\ker(f)$ are $e_2$-vanishing.  In particular, all the
$v_i$ are $e_2$-vanishing.  Letting $U = \ker(e_2^{\ast}) \cap \ker(f)$, this implies that
all the $v_i$ lie in $U$.  Since $U$ is the intersection of two direct summands of $\Z^n$, it follows
that $U$ is also a direct summand of $\Z^n$.  Let $r$ be the rank of $U$.  Since $n \geq 4$, we have $r \geq 2$.  Choose
an isomorphism $\mu\colon \Z^r \rightarrow U$.  Let $u_i = \mu^{-1}(v_i)$.
Recall that we defined $\fQ_{r}$ in Definition \ref{definition:slstd}.
Define a map $\psi\colon \fQ_{r} \rightarrow \fA'_n$ via the formula
\[\psi(\QGen{w}) = \PAGen{f,\mu(w)} \quad \text{for a primitive $w \in \Z^{n-1}$}.\]
This makes sense since each $\mu(w)$ is $e_2$-vanishing, and hence also $e_2$-standard.
The map $\psi$ takes relations to relations, and thus gives a well-defined map.  Since $r \geq 2$,
Theorem \ref{maintheorem:slstd} says that $\fQ_{n-1} \cong \Q^{r}$.  It follows from this theorem that
\[\QGen{\sum_{i=1}^m \lambda_i u_i} = \sum_{i=1}^m \lambda_i \QGen{u_i}.\]
Plugging this into $\psi$, we see that
\[\PAGen{f,\sum_{i=1}^m \lambda_i v_i} = \psi\left(\QGen{\sum_{i=1}^m \lambda_i u_i}\right) = \sum_{i=1}^m \lambda_i \psi\left(\QGen{u_i}\right) = \sum_{i=1}^m \lambda_i \PAGen{f,v_i},\]
as desired.
\end{proof}

Recall that $\cB = \{e_1,\ldots,e_n\}$ is the standard basis for $\Z^n$ and 
$\cB^{\ast} = \{e_1^{\ast},\ldots,e_n^{\ast}\}$ is the corresponding dual basis for $(\Z^n)^{\ast}$.
Let $\Gamma = \SL_n(\Z,e_1,e_2^{\ast})$ be the stabilizer of $e_1$ and $e_2^{\ast}$ in $\SL_n(\Z)$.  The action of $\Gamma$ on
$(\Z^n)^{\ast}$ fixes the $e_1^{\ast}$-coordinate, and the action of $\Gamma$ on $\Z^n$ fixes the $e_2$-coordinate.
It follows that $\Gamma$ acts on
on $\fA'_n$, and we will use this action to prove our theorem.
We follow the outline from \S \ref{section:prooftechnique}, though for readability
we divide Step 3 into Steps 3.A and 3.B.

\begin{step}{1}
\label{step:sladjointvar1}
Let $S = S_1 \cup S_2$, where the $S_i$ are:
\begin{align*}
S_1 &= \Set{$\BPAGen{e_i^{\ast},e_j}$}{$1 \leq i,j \leq n$ distinct}, \\
S_2 &= \Set{$\BPAGen{e_i^{\ast}+e_{i+1}^{\ast},e_i-e_{i+1}}$}{$1 \leq i < n$}.
\end{align*}
Note that all of these are generators of $\fA'_n$.  Like we did here, we will write elements of $S$ in blue.
Then the restriction of $\Phi$ to $\Span{S}$ is an isomorphism.
\end{step}

The proof is identical to that of Step \ref{step:sladjoint1} in the proof of
Theorem \ref{maintheorem:sladjoint}, so we omit the details.

\begin{step}{2}
\label{step:sladjointvar2}
Recall that $\Gamma = \SL_n(\Z,e_1,e_2^{\ast})$.
We prove that the $\Gamma$-orbit of $S$ spans $\fA'_n$.
\end{step}

Consider a generator $\PAGen{f,v}$ of $\fA'_n$, so:
\begin{itemize}
\item $f \in (\Z^n)^{\ast}$ is primitive and $e_1^{\ast}$-standard; and
\item $v \in \Z^n$ is primitive and $e_2$-standard; and
\item $f(v) = 0$.
\end{itemize}
We want to prove that $\PAGen{f,v}$ is in the span of the
$\Gamma$-orbit of $S$.  
There are three (nonexclusive) cases:

\begin{case}{2.1}
\label{case:sladjointvar2.1}
The $e_1^{\ast}$-coordinate of $f$ is $\pm 1$.
\end{case}

\noindent
Using the claim from the beginning of the proof, we have $\PAGen{f,v} = -\PAGen{-f,v}$, so
replacing $\PAGen{f,v}$ with $\PAGen{-f,v}$ if necessary we can assume that
the $e_1^{\ast}$-coordinate of $f$ is $1$.  Let $M$ be the matrix whose
rows are $\{f,e_2^{\ast},\ldots,e_n^{\ast}\}$.  Since the $e_1^{\ast}$-coordinate
of $f$ is $1$, we have $\det(M) = 1$.  By construction we have have $M \in \Gamma = \SL_n(\Z,e_1,e_2^{\ast})$
and $M(e_1^{\ast}) = f$.  Replacing $\PAGen{f,v}$ by $M^{-1}(\PAGen{f,v})$, we can
assume that $f = e_1^{\ast}$.
Write
\[v = \lambda_1 e_1 + \cdots + \lambda_n e_n \quad \text{with $\lambda_1,\ldots,\lambda_n \in \Z$}.\]
Since $v$ is in the kernel of $f = e_1^{\ast}$, we have $\lambda_1 = 0$.  Using the claim
from the beginning of the proof, we have
\[\PAGen{f,v} = \PAGen{e_1^{\ast},\lambda_2 e_2 + \cdots + \lambda_n e_n} = \lambda_2 \PAGen{e_1^{\ast},e_2} + \cdots + \lambda_n \PAGen{e_1^{\ast},e_n} \in \Span{S},\]
as desired.

\begin{case}{2.2}
\label{case:sladjointvar2.2}
The $e_2$-coordinate of $v$ is $\pm 1$.
\end{case}

\noindent
Similar to Case \ref{case:sladjointvar2.1} except that after possibly replacing $\PAGen{f,v}$ with
$\PAGen{f,-v}$ we apply an element of
$\Gamma = \SL_n(\Z,e_1,e_2^{\ast})$ to change $v$ to $e_2$.

\begin{case}{2.3}
The $e_1^{\ast}$-coordinate of $f$ is $0$ and the $e_2$-coordinate of $v$ is $0$.
\end{case}

\noindent
Write $f = p e_2^{\ast} + q g$ with $p,q \in \Z$ and $g \in \Span{e_3^{\ast},\ldots,e_n^{\ast}}$ primitive.
Since $f$ is primitive, we have $\gcd(p,q) = 1$.  Since $n \geq 4$, it follows that the action of
$\SL(\Span{e_3^{\ast},\ldots,e_n^{\ast}})$ on $\Span{e_3^{\ast},\ldots,e_n^{\ast}}$ is transitive on
primitive vectors.  Using this, we can find $M \in \SL_n(\Z)$ such that $M(g) = e_3^{\ast}$ and such that
$M$ acts as the identity on $\Span{e_1^{\ast},e_2^{\ast}}$ and $\Span{e_1,e_2}$.  It follows that
$M \in \Gamma = \SL_n(\Z,e_1,e_2^{\ast})$.  Replacing 
$\PAGen{f,v} = \PAGen{p e_2^{\ast} + q g,v}$ with $M(\PAGen{f,v})$, we can assume that $f = p e_2^{\ast} + q e_3^{\ast}$.
Write
\[v = \lambda_1 e_1 + \cdots + \lambda_n e_n \quad \text{with $\lambda_1,\ldots,\lambda_n \in \Z$}.\]
By assumption, we have $\lambda_2 = 0$.  Since $f(v) = 0$ and $f = p e_2^{\ast} + q e_3^{\ast}$, we have
\[p \lambda_2 + q \lambda_3 = q \lambda_3 = 0.\]
It follows that
either $q$ or $\lambda_3$ must vanish.  If $q = 0$, then since $f = p e_2^{\ast} + q e_3^{\ast}$ is primitive we
must have $p \in \{\pm 1\}$ and we can use the claim from the beginning of the
proof to see that
\begin{align*}
\PAGen{f,v} &= \PAGen{p e_2^{\ast},\lambda_1 e_1 + \lambda_3 e_3 + \cdots + \lambda_n e_n} \\
            &= p \lambda_1 \PAGen{e_2^{\ast},e_1} + p \lambda_3 \PAGen{e_2^{\ast},e_3} + \cdots + p \lambda_n \PAGen{e_2^{\ast},e_n} \in \Span{S},
\end{align*}
as desired.  If instead $\lambda_3 = 0$, then we can again use the claim from beginning of the proof
to see that
\begin{align*}
\PAGen{f,v} &= \PAGen{p e_2^{\ast} + q e_3^{\ast},\lambda_1 e_1 + \lambda_4 e_4 + \cdots + \lambda_n e_n} \\
            &= \sum_{\substack{1 \leq i \leq n \\ i \neq 2,3}} \lambda_i \PAGen{p e_2^{\ast} + q e_3^{\ast},e_i}
             = \sum_{\substack{1 \leq i \leq n \\ i \neq 2,3}} \left(p\lambda_i \PAGen{e_2^{\ast},e_i} + q \lambda_i \PAGen{e_3^{\ast},e_i}\right) \in \Span{S},
\end{align*}
again as desired.

\begin{step}{3.A}
\label{step:sladjointvar3A}
In preparation for proving that $\Gamma = \SL_n(\Z,e_1,e_2^{\ast})$ takes $\Span{S}$ to $\Span{S}$, we prove
that all elements of\hspace{2pt}\footnote{Note that all of these lie in $\fA'_n$; indeed, the only $\GPAGen{f,v} \in E$ where there is any possibility
that either $f$ is not $e_1^{\ast}$-standard or $v$ is not $e_2$-standard are those of the form $\GPAGen{e_i^{\ast}+2e_j^{\ast},2e_i-e_j}$, and 
our assumptions that $j \neq 1$ and $i \neq 2$ ensure that indeed $\GPAGen{f,v} \in \fA'_n$.}
\begin{align*}
E = &\Set{$\GPAGen{e_i^{\ast}+e_j^{\ast},e_i-e_j}$}{$1 \leq i,j \leq n$ distinct} \\
    &\cup \Set{$\GPAGen{e_i^{\ast}-e_j^{\ast},e_i+e_j}$}{$1 \leq i,j \leq n$ distinct} \\
    &\cup \Set{$\GPAGen{e_i^{\ast}+2e_j^{\ast},2e_i-e_j}$}{$1 \leq i,j \leq n$ distinct, $j \neq 1$, $i \neq 2$}.
\end{align*}
lie in $\Span{S}$.  Like we did here, we will write elements of $E$ in green.
\end{step}

The proof is identical to that of Step \ref{step:sladjoint3A} of the proof of Theorem \ref{maintheorem:sladjoint}; indeed,
a careful examination of that proof shows that every term $\PAGen{f,v}$ that appears when handling the elements of $E$
listed above has the property that $f$ is $e_1^{\ast}$-standard and $v$ is $e_2$-standard.
We therefore omit the details.

\begin{step}{3.B}
\label{step:sladjointvar3B}
We prove that $\Gamma = \SL_n(\Z,e_1,e_2^{\ast})$ takes $\Span{S}$ to itself.
By Step \ref{step:sladjointvar2} this will imply that $\Span{S} = \fA'_n$,
and thus by Step \ref{step:sladjointvar1} that $\Phi$ is an isomorphism.
\end{step}

For distinct $1 \leq i,j \leq n$, let $E_{ij} \in \SL_n(\Z)$ be the elementary matrix
obtained from the identity by placing a $1$ at position $(i,j)$.  The matrix $E_{ij}$ lies
in $\Gamma$ if $j \neq 1$ and $i \neq 2$, and the set
\[\Lambda = \Set{$E_{ij}$}{$1 \leq i,j \leq n$ distinct, $j \neq 1$, $i \neq 2$}\]
generates $\Gamma$.  Fixing some $E_{ij} \in \Lambda$ and some $\epsilon = \pm 1$,
it is enough to prove that $E_{ij}^{\epsilon}$ takes $\Span{S}$ to itself.
The proof of this is identical to that of Step \ref{step:sladjoint3B} of the proof of Theorem \ref{maintheorem:sladjoint}.
We therefore omit the details.
\end{proof}

\section{Generating the symplectic group}
\label{section:spgen}

To prove our theorems about $\Sp_{2g}(\Z)$, we need generators
for $\Sp_{2g}(\Z)$.  The most convenient generating set was constructed by
Hua--Reiner \cite{HuaReiner}, which we now describe.  Recall from
\S \ref{section:notation} that $H = \Q^{2g}$ and $H_{\Z} = \Z^{2g}$, which are
equipped with the standard symplectic form $\omega$.  Let $\cB = \{a_1,b_1,\ldots,a_g,b_g\}$
be a symplectic basis for $H_{\Z}$.

Define $\SymSp_g$ to be the subgroup of $\Sp_{2g}(\Z)$ consisting of all $f \in \Sp_{2g}(\Z)$
such that for all $x \in \cB$, we have either $f(x) \in \cB$ or $-f(x) \in \cB$.  This is a finite group.
Associated to each $f \in \Sp_{2g}(\Z)$ is a permutation $p$ of $\{1,\ldots,g\}$ such that
for all $1 \leq i \leq g$ the pair $(f(a_i),f(b_i))$ is one of the following:
\[(a_{p(i)},b_{p(i)}),   \quad \text{or} \quad
  (-a_{p(i)},-b_{p(i)}), \quad \text{or} \quad
  (b_{p(i)},-a_{p(i)}),  \quad \text{or} \quad
  (-b_{p(i)},a_{p(i)}).\]
Next, for $1 \leq i \leq g$ let $X_i \in \Sp_{2g}(\Z)$ be the element defined by
\[X_i(a_i) = a_i+b_i \quad \text{and} \quad \text{$X_i(x) = x$ for $x \in \cB \setminus \{a_i\}$}.\]
Finally, for distinct $1 \leq i,j \leq g$ let $Y_{ij} \in \Sp_{2g}(\Z)$ be the element defined by
\[Y_{ij}(a_i) = a_i + b_j \quad \text{and} \quad Y_{ij}(a_j) = a_j + b_i \quad \text{and} \quad
\text{$Y_{ij}(x) = x$ for $x \in \cB \setminus \{a_i,a_j\}$}.\]
It follows from the calculations in \cite{HuaReiner} that:\footnote{This is not identical to
the generating set from \cite{HuaReiner}, but it is easily seen to be equivalent to it and
fits better into our calculations.}

\begin{theorem}[{Hua--Reiner, \cite{HuaReiner}}]
\label{theorem:huareiner}
For all $g \geq 1$, the group $\Sp_{2g}(\Z)$ is generated by $\SymSp_g$ and $X_1$ and $Y_{12}$.
\end{theorem}

\begin{remark}
The other $X_i$ are not needed since they are all conjugate to $X_1$ by elements of $\SymSp_g$.
Similarly, the other $Y_{ij}$ are not needed.
\end{remark}

The most complicated generator is $Y_{12}$.  The following observation will simplify our calculations
by allowing us to not worry about $Y_{12}^{-1}$:

\begin{corollary}
\label{corollary:gensp}
For all $g \geq 1$, the group $\Sp_{2g}(\Z)$ is generated as a monoid by $\SymSp_g$ and $\{X_1,X_1^{-1},Y_{12}\}$.
\end{corollary}
\begin{proof}
Let $\sigma \in \SymSp_{2g}$ be the element that multiplies $a_1$ and $b_1$ by $-1$ while fixing all other
elements of $\cB$.  We then have $Y_{12}^{-1} = \sigma Y_{12} \sigma$.  The corollary follows.
\end{proof}

\section{Symplectic group I: standard representation}
\label{section:spstd}

Recall that $\fH_n$ is the $\Q$-vector space with the following presentation:
\begin{itemize}
\item {\bf Generators}. A generator $\HGen{v}$ for all primitive vectors $v \in H_{\Z}$.
\item {\bf Relations}.  For a partial basis $\{v_1,v_2\}$ of $H_{\Z}$ with $\omega(v_1,v_2)=0$, the relation
$\HGen{v_1}+\HGen{v_2} = \HGen{v_1+v_2}$.
\end{itemize}
Define a linearization map $\Phi\colon \fH_g \rightarrow H$ via the formula
$\Phi(\HGen{v}) = v$.  This takes relations to relations, and thus gives a well-defined
map.  Our goal is to prove:

\newtheorem*{maintheorem:spstd}{Theorem \ref{maintheorem:spstd}}
\begin{maintheorem:spstd}
For $g \geq 2$, the linearization map $\Phi\colon \fH_g \rightarrow H$ is an isomorphism.
\end{maintheorem:spstd}
\begin{proof}
Let $\cB = \{a_1,b_1,\ldots,a_g,b_g\}$
be a symplectic basis for $H_{\Z}$.
We follow the outline from \S \ref{section:prooftechnique}, though for readability
we divide Step 3 into Steps 3.A and 3.B.

\begin{step}{1}
\label{step:spstd1}
Let $S = \Set{$\HGen{x}$}{$x \in \cB$}$.
Then the restriction of $\Phi$ to $\Span{S}$ is an isomorphism.
\end{step}

Immediate.

\begin{step}{2} 
\label{step:spstd2}
We prove that the $\Sp_{2g}(\Z)$-orbit of $S$ spans $\fH_g$.
\end{step}

Since $\Sp_{2g}(\Z)$ acts transitively on primitive vectors,
it acts transitively on the generators for $\fH_g$.  It follows that the $\Sp_{2g}(\Z)$-orbit
of $S$ spans $\fH_g$.

\begin{step}{3.A}
\label{step:spstd3A}
In preparation for proving that $\Sp_{2g}(\Z)$ takes $\Span{S}$ to $\Span{S}$, we prove
that all elements of
$E = \Set{$\HGen{-x}$}{$x \in \cB$} \cup \{\HGen{a_1+b_1},\HGen{a_1-b_1}\}$
lie in $\Span{S}$.
\end{step}

We divide this into three claims:

\begin{claim}{3.A.1}
For all primitive $v \in H_{\Z}$, we have $\HGen{-v} = -\HGen{v}$.  In particular, all elements of $\Set{$\HGen{-x}$}{$x \in \cB$}$ lie in $S$.
\end{claim}

\noindent
Since $g \geq 2$, we can find $w \in H_{\Z}$ with $\omega(v,w) = 0$ such that $\Span{v,w}$ is a
rank-$2$ direct summand of $\Z^{2g}$.  Let $\mu\colon \Z^2 \rightarrow \Span{v,w}$ be the isomorphism
taking the standard basis $\{e_1,e_2\}$ of $\Z^2$ to $\{v,w\}$.  Recall that we defined
$\fQ_n$ in Definition \ref{definition:slstd}.  Define a map $\psi\colon \fQ_2 \rightarrow \fH_g$
via the formula
\[\psi(\QGen{x}) = \HGen{\mu(x)} \quad \text{for a primitive $x \in \Z^2$}.\]
Since $\omega(-,-)$ vanishes identically on $\Span{v,w}$, this formula
takes generators to generators.  Theorem \ref{maintheorem:slstd} says
that $\fQ_2 \cong \Q^2$.  This implies that $\QGen{-x} = -\QGen{x}$ for all
primitive $x \in \Z^2$.  In particular,
\[\HGen{-v} = \HGen{-\mu(e_1)} = \psi(\QGen{-e_1}) = \psi(-\QGen{e_1}) = -\HGen{v}.\]

\begin{claim}{3.A.2}
We have $\HGen{a_1+b_1} \in \Span{S}$.
\end{claim}

\noindent
The element $\HGen{(a_1+b_1)+(a_2+b_2)}$ equals
$\HGen{a_1+b_1} + \HGen{a_2+b_2}$.
It also equals
\begin{align*}
\HGen{(a_1+b_2) + (a_2+b_1)} &= \HGen{a_1+b_2} + \HGen{a_2+b_1} 
                             = \HGen{a_1} + \HGen{b_2} + \HGen{a_2} + \HGen{b_1}.
\end{align*}
We deduce that
\begin{equation}
\label{eqn:spstd3a2.1}
\HGen{a_1+b_1} + \HGen{a_2+b_2} = \HGen{a_1} + \HGen{b_2} + \HGen{a_2} + \HGen{b_1}.
\end{equation}
Similarly, the element $\HGen{(a_1+b_1)-(a_2+b_2)}$ equals
$\HGen{a_1+b_1} - \HGen{a_2+b_2}$.
It also equals
\begin{align*}
\HGen{(a_1-b_2) + (-a_2+b_1)} &= \HGen{a_1-b_2} + \HGen{-a_2+b_1} 
                             = \HGen{a_1} - \HGen{b_2} - \HGen{a_2} + \HGen{b_1}.
\end{align*}
We deduce that
\begin{equation}
\label{eqn:spstd3a2.2}
\HGen{a_1+b_1} - \HGen{a_2+b_2} = \HGen{a_1} - \HGen{b_2} - \HGen{a_2} + \HGen{b_1}.
\end{equation}
Adding \eqref{eqn:spstd3a2.1} and \eqref{eqn:spstd3a2.2} and dividing by $2$ yield the claim.

\begin{claim}{3.A.3}
We have $\HGen{a_1-b_1} \in \Span{S}$.
\end{claim}

\noindent
The element $\HGen{(a_1-b_1)+(a_2-b_2)}$ equals
$\HGen{a_1-b_1} + \HGen{a_2-b_2}$.
It also equals
\begin{align*}
\HGen{(a_1-b_2) + (a_2-b_1)} &= \HGen{a_1-b_2} + \HGen{a_2-b_1} \\
                             &= \HGen{a_1} - \HGen{b_2} + \HGen{a_2} - \HGen{b_1}.
\end{align*}
We deduce that
\begin{equation}
\label{eqn:spstd3a3.1}
\HGen{a_1-b_1} + \HGen{a_2-b_2} = \HGen{a_1} - \HGen{b_2} + \HGen{a_2} - \HGen{b_1}.
\end{equation}
Similarly, the element $\HGen{(a_1-b_1)-(a_2-b_2)}$ equals
$\HGen{a_1-b_1} - \HGen{a_2-b_2}$.
It also equals
\begin{align*}
\HGen{(a_1+b_2) + (-a_2-b_1)} &= \HGen{a_1+b_2} + \HGen{-a_2-b_1} \\
                             &= \HGen{a_1} + \HGen{b_2} - \HGen{a_2} - \HGen{b_1}.
\end{align*}
We deduce that
\begin{equation}
\label{eqn:spstd3a3.2}
\HGen{a_1-b_1} - \HGen{a_2-b_2} = \HGen{a_1} + \HGen{b_2} - \HGen{a_2} - \HGen{b_1}.
\end{equation}
Adding \eqref{eqn:spstd3a3.1} and \eqref{eqn:spstd3a3.2} and dividing by $2$ yields the claim.

\begin{step}{3.B}
\label{step:spstd3B}
We prove that $\Sp_{2g}(\Z)$ takes $\Span{S}$ to itself.
By Step \ref{step:spstd2} this will imply that $\Span{S} = \fH_g$,
and thus by Step \ref{step:spstd1} that $\Phi$ is an isomorphism.
\end{step}

Corollary \ref{corollary:gensp} says that $\Sp_{2g}(\Z)$ is generated as a monoid by
$\SymSp_g \cup \{X_1,X_1^{-1},Y_{12}\}$.
Let $f \in \SymSp_g \cup \{X_1,X_1^{-1},Y_{12}\}$ and let $s \in S$.
Using Step \ref{step:spstd3A}, it is enough to check that $f(s)$ can both be written
as a linear combination of elements of $S$ and $E$.

The first case is $f \in \SymSp_g$.  This case is easy: we have $s = \HGen{x}$ for some $x \in \cB$, and for some $y \in \cB$ the
element $f(s) = \HGen{f(x)}$ is either
$\HGen{y} \in S$ or $\HGen{-y} \in E$.
The second case is $f = X_1$ or $f = X_1^{-1}$.  Recall that $X_1$ takes $a_1$ to $a_1 + b_1$ and fixes all other
elements of $\cB$.  Both $X_1$ and $X_1^{-1}$ fix all elements of $S$ except for $\HGen{a_1}$, and
for this we have
\begin{align*}
X_1(\HGen{a_1}) &= \HGen{a_1+b_1} \in E,\\
X_1^{-1}(\HGen{a_1}) &= \HGen{a_1-b_1} \in E.
\end{align*}
The final case is $f = Y_{12}$.  Recall that this takes $a_1$ to $a_1 + b_2$ and $a_2$ to $b_1 + a_2$ and fixes
all other elements of $\cB$.  The element $Y_{12}$ fixes all elements of $S$ except for $\HGen{a_1}$ and $\HGen{a_2}$,
for these we have
\begin{align*}
Y_{12}(\HGen{a_1}) &= \HGen{a_1+b_2} = \HGen{a_1} + \HGen{b_2} \in \Span{S}, \\
Y_{12}(\HGen{a_2}) &= \HGen{a_2+b_1} = \HGen{a_2} + \HGen{b_1} \in \Span{S}.\qedhere
\end{align*}
\end{proof}

\section{Symplectic group II: kernel and symmetric representations}
\label{section:spkernel}

We could prove Theorems \ref{maintheorem:spkernelalt} and \ref{maintheorem:spsym} using our
now-standard proof technique, but to illustrate another useful tool we show how to deduce
them from Theorem \ref{maintheorem:sladjoint}.

\subsection{Non-symmetric presentation}
\label{section:nonsymmetricfz}

We defined $\fZ_g^s$ and $\fZ_g^a$ in Definitions \ref{definition:spsym} and \ref{definition:spkernelalt}.  We now define
the following, which is similar to these but does not include their symmetric/anti-symmetric relations:

\begin{definition}
\label{definition:spkernel}
Define $\fZ_g$ to be the $\Q$-vector space with the following presentation:
\begin{itemize}
\item {\bf Generators}.  A generator $\ZGen{v_1,v_2}$ for all orthogonal primitive vectors $v_1,v_2 \in H_{\Z}$.
\item {\bf Relations}.  For all primitive vectors $v \in H_{\Z}$ and all partial bases $\{w_1,w_2\}$ of $v^{\perp}_{\Z}$,
the relations 
\begin{align*}
\ZGen{v,w_1 + w_2} &= \ZGen{v,w_1} + \ZGen{v,w_2}, \quad \text{and} \\
\ZGen{w_1+w_2,v}   &= \ZGen{w_1,v} + \ZGen{w_2,v}.\qedhere
\end{align*}
\end{itemize}
\end{definition}

This combines $\fZ_g^s$ and $\fZ_g^a$ in the following way.
There is an involution $I\colon \fZ_g \rightarrow \fZ_g$ defined by $I(\ZGen{v_1,v_2}) = \ZGen{v_2,v_1}$ that
we will call the {\em canonical involution}.  We then have:

\begin{lemma}
\label{lemma:decomposespkernel}
We have $\fZ_g = \fZ_g^s \oplus \fZ_g^a$, where $\fZ_g^s$ and $\fZ_g^a$ are identified with
the $+1$ and $-1$ eigenspaces of the canonical involution.
\end{lemma}
\begin{proof}
Define a map $\pi\colon \fZ_g \rightarrow \fZ_g^s \oplus \fZ_g^a$ via the formula
\[\pi(\ZGen{v_1,v_2}) = (\ZGenS{v_1,v_2},\ZGenA{v_1,v_2}).\]
This take relations to relations, and thus gives a well-defined map.  Next, define
$\iota\colon \fZ_g^s \oplus \fZ_g^a$ as $\iota = \iota_s + \iota_a$, where $\iota_s\colon \fZ_g^s \rightarrow \fZ_g$
and $\iota_a\colon \fZ_g^a \rightarrow \fZ_g$ are the maps defined via the formulas
\begin{align*}
\iota_s(\ZGenS{v_1,v_2}) = \frac{1}{2}\left(\ZGen{v_1,v_2} + \ZGen{v_2,v_1}\right), \\
\iota_a(\ZGenA{v_1,v_2}) = \frac{1}{2}\left(\ZGen{v_1,v_2} - \ZGen{v_2,v_1}\right).
\end{align*}
Again, these take relations to relations and thus give well-defined maps.  The maps
$\pi$ and $\iota$ are inverses, so both are isomorphisms.  This gives the decomposition
$\fZ_g = \fZ_g^s \oplus \fZ_g^a$, and the fact that $\fZ_g^s$ and $\fZ_g^a$ are identified with
the $+1$ and $-1$ eigenspaces of the canonical involution follows from the above formulas.
\end{proof}

\subsection{Identifying the non-symmetric presentation}
\label{section:nonsymmetriczero}

Define a linearization map $\Phi\colon \fZ_g \rightarrow H^{\otimes 2}$ via the formula
$\Phi(\ZGen{v_1,v_2}) = v_1 \otimes v_2$.  This takes relations to relations, and thus
gives a well-defined map.  Its image lies in the kernel $\cZ_g$ of the map
\[\begin{tikzcd}
H^{\otimes 2} \arrow{r} & \wedge^2 H \arrow{r}{\omega} & \Q,
\end{tikzcd}\]
where the second map is the one induced by the symplectic form $\omega$.  We will prove:

\begin{theorem}
\label{theorem:spkernel}
For $g \geq 2$, the linearization map $\Phi\colon \fZ_g \rightarrow \cZ_g$ is an isomorphism.
\end{theorem}
\begin{proof}
Pick an isomorphism $\mu_2 \colon H_{\Z} \rightarrow \Z^{2g}$.  Since $\omega$ identifies
$H_{\Z}$ with its dual $H_{\Z}^{\ast} = \Hom(H,\Z)$, we can find a corresponding isomorphism
$\mu_1\colon H_{\Z} \rightarrow (\Z^{2g})^{\ast}$ such that
\begin{equation}
\label{eqn:dualiso}
\omega(v_1,v_2) = \mu_1(v_1)(\mu_2(v_2)) \quad \text{for all $v_1,v_2 \in H_{\Z}$}.
\end{equation}
The map $\mu_1 \otimes \mu_2 \colon H_{\Z}^{\otimes 2} \rightarrow (\Z^{2g})^{\ast} \otimes \Z^{2g}$
is then an $\Sp_{2g}(\Z)$-equivariant isomorphism.  

Recall that we defined the vector space $\fA_{2g}$ in Definition \ref{definition:sladjoint}.  Using
$\mu_1$ and $\mu_2$, we can define a map $M\colon \fZ_g \rightarrow \fA_{2g}$ via the formula
$M(\ZGen{v_1,v_2}) = \AGen{\mu_1(v_1),\mu_2(v_2)}$.  The identity \eqref{eqn:dualiso} implies that 
$\AGen{\mu_1(v_1),\mu_2(v_2)}$ is a generator for $\fA_{2g}$.  The map $M$ takes relations to relations, and thus gives
a well-defined map.  In fact, even more is true: $M$ is an isomorphism with inverse
the map taking $\AGen{f,v}$ to $\ZGen{\mu_1^{-1}(f),\mu_2^{-1}(v)}$.

The vector space $\fA_{2g}$ is equipped with a linearization map $\Phi\colon \fA_{2g} \rightarrow (\Q^{2g})^{\ast} \otimes \Q^{2g}$
defined via the formula $\Phi(\AGen{f,v}) = f \otimes v$.  The image of this map is
the kernel $\fsl_{2g}(\Q)$ of the trace map
\[\trace\colon (\Q^n)^{\ast} \otimes \Q^n \longrightarrow \Q\]
defined by $\trace(f,v) = f(v)$.  Theorem \ref{maintheorem:sladjoint} says that
$\Phi\colon \fA_{2g} \rightarrow \fsl_{2g}(\Q)$ is an isomorphism.  Abusing
notation slightly, we will identify the isomorphism
$\mu_1 \otimes \mu_2\colon H_{\Z}^{\otimes 2} \rightarrow (\Z^n)^{\ast} \otimes \Z^n$
with the corresponding isomorphism $\mu_1 \otimes \mu_2\colon H^{\otimes 2} \rightarrow (\Q^n)^{\ast} \otimes \Q^n$.
By \eqref{eqn:dualiso}, this takes $\cZ_g$ to $\fsl_{2g}(\Q)$.  
This all fits into a commutative diagram
\[\begin{tikzcd}
\fZ_g \arrow{r}{M}[swap]{\cong} \arrow{d}{\Phi} & \fA_{2g} \arrow{d}{\Phi}[swap]{\cong} \\
\cZ_g   \arrow{r}{\mu_1 \otimes \mu_2}[swap]{\cong} & \fsl_{2g}(\Q).
\end{tikzcd}\]
From this, we conclude that $\Phi\colon \fZ_g \rightarrow \cZ_g$ is an isomorphism.
\end{proof}

\subsection{Consequences}

Let $\iota\colon H^{\otimes 2} \rightarrow H^{\otimes 2}$ be the involution $\iota(v_1 \otimes v_2) = v_2 \otimes v_1$.
This involution induces a decomposition of $H^{\otimes 2}$ into into its $+1$ and $-1$ eigenspaces.  This
gives the familiar decomposition
\[H^{\otimes 2} = \Sym^2(H) \oplus \wedge^2 H.\]
Recall that $\cZ_g^a$ is the kernel of the map $\wedge^2 H \rightarrow \Q$ given by the symplectic form.  The
vector spaces $\cZ_g$ and $\cZ_g^a$ fit into the above decomposition as
\[\cZ_g = \Sym^2(H) \oplus \cZ_g^a.\]
The involution $\iota$ lifts to the canonical involution $I\colon \fZ_g \rightarrow \fZ_g$
in the sense that diagram
\[\begin{tikzcd}
\fZ_g \arrow{r}{I} \arrow{d}{\Phi}[swap]{\cong} & \fZ_g \arrow{d}{\Phi}[swap]{\cong} \\
\cZ_g \arrow{r}{\iota}             & \cZ_g
\end{tikzcd}\]
commute.  This implies that the isomorphism $\Phi\colon \fZ_g \rightarrow \cZ_g$ from
Theorem \ref{theorem:spkernel} matches up the $+1$ and $-1$ eigenspaces of $I$ and $\iota$.
For $\fZ_g$, these eigenspaces were identified by Lemma \ref{lemma:decomposespkernel},
and the following theorems from the introduction follow:

\newtheorem*{maintheorem:spsym}{Theorem \ref{maintheorem:spsym}}
\begin{maintheorem:spsym}
For $g \geq 2$, the linearization map $\Phi\colon \fZ_g^s \rightarrow \Sym^2(H)$ is an isomorphism.
\end{maintheorem:spsym}

\newtheorem*{maintheorem:spkernelalt}{Theorem \ref{maintheorem:spkernelalt}}
\begin{maintheorem:spkernelalt}
For $g \geq 1$, the linearization map $\Phi\colon \fZ_g^a \rightarrow \cZ_g^a$ is an isomorphism.
\end{maintheorem:spkernelalt}

\begin{remark}
One tiny issue with the above argument is that it only works for $g \geq 2$, while Theorem \ref{maintheorem:spkernelalt}
is also supposed to hold for $g=1$.  However, for $g=1$ this theorem is trivial since $\fZ_1^a = 0$ and $\cZ_1^a = 0$.
\end{remark}

\part{Improving the presentation for the symmetric kernel}
\label{part:2}

We now turn to our theorems on the symmetric kernel.  In this part of the paper,
we enlarge its purported presentation by adding some additional generators.
The key result needed to add these generators (Proposition \ref{proposition:isotropicpairiso} below)
uses the proof technique from \S \ref{section:prooftechnique}
that we have already used to prove Theorems \ref{maintheorem:slstd} -- \ref{maintheorem:spsym}.
See the introductory \S \ref{section:part2intro} for a more detailed discussion of what we will
do.  Our main theorems will be proved in Parts \ref{part:alt} and \ref{part:sym}, again using the proof technique
from \S \ref{section:prooftechnique}.

To avoid having to constantly
impose genus hypotheses, we make the following blanket assumption:

\begin{assumption}
\label{assumption:genuspart2}
Throughout Part \ref{part:2}, we assume that $g \geq 4$.
\end{assumption}

\section{Introduction to Part \ref{part:2}}
\label{section:part2intro}

Recall from \S \ref{section:notation} that $H = \Q^{2g}$ and $H_{\Z} = \Z^{2g}$ and
$\omega\colon H \times H \rightarrow \Q$ is the standard symplectic form on $H$.
We start by recalling some definitions and notation from the introduction and proving some preliminary results,
and then we will outline the rest of this part.

\subsection{Quotient representation}

The symplectic form $\omega$ on $H$ identifies $H$ with its dual.  Using this, we can
identify alternating forms on $H$ with elements of $\wedge^2 H$.  If
$\{a_1,b_1,\ldots,a_g,b_g\}$ is a symplectic basis for $H$, then
\[\omega = a_1 \wedge b_1 + \cdots a_g \wedge b_g.\]
The $\Q$-span of $\omega$ in $\wedge^2 H$ is a copy of $\Q$. 
The quotient $(\wedge^2 H)/\Q$ will always mean the quotient by the $\Q$-span of $\omega$.
Similarly, $(\wedge^2 H_{\Z})/\Z$ will always mean the quotient of $\wedge^2 H_{\Z}$ by the $\Z$-span
of $\omega$.

\subsection{Symmetric contraction}
\label{section:symmetriccontraction}

As we discussed in the introduction, the symmetric contraction is the bilinear map
\[\fc\colon ((\wedge^2 H)/\Q) \times ((\wedge^2 H)/\Q) \longrightarrow \Sym^2(H)\]
defined by the formula
\[\text{$\fc(x \wedge y,z \wedge w) = \omega(x,z) y \Cdot w - \omega(x,w) y \Cdot z - \omega(y,z) x \Cdot w + \omega(y,w) x \Cdot z$ for $x,y,z,w \in H$}.\]
The bilinear form $\fc$ is alternating:
\[\text{$\fc(\kappa_2,\kappa_1) = -\fc(\kappa_1,\kappa_2)$ for all $\kappa_1,\kappa_2 \in (\wedge^2 H)/\Q$}.\]
It induces a map
\[((\wedge^2 H)/\Q)^{\otimes 2} \longrightarrow \Sym^2(H)\]
whose kernel $\cK_g$ is the {\em symmetric kernel}.
We say that $\kappa_1,\kappa_2 \in (\wedge^2 H)/\Q$ are {\em sym-orthogonal} if 
\[\fc(\kappa_1,\kappa_2)=-\fc(\kappa_2,\kappa_1) = 0,\]
or equivalently if $\kappa_1 \otimes \kappa_2$ and $\kappa_2 \otimes \kappa_1$ lie in $\cK_g$.
For $\kappa \in (\wedge^2 H)/\Q$, the {\em symmetric orthogonal complement} of $\kappa$, denoted
$\kappa^{\perp}$, is the subspace of all $\kappa' \in (\wedge^2 H)/\Q$ that are sym-orthogonal to $\kappa$.

\subsection{Symplectic pairs}
\label{section:semisymplectic}

A {\em symplectic pair} is an element of $(\wedge^2 H_{\Z})/\Z$ of the form
$a \wedge b$, where $a,b \in H_{\Z}$ are such that $\omega(a,b) = 1$.
Equivalently, there exists a symplectic basis $\{a_1,b_1,\ldots,a_g,b_g\}$
for $H_{\Z}$ with $a_1 = a$ and $b_1 = b$.  
For $X \subset \wedge^2 H$, let $\overline{X}$
be its image in $(\wedge^2 H)/\Q$.  Also, for $V \subset H_{\Z}$ let $V_{\Q} = V \otimes \Q \subset H$.
We have:

\begin{lemma}
\label{lemma:symplecticorthogonal}
Let $a \wedge b$ be a symplectic pair and let $V = \Span{a,b}$.  Then
$(a \wedge b)^{\perp} = \overline{\wedge^2 V_{\Q}^{\perp}}$.
\end{lemma}
\begin{proof}
Let $\{a_1,b_1,\ldots,a_g,b_g\}$ be a symplectic basis for $H$ with $a_1 = a$
and $b_1 = b$, so $V_{\Q}^{\perp} = \Span{a_{2},b_{2},\ldots,a_g,b_g}$.
It is immediate from the formula for $\fc$ in \S \ref{section:symmetriccontraction} that
$\fc(a_1 \wedge b_1,\kappa) = 0$ for $\kappa \in \overline{\wedge^2 V_{\Q}^{\perp}}$, so
$\overline{\wedge^2 V_{\Q}^{\perp}} \subset (a_1 \wedge b_1)^{\perp}$.
We must show the other inclusion.  Via the decomposition
\[\wedge^2 H = \wedge^2 (V_{\Q} \oplus V_{\Q}^{\perp}) = \left(\wedge^2 V_{\Q}\right) \oplus \left(\wedge^2 V_{\Q}^{\perp}\right) \oplus \left(V_{\Q} \wedge V_{\Q}^{\perp}\right),\]
this is equivalent to showing that the intersection of $(a_1 \wedge b_1)^{\perp}$
and
\begin{equation}
\label{eqn:symplecticexclude}
\overline{\left(\wedge^2 V_{\Q}\right) \oplus \left(V_{\Q} \wedge V_{\Q}^{\perp}\right)} = \overline{\Span{a_1 \wedge b_1} \oplus \left(V_{\Q} \wedge V_{\Q}^{\perp}\right)}
\end{equation}
is contained in $\overline{\wedge^2 V_{\Q}^{\perp}}$.

Note that \eqref{eqn:symplecticexclude} is spanned by $a_1 \wedge b_1$ and $a_1 \wedge z$ and $b_1 \wedge z$
as $z$ ranges over $\{a_{2},b_{2},\ldots,a_g,b_g\}$.  For such $z$, we have
\begin{align*}
\fc(a_1 \wedge b_1,a_1 \wedge z)   &= -a_1 \Cdot z,\\
\fc(a_1 \wedge b_1,b_1 \wedge z)   &= b_1 \Cdot z,\\
\fc(a_1 \wedge b_1,a_1 \wedge b_1) &= -b_1 \Cdot a_1 + a_1 \Cdot b_1 = 0.
\end{align*}
Other than $0$, the elements of $\Sym^2(H)$ appearing on the right hand side of this equation as
$z$ ranges over $\{a_{2},b_{2},\ldots,a_g,b_g\}$ are linearly independent.  It follows
that the intersection of $(a_1 \wedge b_1)^{\perp}$ with \eqref{eqn:symplecticexclude} is spanned
by
\[a_1 \wedge b_1 = -(a_2 \wedge b_2 + \cdots + a_g \wedge b_g) \in \overline{\wedge^2 V_{\Q}^{\perp}},\]
as desired.  Note that here we are using the fact that we are working in $(\wedge^2 H)/\Q$, so $\omega \in \wedge^2 H$
equals $0$.
\end{proof}

\subsection{Isotropic pairs}
\label{section:isotropicpairs}

An {\em isotropic pair}
is an element of $(\wedge^2 H_{\Z})/\Z$ of the form
$a \wedge a'$, where $a,a' \in H_{\Z}$ are linearly independent elements such that
$\omega(a,a') = 0$.  
The following is the analogue for isotropic pairs of Lemma \ref{lemma:symplecticorthogonal}:

\begin{lemma}
\label{lemma:isotropicorthogonal}
Let $a \wedge a'$ be an isotropic pair and let $I = \Span{a,a'}$.  Then
$(a \wedge a')^{\perp} = \overline{\wedge^2 I_{\Q}^{\perp}}$.
\end{lemma}
\begin{proof}
We can find a symplectic basis $\{a_1,b_1,\ldots,a_g,b_g\}$ for $H$ with $a_1 = a$
and $a_2 = a'$, so $I_{\Q}^{\perp} = \Span{a_1,a_2,a_3,b_3,\ldots,a_g,b_g}$.
Note that we might not be able to find such a basis for $H_{\Z}$ since $\{a,a'\}$ might
not span a direct summand of $H_{\Z}$ (see \S \ref{section:strongisotropicpairs} below).
It is immediate from the formula for $\fc$ in \S \ref{section:symmetriccontraction} that
$\fc(a_1 \wedge a_2,\kappa) = 0$ for $\kappa \in \overline{\wedge^2 I_{\Q}^{\perp}}$, so
$\overline{\wedge^2 I_{\Q}^{\perp}} \subset (a_1 \wedge a_2)^{\perp}$.
We must show the other inclusion.  Via the decomposition
\[\wedge^2 H = \wedge^2 (I_{\Q}^{\perp} \oplus \Span{b_1,b_2}) = \left(\wedge^2 I_{\Q}^{\perp}\right) \oplus \left(b_1 \wedge I_{\Q}^{\perp}\right) \oplus \left(b_2 \wedge I_{\Q}^{\perp}\right) \oplus \Span{b_1 \wedge b_2},\]
this is equivalent to showing that the intersection of $(a_1 \wedge a_2)^{\perp}$
and
\begin{equation}
\label{eqn:isotropicexclude}
\overline{\left(b_1 \wedge I_{\Q}^{\perp}\right) \oplus \left(b_2 \wedge I_{\Q}^{\perp}\right) \oplus \Span{b_1 \wedge b_2}}.
\end{equation}
is contained in $\overline{\wedge^2 I_{\Q}^{\perp}}$.
For $z \in \{a_1,a_{2},a_3,b_{3},\ldots,a_g,b_g\}$, we have
\begin{align*}
\fc(a_1 \wedge a_2,b_1 \wedge z)   &= a_2 \Cdot z,\\
\fc(a_1 \wedge a_2,b_2 \wedge z)   &= -a_1 \Cdot z,\\
\fc(a_1 \wedge a_2,b_1 \wedge b_2) &= a_2 \Cdot b_2 + a_1 \Cdot b_1.
\end{align*}
The only linear dependence among the elements of $\Sym^2(H)$ appearing on the right hand side of this equation as
$z$ ranges over $\{a_1,a_{2},a_3,b_{3},\ldots,a_g,b_g\}$ is
\[\fc(a_1 \wedge a_2,b_1 \wedge a_1) + \fc(a_1 \wedge a_2,b_2 \wedge a_2) = a_2 \Cdot a_1 - a_1 \Cdot a_2 = 0.\]
It follows that the intersection of $(a_1 \wedge a_2)^{\perp}$ with \eqref{eqn:isotropicexclude} is spanned
by
\[b_1 \wedge a_1 + b_2 \wedge a_2 = -(a_1 \wedge b_1 + a_2 \wedge b_2) = a_3 \wedge b_3 + \cdots + a_g \wedge b_g \in \overline{\wedge^2 I_{\Q}^{\perp}},\]
as desired.
\end{proof}

\subsection{Strong isotropic pairs}
\label{section:strongisotropicpairs}

A {\em strong isotropic pair} is an isotropic pair $a \wedge a'$ such that $\{a,a'\}$ forms a basis
for a rank-$2$ direct summand of $H_{\Z}$.
Equivalently, there exists a symplectic basis $\{a_1,b_1,\ldots,a_g,b_g\}$
for $H_{\Z}$ with $a_1 = a$ and $a_2 = a'$.  We will prove that every isotropic
pair is a multiple of a strong isotropic pair.  This requires the following lemma:

\begin{lemma}
\label{lemma:summand}
Let $V$ be a subspace of $\Q^n$.  Then $V_{\Z} = V \cap \Z^n$
is a direct summand of $\Z^n$.
\end{lemma}
\begin{proof}
The short exact sequence
\[\begin{tikzcd}
0 \arrow{r} & V \arrow{r} & \Q^n \arrow{r}{\pi} & \Q^n/V \arrow{r} & 0
\end{tikzcd}\]
restricts to a short exact sequence
\begin{equation}
\label{eqn:bijection}
\begin{tikzcd}
0 \arrow{r} & V_{\Z} \arrow{r} & \Z^n \arrow{r}{\pi} & \pi(\Z^n) \arrow{r} & 0.
\end{tikzcd}
\end{equation}
The subgroup $\pi(\Z^n)$ of $\Q^n/V \cong \Q^{n-\dim(V)}$ is finitely generated and torsion-free, and hence
free abelian.  The short exact sequence \eqref{eqn:bijection} thus splits, so $V_{\Z}$ is a direct summand of $\Z^n$.
\end{proof}

For a subspace $V$ of $H_{\Z}$, the {\em saturation} of $V$ in $H_{\Z}$ is
$V_{\Q} \cap H_{\Z}$.  By Lemma \ref{lemma:summand}, this is a direct summand
of $H_{\Z}$.  We have:

\begin{lemma}
\label{lemma:strongisotropic}
Let $a \wedge a'$ be an isotropic pair.  Then there exists a strong isotropic
pair $a_0 \wedge a'_0$ and $n \in \Z$ such that $a \wedge a' = n a_0 \wedge a'_0$.
Moreover, $\Span{a_0,a'_0}$ is the saturation in $H_{\Z}$ of $\Span{a,a'}$.
\end{lemma}
\begin{proof}
Set $I = \Span{a,a'}$ and let $\oI$ be the saturation of $I$ in $H_{\Z}$.  Let
$\{a_0,a'_0\}$ be a basis for $\oI$.  Regarding $a \wedge a'$ and $a_0 \wedge a'_0$
as elements of $\wedge^2 H$, they correspond to the same $2$-dimensional subspace
of $H$, namely $I_{\Q} = \oI_{\Q}$. It follows that there exists some $n \in \Q$ such that
$a \wedge a' = n a_0 \wedge a'_0$.  Since $a \wedge a' \in \wedge^2 H_{\Z}$
and $a_0 \wedge a'_0$ is a primitive element of $\wedge^2 H_{\Z}$, we must have
$n \in \Z$, as desired.
\end{proof}

\subsection{Symmetric kernel presentation}

We defined $\fZ_g^s$ and $\fZ_g^a$ in Definitions \ref{definition:kgsym} and \ref{definition:kgalt}.  Just like
we did for $\fZ_g^s$ and $\fZ_g^a$ in \S \ref{section:nonsymmetricfz}, we now define a version of them
that does not include their symmetric/anti-symmetric relations: 

\begin{definition}
\label{definition:kg}
Define $\fK_g$ to be the $\Q$-vector space with the following presentation:
\begin{itemize}
\item {\bf Generators}.
A generator $\Pres{\kappa_1,\kappa_2}$ for all sym-orthogonal
$\kappa_1,\kappa_2 \in (\wedge^2 H)/\Q$ such that either
$\kappa_1$ or $\kappa_2$ (or both) is a symplectic pair in $(\wedge^2 H_{\Z})/\Z$.
\item {\bf Relations}.  For all symplectic pairs $a \wedge b \in (\wedge^2 H_{\Z})/\Z$ and all
$\kappa_1,\kappa_2 \in (\wedge^2 H)/\Q$ that are sym-orthogonal to $a \wedge b$
and all $\lambda_1,\lambda_2 \in \Q$, the relations
\begin{align*}
\Pres{a \wedge b,\lambda_1 \kappa_1 + \lambda_2 \kappa_2} &= \lambda_1 \Pres{a \wedge b,\kappa_1} + \lambda_2 \Pres{a \wedge b,\kappa_2} \quad \text{and} \\
\Pres{\lambda_1 \kappa_1 + \lambda_2 \kappa_2,a \wedge b} &= \lambda_1 \Pres{\kappa_1,a \wedge b} + \lambda_2 \Pres{\kappa_2,a \wedge b}.\qedhere
\end{align*}
\end{itemize}
\end{definition}

There is an involution $I\colon \fK_g \rightarrow \fK_g$ defined by $I(\Pres{\kappa_1,\kappa_2}) = \Pres{\kappa_2,\kappa_1}$ that
we will call the {\em canonical involution}.  We have:

\begin{lemma}
\label{lemma:decomposekg}
We have $\fK_g = \fK_g^s \oplus \fK_g^a$, where $\fK_g^s$ and $\fK_g^a$ are identified with
the $+1$ and $-1$ eigenspaces of the canonical involution.
\end{lemma}
\begin{proof}
Identical to the proof of Lemma \ref{lemma:decomposespkernel}.
\end{proof}

There is a linearization map $\Phi\colon \fK_g \rightarrow ((\wedge^2 H)/\Q)^{\otimes 2}$ defined
by $\Phi(\Pres{\kappa_1,\kappa_2}) = \kappa_1 \otimes \kappa_2$.  This takes relations to relations,
and thus gives a well-defined map.  Since in the generator $\Pres{\kappa_1,\kappa_2}$ the elements
$\kappa_1$ and $\kappa_2$ are sym-orthogonal, the image of $\Phi$ lies in $\cK_g$.
In light of Lemma \ref{lemma:decomposekg}, Theorems \ref{maintheorem:presentationalt} and \ref{maintheorem:presentationsym}
are equivalent to:

\begin{theorem}
\label{theorem:presentation}
For $g \geq 4$, the linearization map $\Phi\colon \fK_g \rightarrow \cK_g$ is an isomorphism.
\end{theorem}

\subsection{Goal of Part \ref{part:2}}

Our goal in the rest of this paper is to prove Theorem \ref{theorem:presentation}.  Actually,
it will turn out that it is more convenient to prove Theorems \ref{maintheorem:presentationalt} 
and \ref{maintheorem:presentationsym} separately.  We introduced the representation
$\fK_g$ from Theorem \ref{theorem:presentation} for the sake of the calculations
in this part of the paper, which later will give results about $\fK_g^a$ and $\fK_g^s$
that will be needed for the proofs of Theorems \ref{maintheorem:presentationalt}
and \ref{maintheorem:presentationsym}.

Our goal in this part is to enhance $\fK_g$ by showing that in the above
presentation we can add generators $\Pres{\kappa_1,\kappa_2}$ such that
$\kappa_1,\kappa_2 \in (\wedge^2 H)/\Q$ are sym-orthogonal elements with
either $\kappa_1$ or $\kappa_2$ (or both) a symplectic pair or an isotropic pair.\footnote{We will
actually prove something slightly more general.}  Lemma \ref{lemma:decomposekg}
will then imply a corresponding result about $\fK_g^s$ and $\fK_g^a$.  We accomplish this
in \S \ref{section:isotropicpairs5}.  This is preceded by a series of preliminary results
in \S \ref{section:isotropicpairs1} -- \S \ref{section:isotropicpairs4}.

\section{Isotropic pairs I: setup}
\label{section:isotropicpairs1}

This section contains the basic framework for constructing our new generators.

\subsection{Generation by symplectic pairs}
\label{section:generationomega}

We start with a technical lemma.
Let $X$ be a direct summand of $H_{\Z}$.  Define $\ker(X)$ to be the subspace of all $x_0 \in X$ such that
$\omega(x_0,x) = 0$ for all $x \in X$.  The rank of $\ker(X)$ is the {\em kernel rank} of $X$.

The restriction of $\omega$ to $X$ induces an alternating bilinear 
form $\iota$ on $X/\ker(X)$, and we say that $X$ is a {\em near symplectic summand} of $H_{\Z}$
if $\iota$ is a symplectic form.  This is equivalent to requiring that there be a symplectic basis $\{a_1,b_1,\ldots,a_g,b_g\}$
for $H_{\Z}$ such that $X = \Span{a_1,b_1,\ldots,a_h,b_h,a_{h+1},\ldots,a_{h+k}}$ for some $h \leq g$ and $k \leq g-h$.
The integer $k$ is the kernel rank of $X$, and we call $h$ the {\em genus} of $X$.  Here is an example of this:

\begin{lemma}
\label{lemma:isotropicnear}
Let $I$ be a rank-$k$ subgroup of $H_{\Z}$ on which $\omega$ vanishes identically.  Then
$I^{\perp}$ is a near symplectic summand of genus $g-k$ and kernel rank $k$.
\end{lemma}
\begin{proof}
Let $\oI$ be the saturation of $I$ in $H_{\Z}$.  Since $I_{\Q} = \oI_{\Q}$ we have that
$I^{\perp} = \oI^{\perp}$.  Lemma \ref{lemma:summand} implies that $\oI$ is a direct summand
of $H_{\Z}$.  We can therefore find a symplectic basis $\{a_1,b_1,\ldots,a_g,b_g\}$ for
$H_{\Z}$ such that $\oI = \Span{a_1,a_2,\ldots,a_k}$.  It follows that
\[I^{\perp} = \oI^{\perp} = \Span{a_1,a_2,\ldots,a_k,a_{k+1},b_{k+1},\ldots,a_g,b_g}.\]
The lemma follows.
\end{proof}

Our main result about near symplectic summands is:

\begin{lemma}
\label{lemma:generationomega}
Let $X$ be a near symplectic summand of $H_{\Z}$ of genus $h \geq 1$.  Then
$\overline{\wedge^2 X_{\Q}}$ is spanned by symplectic pairs $\sigma$
with $\sigma \in \overline{\wedge^2 X_{\Q}}$.
\end{lemma}
\begin{proof}
Let $k$ be the kernel rank of $X$ and let $\{a_1,b_1,\ldots,a_g,b_g\}$ be a symplectic basis for $H_{\Z}$ such that
$X = \Span{a_1,b_1,\ldots,a_h,b_h,a_{h+1},\ldots,a_{h+k}}$.  The
vector space $\overline{\wedge^2 X_{\Q}}$ is spanned by elements of the form
$x \wedge y$ with $x,y \in \{{a_1,b_1,\ldots,a_h,b_h,a_{h+1},\ldots,a_{h+k}}\}$ distinct.  We must write each of these
as a linear combination of symplectic pairs $\sigma$ with $\sigma \in \overline{\wedge^2 X_{\Q}}$.

Up to flipping $x$ and $y$, there are several cases.  In each of them, we will use blue to denote
symplectic pairs $\sigma$ with $\sigma \in \overline{\wedge^2 X_{\Q}}$.
\begin{itemize}
\item If $\omega(x,y) = 1$, then we have $x = a_i$ and $y = b_i$ for some $1 \leq i \leq h$ and $\blue{x \wedge y}$ is already
of the desired form.
\item If $\omega(x,y) = 0$ and $x \in \{a_1,b_1,\ldots,a_h,b_h\}$ and $y \in \{a_1,b_1,\ldots,a_h,b_h,a_{h+1},\ldots,a_{h+k}\}$,
then for $1 \leq i \leq h$ we have either:
\begin{align*}
x \wedge y &= a_i \wedge y = \blue{a_i \wedge (b_i + y)} - \blue{a_i \wedge b_i},\text{ or}\\
x \wedge y &= b_i \wedge y = -\blue{(a_i + y) \wedge b_i} + \blue{a_i \wedge b_i}.
\end{align*}
\item If $x,y \in \{a_{h+1},\ldots,a_{h+k}\}$, then we have
\[x \wedge y = \blue{(a_1 + x) \wedge (b_1 + y)} - \blue{a_1 \wedge (b_1 + y)} - \blue{(a_1+x) \wedge b_1} + \blue{a_1 \wedge b_1}.\qedhere\]
\end{itemize}
\end{proof}

\subsection{Right compatible subspaces}
\label{section:compatiblekg}

Let $a \wedge a'$ 
be an isotropic pair.  Define $\fK_g[-,a \wedge a']$ to be the subspace of $\fK_g$
spanned by elements of the form $\Pres{\sigma,a \wedge a'}$ with $\sigma$ a symplectic pair
that is sym-orthogonal to $a \wedge a'$.
Let $\Phi\colon \fK_g \rightarrow ((\wedge^2 H)/\Q)^{\otimes 2}$ be the linearization map.
Our main technical result will be that
$\Phi$ takes $\fK_g[-,a \wedge a']$ isomorphically onto $(a \wedge a')^{\perp} \otimes (a \wedge a')$.
The proof of this is spread over \S \ref{section:isotropicpairs1} -- \S \ref{section:isotropicpairs3}, with the result
being Proposition \ref{proposition:isotropicpairiso}.  
In \S \ref{section:isotropicpairs4}, we use this to construct our new generators.

\begin{remark}
We could also define $\fK_g[a \wedge a',-]$ similarly, and all of our results would
have analogues for $\fK_g[a \wedge a',-]$.  To avoid repetition, we will focus on
$\fK_g[-,a \wedge a']$ and then at the very end formally derive these
analogues; see \S \ref{section:leftcompatible}.
\end{remark}

\subsection{Calculating the image}

We start by proving:

\begin{lemma}
\label{lemma:isotropiccompatibleimage}
Let $a \wedge a'$ be an isotropic pair and let
$\Phi\colon \fK_g \rightarrow ((\wedge^2 H)/\Q)^{\otimes 2}$ be the linearization map.  Then
$\Phi$ takes $\fK_g[-,a \wedge a']$ onto $(a \wedge a')^{\perp} \otimes (a \wedge a')$.
\end{lemma}
\begin{proof}
By definition, $\fK_g[-,a \wedge a']$ is spanned by elements of the form
$\Pres{\sigma,a \wedge a'}$ with $\sigma$ a symplectic pair such that
$\sigma \in (a \wedge a')^{\perp}$.  Since $\Phi(\Pres{\sigma,a \wedge a'}) = \sigma \otimes (a \wedge a')$, this implies that
\[\Phi(\fK_g[-,a \wedge a']) \subset (a \wedge a')^{\perp} \otimes (a \wedge a').\]
To see that this is an equality, let $I = \Span{a,a'}$.  Lemma \ref{lemma:isotropicorthogonal} says that
$(a \wedge a')^{\perp} = \overline{\wedge^2 I_{\Q}^{\perp}}$.
Lemma \ref{lemma:isotropicnear} says that $I^{\perp}$ is a near symplectic summand of 
genus\footnote{Here we are using our standing assumption that $g \geq 4$; see
Assumption \ref{assumption:genuspart2}.} $g-2 \geq 1$.
Lemma \ref{lemma:generationomega} therefore implies that $\overline{\wedge^2 I_{\Q}^{\perp}}$
is spanned by symplectic pairs $\sigma$ such that $\sigma \in \overline{\wedge^2 I_{\Q}^{\perp}}$.  The desired
equality follows.
\end{proof}

\section{Isotropic pairs II: lifting orthogonal elements}
\label{section:isotropicpairs2}

Let $a \wedge a'$ be an isotropic pair and
let $\kappa \in (a \wedge a')^{\perp}$.  In
this section, for certain $\kappa$ we show how to find specific elements of $\fK_g[-,a \wedge a']$
projecting to $\kappa \otimes (a \wedge a')$.

\subsection{Separating classes}

A subgroup $X$ of $H_{\Z}$ is said to {\em separate} $\kappa$ from $a \wedge a'$ if:
\begin{itemize}
\item $X \subset \Span{a,a'}^{\perp}$; and
\item $\kappa \in \overline{\wedge^2 X_{\Q}}$; and
\item $X$ is a near symplectic summand of $H_{\Z}$ of positive genus.  This implies in particular
that $X$ is a direct summand of $H_{\Z}$.
\end{itemize}
Let $X$ be a direct summand of $H_{\Z}$ separating $\kappa$ from $a \wedge a'$.  Use Lemma \ref{lemma:generationomega} to write
\begin{equation}
\label{eqn:expresskappax}
\kappa = \sum_{i=1}^n \lambda_i \sigma_i \quad \text{with $\lambda_i \in \Q$ and $\sigma_i$ a symplectic pair with $\sigma_i \in \overline{\wedge^2 X_{\Q}}$}.
\end{equation}

\subsection{Constructing the lift}
\label{section:constructlift}

We would like to define
\[\OPres{(\kappa;X),a \wedge a'} = \sum_{i=1}^n \lambda_i \Pres{\sigma_i,a \wedge a'} \in \fK_g.\]
This is in orange to emphasize that it is not one of our generators.  It appears to depend on
the expression \eqref{eqn:expresskappax}, but below we will prove that under favorable
circumstances it does not depend on this expression.  

To state our result, recall that a {\em Lagrangian} in $H_{\Z}$ is a direct
summand $L$ with $L^{\perp} = L$.  Equivalently, we can find a symplectic basis
$\{a_1,b_1,\ldots,a_g,b_g\}$ for $H_{\Z}$ with $L = \Span{a_1,\ldots,a_g}$.  We
say that $X$ is {\em Lagrangian-free} if $X$ does not contain a Lagrangian of $H_{\Z}$.  Then:

\begin{lemma}
\label{lemma:orangeindependent}
Let the notation be as above, and assume that $X$ is Lagrangian-free.  Then $\OPres{(\kappa;X),a \wedge a'}$ does
not depend on \eqref{eqn:expresskappax}.
\end{lemma}
\begin{proof}
Let
\[\kappa = \sum_{j=1}^m \lambda'_j \sigma'_j \quad \text{with $\lambda'_j \in \Q$ and $\sigma'_j$ a symplectic pair with $\sigma'_j \in \overline{\wedge^2 X_{\Q}}$}\]
be another expression.  We must prove that
\begin{equation}
\label{eqn:orangeindependent}
\sum_{i=1}^n \lambda_i \Pres{\sigma_i,a \wedge a'} = \sum_{j=1}^m \lambda'_j \Pres{\sigma'_j,a \wedge a'}.
\end{equation}
Let $h \geq 1$ be the genus of $X$ and let $k$ be the kernel rank of $X$.  Pick a symplectic
basis $\{a_1,b_1,\ldots,a_g,b_g\}$ for $H_{\Z}$ with $X = \Span{a_1,b_1,\ldots,a_h,b_h,a_{h+1},\ldots,a_{h+k}}$.
We then have
\[X^{\perp} = \Span{a_{h+1},\ldots,a_{h+k},a_{h+k+1},b_{h+k+1},\ldots,a_g,b_g}.\]
Since $X$ is Lagrangian-free, we have $h+k < g$.  It follows that $X^{\perp}$ is a near symplectic
summand of $H_{\Z}$ of positive genus.  By assumption, $a,a' \in X^{\perp}$.  Using
Lemma \ref{lemma:generationomega}, we can write
\[a \wedge a' = \sum_{\ell=1}^p c_{\ell} s_{\ell} \quad \text{with $c_{\ell} \in \Q$ and $s_{\ell}$ a symplectic pair with $s_{\ell} \in \overline{\wedge^2 X_{\Q}^{\perp}}$}.\]
We have the following relation in $\fK_g$:
\[\Pres{\sigma_i,a \wedge a'} = \sum_{\ell=1}^p c_{\ell} \Pres{\sigma_i,s_{\ell}}.\]
Using the relations in $\fK_h$ again, it follows that
\[\sum_{i=1}^n \lambda_i \Pres{\sigma_i,a \wedge a'} = \sum_{\ell=1}^p \left(c_{\ell} \sum_{i=1}^n \lambda_i \Pres{\sigma_i,s_{\ell}}\right)
                                                     = \sum_{\ell=1}^p \left(c_{\ell} \Pres{\sum_{i=1}^n \lambda_i \sigma_i,s_{\ell}}\right) 
                                                     = \sum_{\ell=1}^p c_{\ell} \Pres{\kappa,s_{\ell}}.\]
Similarly, we have
\[\sum_{j=1}^m \lambda'_j \Pres{\sigma'_j,a \wedge a'} = \sum_{\ell=1}^p c_{\ell} \Pres{\kappa,s_{\ell}}.\]
The equality \eqref{eqn:orangeindependent} follows.
\end{proof}

\subsection{Properties of the lift}
\label{section:liftproperties}

We now give three properties of our lifts.  The first is linearity:

\begin{lemma}[Linearity of the lifts]
\label{lemma:liftadd}
Let $a \wedge a'$ be an isotropic pair, let
$\kappa_1,\kappa_2 \in (a \wedge a')^{\perp}$, and 
let $\lambda_1,\lambda_2 \in \Q$.  Let $X$ be a direct summand of $H_{\Z}$ that is Lagrangian-free and
separates both $\kappa_1$ and $\kappa_2$ from $a \wedge a'$.  Then
\[\OPres{(\lambda_1 \kappa_1+\lambda_2 \kappa_2;X),a \wedge a'} = \lambda_1 \OPres{(\kappa_1;X),a \wedge a'} + \lambda_2 \OPres{(\kappa_2;X),a \wedge a'}.\]
\end{lemma}
\begin{proof}
By taking the corresponding linear combination of the expressions in $\fK_g[-,a \wedge a']$ that we used to
write $\OPres{(\kappa_1;X);a \wedge a'}$ and $\OPres{(\kappa_2;X),a \wedge a'}$, we obtain an expression
that can be used to write $\OPres{(\lambda_1 \kappa_1+\lambda_2 \kappa_2;X),a \wedge a'}$.  The lemma follows.
\end{proof}

The second is equivariance.  For a strong isotropic pair $a \wedge a'$ and $f \in \Sp_{2g}(\Z)$, note that
$f(a) \wedge f(a')$ is another strong isotropic pair.  The group $\Sp_{2g}(\Z)$ also acts on
$\fK_g$, and $f$ takes $\fK_g[-,a \wedge a']$ to $\fK_g[-,f(a) \wedge f(a')]$.  We have:

\begin{lemma}[Equivariance of the lifts]
\label{lemma:liftequivariance}
Let $a \wedge a'$ be an isotropic pair and let $\kappa \in (a \wedge a')^{\perp}$.
Let $X$ be a direct summand of $H_{\Z}$ that is Lagrangian-free and separates $\kappa$ from $a \wedge a'$.  Then
for all $f \in \Sp_{2g}(\Z)$ we have
$f(\OPres{(\kappa;X),a \wedge a'}) = \OPres{(f(\kappa);f(X)),f(a) \wedge f(a')}$.
\end{lemma}
\begin{proof}
The map $f$ takes the expression in $\fK_g[-,a \wedge a']$ we used to write $\OPres{(\kappa;X),a \wedge a'}$ to one that can
be used to write $\OPres{(f(\kappa);f(X)),f(a) \wedge f(a')}$.  The lemma follows.
\end{proof}

Our final lemma lets us change $X$:

\begin{lemma}[Changing the separator in the lifts]
\label{lemma:liftchange}
Let $a \wedge a'$ be an isotropic pair and let $\kappa \in (a \wedge a')^{\perp}$.
Let $X$ and $X'$ be direct summands of $H_{\Z}$ that are Lagrangian-free and separate $\kappa$ from $a \wedge a'$.  Assume
that $X \subset X'$.  Then $\OPres{(\kappa;X),a \wedge a'} = \OPres{(\kappa;X'),a \wedge a'}$.
\end{lemma}
\begin{proof}
Since $X \subset X'$, the expression in $\fK_g[-,a \wedge a']$ we used to write $\OPres{(\kappa;X),a \wedge a'}$ can
also be used to write $\OPres{(\kappa;X'),a \wedge a'}$.  The lemma follows.
\end{proof}

\subsection{Symplectic automorphism group}
\label{section:symplecticautomorphism}

We pause now to prove a lemma about the symplectic group.
Let $I$ be a rank-$k$ direct summand of $H_{\Z}$ on which $\omega$ vanishes identically.  Let $\Sp_{2g}(\Z,I)$ be the subgroup
of all $f \in \Sp_{2g}(\Z)$ such that $f$ fixes $I$ pointwise.  The group $\Sp_{2g}(\Z,I)$ acts on
$I^{\perp}$.  Let $\Sp_{2g}(\Z,I)|_{I^{\perp}}$ be the image of $\Sp_{2g}(\Z,I)$ in $\Aut(I^{\perp})$.  Lemma \ref{lemma:isotropicnear}
says that $I^{\perp}$ is a near symplectic summand of $H_{\Z}$ of genus $g-k$, so we can find a symplectic summand $X$ of
$H_{\Z}$ of genus $g-k$ such that $I^{\perp} = X \oplus I$.  We have:

\begin{lemma}
\label{lemma:symplecticstabilizer}
Let $X$ and $I$ be as above.  We then have a semidirect product decomposition
\[\Sp_{2g}(\Z,I)|_{I^{\perp}} = \Hom(X,I) \ltimes \Sp(X),\]
where for $\lambda \in \Hom(X,I)$ the associated element $f \in \Sp_{2g}(\Z,I)|_{I^{\perp}}$ satisfies
$f(x) = x + \lambda(x)$ for all $x \in X$.
\end{lemma}
\begin{proof}
Set $\Gamma = \Sp_{2g}(\Z,I)|_{I^{\perp}}$.  
The action of $\Gamma$ on $I^{\perp}$ descends to an action
on $I^{\perp}/I$.  The symplectic form on $H_{\Z}$ induces a symplectic form on
$I^{\perp}/I$.  We thus get a homomorphism
\[\rho\colon \Gamma \longrightarrow \Sp(I^{\perp}/I).\]
The map $I^{\perp} \rightarrow I^{\perp}/I$ restricts to an isomorphism $X \cong I^{\perp}/I$.  Identifying
$\Sp(X)$ with the subgroup of $\Gamma$ consisting of automorphisms
that act trivially on $X^{\perp}$, the homomorphism $\rho$ splits via the map
\[\begin{tikzcd}
\Sp(I^{\perp}/I) \cong \Sp(X) \arrow[hook]{r} & \Gamma.
\end{tikzcd}\]
We therefore get a semidirect product decomposition
\[\Gamma = \ker(\rho) \ltimes \Sp(X).\]
To identify $\ker(\rho)$ with $\Hom(X,I)$, consider $f \in \ker(\rho)$.  By definition, for $x \in X$ we have
$f(x)-x \in I$.  We can therefore define a homomorphism
$\lambda_f\colon X \rightarrow I$ via the formula
$\lambda_f(x) = f(x) - x$.  
If $\lambda_f = 0$, then $f$ fixes both $X$ and $I$, so it fixes $I^{\perp} = X \oplus I$ and is the identity.

The map $f \mapsto \lambda_f$ is thus an injective homomorphism from $\ker(\rho)$ to $\Hom(X,I)$.  We
remark that the fact that is a homomorphism uses the fact that $f(y) = y$ for all $y \in I$.  To see
that it is a surjection and thus an isomorphism, consider $\lambda \in \Hom(X,I)$.  
Let $\{a_1,b_1,\ldots,a_k,b_k\}$ be a symplectic basis for $X^{\perp}$ such that
$I = \Span{a_1,a_2,\ldots,a_k}$.
Set $J = \Span{b_1,b_2,\ldots,b_k}$.
Since $\omega$ restricts to a symplectic form
on $X$, it identifies $X$ with $\Hom(X,\Z)$.  There is thus a unique homomorphism
$\delta\colon J \rightarrow X$ such that
\[\omega(x,\delta(z)) = - \omega(\lambda(x),z) \quad \text{for all $x \in X$ and $z \in J$}.\]
Since $X$ and $J$ are orthogonal to each other, for $x \in X$ and $z \in J$ we have
\[\omega(x+\lambda(x),z+\delta(z)) = \omega(\lambda(x),z) + \omega(x,\delta(z)) = 0.\]
In other words, the map $f\colon H_{\Z} \rightarrow H_{\Z}$ defined by
$f(x) = x + \lambda(x)$ for $x \in X$ and
\[f(a_i) = a_i \quad \text{and} \quad f(b_i) = b_i + \delta(b_i)\]
for $1 \leq i \leq k$
is an element of $\Sp_{2g}(\Z,I)$ whose restriction to $I^{\perp}$ satisfies $\lambda = \lambda_f$.  The lemma follows.
\end{proof}

\subsection{Fixed lift}

Let $a \wedge a'$ be a strong isotropic pair, so $I = \Span{a,a'}$ is a direct summand of $H_{\Z}$.  There is a special element
$(a \wedge a') \otimes (a \wedge a')$ in $(a \wedge a')^{\perp} \otimes (a \wedge a')$
that is fixed by $\Sp_{2g}(\Z,I)$.  We close this section by showing
that we can lift this to an element of $\fK_g[-,a \wedge a']$ that is fixed
by $\Sp_{2g}(\Z,I)$.
To state our result, let $\{a_1,b_1,\ldots,a_g,b_g\}$ be a symplectic basis for $H_{\Z}$ with\footnote{Indexing it like this rather
than $a_1 = a$ and $a_2 = a'$ will simplify our notation later.} $a_{g-1} = a$ and $a_g = a'$.  For
$1 \leq i \leq g-2$, let $W_i = \Span{a_i,b_i,a_{g-1},a_g}$.  We then have:

\begin{lemma}
\label{lemma:liftspecial}
Let $a_{g-1}$ and $a_g$ and $W_i$ be as above.  The following then hold:
\begin{itemize}
\item For $1 \leq i,j \leq g-2$ we have 
\[\OPres{(a_{g-1} \wedge a_g;W_i),a_{g-1} \wedge a_g} = \OPres{(a_{g-1} \wedge a_g;W_j),a_{g-1} \wedge a_g}.\]
\item For $1 \leq i \leq g-2$, the group $\Sp_{2g}(\Z,I)$ fixes $\OPres{(a_{g-1} \wedge a_g;W_i),a_{g-1} \wedge a_g}$.
\end{itemize}
\end{lemma}
\begin{proof}
For $\kappa \in (a \wedge a')^{\perp}$ and a summand $X$ of $H_{\Z}$ that is Lagrangian-free and separates
$\kappa$ from $a_{g-1} \wedge a_g$, we will drop the $a_{g-1} \wedge a_g$ from our notation
and write $\OPres{\kappa;X}$ instead of $\OPres{(\kappa;X),a_{g-1} \wedge a_g}$.  
We encourage the reader to verify that all
the summands appearing in our calculations are Lagrangian-free and separate the appropriate elements
of $(a \wedge a')^{\perp}$ from $a_{g-1} \wedge a_g$.  In particular, they all have
positive genus.  We have:

\begin{unnumberedclaim}
For $1 \leq i,j \leq g-2$, we have $\OPres{a_{g-1} \wedge a_g;W_i} = \OPres{a_{g-1} \wedge a_g;W_j}$.
\end{unnumberedclaim}

Using Lemma \ref{lemma:liftadd} (linearity of the lifts), we have
\begin{align*}
\OPres{a_{g-1} \wedge a_g;W_i} &= \OPres{(a_i+a_{g-1}) \wedge a_g;W_i} - \OPres{a_i \wedge a_g;W_i},\\
\OPres{a_{g-1} \wedge a_g;W_j} &= \OPres{(a_j+a_{g-1}) \wedge a_g;W_j} - \OPres{a_j \wedge a_g;W_j}.
\end{align*}
To prove that these are equal, we must show that
\begin{equation}
\label{eqn:liftspecial1}
\OPres{(a_i+a_{g-1}) \wedge a_g;W_i} + \OPres{a_j \wedge a_g;W_j} = \OPres{(a_j+a_{g-1}) \wedge a_g;W_j} + \OPres{a_i \wedge a_g;W_i}.
\end{equation}
Using Lemma \ref{lemma:liftchange} (changing the separator in the lifts), we have
\begin{align*}
\OPres{(a_i+a_{g-1}) \wedge a_g;W_i} &= \OPres{(a_i+a_{g-1}) \wedge a_g;\Span{a_i+a_{g-1},b_i,a_g}} \\
                                     &= \OPres{(a_i+a_{g-1}) \wedge a_g;\Span{a_i+a_{g-1},b_i,a_j,b_j,a_g}},\\
\OPres{a_j \wedge a_g;W_j}           &= \OPres{a_j \wedge a_g;\Span{a_j,b_j,a_g}} \\
                                     &= \OPres{a_j \wedge a_g;\Span{a_i+a_{g-1},b_i,a_j,b_j,a_g}}.
\end{align*}
Adding these and using Lemmas \ref{lemma:liftadd} (linearity of the lifts) and \ref{lemma:liftchange} (changing the separator in the lifts),
we see that\footnote{The term
$\Span{a_i+a_j+a_{g-1},b_i,b_j,a_g}$ appearing here is a near symplectic summand even though the given basis does not
reflect this.}
\begin{align*}
\OPres{(a_i+a_{g-1}) \wedge a_g;W_i} + \OPres{a_j \wedge a_g;W_j} &= \OPres{(a_i+a_j+a_{g-1}) \wedge a_g;\Span{a_i+a_{g-1},b_i,a_j,b_j,a_g}} \\
                                                                  &= \OPres{(a_i+a_j+a_{g-1}) \wedge a_g;\Span{a_i+a_j+a_{g-1},b_i,b_j,a_g}}.
\end{align*}
Similarly, we have
\[\OPres{(a_j+a_{g-1}) \wedge a_g;W_j} + \OPres{a_i \wedge a_g;W_i} = \OPres{(a_i+a_j+a_{g-1}) \wedge a_g;\Span{a_i+a_j+a_{g-1},b_i,b_j,a_g}}.\]
The identity \eqref{eqn:liftspecial1} follows.

\begin{unnumberedclaim}
For $1 \leq i \leq g-2$, the group $\Sp_{2g}(\Z,I)$ fixes $\OPres{a_{g-1} \wedge a_g;W_i}$.
\end{unnumberedclaim}

As we noted in \S \ref{section:symplecticautomorphism}, the action of $\Sp_{2g}(\Z,I)$ on our lifts
factors through $\Gamma = \Sp_{2g}(\Z,I)|_{I^{\perp}}$.
Lemma \ref{lemma:symplecticstabilizer} says that
\[\Gamma = \Hom(\Z^{2(g-2)},I) \ltimes \Sp_{2(g-2)}(\Z).\]
We must prove that the subgroups $\Hom(\Z^{2(g-2)},I)$ and $\Sp_{2(g-2)}(\Z)$ both fix
$\OPres{a_{g-1} \wedge a_g;W_i}$.

We start with $\Sp_{2(g-2)}(\Z)$.  It is classical that $\Sp_{2(g-2)}(\Z)$ is generated by
the stabilizer of $a_1$ and the stabilizer of $a_2$.  For instance,\footnote{This could also be deduced from the generating set of Hua--Reiner \cite{HuaReiner} discussed
in \S \ref{section:spgen}, but be warned that their generating set does not consist of elements that fix either $a_1$ or $a_2$.}
the mapping class group $\Mod_{g-2}$ surjects onto $\Sp_{2(g-2)}(\Z)$, and choosing a basis for
$\HH_1(\Sigma_{g-2})$ appropriately the usual Dehn twist generators for $\Mod_{g-2}$ from \cite[Theorem 4.13]{FarbMargalitPrimer}
each fix either $a_1$ or $a_2$.  It is thus enough to prove that both of these stabilizers fix
$\OPres{a_{g-1} \wedge a_g;W_i}$.  

The proofs for both are similar, so we will give the details for the stabilizer of $a_2$ and leave the other
case to the reader.  Consider $f \in \Sp_{2(g-2)}(\Z)$ with $f(a_2) = a_2$.  By the previous claim, it is enough to prove that
$f$ fixes $\OPres{a_{g-1} \wedge a_g;W_2}$.  Using Lemma \ref{lemma:liftadd} (linearity of the lifts), we have
\[\OPres{a_{g-1} \wedge a_g;W_2} = \OPres{(a_2 + a_{g-1}) \wedge a_g;W_2} - \OPres{a_2 \wedge a_g;W_2}.\]
We will prove that $f$ fixes $\OPres{(a_2 + a_{g-1}) \wedge a_g;W_2}$ and $\OPres{a_2 \wedge a_g;W_2}$.

For the first, Lemma \ref{lemma:liftchange} (changing the separator in the lifts) says that
\begin{align*}
\OPres{(a_2 + a_{g-1}) \wedge a_g;W_2} &= \OPres{(a_2 + a_{g-1}) \wedge a_g;\Span{a_2+a_{g-1},b_2,a_g}} \\
                                       &= \OPres{(a_2 + a_{g-1}) \wedge a_g;\Span{a_1,b_1,a_2+a_{g-1},b_2,a_g}}.
\end{align*}
Lemma \ref{lemma:liftequivariance} (equivariance of the lifts) says that $f$ takes this to
\begin{align*}
&\OPres{(f(a_2) + f(a_{g-1})) \wedge f(a_g);f(\Span{a_1,b_1,a_2+a_{g-1},b_2,a_g})} \\
&\quad\quad\quad\quad\quad\quad\quad\quad= \OPres{(a_2 + a_{g-1}) \wedge a_g;\Span{a_1,b_1,a_2+a_{g-1},b_2,a_g}},
\end{align*}
as desired.

For the second, Lemma \ref{lemma:liftchange} (changing the separator in the lifts) says that
\begin{align*}
\OPres{a_2 \wedge a_g;W_2} &= \OPres{a_2 \wedge a_g;\Span{a_2,b_2,a_g}} \\
                           &= \OPres{a_2 \wedge a_g;\Span{a_1,b_1,a_2,b_2,a_g}}.
\end{align*}
Lemma \ref{lemma:liftequivariance} (equivariance of the lifts) says that $f$ takes this to
\[\OPres{f(a_2) \wedge f(a_g);f(\Span{a_1,b_1,a_2,b_2,a_g})} 
                                                     = \OPres{a_2 \wedge a_g;\Span{a_1,b_1,a_2,b_2,a_g}},\]
as desired.

It remains to prove that the subgroup $\Hom(\Z^{2(g-2)},I)$ of $\Gamma$ fixes our lift.  Observe 
that $\Hom(\Z^{2(g-2)},I)$ is generated by elements that fix all but one element of the basis
$\{a_1,b_1,\ldots,a_{g-2},b_{g-2}\}$.  It is enough to prove that such elements fix our lift.
Consider $f \in \Hom(\Z^{2(g-2)},I)$ that fixes all elements of $\{a_1,b_1,\ldots,a_{g-2},b_{g-2}\}$
except for $x \in \{a_j,b_j\}$.  Letting $1 \leq i \leq g-2$ be such that $i \neq j$, 
it is enough to prove that $f$ fixes $\OPres{a_{g-1} \wedge a_g;W_i}$.  But this is immediate from the
fact that $f$ fixes $a_{g-1}$ and $a_g$ and $W_i = \Span{a_i,b_i,a_{g-1},a_g}$.
\end{proof}

\section{Isotropic pairs III: isomorphism theorem}
\label{section:isotropicpairs3}

We now prove the following theorem using the proof outline from \S \ref{section:prooftechnique}.

\begin{proposition}
\label{proposition:isotropicpairiso}
Let $a \wedge a'$ be an isotropic pair and let $\Phi\colon \fK_g \rightarrow ((\wedge^2 H)/\Q)^{\otimes 2}$ be the
linearization map.  Then $\Phi$ takes $\fK_g[-,a \wedge a']$ isomorphically to $(a \wedge a')^{\perp} \otimes (a \wedge a')$.
\end{proposition}
\begin{proof}
By Lemma \ref{lemma:strongisotropic}, there exists a strong isotropic pair $a_0 \wedge a'_0$ and $n \in \Z$ such that
$a \wedge a' = n a_0 \wedge a'_0$.  Moreover, $\Span{a_0,a'_0}_{\Q} = \Span{a,a'}_{\Q}$, so
by Lemma \ref{lemma:isotropicorthogonal} we have $(a \wedge a')^{\perp} = (a_0 \wedge a'_0)^{\perp}$.  Using the linearity relations in $\fK_g$, 
multiplication by $n$ gives an isomorphism
$\fK_g[-,a_0 \wedge a_0'] \cong \fK_g[-,a \wedge a']$ taking a generator $\Pres{\sigma,a_0 \wedge a_0'}$ with
$\sigma$ a symplectic pair in $(a_0 \wedge a'_0)^{\perp}$ to a
generator $\Pres{\sigma,a \wedge a'}$.  It is thus enough to prove the proposition for
$a_0 \wedge a'_0$.  Replacing $a \wedge a'$ with $a_0 \wedge a'_0$, we can therefore assume that
$a \wedge a'$ is a strong isotropic pair.

To simplify our notation, we will drop $a \wedge a'$ from our notation in two places:
\begin{itemize}
\item For $\kappa \in (a \wedge a')^{\perp}$
and $X$ a Lagrangian-free direct summand of $H_{\Z}$ that separates $\kappa$ from $a \wedge a'$, we will drop the $a \wedge a'$ from our notation
and write $\OPres{\kappa;X}$ instead of $\OPres{(\kappa;X),a \wedge a'}$.
\item We will also drop the $a \wedge a'$ from our notation for the codomain of the restriction of $\Phi$ to $\fK_g[-,a \wedge a']$.  Thus for 
$\OPres{\kappa;X}$ as in the previous bullet point, we will write $\Phi(\OPres{\kappa;X}) = \kappa$ rather than $\Phi(\OPres{\kappa;X}) = \kappa \otimes (a \wedge a')$.
\end{itemize}
The proof has three steps.

\begin{step}{1}
\label{step:isotropicpairsiso1}
We construct a set $S \subset \fK_g[-,a \wedge a']$ such that the restriction of $\Phi$ to $\Span{S}$ is an isomorphism
to $(a \wedge a')^{\perp} \otimes (a \wedge a')$.
\end{step}

Let $\cB = \{a_1,b_1,\ldots,a_g,b_g\}$ be a symplectic basis for $H_{\Z}$ with
$a_{g-1} = a$ and $a_g = a'$.  Letting $I = \Span{a,a'}$, we have
\[I_{\Q}^{\perp} = \Span{a_1,b_1,\ldots,a_{g-2},b_{g-2},a_{g-1},a_g}.\]
Set 
\[\cB' = \{a_1,b_1,\ldots,a_{g-2},b_{g-2}\},\] 
with the total order $\prec$ as indicated in this list.  Lemma \ref{lemma:isotropicorthogonal} says that
\[(a \wedge a')^{\perp} = \overline{\wedge^2 I_{\Q}^{\perp}}.\]
This vector space has the basis\footnote{This clearly forms
a basis for $\wedge^2 I_{\Q}^{\perp}$, and since the restriction of the map $\wedge^2 H \rightarrow (\wedge^2 H)/\Q$ 
to $\wedge^2 I_{\Q}^{\perp}$ is an injection the image of $T$ in $(\wedge^2 H)/\Q$ also forms a basis
for $\overline{\wedge^2 I_{\Q}^{\perp}}$.}
\[T = \Set{$x \wedge y$}{$x,y \in \cB'$, $x \prec y$} \cup \Set{$x \wedge a_{g-1}$, $x \wedge a_g$}{$x \in \cB'$} \cup \{a_{g-1} \wedge a_g\}.\]
Define
\[X = \Span{\cB'} \quad \text{and} \quad Y = \Span{\cB',a_{g-1}} \quad \text{and} \quad Z = \Span{\cB',a_g}\]
and
\[W_i = \Span{a_i,b_i,a_{g-1},a_g} \quad \text{for $1 \leq i \leq g-2$}.\]
These are all Lagrangian-free near symplectic summands of $H_{\Z}$ of positive genus. 
Finally, define
\begin{align*}
S = &\Set{$\OPres{x \wedge y;X}$}{$x,y \in \cB'$, $x \prec y$} \\
    &\quad\quad\cup \Set{$\OPres{x \wedge a_{g-1};Y}$, $\OPres{x \wedge a_g;Z}$}{$x \in \cB'$} \cup \{\OPres{a_{g-1} \wedge a_g;W_1}\}.
\end{align*}
By construction, $\Phi$ takes $S$ bijectively to $T$.  Since $T$ is a basis for $(a \wedge a')^{\perp}$, it follows that
the restriction of $\Phi$ to $\Span{S}$ is an isomorphism.

\begin{step}{2}
\label{step:isotropicpairsiso2}
We prove that the $\Sp_{2g}(\Z)$-orbit of $S$ spans $\fK_g[-,a_{g-1} \wedge a_g]$.
\end{step}

By definition, $\fK_g[-,a_{g-1} \wedge a_g]$ is spanned by elements of the form
$\Pres{\sigma,a_{g-1} \wedge a_g}$, where $\sigma$ is a symplectic pair
with $\sigma \in (a \wedge a')^{\perp}$.  The image of $\Phi$ contains some elements of this form; for instance,
it contains
\[\OPres{a_1 \wedge b_1;X} = \Pres{a_1 \wedge b_1,a_{g-1} \wedge a_g}.\]
Since $\Sp_{2g}(\Z,I)$ acts transitively on symplectic pairs lying in $(a \wedge a')^{\perp} = \overline{\wedge^2 I_{\Q}^{\perp}}$,
it follows that the $\Sp_{2g}(\Z)$-orbit of $S$ spans $\fK_g[-,a_{g-1} \wedge a_g]$..

\begin{step}{3}
\label{step:isotropicpairsiso3}
We prove that $\Sp_{2g}(\Z)$ takes $\Span{S}$ to itself. 
By Step \ref{step:isotropicpairsiso2} this will imply that $\Span{S} = \fK_g[-,a_{g-1} \wedge a_g]$,
and thus by Step \ref{step:isotropicpairsiso1} that $\Phi$ is an isomorphism.
\end{step}

The action of $\Sp_{2g}(\Z,I)$ on $\fK_g[-,a_{g-1} \wedge a_g]$ factors through
$\Gamma = \Sp_{2g}(\Z,I)|_{I^{\perp}}$, and by
Lemma \ref{lemma:symplecticstabilizer} we have
\[\Gamma = \Hom(X,I) \ltimes \Sp_{2(g-2)}(\Z).\]
We must prove that $\Hom(X,I)$ and $\Sp_{2(g-2)}(\Z)$ both take $\Span{S}$ to itself.  We
divide this into two claims:

\begin{claim}{3.1}
The action of $\Sp_{2(g-2)}(\Z)$ on $\fK_g[-,a_{g-1} \wedge a_g]$ takes $\Span{S}$ to itself.
\end{claim}

\noindent
Consider $f \in \Sp_{2(g-2)}(\Z)$ and $s \in S$.  We must prove that $f(s)$ is a linear
combination of elements of $S$.  This is trivial for $s = \OPres{a_{g-1} \wedge a_g;W_1}$ since
Lemma \ref{lemma:liftspecial} implies that $f$ fixes $s$.
The other $s$ fall into three cases.

The first is $s = \OPres{x \wedge y; X}$ with
\[\text{$x,y \in \cB' = \{a_1,b_1,\ldots,a_{g-2},b_{g-2}\}$ such that $x \prec y$}.\]
Since $f(X) = X$, Lemma \ref{lemma:liftequivariance} (equivariance of the lifts) implies that
\[f(\OPres{x \wedge y; X}) = \OPres{f(x) \wedge f(y); f(X)} = \OPres{f(x) \wedge f(y); X}.\]
The element $f(x) \wedge f(y) \in \overline{\wedge^2 X_{\Q}}$ is a linear combination of terms of
the form $x' \wedge y'$ with $x',y' \in \cB'$ such that $x' \prec y'$, and by Lemma \ref{lemma:liftadd} (linearity of the lifts) the
element $\OPres{f(x) \wedge f(y);X}$ equals the corresponding linear combination of
elements of the form $\OPres{x' \wedge y';X} \in S$, as desired.

The second is $s = \OPres{x \wedge a_{g-1};Y}$ with $x \in \cB'$.  
Since $f(a_{g-1}) = a_{g-1}$ and $f(Y) = Y$, Lemma \ref{lemma:liftequivariance} (equivariance of the lifts) implies that
\[f(\OPres{x \wedge a_{g-1};Y}) = \OPres{f(x) \wedge f(a_{g-1});f(Y)} = \OPres{f(x) \wedge a_{g-1};Y}.\]
The element $f(x) \wedge a_{g-1} \in \overline{\wedge^2 Y_{\Q}}$ is a linear combination of terms
of the form $x' \wedge a_{g-1}$ with $x' \in \cB'$, and by Lemma \ref{lemma:liftadd} (linearity of the lifts) the
element $\OPres{f(x) \wedge a_{g-1};Y}$ equals the corresponding linear combination of
elements of the form $\OPres{x' \wedge a_{g-1};Y} \in S$, as desired.

The third is $s = \OPres{x \wedge a_g;Z}$ with $x \in \cB'$.  This is handled in the same way
as the previous case, so we omit the details.

\begin{claim}{3.2}
The action of $\Hom(X,I)$ on $\fK_g[-,a_{g-1} \wedge a_g]$ takes $\Span{S}$ to itself.
\end{claim}

\noindent
Recall that $X = \Span{a_1,b_1,\ldots,a_{g-2},b_{g-2}}$ and $I = \{a_{g-1},a_g\}$.  The group
$\Hom(X,I)$ is generated by $\Hom(X,\Span{a_{g-1}})$ and $\Hom(X,\Span{a_g})$.   
It is enough to check that all $\lambda$ lying in one of these two subgroups take $\Span{S}$ to itself.
For concreteness, we will explain how to do this for $\lambda \in \Hom(X,\Span{a_{g-1}})$.  The
other case is similar.  The corresponding $f \in \Sp_{2g}(\Z,I)|_{I^{\perp}}$ satisfies
\[f(a_{g-1}) = a_{g-1} \quad \text{and} \quad f(a_g) = a_g \quad \text{and} \quad \text{$f(x) = x + \lambda(x)$ for all $x \in X$}.\]
Consider $s \in S$.  We must prove that $f(s)$ is a linear
combination of elements of $S$.  This is trivial for $s = \OPres{a_{g-1} \wedge a_g;W_1}$ since
in this case Lemma \ref{lemma:liftspecial} implies that $f$ fixes $s$.
The other $s$ fall into three cases.

The first is $s = \OPres{x \wedge y; X}$ with $x,y \in \cB'$ such that $x \prec y$.
By Lemma \ref{lemma:liftchange} (changing the separator in the lifts), this
equals $\OPres{x \wedge y; Y}$.  The reason for doing this is that $f(Y) = Y$.
Write $f(x) = x + c a_{g-1}$ and $f(y) = y + d a_{g-1}$ with $c,d \in \Z$.  Using all three
properties of our lifts from \S \ref{section:liftproperties}, we have
\begin{align*}
f(\OPres{x \wedge y;Y}) &= \OPres{(x+c a_{g-1}) \wedge (y+d a_{g-1});Y} \\
                        &= \OPres{x \wedge y;Y} + d \OPres{x \wedge a_{g-1};Y} - c \OPres{y \wedge a_{g-1};Y} \\
                        &= \OPres{x \wedge y;X} + d \OPres{x \wedge a_{g-1};Y} - c \OPres{y \wedge a_{g-1};Y}.
\end{align*}
This is a linear combination of elements of $S$, as desired.

The second is $s = \OPres{x \wedge a_{g-1};Y}$ with $x \in \cB'$.
Write $f(x) = x + c a_{g-1}$ with $c \in \Z$.  We then
have 
\[f(x \wedge a_{g-1}) = (x+c a_{g-1}) \wedge a_{g-1} = x \wedge a_{g-1}.\]
Since $f$ fixes $x \wedge a_{g-1}$ and $Y$, by Lemma \ref{lemma:liftequivariance} (equivariance of the lifts)
the map $f$ also fixes $\OPres{x \wedge a_{g-1}; Y}$ and there is nothing to prove.  

The third is $s = \OPres{x \wedge a_g;Z}$ with $x \in \cB'$.
Write $f(x) = x + c a_{g-1}$ with $c \in \Z$.  Pick $i$ such that $x \in \{a_i,b_i\}$.  By Lemma \ref{lemma:liftchange} (changing the separator in the lifts),
we have
\[\OPres{x \wedge a_g;Z} = \OPres{x \wedge a_g;\Span{a_i,b_i,a_g}} = \OPres{x \wedge a_g;W_i}.\]
The reason for doing is that $f(W_i) = W_i$.
Using all three properties of our lifts from \S \ref{section:liftproperties} along with Lemma \ref{lemma:liftspecial}, we have
\begin{align*}
f(\OPres{x \wedge a_g;W_i}) &= \OPres{(x+c a_{g-1}) \wedge a_g;W_i}
                            = \OPres{x \wedge a_g;W_i} + c \OPres{a_{g-1} \wedge a_g;W_i}\\
                            &= \OPres{x \wedge a_g;Z} + c \OPres{a_{g-1} \wedge a_g;W_1}.
\end{align*}
This is a linear combination of elements of $S$, as desired.
\end{proof}

\section{Isotropic pairs IV: refining the presentation I}
\label{section:isotropicpairs4}

We now bring all our work together to add new generators to $\fK_g$ involving isotropic pairs.
Let $a \wedge a'$ be an isotropic pair and 
let $\Phi\colon \fK_g \rightarrow ((\wedge^2 H)/\Q)^{\otimes 2}$ be the linearization map.

\subsection{Right elements}
Proposition \ref{proposition:isotropicpairiso} says that for all $\kappa \in (a \wedge a')^{\perp}$, 
there is a unique element $\RPres{\kappa,a \wedge a'} \in \fK_g[-,a \wedge a']$ satisfying
\[\Phi(\RPres{\kappa,a \wedge a'}) = \kappa \otimes (a \wedge a').\]
For $\lambda_1,\lambda_2 \in \Q$ and $\kappa_1,\kappa_2 \in (a \wedge a')^{\perp}$,
Proposition \ref{proposition:isotropicpairiso} implies that
\[\RPres{\lambda_1 \kappa_1 + \lambda_2 \kappa_2,a \wedge a'} = \lambda_1 \RPres{\kappa_1,a \wedge a'} + \lambda_2 \RPres{\kappa_2,a \wedge a'}.\]

\subsection{Left elements}
\label{section:leftcompatible}

Define $\fK_g[a \wedge a',-]$ to be the subspace of $\fK_g$
spanned by elements of the form $\Pres{a \wedge a',\sigma}$ with $\sigma$ a symplectic pair
such that $\sigma \in (a \wedge a')^{\perp}$.  There is an involution
$\iota\colon \fK_g \rightarrow \fK_g$ taking a generator $\Pres{\kappa_1,\kappa_2}$ to $\Pres{\kappa_2,\kappa_1}$, and
$\iota$ takes $\fK_g[a \wedge a',-]$ isomorphically to $\fK_g[-,a \wedge a']$.  For
$\kappa \in (a \wedge a')^{\perp}$, define
\[\LPres{a \wedge a', \kappa} = \iota(\RPres{\kappa, a \wedge a'}).\]
The element $\LPres{a \wedge a',\kappa}$ is then the unique element of $\fK_g[a \wedge a',-]$ 
satisfying
\[\Phi(\LPres{a \wedge a',\kappa}) = (a \wedge a') \otimes \kappa.\]
For $\lambda_1,\lambda_2 \in \Q$ and $\kappa_1,\kappa_2 \in (a \wedge a')^{\perp}$,
we have
\[\LPres{a \wedge a',\lambda_1 \kappa_1 + \lambda_2 \kappa_2} = \lambda_1 \LPres{a \wedge a',\kappa_1} + \lambda_2 \LPres{a \wedge a',\kappa_2}.\]

\subsection{Ambiguity}

We would like to drop the $L$ and $R$ from $\LPres{a \wedge a',\kappa}$ and $\RPres{\kappa,a \wedge a'}$.  To
do this, we must first show that this does not introduce ambiguity into our notation.  The issue
is that there exist isotropic pairs $a_1 \wedge a'_1$ and $a_2 \wedge a'_2$ that are 
sym-orthogonal to each other.  In this case, we have 
elements $\LPres{a_1 \wedge a'_1,a_2 \wedge a'_2}$ and $\RPres{a_1 \wedge a'_1,a_2 \wedge a'_2}$, and
we need to prove that they are equal:

\begin{lemma}
\label{lemma:isotropicambiguityresolve}
Let $a_1 \wedge a'_1$ and $a_2 \wedge a'_2$ be isotropic pairs that are sym-orthogonal.  Then
$\LPres{a_1 \wedge a'_1,a_2 \wedge a'_2} = \RPres{a_1 \wedge a'_1,a_2 \wedge a'_2}$.
\end{lemma}
\begin{proof}
The proof uses the same idea as the proof of Lemma \ref{lemma:orangeindependent}.
Set $I_1 = \Span{a_1,a'_1}$ and $I_2 = \Span{a_2,a'_2}$.  By Lemma \ref{lemma:isotropicorthogonal}, we have
$a_1 \wedge a'_1 \in \overline{\wedge^2 (I_2)_{\Q}^{\perp}}$
and $a_2 \wedge a'_2 \in \overline{\wedge^2 (I_1)_{\Q}^{\perp}}$.  This implies that $I_1 \subset I_2^{\perp}$
and $I_2 \subset I_1^{\perp}$.  Recall that from \S \ref{section:constructlift} that a Lagrangian in
$H_{\Z}$ is a direct summand $L$ of $H_{\Z}$ with $L^{\perp} = L$.  We start with:

\begin{unnumberedclaim}
There exists a Lagrangian $L$ in $H_{\Z}$ such that $I_1,I_2 \subset L$.
\end{unnumberedclaim}
\begin{proof}[Proof of claim]
Recall that a subspace $J$ of $H$ is isotropic if $J \subset J^{\perp}$ and is a Lagrangian if
$J = J^{\perp}$.  It is standard that $J$ being a Lagrangian is equivalent to $J$
being isotropic and $g$-dimensional, and also that every isotropic subspace is contained in a Lagrangian.
The subspace $\Span{(I_1)_{\Q},(I_2)_{\Q}}$ of $H$ is isotropic, so it is contained in a Lagrangian $L_{\Q}$.
Define $L = L_{\Q} \cap H_{\Z}$.  Lemma \ref{lemma:summand} implies that $L$ 
is a direct summand of $H_{\Z}$, and by construction it is a Lagrangian
containing $I_1$ and $I_2$.
\end{proof}

Using this, we will prove:

\begin{unnumberedclaim}
There exists a Lagrangian-free near symplectic summand $X$ of $H_{\Z}$ of genus $1$ such that
$I_1 \subset X \subset I_2^{\perp}$.
\end{unnumberedclaim}
\begin{proof}[Proof of claim]
By the previous claim, we can find a Lagrangian $L$ in $H_{\Z}$ with $I_1,I_2 \subset L$.  Since
$I_2 \cong \Z^2$ is a subspace of $L \cong \Z^g$ and $g \geq 4$ (see Assumption \ref{assumption:genuspart2}),
the quotient $L/I_2$ cannot consist entirely of torsion.  It follows that there exists
a surjection $\pi\colon L \rightarrow \Z$ with $I_2 \subset \ker(\pi)$.  Since the symplectic
form $\omega$ identifies $H_{\Z}$ with its dual, we can find $y_1 \in H_{\Z}$
with
\[\text{$\omega(z,y_1) = \pi(z)$ for all $z \in L$}.\]
In particular, $y_1 \in (I_2)^{\perp}$.  Pick $x_1 \in L$ with $\omega(x_1,y_1) = 1$.  
Define $J = \Span{(I_1)_{\Q},\Span{x_1}_{\Q}} \cap L$, so by Lemma \ref{lemma:summand} the subgroup $J$ is a direct summand of
$L$ with $I_1 \subset J$ and $x_1 \in J$.  Let $r$ be the rank of $J$.  Since the rank
of $I_1$ is $2$, we have $2 \leq r \leq 3$.  We can now extend $x_1$ to a basis $\{x_1,\ldots,x_r\}$ for $J$ such that $\omega(x_i,y_1)=0$ for $2 \leq i \leq r$.  Set
$X = \Span{x_1,y_1,x_2,\ldots,x_r}$.  By construction, $X$ is a near-symplectic summand of $H_{\Z}$ of genus $1$ such that $I_1 \subset X \subset I_2^{\perp}$.  Since
$r \leq 3$, our standing assumption that $g \geq 4$ (see Assumption \ref{assumption:genuspart2}) implies that $X$ is Lagrangian-free.
\end{proof}

Let $X$ be as in the previous claim.  We have:

\begin{unnumberedclaim}
The subspace $X^{\perp}$ of $H_{\Z}$ is a near symplectic summand of positive genus.
\end{unnumberedclaim}
\begin{proof}[Proof of claim]
Since $X$ is a near-symplectic summand of genus $1$, we can find a symplectic basis 
$\{x_1,y_1,\ldots,x_g,y_g\}$ for $H_{\Z}$ such that $X = \Span{x_1,y_1,x_2,\ldots,x_r}$.  Since
$X$ is Lagrangian-free, we have $r < g$.  We have $X^{\perp} = \Span{x_2,\ldots,x_r,x_{r+1},y_{r+1},\ldots,x_g,y_g}$.
This is a near symplectic summand, and since $r<g$ its genus is positive.
\end{proof}

Since $I_1 \subset X$ and $I_2 \subset X^{\perp}$, we can use Lemma \ref{lemma:generationomega} to write
\[a_1 \wedge a'_1 = \sum_{i=1}^n \lambda_i \sigma_i \quad \text{and} \quad a_2 \wedge a'_2 = \sum_{j=1}^m c_j s_j\]
with $\lambda_i,c_j \in \Q$ and with each $\sigma_i$ and $s_j$ a symplectic pair with $\sigma_i \in \overline{\wedge^2 X_{\Q}}$ and $s_j \in \overline{\wedge^2 X^{\perp}_{\Q}}$, respectively.
The element $\RPres{a_1 \wedge a'_1,a_2 \wedge a'_2}$ then equals
\begin{align*}
\sum_{i=1}^n \lambda_i \Pres{\sigma_i,a_2 \wedge a'_2} &= \sum_{i=1}^n \lambda_i \left(\sum_{j=1}^m c_j \Pres{\sigma_i,s_j}\right) 
                                                        = \sum_{j=1}^m c_j \left(\sum_{i=1}^n \lambda_i \Pres{\sigma_i,s_j}\right) \\
                                                       &= \sum_{j=1}^m c_j \Pres{a_1 \wedge a'_1,s_j} = \LPres{a_1 \wedge a'_1,a_2 \wedge a'_2}.\qedhere
\end{align*}
\end{proof}

\subsection{New generators}

Let $a \wedge a'$ be an isotropic pair and let $\kappa \in (a \wedge a')^{\perp}$.
We then have elements $\RPres{\kappa,a \wedge a'}$ and $\LPres{a \wedge a',\kappa}$ of $\fK_g$.  
By Lemma \ref{lemma:isotropicambiguityresolve}, we can unambiguously drop the ``L'' and ``R'' from our notation.  Since everything is
now canonical, we will also stop writing our elements in orange and define
\[\Pres{\kappa,a \wedge a'} = \RPres{\kappa,a \wedge a'} \quad \text{and} \quad \Pres{a \wedge a',\kappa} = \LPres{a \wedge a',\kappa}.\]

\subsection{Summary}

The following summarizes what we have accomplished in
Proposition \ref{proposition:isotropicpairiso} and Lemma \ref{lemma:isotropicambiguityresolve}:

\begin{theorem}
\label{theorem:summarypresentationweak}
The vector space $\fK_g$ has the following presentation:
\begin{itemize}
\item {\bf Generators}. 
A generator $\Pres{\kappa_1,\kappa_2}$ for all sym-orthogonal
$\kappa_1,\kappa_2 \in (\wedge^2 H)/\Q$ such that either
$\kappa_1$ or $\kappa_2$ (or both) is a symplectic pair or an isotropic pair.
\item {\bf Relations}.  For all $\zeta \in (\wedge^2 H)/\Q$ that are symplectic pairs or strong isotropic pairs
and all $\kappa_1,\kappa_2 \in (\wedge^2 H)/\Q$ that are sym-orthogonal to $\zeta$
and all $\lambda_1,\lambda_2 \in \Q$, the relations
\begin{align*}
\Pres{\zeta,\lambda_1 \kappa_1 + \lambda_2 \kappa_2} &= \lambda_1 \Pres{\zeta,\kappa_1} + \lambda_2 \Pres{\zeta,\kappa_2} \quad \text{and} \\
\Pres{\lambda_1 \kappa_1 + \lambda_2 \kappa_2,\zeta} &= \lambda_1 \Pres{\kappa_1,\zeta} + \lambda_2 \Pres{\kappa_2,\zeta}.\qedhere
\end{align*}
\end{itemize}
\end{theorem}

\section{Isotropic pairs V: refining the presentation II}
\label{section:isotropicpairs5}

In Theorem \ref{theorem:summarypresentationweak}, we added generators involving isotropic pairs.
In this section, we add a few more generators and verify some additional relations.  A {\em special pair} is an element
$x \wedge y$ of $(\wedge^2 H_{\Z})/\Z$ such that $\omega(x,y) \in \{-1,0,1\}$.  These fall
into four classes:
\begin{itemize}
\item $\omega(x,y) = 1$, so $x \wedge y$ is a symplectic pair; and
\item $\omega(x,y) = 0$ with $x$ and $y$ linearly independent, so $x \wedge y$ is an isotropic pair;\footnote{They
are strong isotropic pairs if $x$ and $y$ also span a direct summand of $H_{\Z}$; see \S \ref{section:isotropicpairs}.} and
\item $\omega(x,y) = 0$ with $x$ and $y$ linearly dependent, so $x \wedge y=0$; and
\item $\omega(x,y) = -1$, so $y \wedge x = -x \wedge y$ is a symplectic pair.
\end{itemize}
Our main result is:

\begin{theorem}
\label{theorem:summarypresentation}
For\footnote{This is our standing assumption in this part of the paper; see Assumption \ref{assumption:genuspart2}.} 
$g \geq 4$, the vector space $\fK_g$ has the following presentation:
\begin{itemize}
\item {\bf Generators}. 
A generator $\Pres{\kappa_1,\kappa_2}$ for all sym-orthogonal
$\kappa_1,\kappa_2 \in (\wedge^2 H)/\Q$ such that either
$\kappa_1$ or $\kappa_2$ (or both) is a special pair.
\item {\bf Relations}.  The following two families of relations:
\begin{itemize}
\item For all special pairs $\zeta \in (\wedge^2 H)/\Q$ 
and all $\kappa_1,\kappa_2 \in (\wedge^2 H)/\Q$ that are sym-orthogonal to $\zeta$
and all $\lambda_1,\lambda_2 \in \Q$, the linearity relations
\begin{align*}
\Pres{\zeta,\lambda_1 \kappa_1 + \lambda_2 \kappa_2} &= \lambda_1 \Pres{\zeta,\kappa_1} + \lambda_2 \Pres{\zeta,\kappa_2} \quad \text{and} \\
\Pres{\lambda_1 \kappa_1 + \lambda_2 \kappa_2,\zeta} &= \lambda_1 \Pres{\kappa_1,\zeta} + \lambda_2 \Pres{\kappa_2,\zeta}.
\end{align*}
\item For all special pairs $\zeta \in (\wedge^2 H)/\Q$ and all $\kappa \in (\wedge^2 H)/\Q$ that are sym-orthogonal to $\zeta$
and all $n \in \Z$ such that $n \zeta$ is a special pair,\footnote{If $\zeta = x \wedge y$ with $\omega(x,y) = 0$, then any
$n$ works.  However, if $\zeta = x \wedge y$ with $\omega(x,y) \in \{-1,1\}$ then we must take $n \in \{-1,0,1\}$.}
the relations
\begin{align*}
\Pres{n \zeta,\kappa} &= n \Pres{\zeta,\kappa} \quad \text{and} \\
\Pres{\kappa,n \zeta} &= n \Pres{\kappa,\zeta}.\qedhere
\end{align*}
\end{itemize}
\end{itemize}
\end{theorem}
\begin{proof}
We divide the proof into three steps: we first define our new generators in terms of the generators
given by Theorem \ref{theorem:summarypresentationweak}, and then we check the two families of relations.

\begin{step}{1}
We define our new generators.
\end{step}

Consider a special pair $x \wedge y \in (\wedge^2 H)/\Q$ and $\kappa \in (\wedge^2 H)/\Q$ that is sym-orthogonal
to $x \wedge y$.  We will express $\Pres{x \wedge y,\kappa}$ and $\Pres{\kappa,x \wedge y}$ in terms of
the generators given by Theorem \ref{theorem:summarypresentationweak}.  By definition,
we have $\omega(x,y) \in \{-1,0,1\}$.  There are four cases:
\begin{itemize}
\item If $\omega(x,y) = 1$, then $x \wedge y$ is a symplectic pair and $\Pres{x \wedge y,\kappa}$ and $\Pres{\kappa,x \wedge y}$
are already defined.
\item If $\omega(x,y) = 0$ and $x$ and $y$ are linearly independent, then $x \wedge y$ is an isotropic
pair and $\Pres{x \wedge y,\kappa}$ and $\Pres{\kappa,x \wedge y}$
are already defined.
\item If $\omega(x,y) = 0$ and $x$ and $y$ are linearly dependent, then $x \wedge y = 0$.  We define
$\Pres{0,\kappa} = 0$ and $\Pres{\kappa,0} = 0$.
\item If $\omega(x,y) = -1$, then $x \wedge y = -y \wedge x$ and $y \wedge x$ is a symplectic pair.  We define
$\Pres{x \wedge y,\kappa} = -\Pres{y \wedge x,\kappa}$ and $\Pres{\kappa,x \wedge y} = -\Pres{\kappa,y \wedge x}$.
\end{itemize}
The only issue with this definition is that if $\kappa = z \wedge w$ is also a special pair, then in
a few cases we have two different definitions of $\Pres{x \wedge y,z \wedge w}$.  We must check that they give the same
element of $\fK_g$.  If both $x \wedge y$ and $z \wedge w$ are symplectic or isotropic pairs, then
there is no ambiguity.  Also, if one of them is $0$, then both definitions give $0$.  The only
potential issue is therefore when either $\omega(x,y)$ or $\omega(z,w)$ (or both) is $-1$. 

There are several cases.  All are handled the same way, so we will give the details
for when $\omega(x,y) = \omega(z,w) = -1$, which is slightly harder.  Our two definitions are
\[\Pres{x \wedge y,z \wedge w} = -\Pres{y \wedge x,z \wedge w} \quad \text{and} \quad \Pres{x \wedge y,z \wedge w} = -\Pres{x \wedge y,w \wedge z}.\]
We appeal to the linearity relations from Theorem \ref{theorem:summarypresentationweak} to see these are equal:
\[-\Pres{y \wedge x,z \wedge w} = -\Pres{y \wedge x,-w \wedge z} = \Pres{y \wedge x,w \wedge z} = \Pres{-x \wedge y,w \wedge z} = -\Pres{x \wedge y,w \wedge z}.\]

\begin{step}{2}
Let $\zeta \in (\wedge^2 H)/\Q$ be a special pair, let 
$\kappa_1,\kappa_2 \in (\wedge^2 H)/\Q$ be sym-orthogonal to $\zeta$, and let
$\lambda_1,\lambda_2 \in \Q$.  Then:
\begin{align*}
\Pres{\zeta,\lambda_1 \kappa_1 + \lambda_2 \kappa_2} &= \lambda_1 \Pres{\zeta,\kappa_1} + \lambda_2 \Pres{\zeta,\kappa_2} \quad \text{and} \\
\Pres{\lambda_1 \kappa_1 + \lambda_2 \kappa_2,\zeta} &= \lambda_1 \Pres{\kappa_1,\zeta} + \lambda_2 \Pres{\kappa_2,\zeta}.
\end{align*}
\end{step}

These are trivial if $\zeta = 0$, so we can assume that $\zeta \neq 0$.  Also, these are special cases
of the linearity relations from Theorem \ref{theorem:summarypresentationweak} if $\zeta$ is either a symplectic
pair or an isotropic pair.  The remaining case is where $\zeta = x \wedge y$ with $\omega(x,y) = -1$.  In that case,
using the linearity relations from Theorem \ref{theorem:summarypresentationweak} we have that
$\Pres{x \wedge y,\lambda_1 \kappa_1+\lambda_2 \kappa_2}$ equals
\[-\Pres{y \wedge x,\lambda_1 \kappa_1+\lambda_2 \kappa_2}
                                                          = -\lambda_1 \Pres{y \wedge x,\kappa_1}  - \lambda_2 \Pres{y \wedge x,\kappa_2}
                                                          = \lambda_1 \Pres{x \wedge y,\kappa_1} + \lambda_2 \Pres{x \wedge y,\kappa_2},\]
and similarly for $\Pres{\lambda_1 \kappa_1+\lambda_2 \kappa_2,x \wedge y}$.

\begin{step}{3}
Let $\zeta \in (\wedge^2 H)/\Q$ be a special pair, let $\kappa \in (\wedge^2 H)/\Q$ be sym-orthogonal to $\zeta$,
and let $m \in \Z$ be such that such that $m \zeta$ is a special pair.  Then:
\begin{align*}
\Pres{m \zeta,\kappa} &= m \Pres{\zeta,\kappa} \quad \text{and} \\
\Pres{\kappa,m \zeta} &= m \Pres{\kappa,\zeta}.
\end{align*}
\end{step}

This is trivial if $\zeta = 0$ or if $m = 0$, so assume that both are nonzero.  This is immediate
from the definitions if $\zeta = x \wedge y$ with $\omega(x,y) \in \{\pm 1\}$, in which case we necessarily
also have $m \in \{\pm 1\}$.  The remaining case is when $\zeta = x \wedge y$ is an isotropic pair.  Write
$\kappa = \sum_{i=1}^n \lambda_i \zeta_i$ with $\zeta_i$ a symplectic pair in $(x \wedge y)^{\perp}$.  We also
have $\zeta_i \in (m x \wedge y)^{\perp}$ for $1 \leq i \leq n$.  By definition, we therefore have
\[\Pres{x \wedge y,\kappa} = \sum_{i=1}^n \lambda_i \Pres{x \wedge y,\zeta_i} \quad \text{and} \quad
\Pres{m x \wedge y,\kappa} = \sum_{i=1}^n \lambda_i \Pres{m x \wedge y,\zeta_i}.\]
The linearity relations from Theorem \ref{theorem:summarypresentationweak} imply
that for $1 \leq i \leq n$ we have $\Pres{m x \wedge y,\zeta_i} = m \Pres{x \wedge y,\zeta_i}$.  Plugging
this into the above formulas, we therefore have
\[\Pres{m x \wedge y,\kappa} = \sum_{i=1}^n \lambda_i m \Pres{x \wedge y,\zeta_i} = m \sum_{i=1}^n \lambda_i \Pres{x \wedge y,\zeta_i} = m \Pres{x \wedge y,\kappa}.\]
A similar argument shows that $\Pres{\kappa,m \zeta} = m \Pres{\kappa,\zeta}$.
\end{proof}

\part{Verifying the presentation for the symmetric kernel, alternating version}
\label{part:alt}

Our goal in the rest of the paper is to prove Theorems \ref{maintheorem:presentationalt} and
\ref{maintheorem:presentationsym}.  This part of the paper proves Theorem \ref{maintheorem:presentationalt}, while
Part \ref{part:sym} proves Theorem \ref{maintheorem:presentationsym}.  See
the introductory \S \ref{section:presentationalt0} for an outline of what we do in this part.
Throughout, we make the following genus assumption:

\begin{assumption}
\label{assumption:genusalt}
Throughout Part \ref{part:alt}, unless otherwise specified we assume that $g \geq 4$. 
\end{assumption}

\section{Symmetric kernel, alternating version: introduction}
\label{section:presentationalt0}

We start by recalling some results and definitions from earlier in the paper, and
then outline what we prove in this part.

\subsection{Symmetric kernel and contraction}
Recall that $\omega$ is the symplectic form on $H$.
The {\em symmetric contraction} is the alternating $\Sym^2(H)$-valued alternating
form $\fc$ on $(\wedge^2 H)/\Q$ defined via the formula
\[\text{$\fc(x \wedge y,z \wedge w) = \omega(x,z) y \Cdot w - \omega(x,w) y \Cdot z - \omega(y,z) x \Cdot w + \omega(y,w) x \Cdot z$ for $x,y,z,w \in H$}.\]
It induces a map $\wedge^2 ((\wedge^2 H)/\Q) \rightarrow \Sym^2(H)$ whose kernel $\cK_g^a$ is
the {\em symmetric kernel}.  Elements $\kappa_1,\kappa_2 \in (\wedge^2 H)/\Q$ are
{\em sym-orthogonal} if $\fc(\kappa_1,\kappa_2) = 0$, or equivalently if $\kappa_1 \wedge \kappa_2 \in \cK_g^a$.  The
{\em sym-orthogonal complement} of $\kappa \in (\wedge^2 H)/\Q$ is the subspace $\kappa^{\perp}$
consisting of all elements that are sym-orthogonal to $\kappa$.

\subsection{Special pairs}
A {\em special pair} in $(\wedge^2 H_{\Z})/\Z$ is an element of the form $x \wedge y$ with
$\omega(x,y) \in \{-1,0,1\}$.  Examples include symplectic pairs and isotropic pairs.
Lemmas \ref{lemma:symplecticorthogonal} and \ref{lemma:isotropicorthogonal} say that
the sym-orthogonal complements in $(\wedge^2 H)/\Q$ of these are:
\begin{itemize}
\item for a symplectic pair $a \wedge b$,  we have $(a \wedge b)^{\perp} = \overline{\wedge^2 \Span{a,b}_{\Q}^{\perp}}$; and
\item for an isotropic pair $a \wedge a'$, we have $(a \wedge a')^{\perp} = \overline{\wedge^2 \Span{a,a'}_{\Q}^{\perp}}$.
\end{itemize}

\subsection{Non-symmetric presentation}
We will use the generators and relations for $\fK_g$ from Theorem \ref{theorem:summarypresentation}, whose statement
we recall:

\newtheorem*{theorem:summarypresentation1}{Theorem \ref{theorem:summarypresentation}}
\begin{theorem:summarypresentation1}
For $g \geq 4$, the vector space $\fK_g$ has the following presentation:
\begin{itemize}
\item {\bf Generators}.
A generator $\Pres{\kappa_1,\kappa_2}$ for all sym-orthogonal
$\kappa_1,\kappa_2 \in (\wedge^2 H)/\Q$ such that either
$\kappa_1$ or $\kappa_2$ (or both) is a special pair.
\item {\bf Relations}.  The following two families of relations:
\begin{itemize}
\item For special pairs $\zeta \in (\wedge^2 H)/\Q$
and all $\kappa_1,\kappa_2 \in (\wedge^2 H)/\Q$ that are sym-orthogonal to $\zeta$
and all $\lambda_1,\lambda_2 \in \Q$, the linearity relations
\begin{align*}
\Pres{\zeta,\lambda_1 \kappa_1 + \lambda_2 \kappa_2} &= \lambda_1 \Pres{\zeta,\kappa_1} + \lambda_2 \Pres{\zeta,\kappa_2} \quad \text{and} \\
\Pres{\lambda_1 \kappa_1 + \lambda_2 \kappa_2,\zeta} &= \lambda_1 \Pres{\kappa_1,\zeta} + \lambda_2 \Pres{\kappa_2,\zeta}.
\end{align*}
\item For all special pairs $\zeta \in (\wedge^2 H)/\Q$ and all $\kappa \in (\wedge^2 H)/\Q$ that are sym-orthogonal to $\zeta$
and all $n \in \Z$ such that $n \zeta$ is a special pair, the relations
\begin{align*}
\Pres{n \zeta,\kappa} &= n \Pres{\zeta,\kappa} \quad \text{and} \\
\Pres{\kappa,n \zeta} &= n \Pres{\kappa,\zeta}.\qedhere
\end{align*}
\end{itemize}
\end{itemize}
\end{theorem:summarypresentation1}

\begin{remark}
\label{remark:genusalt}
Our standing assumption is that $g \geq 4$ (Assumption \ref{assumption:genusalt}).  However, in a few places
we will need to work with $g=3$ for inductive proofs.  In those cases, Theorem \ref{theorem:summarypresentation}
does not apply.  To fix this, in this part of the paper we will redefine $\fK_3$ to be the vector space given by the
presentation from Theorem \ref{theorem:summarypresentation}.  Note that we will {\em not} extend Theorem
\ref{maintheorem:presentationalt} to $g = 3$, and we do not know if this $\fK_3$
is isomorphic to $\cK_3$.
\end{remark}

\subsection{Anti-symmetrizing}

Recall from Lemma \ref{lemma:decomposekg} that $\fK_g^a$ is the $-1$-eigenspace of the involution of
$\fK_g$ that takes a generator $\Pres{\kappa_1,\kappa_2}$ to $\Pres{\kappa_2,\kappa_1}$.  We anti-symmetrize
a generator $\Pres{\kappa_1,\kappa_2}$ of $\fK_g$ to
\[\PresA{\kappa_1,\kappa_2} = \frac{1}{2} \left(\Pres{\kappa_1,\kappa_2} - \Pres{\kappa_2,\kappa_1}\right) \in \fK_g^a.\]
The anti-symmetrized generators generate $\fK_g^a$.  They satisfy the same relations as the generators
of $\fK_g$, and also the anti-symmetry relation
$\PresA{\kappa_2,\kappa_1} = -\PresA{\kappa_1,\kappa_2}$.

\subsection{Goal and outline}

We have a linearization map $\Phi\colon \fK_g^a \rightarrow \wedge^2((\wedge^2 H)/\Q)$.  On generators,
it satisfies
\[\Phi(\PresA{\kappa_1,\kappa_2}) = \kappa_1 \wedge \kappa_2 \in \wedge^2((\wedge^2 H)/\Q).\]
Its image is contained in the symmetric kernel $\cK_g^a$.
Our goal in this part of the paper is to prove Theorem \ref{maintheorem:presentationalt}, which says that $\Phi$ is an isomorphism
from $\fK_g^a$ to $\cK_g^a$.  
The proof uses the proof technique described in \S \ref{section:prooftechnique}, and is modeled
on the proofs of Theorems \ref{maintheorem:slstd}--\ref{maintheorem:spsym}.  However, since
the calculations are lengthy we spread them out over nine sections:
\begin{itemize}
\item In the preliminary \S \ref{section:presentationalt1}, we identify some important subspaces
of $\fK_g^a$.
\item In \S \ref{section:presentationalt2} -- \S \ref{section:presentationalt4}, we construct
a subset $S$ of $\fK_g^a$ and prove that $\Span{S} = \fK_g^a$.  The proof of this
uses the action of $\Sp_{2g}(\Z)$ on $\fK_g^a$: we first prove that the $\Sp_{2g}(\Z)$-orbit
of $S$ spans $\fK_g^a$, and then we prove that $\Sp_{2g}(\Z)$ takes $\Span{S}$ to itself.
This corresponds to Steps 2 and 3 of the proof outline from \S \ref{section:prooftechnique}.
We do these steps first because they are easier than Step 1.\footnote{Though $S$ will be infinite,
it will follow from our results in these sections that $\Span{S}$ is finite-dimensional.  At
the end of \S \ref{section:presentationalt4}, we will therefore already know that $\fK_g^a$ is a finite-dimensional
representation of $\Sp_{2g}(\Z)$.}
\item The set $S$ is the union of sets $S_{12}$ and $S_3$.
In \S \ref{section:presentationalt5} -- \S \ref{section:presentationalt8}, we prove
that $\Phi$ is an isomorphism by first proving that its restriction to $S_{12}$ is an isomorphism onto
its image (\S \ref{section:presentationalt5} -- \S \ref{section:presentationalt6}) and then
extending this to $S_3$ and hence all of $\Span{S} = \fK_g^a$ 
(\S \ref{section:presentationalt7} -- \S \ref{section:presentationalt9}).  This roughly speaking
corresponds to Step 1.
\end{itemize}
Throughout the following nine sections, $\Phi$ will always mean the linearization map
$\Phi\colon \fK_g^a \rightarrow \wedge^2((\wedge^2 H)/\Q)$.  Also, $\fc$ will always mean
the symmetric contraction.

\section{Symmetric kernel, alternating version I: fixing the 1st coordinates of generators}
\label{section:presentationalt1}

In this preliminary section, we identify some important subspaces of $\fK_g^a$ and
their images in $\wedge^2((\wedge^2 H)/\Q)$ under the linearization map.  Throughout this
section, we relax our standing assumption that $g \geq 4$ (Assumption \ref{assumption:genusalt}),
so our results include explicit genus ranges when they are necessary.  See Remark \ref{remark:genusalt}.

\begin{warning}
\label{warning:alternating}
The vector space $\wedge^2((\wedge^2 H)/\Q)$ is {\em not} a quotient of $\wedge^4 H$.  Because of this, for $x,y,z,w \in H$
care must be taken when working with elements like $(x \wedge y) \wedge (z \wedge w) \in \wedge^2((\wedge^2 H)/\Q)$.
The wedges $\wedge$ cannot be rearranged like in $\wedge^4 H$; for instance,
$(x \wedge y) \wedge (z \wedge w)$ is not equal to $-(x \wedge z) \wedge (y \wedge w)$.
\end{warning}

\subsection{Setup}
\label{section:presentationalt1setup}

Let $a,a' \in H_{\Z}$ satisfy $\omega(a,a') = 0$.  Define $\Fix{a,a'}$ to
be the subspace of $\fK_g^a$ spanned by elements $\PresA{a \wedge x,a' \wedge y}$ where
$x,y \in H_{\Z}$ satisfy the following two conditions:
\begin{itemize}
\item[(i)] We have $\omega(a,y)=\omega(a',x)=\omega(x,y) = 0$.  This ensures that
$\fc(a \wedge x,a' \wedge y) = 0$.
\item[(ii)] Both $a \wedge x$ and $a' \wedge y$ are special pairs, so in particular
$\PresA{a \wedge x,a' \wedge y}$ is defined.  Equivalently, 
$\omega(a,x) \in \{-1,0,1\}$ and $\omega(a',y) \in \{-1,0,1\}$.
\end{itemize}
We will call these $\PresA{a \wedge x,a' \wedge y}$ the {\em generators} of
$\Fix{a,a'}$.  Here are some easy properties
of these subspaces:

\begin{lemma}
\label{lemma:isotropicbilineareasy}
Let $g \geq 3$ and let $a,a' \in H_{\Z}$ satisfy $\omega(a,a') = 0$.  Then:
\begin{itemize}
\item[(a)] $\Fix{a',a} = \Fix{a,a'}$.
\item[(b)] $\Fix{-a,a'} = \Fix{a,-a'} = \Fix{-a,-a'} = \Fix{a,a'}$.
\end{itemize}
\end{lemma}
\begin{proof}
Conclusion (a) follows from the fact that for all generators
$\PresA{a \wedge x,a' \wedge y}$ of $\Fix{a,a'}$, the 
element $\PresA{a' \wedge y,a \wedge x}$ is a generator of $\Fix{a',a}$ satisfying
\[\PresA{a' \wedge y,a \wedge x} = -\PresA{a \wedge x,a' \wedge y}.\]
In light of (a), to prove (b) it is enough to prove that 
$\Fix{-a,a'} = \Fix{a,a'}$.
For this, note that if $\PresA{a \wedge x,a' \wedge y}$ is a generator
of $\Fix{a,a'}$
then $\PresA{-a \wedge x,a' \wedge y}$
is a generator of $\Fix{-a,a'}$ and we have
\[\PresA{-a \wedge x,a' \wedge y} = -\PresA{a \wedge x,a' \wedge y}.\qedhere\]
\end{proof}

\subsection{Image}
\label{section:presentationalt1image}

Define $\FixIm{a,a'}$ to be the subspace of $\wedge^2((\wedge^2 H)/\Q)$
spanned by elements of the form $(a \wedge x) \wedge (a' \wedge y)$ such that
$\PresA{a \wedge x,a' \wedge y}$ is a generator of $\Fix{a,a'}$.  We
thus have $\FixIm{a,a'} \subset \cK_g^a$ and
\[\Phi(\Fix{a,a'}) \subset \FixIm{a,a'}.\]
Our goal in the rest of this section is to prove that in two important cases
the map $\Phi$ takes $\Fix{a,a'}$ isomorphically
to $\FixIm{a,a'}$.

Before we do this, we introduce one further piece of notation.  Define $\FixBigIm{a,a'}$
to be the subspace of $\wedge^2((\wedge^2 H)/\Q)$ spanned by elements
of the form $(a \wedge x) \wedge (a' \wedge y)$ such that:
\begin{itemize}
\item We have $\omega(a,y)=\omega(a',x)= 0$.  Note that we are not requiring that $\omega(x,y) = 0$.
\end{itemize}
For such an element, we have
\[\fc(a \wedge x,a' \wedge y) = \omega(a,a') x \Cdot y - \omega(a,y)x \Cdot a' - \omega(x,a') a \Cdot y + \omega(x,y) a \Cdot a' = \omega(x,y) a \Cdot a'.\]
It follows that $\fc$ takes $\FixBigIm{a,a'}$ to the $1$-dimensional subspace of
$\Sym^2(H)$ spanned by $a \Cdot a'$.  It is easy to see that we can find elements $(a \wedge x) \wedge (a' \wedge y)$
as above with $\omega(x,y) \neq 0$, so in fact the kernel $\cK_g^a \cap \FixBigIm{a,a'}$ of $\fc$ restricted to 
$\FixBigIm{a,a'}$ has codimension $1$.  We have $\FixIm{a,a'} \subset \cK_g^a \cap \FixBigIm{a,a'}$,
and later in this section we will prove that in two cases we have
$\FixIm{a,a'} = \cK_g^a \cap \FixBigIm{a,a'}$.

\subsection{Identification I}
\label{section:identificationalt1}

We now commence with identifying our subspaces.
For the first, let $a \in H_{\Z}$ be primitive.  We will identify $\Fix{a,a}$.
In a generator $\PresA{a \wedge x,a \wedge y}$ of $\Fix{a,a}$, both
$x$ and $y$ are orthogonal to $a$ and also are only well-defined
up to multiples of $a$.  This suggests defining $\cU(a) = \Span{a}_{\Q}^{\perp}/\Span{a}$.
We can embed $\cU(a)$ into $(\wedge^2 H)/\Q$ by taking $x \in \cU(a)$ to
$a \wedge x \in (\wedge^2 H)/\Q$.  Using this, we have
\[\FixBigIm{a,a} = \wedge^2 \cU(a).\]
The symplectic form $\omega$ induces a symplectic form $\oomega$ on $\cU(a)$.  Let $\cK(a)$ be the kernel
of the map $\wedge^2 \cU(a) \rightarrow \Q$ induced by $\oomega$.  
We then have:

\begin{lemma}
\label{lemma:mapbilinear1}
Let $g \geq 3$ and let $a \in H_{\Z}$ be primitive.  Then:
\begin{itemize}
\item[(a)] $\FixBigIm{a,a} = \wedge^2 \cU(a)$; and
\item[(b)] $\FixIm{a,a} = \cK_g^a \cap \FixBigIm{a,a}$; and
\item[(c)] the linearization map $\Phi\colon \fK_g^a \rightarrow \wedge^2((\wedge^2 H)/\Q)$
takes $\Fix{a,a}$ isomorphically onto $\cK(a)$.   
\end{itemize}
\end{lemma}
\begin{proof}
We noted that (a) held right before the lemma.  We will prove (c) and then (b).

\begin{step}{1}
Conclusion (c) holds: the linearization map $\Phi\colon \fK_g^a \rightarrow \wedge^2((\wedge^2 H)/\Q)$
takes $\Fix{a,a}$ isomorphically onto $\cK(a)$.
\end{step}

Endow $\Z^{2g-2}$ with the standard symplectic form.  Let 
$\mu\colon \Z^{2g-2} \rightarrow \Span{a}^{\perp}/\Span{a}$ be an isomorphism
of abelian groups equipped with symplectic forms.  If $v_1,v_2 \in \Z^{2g-2}$ are
orthogonal vectors, then $a \wedge \mu(v_1)$ and $a \wedge \mu(v_2)$
are both either isotropic pairs or $0$, so we have a generator $\PresA{a \wedge \mu(v_1),a \wedge \mu(v_2)}$
of $\Fix{a,a}$.

Recall that we defined the vector
space $\fZ_{g-1}^a$ in Definition \ref{definition:spkernelalt}.  It is generated
by elements $\ZGenA{v_1,v_2}$ with $v_1,v_2 \in \Z^{2g-2}$ orthogonal primitive vectors.
Define a map $\psi\colon \fZ_{g-1}^a \rightarrow \Fix{a,a}$ via the formula
\[\psi(\ZGenA{v_1,v_2}) = \PresA{a \wedge \mu(v_1),a \wedge \mu(v_2)} \quad \text{for orthogonal primitive 
vectors $v_1,v_2 \in \Z^{2g-2}$}.\]
This takes relations to relations, and thus gives a well-defined map.

We claim that $\psi$ is surjective.  To see this, consider a generator $\PresA{a \wedge x,a \wedge y}$
of $\Fix{a,a}$.  We must check that $\PresA{a \wedge x,a \wedge y}$ is in
the image of $\psi$.  Let $w_1,w_2 \in \Z^{2g-2}$ be such that $\mu(w_1) = x$ and $\mu(w_2) = y$.
Write $w_1 = \lambda_1 v_1$ and $w_2 = \lambda_2 v_2$ with $\lambda_1,\lambda_2 \in \Z$ and $v_1,v_2 \in \Z^{2g-2}$
primitive.  We then have
\begin{align*}
\psi(\lambda_1 \lambda_2 \ZGenA{v_1,v_2}) &= \lambda_1 \lambda_2 \PresA{a \wedge \mu(v_1),a \wedge \mu(v_2)} = \PresA{a \wedge \left(\lambda_1 \mu(v_1)\right),a \wedge \left(\lambda_2 \mu(v_2)\right)} \\
&= \PresA{a \wedge \mu(w_1),a \wedge \mu(w_2)} = \PresA{a \wedge x,a \wedge y},
\end{align*}
as desired.

We can identify
the composition
\[\Phi \circ \psi\colon \fZ_{g-1}^a \longrightarrow \cK(a)\]
with the linearization map for $\fZ_{g-1}^a$.  Our assumption $g \geq 3$ implies that $g-1 \geq 1$, 
so Theorem \ref{maintheorem:spkernelalt} says
that $\Phi \circ \psi$ is an isomorphism.  Since $\psi$ is a surjection, this implies that
$\Phi$ takes $\Fix{a,a}$ isomorphically to $\cK(a)$, as desired.

\begin{step}{2}
Conclusion (b) holds: $\FixIm{a,a} = \cK_g^a \cap \FixIm{a,a} = \cK(a)$.
\end{step}

That $\FixIm{a,a} = \cK(a)$ is immediate from (c).  Conclusion (a) says
that $\FixBigIm{a,a} = \wedge^2 \cU(a)$.  Since $\cK_g^a \cap \FixBigIm{a,a}$ is
a codimension-$1$ subspace of $\FixBigIm{a,a} = \wedge^2 \cU(a)$ containing
$\FixIm{a,a}$ and $\cK(a)$ is also a codimension-$1$ subspace of $\wedge^2 \cU(a)$,
it follows that $\FixIm{a,a} = \cK(a)$ equals $\cK_g^a \cap \FixBigIm{a,a}$, as desired.
\end{proof}

\subsection{Identification II}
\label{section:identificationalt2}

Let $(a,a')$ be a pair of elements of $H_{\Z}$ such that
$\omega(a,a') = 0$ and such that
$\{a,a'\}$ is a basis for a rank-$2$ direct summand of $H_{\Z}$.
This latter condition implies that $a \wedge a'$ is a strong isotropic pair.
We will call such a pair $(a,a')$ an {\em isotropic basis}.
Our next goal is to identify $\Fix{a,a'}$.

In a generator $\PresA{a \wedge x,a' \wedge y}$ of $\Fix{a,a'}$,
we have that $x$ is orthogonal to $a'$ and
$y$ is orthogonal to $a$.  Moreover, $x$ is only well-defined up
to multiples of $a$ and $y$ is only well-defined up to multiples of $a'$.
This suggests defining
\[\cV(a,a') = \Span{a'}_{\Q}^{\perp}/\Span{a} \quad \text{and} \quad \cW(a,a') = \Span{a}_{\Q}^{\perp}/\Span{a'}.\]
We can embed $\cV(a,a')$ and $\cW(a,a')$ into $(\wedge^2 H)/\Q$ by taking $x \in \cV(a,a')$ to
$a \wedge x$ and $y \in \cW(a,a')$ to $a' \wedge y$.  Identifying $\cV(a,a')$ and $\cW(a,a')$
with the corresponding subspaces of $(\wedge^2 H)/\Q$, the intersection
$\cV(a,a') \cap \cW(a,a')$ is spanned by $a \wedge a'$.  Here
$a \wedge a'$ corresponds to $a' \in \cV(a,a')$ and $-a \in \cW(a,a')$.
It follows that as subspaces of $\wedge^2((\wedge^2 H)/\Q)$ we have
\[\FixBigIm{a,a'} = \cV(a,a') \wedge \cW(a,a') \cong \frac{\cV(a,a') \otimes \cW(a,a')}{\Span{a' \otimes a}}.\]
The symplectic form $\omega$ induces a bilinear pairing
\[\oomega\colon \cV(a,a') \times \cW(a,a') \longrightarrow \Q.\]
Let $\cK(a,a')$ be the kernel of the map
\[\cV(a,a') \otimes \cW(a,a') \longrightarrow \Q\]
induced by $\oomega$.  We have $a' \otimes a \in \cK(a,a')$, and by the previous paragraph
$\cK(a,a') / \Span{a' \otimes a}$ is a subspace of $\wedge^2((\wedge^2 H)/\Q)$.
We then have:

\begin{lemma}
\label{lemma:mapbilinear2}
Let $g \geq 3$ and let $(a,a')$ be an isotropic basis.  Then:
\begin{itemize}
\item[(a)] $\FixBigIm{a,a'} = \cV(a,a') \wedge \cW(a,a') \cong \left(\cV(a,a') \otimes \cW(a,a')\right)/\Span{a' \otimes a}$; and
\item[(b)] $\FixIm{a,a'} = \cK_g^a \cap \FixBigIm{a,a'} = \cK(a,a')/\Span{a' \otimes a}$; and
\item[(c)] the linearization map $\Phi\colon \fK_g^a \rightarrow \wedge^2((\wedge^2 H)/\Q)$
takes $\Fix{a,a'}$ isomorphically onto $\cK(a,a')/\Span{a' \otimes a}$.   
\end{itemize}
\end{lemma}
\begin{proof}
We noted that (a) held right before the lemma, and (b) follows from (c) just like in the
proof of Lemma \ref{lemma:mapbilinear1}.  We must prove (c).
Since $(a,a')$ is an isotropic basis, we can find a symplectic basis $\{a_1,b_1,\ldots,a_g,b_g\}$
for $H_{\Z}$ such that $a_1 = a$ and $a_2 = a'$.
Define $\cV_{\Z} = \Span{a_2}^{\perp} / \Span{a_1}$ and $\cW_{\Z} = \Span{a_1}^{\perp} / \Span{a_2}$.
We can identify:
\begin{itemize}
\item $\cV_{\Z}$ with $\Span{b_1,a_2,a_3,b_3,\ldots,a_g,b_g} \cong \Z^{2g-2}$; and
\item $\cW_{\Z}$ with $\Span{a_1,b_2,a_3,b_3,\ldots,a_g,b_g} \cong \Z^{2g-2}$.
\end{itemize}
Under these identifications, the bilinear pairing $\oomega$ between $\cV_{\Z}$ and
$\cW_{\Z}$ induced by $\omega$ is identified with the bilinear pairing between
$\Span{b_1,a_2,a_3,b_3,\ldots,a_g,b_g}$ and $\Span{a_1,b_2,a_3,b_3,\ldots,a_g,b_g}$ given by $\omega$.
Let $\{e_1,\ldots,e_{2g-2}\}$ be the standard basis for $\Z^{2g-2}$ and $\{e_1^{\ast},\ldots,e_{2g-2}^{\ast}\}$
be the corresponding dual basis for $(\Z^{2g-1})^{\ast}$.  Let $\mu_1\colon (\Z^{2g-2})^{\ast} \rightarrow \cV_{\Z}$
and $\mu_2\colon \Z^{2g-1} \rightarrow \cW_{\Z}$ be the isomorphisms defined by
\begin{alignat}{4}
\label{eqn:mathfirst}
&&\mu_1(e_1^{\ast}) &= b_1,\quad  &\mu_1(e_2^{\ast}) &= a_2\\
&&\mu_2(e_1)        &= -a_1,\quad &\mu_2(e_2)        &= b_2\notag
\end{alignat}
and
\begin{alignat*}{5}
&&\mu_1(e_{2i-3}^{\ast}) &= a_i,\quad &\mu_1(e_{2i-2}^{\ast}) &= b_i\quad\quad&\text{for $3 \leq i \leq g$},\\
&&\mu_2(e_{2i-3})        &= b_i,\quad &\mu_2(e_{2i-2})        &= -a_i,\quad\quad&\text{for $3 \leq i \leq g$}.
\end{alignat*}
We chose these isomorphisms in part because they ensure that
\[f(x) = \oomega(\mu_1(f),\mu_2(x)) \quad \text{for all $f \in (\Z^{2g-2})^{\ast}$ and $x \in \Z^{2g-2}$}.\]
The precise form of \eqref{eqn:mathfirst} will be important in the next paragraph.

Recall that we defined the vector
space $\fA'_{2g-2}$ in Definition \ref{definition:sladjointvar}.  It is generated
by elements $\PAGen{f,v}$ with $f \in (\Z^{2g-2})^{\ast}$ and $v \in \Z^{2g-2}$ primitive
vectors such that $f(v) = 0$ and such that:
\begin{itemize}
\item $f$ is $e_1^{\ast}$-standard, which means that the $e_1^{\ast}$-coordinate of $f$ lies
in $\{-1,0,1\}$; and
\item $v$ is $e_2$-standard, which
means that the $e_2$-coordinate of $v$ lies in $\{-1,0,1\}$.
\end{itemize}
In light of \eqref{eqn:mathfirst}, this implies that both $a_1 \wedge \mu_1(v)$ and $a_2 \wedge \mu_2(v)$ are special pairs.
We can therefore define a map $\psi\colon \fA'_{2g-2} \rightarrow \Fix{a_1,a_2}$ via the formula
\[\psi(\PAGen{f,v}) = \PresA{a_1 \wedge \mu_1(f),a_2 \wedge \mu(v)} \quad \text{for a generator $\PAGen{f,v}$ of $\fA'_{2g-2}$.}\]
This takes relations to relations, and thus gives a well-defined map.  We also have:

\begin{unnumberedclaim}
The map $\psi$ is surjective.
\end{unnumberedclaim}
\begin{proof}[Proof of claim]
Consider a generator $\PresA{a_1 \wedge x,a_2 \wedge y}$ of
$\Fix{a_1,a_2}$.  We must show that $\PresA{a_1 \wedge x,a_2 \wedge y}$ is in
the image of $\psi$.  We can assume that our generator is nonzero, so $x \neq 0$ and $y \neq 0$.
Write $x = \mu_1(f)$ and $y = \mu_2(v)$ with $f \in (\Z^{2g-2})^{\ast}$ and $v \in \Z^{2g-2}$.  By definition,
$f$ is $e_1^{\ast}$-standard and $v$ is $e_2$-standard.  Write
$f = \lambda f'$ and $v = \eta v'$ with $\lambda,\eta \in \Z$ and $f'$ and $v'$ primitive.  
The elements $f'$ and $v'$ are $e_1^{\ast}$- and $e_2$-standard, respectively.
We therefore have a generator
$\PAGen{f',v'}$ of $\fA'_{2g-2}$, and
\begin{align*}
\psi(\lambda \eta \PAGen{f',v'}) &= \lambda \eta \PresA{a_1 \wedge \mu_1(f'),a_2 \wedge \mu_2(v')} = \lambda \PresA{a_1 \wedge \mu_1(f'),a_2 \wedge \left(\eta \mu_2(v')\right)}\\
                                 &= \lambda \PresA{a_1 \wedge \mu_1(f'),a_2 \wedge y} = \PresA{a_1 \wedge \left(\lambda \mu_1(f')\right),a_2 \wedge y} = \PresA{a_1 \wedge x,a_2 \wedge y},
\end{align*}
as desired.
\end{proof}

Since $g \geq 3$, we have $2g-2 \geq 4$.  Theorem \ref{maintheorem:sladjointvar} thus gives a linearization map
$\tPhi\colon \fA'_{2g-2} \rightarrow \fsl_{2g-2}(\Q)$ that is an isomorphism.  
Here $\fsl_{2g}(\Q)$ is the kernel of the trace map $(\Q^{2g-2})^{\ast} \otimes \Q^{2g-2} \rightarrow \Q$.  Recall
that $\cK(a_1,a_2)$ is the kernel of the symplectic pairing $\cV(a_1,a_2) \otimes \cW(a_1,a_2) \rightarrow \Q$.  Identifying
$\mu_1$ and $\mu_2$ with maps $(\Q^{2g-2})^{\ast} \rightarrow \cV(a_1,a_2)$ and $\Q^{2g-2} \rightarrow \cW(a_1,a_2)$, this
all fits into a commutative diagram
\[\begin{tikzcd}
\fA'_{2g-2} \arrow[two heads]{d}{\psi} \arrow{r}{\tPhi}[swap]{\cong} & \fsl_{2g-2}(\Q) \arrow{r}{\mu_1 \otimes \mu_2}[swap]{\cong} & \cK(a_1,a_2) \arrow[two heads]{d} \\
\Fix{a_1,a_2} \arrow{rr}{\Phi} & & \cK(a_1,a_2) / \Span{a_2 \otimes a_1}.
\end{tikzcd}\]
By \eqref{eqn:mathfirst}, we have $\psi(\PAGen{e_2^{\ast},-e_1}) = \PresA{a_1 \wedge a_2,a_2 \wedge a_1} = 0$.  Since the isomorphism on the top row
of this diagram takes $\PAGen{e_2^{\ast},-e_1}$ to $a_2 \otimes a_1$, we conclude that the map $\Phi$ on the bottom row
is an isomorphism, as desired.
\end{proof}

\section{Symmetric kernel, alternating version II: the set \texorpdfstring{$S$}{S} and \texorpdfstring{$\SymSp_g$}{SymSpg}}
\label{section:presentationalt2}

In the next three sections, we construct a set $S \subset \fK_g^a$ with $\Span{S} = \fK_g^a$.  This section defines
$S$ and studies its symmetries.  In \S \ref{section:presentationalt3} we will construct a large
number of elements in $\Span{S}$, and then finally in \S \ref{section:presentationalt4} we will prove that $\Span{S} = \fK_g^a$.
Throughout this section, like in the last section we relax our standing assumption that $g \geq 4$ (Assumption \ref{assumption:genusalt}),
so our results include explicit genus ranges when they are necessary.  See Remark \ref{remark:genusalt}.

\subsection{The set \texorpdfstring{$S$}{S}}

Fix a symplectic basis $\cB = \{a_1,b_1,\ldots,a_g,b_g\}$ for $H_{\Z}$.  Define $S = S_{12} \cup S_3$, where\footnote{The reason for calling this set $S_{12}$ is that later it will be expressed
as the union of sets $S_1$ and $S_2$, which also explains why the other set is $S_3$.}
\[S_{12} = \bigcup_{\substack{a,a' \in \cB \\ \omega(a,a') = 0}} \pFix{a,a'} \quad \text{and} \quad
S_3    = \bigcup_{\substack{1 \leq i,j \leq g \\ i \neq j}} \oFix{a_i-b_j,b_i-a_j}.\]
Just like we did here, we will write elements of $\Span{S_{12}}$ in purple and elements of $\Span{S_3}$ in orange; for instance,
\begin{align*}
&\PPresA{a_1 \wedge a_3,a_2 \wedge (a_3 - b_4)} \in S_{12},\\
&\OPresA{(a_1-b_2) \wedge (a_3+a_4),(b_1 - a_2) \wedge (b_3-b_4)} \in S_3.
\end{align*}
The sets $S_{12}$ and $S_3$ are not disjoint, so some elements could be written in either color; for instance,
\begin{align*}
&\PPresA{a_1 \wedge (a_2 - b_3),a_1 \wedge (b_2 - a_3)} \in S_{12},\\
&\OPresA{a_1 \wedge (a_2 - b_3),a_1 \wedge (b_2-a_3)} = \OPresA{(a_2-b_3) \wedge a_1,(b_2 - a_3) \wedge a_1} \in S_3.
\end{align*}

\subsection{Signed symmetric group}

Recall that $\SymSp_g$ consists of all $f \in \Sp_{2g}(\Z)$ such that
for all $x \in \cB$, we have either $f(x) \in \cB$ or $-f(x) \in \cB$.  This is a finite group.
Associated to each $f \in \Sp_{2g}(\Z)$ is a permutation $p$ of $\{1,\ldots,g\}$ such that
for all $1 \leq i \leq g$ the pair $(f(a_i),f(b_i))$ is one of the following:
\[(a_{p(i)},b_{p(i)}),   \quad \text{or} \quad
  (-a_{p(i)},-b_{p(i)}), \quad \text{or} \quad
  (b_{p(i)},-a_{p(i)}),  \quad \text{or} \quad
  (-b_{p(i)},a_{p(i)}).\]
Our main goal in this section is to prove:

\begin{lemma}
\label{lemma:altsymspacts}
For all $g \geq 3$, the action of $\SymSp_g$ on $\fK_g^a$ takes $\Span{S}$ to $\Span{S}$.
\end{lemma}

The proof of Lemma \ref{lemma:altsymspacts} is at the end of this section after some
preliminaries.

\subsection{Symmetric group}

Embed the symmetric group $\fS_g$ on $g$ generators into $\Sp_{2g}(\Z)$ by letting $p\in \fS_g$ act as
$p(a_i) = a_{p(i)}$ and $p(b_i) = b_{p(i)}$ for $1 \leq i \leq g$.
The group $\fS_g$ is a subgroup of $\SymSp_g$.  We start with:

\begin{lemma}
\label{lemma:symmetrics}
For all $g \geq 3$, the action of $\fS_g$ on $\fK_g^a$ takes $\Span{S}$ to $\Span{S}$.
\end{lemma}
\begin{proof}
For $p \in \fS_g$, we have
\begin{alignat*}{3}
&&p(\pFix{a,a'}) &= \pFix{p(a),p(a')} \subset S_{12} && \quad \text{for $a,a \in \cB$ with $\omega(a,a') = 0$},\\
&&p(\oFix{a_i-b_j,b_i-a_j}) &= \oFix{a_{p(i)} - b_{p(j)},b_{p(i)} - a_{p(j)}} \subset S_3 && \quad \text{for $1 \leq i,j \leq g$ distinct}.
\end{alignat*}
The lemma follows.
\end{proof}

\subsection{Some elements, I}

To extend Lemma \ref{lemma:symmetrics} to $\SymSp_g$, we need to construct some
elements in $\Span{S}$.  We start with:

\begin{lemma}
\label{lemma:s12twoa1weak}
Let $g \geq 3$.  For $1 \leq i \leq g$, we have $\Fix{a_i+b_i,a_i+b_i} \subset \Span{S}$.
\end{lemma}
\begin{proof}
By Lemma \ref{lemma:symmetrics}, we can apply any element of $\fS_g$ to $\Fix{a_i+b_i,a_i+b_i}$ without changing
whether or not it lies in $\Span{S}$.  Applying an appropriate such element, we reduce to 
showing that $\Fix{a_1+b_1,a_1+b_1} \subset \Span{S}$.  Following the notation in
\S \ref{section:identificationalt1}, define 
\[\cU = \Span{a_1+b_1}_{\Q}^{\perp}/\Span{a_1+b_1} \cong \Span{A}_{\Q} \quad \text{with} \quad A = \{a_2,b_2,\ldots,a_g,b_g\}.\]
We proved in Lemma \ref{lemma:mapbilinear1} that $\Fix{a_1+b_1,a_1+b_1}$ is isomorphic
to the kernel of the map $\wedge^2 \cU \rightarrow \Q$
induced by the symplectic form $\omega$.  Under this isomorphism, a generator
$\PresA{(a_1+b_1) \wedge x,(a_1+b_1) \wedge y}$ of $\Fix{a_1+b_1,a_1+b_1}$ maps 
to $x \wedge y \in \wedge^2 \cU$.

The kernel of $\wedge^2 \cU \rightarrow \Q$ is spanned by $X = \Set{$x \wedge y$}{$x,y \in A$, $\omega(x,y) = 0$}$
and $Y = \Set{$a_i \wedge b_i - a_j \wedge b_j$}{$2 \leq i < j \leq g$}$.  Since for $2 \leq i < j \leq g$ we have
\[(a_i - b_j) \wedge (b_i - a_j) = a_i \wedge b_i - a_j \wedge b_j + \text{an element of $\Span{X}$},\]
we can replace $Y$ by $\Set{$(a_i - b_j) \wedge (b_i - a_j)$}{$2 \leq i < j \leq g$}$.  It
follows that $\Fix{a_1+b_1,a_1+b_1}$ is generated by the following elements:

\begin{case}{1}
\label{case:s12twoa1weak.1}
$\PresA{(a_1+b_1) \wedge x,(a_1+b_1) \wedge y}$ for $x,y \in A$ with $\omega(x,y) = 0$.  
\end{case}

\noindent
These equal $\PPresA{x \wedge (a_1+b_1),y \wedge (a_1+b_1)} \in S_{12}$.

\begin{case}{2}
\label{case:s12twoa1weak.2}
$\PresA{(a_1+b_1) \wedge (a_i-b_j),(a_1+b_1) \wedge (b_i-a_j)}$ for $2 \leq i < j \leq g$.
\end{case}

\noindent
These equal $\OPresA{(a_i-b_j) \wedge (a_1+b_1),(b_i-a_j) \wedge (a_1+b_1)} \in S_3$.
\end{proof}

\subsection{Some elements, II}

We next handle the following variants of the elements
of $S_3$:

\begin{lemma}
\label{lemma:s3symmetricalt}
Let $g \geq 3$.  For all distinct $1 \leq i,j \leq g$ and $\epsilon,\epsilon' \in \{\pm 1\}$, both
$\Fix{\epsilon a_i + \epsilon' b_j,\epsilon b_i + \epsilon' a_j}$ and
$\Fix{\epsilon a_i + \epsilon' a_j,\epsilon b_i - \epsilon' b_j}$ are subsets of $\Span{S}$.
\end{lemma}
\begin{proof}
Since $\Fix{-,-}$ is not changed when its entries are multiplied by $-1$ (Lemma \ref{lemma:isotropicbilineareasy}),
it is enough to prove this for $\epsilon = 1$.  Also, using Lemma \ref{lemma:symmetrics} we can multiply
our elements by appropriate elements of the symmetric group $\fS_g$ and assume that $(i,j) = (1,2)$.  Since
we already know that $\oFix{a_1 - b_2,b_1-a_2} \subset S_3$, this reduces us to proving
that $\Fix{a_1+b_2,b_1+a_2}$ and $\Fix{a_1+a_2,b_1-b_2}$ and $\Fix{a_1-a_2,b_1+b_2}$ are contained
in $\Span{S}$.  We do this in Lemmas \ref{lemma:s3symmetricalt.1} and \ref{lemma:s3symmetricalt.2}
and \ref{lemma:s3symmetricalt.1} below.
\end{proof}

\begin{lemma}
\label{lemma:s3symmetricalt.1}
Let $g \geq 3$.  We have $\Fix{a_1+b_2,b_1+a_2} \subset \Span{S}$.
\end{lemma}
\begin{proof}
Following the notation in \S \ref{section:identificationalt2}, define
\begin{alignat*}{8}
&&\cV &= &\Span{b_1+a_2}_{\Q}^{\perp}/\Span{a_1+b_2} &\cong &\Span{A_V}_{\Q} &\quad \text{with} \quad &A_V = \{b_1,a_2,a_3,b_3,\ldots,a_g,b_g\},&\\
&&\cW &= &\Span{a_1+b_2}_{\Q}^{\perp}/\Span{b_1+a_2} &\cong &\Span{A_W}_{\Q} &\quad \text{with} \quad &A_W = \{a_1,b_2,a_3,b_3,\ldots,a_g,b_g\}.&
\end{alignat*}
We proved in Lemma \ref{lemma:mapbilinear2} that $\Fix{a_1+b_2,b_1+a_2}$ is isomorphic
to a quotient of the kernel of the map $\cV \otimes \cW \rightarrow \Q$
induced by the symplectic form $\omega$.  Under this isomorphism, a generator
$\PresA{(a_1+b_2) \wedge x,(b_1+a_2) \wedge y}$ of $\Fix{a_1+b_2,b_1+a_2}$ maps
to $x \otimes y \in \cV \otimes \cW$.

The kernel of $\cV \otimes \cW \rightarrow \Q$ is spanned by $X \cup Y$ where\footnote{The first part of $Y$ along with the portion
of $X$ lying in it spans the kernel of the map $\Span{a_3,b_3,\ldots,a_g,b_g}^{\otimes 2} \rightarrow \Q$.  Indeed, quotienting
$\Span{a_3,b_3,\ldots,a_g,b_g}^{\otimes 2}$ by the portion of $X$ lying in it results in
\[U = \Span{a_3 \otimes b_3, b_3 \otimes a_3,\ldots,a_g \otimes b_g,b_g \otimes a_g}.\]
In $Y$, we have $a_i \otimes b_i + b_i \otimes a_i$ for all $3 \leq i \leq g$, and also for $3 \leq i,j \leq g$ distinct
the elements
\begin{align*}
a_i \otimes b_i - a_j \otimes b_j &= (a_i \otimes b_i + b_j \otimes a_j) - (a_j \otimes b_j + b_j \otimes a_j),\\
b_i \otimes a_i - b_j \otimes a_j &= (a_j \otimes b_j + b_i \otimes a_i) - (a_j \otimes b_j + b_j \otimes a_j).
\end{align*}
Quotienting $U$ by these gives $\Q$, as desired.  We will silently use calculations
like this in the next two sections.}
\begin{align*}
X = &\Set{$x \otimes y$}{$x \in A_V$, $y \in A_W$, $\omega(x,y)=0$},\\
Y = &\Set{$a_i \otimes b_i + b_j \otimes a_j$}{$3 \leq i,j \leq g$} \\
    &\cup \{a_3 \otimes b_3 + b_1 \otimes a_1,a_2 \otimes b_2 + b_3 \otimes a_3\}.
\end{align*}
Since for $1 \leq i,j \leq g$ distinct with either $i,j \geq 3$ or
$(i,j) \in \{(3,1),(2,3)\}$ and for $3 \leq k \leq g$ we have
\begin{align*}
(a_i-b_j) \otimes (b_i - a_j) &= a_i \otimes b_i + b_j \otimes a_j + \text{an element of $\Span{X}$},\\
(a_k+b_k) \otimes (a_k + b_k) &= a_k \otimes b_k + b_k \otimes a_k + \text{an element of $\Span{X}$},
\end{align*}
we can replace $Y$ by the set
\begin{align*}
&\Set{$(a_i-b_j) \otimes (b_i-a_j)$}{$1 \leq i,j \leq g$ distinct, either $i,j \geq 3$ or $(i,j) \in \{(3,1),(2,3)\}$} \\
&\quad\quad\cup \Set{$(a_i+b_i) \otimes (a_i + b_i)$}{$3 \leq i \leq g$}.
\end{align*}
From this, we see that $\Fix{a_1+b_2,b_1+a_2}$ is generated by the following elements:\footnote{The elements
$\PresA{(a_1+b_2) \wedge z,(b_1+a_2) \wedge w}$ listed below
each map to one of the generators for the kernel of $\cV \otimes \cW \rightarrow \Q$ we
identified.  It is important that these are indeed generators of $\Fix{a_1+b_2,b_1+a_2}$, i.e., that
both $(a_1+b_2) \wedge z$ and $(b_1+a_2) \wedge w$ are special pairs.}

\begin{case}{1}
$\PresA{(a_1+b_2) \wedge x,(b_1+a_2) \wedge y}$ for $x \in A_V$ and $y \in A_W$ with $\omega(x,y) = 0$.
\end{case}

\noindent
These equal $\PPresA{x \wedge (a_1+b_2),y \wedge (b_1+a_2)} \in S_{12}$.

\begin{case}{2}
$\PresA{(a_1+b_2) \wedge (a_i-b_j),(b_1+a_2) \wedge (b_i-a_j)}$ for $1 \leq i,j \leq g$ distinct
with either $i,j \geq 3$ or $(i,j) \in \{(3,1),(2,3)\}$.
\end{case}

\noindent
These equal $\OPresA{(a_i-b_j) \wedge (a_1+b_2),(b_i-a_j) \wedge (b_1+a_2)} \in S_3$.

\begin{case}{3}
$\PresA{(a_1+b_2) \wedge (a_i+b_i),(b_1+a_2) \wedge (a_i+b_i)}$ for $3 \leq i \leq g$.
\end{case}

\noindent
These equal $\PresA{(a_i+b_i) \wedge (a_1+b_2),(a_i+b_i) \wedge (b_1+a_2)}$, which lie
in $\Span{S}$ by Lemma \ref{lemma:s12twoa1weak}.
\end{proof}

\begin{lemma}
\label{lemma:s3symmetricalt.2}
Let $g \geq 3$.  We have $\Fix{a_1+a_2,b_1-b_2} \subset \Span{S}$.
\end{lemma}
\begin{proof}
This is similar to Lemma \ref{lemma:s3symmetricalt.1}, so we just sketch the argument.  Like
in that lemma, define
\begin{alignat*}{8}
&&\cV &= &\Span{b_1-b_2}_{\Q}^{\perp}/\Span{a_1+a_2} &\cong &\Span{A_V}_{\Q} &\quad \text{with} \quad &A_V = \{b_1,b_2,a_3,b_3,\ldots,a_g,b_g\},&\\
&&\cW &= &\Span{a_1+a_2}_{\Q}^{\perp}/\Span{b_1-b_2} &\cong &\Span{A_W}_{\Q} &\quad \text{with} \quad &A_W = \{a_1,a_2,a_3,b_3,\ldots,a_g,b_g\}.&
\end{alignat*} 
Elements of $\Fix{a_1+a_2,b_1-b_2}$ correspond to elements of a quotient of the kernel
of the map $\cV \otimes \cW \rightarrow \Q$ induced by the symplectic form $\omega$.  Using this, just
like in Lemma \ref{lemma:s3symmetricalt.1} we can reduce to the following three cases, each of which is
handled just like in Lemma \ref{lemma:s3symmetricalt.1}.
\begin{itemize}
\item $\PresA{(a_1+a_2) \wedge x,(b_1-b_2) \wedge y}$ for $x \in A_V$ and $y \in A_W$ with $\omega(x,y) = 0$. 
\item $\PresA{(a_1+a_2) \wedge (a_i-b_j),(b_1-b_2) \wedge (b_i-a_j)}$ for 
$1 \leq i,j \leq g$ distinct with either $i,j \geq 3$ or $(i,j) \in \{(3,1),(3,2)\}$.
\item $\PresA{(a_1+a_2) \wedge (a_i+b_i),(b_1-b_2) \wedge (a_i+b_i)}$ for $3 \leq i \leq g$.\qedhere
\end{itemize}
\end{proof}

\begin{lemma}
\label{lemma:s3symmetricalt.3}
Let $g \geq 3$.  We have $\Fix{a_1-a_2,b_1+b_2} \subset \Span{S}$.
\end{lemma}
\begin{proof}
This is also similar to Lemma \ref{lemma:s3symmetricalt.1}, so we just sketch the argument.  Like
in that lemma, define
\begin{alignat*}{8}
&&\cV &= &\Span{b_1+b_2}_{\Q}^{\perp}/\Span{a_1-a_2} &\cong &\Span{A_V}_{\Q} &\quad \text{with} \quad &A_V = \{b_1,b_2,a_3,b_3,\ldots,a_g,b_g\},&\\
&&\cW &= &\Span{a_1-a_2}_{\Q}^{\perp}/\Span{b_1+b_2} &\cong &\Span{A_W}_{\Q} &\quad \text{with} \quad &A_W = \{a_1,a_2,a_3,b_3,\ldots,a_g,b_g\}.&
\end{alignat*}
Elements of $\Fix{a_1-a_2,b_1+b_2}$ correspond to elements of a quotient of the kernel
of the map $\cV \otimes \cW \rightarrow \Q$ induced by the symplectic form $\omega$.  Using this, just
like in Lemma \ref{lemma:s3symmetricalt.1} we can reduce to the following three cases, each of which is
handled just like in Lemma \ref{lemma:s3symmetricalt.1}.
\begin{itemize}
\item $\PresA{(a_1-a_2) \wedge x,(b_1+b_2) \wedge y}$ for $x \in A_V$ and $y \in A_W$ with $\omega(x,y) = 0$.
\item $\PresA{(a_1-a_2) \wedge (a_i-b_j),(b_1+b_2) \wedge (b_i-a_j)}$ for 
$1 \leq i,j \leq g$ distinct with either $i,j \geq 3$ or $(i,j) \in \{(3,1),(3,2)\}$.
\item $\PresA{(a_1-a_2) \wedge (a_i+b_i),(b_1+b_2) \wedge (a_i+b_i)}$ for $3 \leq i \leq g$.\qedhere
\end{itemize}
\end{proof}

\subsection{Closure under signed symmetric group}

We now prove Lemma \ref{lemma:altsymspacts}, whose statement we recall:

\newtheorem*{lemma:altsymspacts}{Lemma \ref{lemma:altsymspacts}}
\begin{lemma:altsymspacts}
For all $g \geq 3$, the action of $\SymSp_g$ on $\fK_g^a$ takes $\Span{S}$ to $\Span{S}$.
\end{lemma:altsymspacts}
\begin{proof}
It is enough to prove that $\SymSp_g$ takes both $S_{12}$ and $S_3$ into $\Span{S}$:

\begin{claim}{1}
For $f \in \SymSp_g$ and $a,a' \in \cB$ with $\omega(a,a') = 0$, we have $f(\pFix{a,a'}) \subset \Span{S}$.
\end{claim}

Write $f(a) = \epsilon a_0$ and $f(a') = \epsilon' a'_0$ with $a_0,a'_0 \in \cB$ and $\epsilon,\epsilon \in \{\pm 1\}$.
We have $\omega(a_0,a'_0) = 0$, and by Lemma \ref{lemma:isotropicbilineareasy} we have
\[f(\pFix{a,a'}) = \Fix{\epsilon a_0,\epsilon' a'_0} = \pFix{a_0,a'_0} \subset S_{12}.\]

\begin{claim}{2}
For $f \in \SymSp_g$ and $1 \leq i < j \leq g$, we have $f(\oFix{a_i-b_j,b_i-a_j}) \subset \Span{S}$.
\end{claim}

For some distinct $1 \leq i_0,j_0 \leq g$ we have
\[f(a_i),f(b_i) \in \{\pm a_{i_0},\pm b_{i_0}\} \quad \text{and} \quad f(a_j),f(b_j) \in \{\pm a_{j_0},\pm b_{j_0}\}.\]
Since
\[\omega(f(a_i-b_i),f(b_i-a_j)) = \omega(a_i-b_i,b_i-a_j) = 0,\]
it follows that for some distinct $1 \leq i_0,j_0 \leq g$ and $\epsilon,\epsilon' \in \{\pm 1\}$ the
set $f(\oFix{a_i-b_j,b_i-a_j})$ equals either
\[\Fix{\epsilon a_{i_0} + \epsilon' b_{j_0},\epsilon b_{i_0} + \epsilon' a_{j_0}} \quad \text{or} \quad
\Fix{\epsilon a_{i_0} + \epsilon' a_{j_0},\epsilon b_{i_0} - \epsilon' b_{j_0}}.\]
Lemma \ref{lemma:s3symmetricalt} says that both of these are contained in $\Span{S}$.
\end{proof}

One useful consequence of Lemma \ref{lemma:altsymspacts} is the following
generalization of Lemma \ref{lemma:s12twoa1}:

\begin{lemma}
\label{lemma:s12twoa1}
Let $g \geq 3$.  For $1 \leq i \leq g$ and $\epsilon \in \{\pm 1\}$, we have $\Fix{a_i+\epsilon b_i,a_i+\epsilon b_i} \subset \Span{S}$.
\end{lemma}
\begin{proof}
The group $\SymSp_g$ can take $a_i+b_i$ to $a_i-b_i$ by mapping $a_i$ to $-b_i$ and $b_i$ to $a_i$.  From
this and Lemma \ref{lemma:altsymspacts}, the lemma reduces to Lemma \ref{lemma:s12twoa1weak}.
\end{proof}

\section{Symmetric kernel, alternating version III: eight elements}
\label{section:presentationalt3}

We now re-impose our standing assumption $g \geq 4$ (Assumption \ref{assumption:genusalt}), which will
remain in place until we say otherwise.
We continue using all the notation from \S \ref{section:presentationalt2}.  Our goal
in this section is to prove eight lemmas saying that certain elements lie in $\Span{S}$.
They might appear random, but they are exactly\footnote{Except for two that are needed
for the proofs of ones needed in the next section.} the ones needed in the next section (\S \ref{section:presentationalt4}) to
prove that $\Span{S} = \fK_g^a$, and as motivation a reader might first consult that
proof.

\begin{remark}
We apologize for the repetitiveness of the proofs of our lemmas.  They all follow the same
pattern, but each has small twists and it is important to prove them in the right order
so that earlier lemmas can be invoked during the proofs of later ones.
\end{remark}

\subsection{Two vs one}

Our first three lemmas are about $\Fix{a,a'}$ where $a$ uses two generators and $a'$ is a single generator.
We remark that during our proofs
we will frequently silently be using the fact that $g \geq 4$ (Assumption \ref{assumption:genusalt}).

\begin{lemma}[Two vs one, I]
\label{lemma:twovsone1}
For all $x \in \cB \setminus \{a_1,b_1\}$, both $\Fix{a_1+b_1,x}$ and $\Fix{a_1-b_1,x}$ are subsets
of $\Span{S}$.
\end{lemma}
\begin{proof}
Using Lemma \ref{lemma:altsymspacts}, we can apply an appropriate element of $\SymSp_g$ and
reduce ourselves to proving that $\Fix{a_1+b_1,a_2} \subset \Span{S}$.
Following the notation in \S \ref{section:identificationalt2}, define
\begin{alignat*}{8}
&&\cV &= &\Span{a_2}_{\Q}^{\perp}/\Span{a_1+b_1} &\cong &\Span{A_V,A_U}_{\Q} &\quad \text{with} \quad &&A_V = \{a_1,a_2,a_3,b_3,\ldots,a_g,b_g\},&&\\
&&\cW &= &\Span{a_1+b_1}_{\Q}^{\perp}/\Span{a_2} &\cong &\Span{A_W,A_U}_{\Q} &\quad \text{with} \quad &&A_W = \{a_1+b_1,b_2,a_3,b_3,\ldots,a_g,b_g\}.&&
\end{alignat*}
We proved in Lemma \ref{lemma:mapbilinear2} that $\Fix{a_1+b_1,a_2}$ is isomorphic
to a quotient of the kernel of the map $\cV \otimes \cW \rightarrow \Q$
induced by the symplectic form $\omega$.  Under this isomorphism, a generator $\PresA{(a_1+b_1) \wedge x,a_2 \wedge y}$
of $\Fix{a_1+b_1,a_2}$ maps to $x \otimes y \in \cV \otimes \cW$.

The kernel of $\cV \otimes \cW \rightarrow \Q$ is spanned by $X \cup Y$ where\footnote{See the footnotes in the proof of
Lemma \ref{lemma:s3symmetricalt.1} for more on this.  The elements 
are carefully chosen such that all the elements $\PresA{(a_1+b_1) \wedge z,a_2 \wedge w}$ that appears
in Cases \ref{case:twovsone1.1} -- \ref{case:twovsone1.4} below are actually generators of $\Fix{a_1+b_1,a_2}$, i.e., both
$(a_1+b_1) \wedge z$ and $a_2 \wedge w$ are special pairs.  The most delicate choice here is
$a_2 \otimes b_2 - a_1 \otimes (a_1+b_1)$, which is also chosen to make the calculation in Case \ref{case:twovsone1.4} easier.}
\begin{align*}
X = &\Set{$x \otimes y$}{$x \in A_V$, $y \in A_W$, $\omega(x,y)=0$},\\
Y = &\Set{$a_i \otimes b_i + b_j \otimes a_j$}{$3 \leq i,j \leq g$}\\
    &\cup \{a_2 \otimes b_2 - a_1 \otimes (a_1+b_1),a_2 \otimes b_2 + b_3 \otimes a_3\}.
\end{align*}
Since for $1 \leq i,j \leq g$ distinct with either $i,j \geq 3$ or $(i,j) = (2,3)$ and for $3 \leq k \leq g$ we have
\begin{align*}
(a_i-b_j) \otimes (b_i - a_j)   &= a_i \otimes b_i       + b_j \otimes a_j + \text{an element of $\Span{X}$},\\
(a_k+b_k) \otimes (a_k + b_k)   &= a_k \otimes b_k       + b_k \otimes a_k + \text{an element of $\Span{X}$},\\
(a_1-a_2) \otimes (a_1+b_1+b_2) &= a_1 \otimes (a_1+b_1) - a_2 \otimes b_2 + \text{an element of $\Span{X}$},
\end{align*}
we can replace $Y$ by the set
\begin{align*}
&\Set{$(a_i-b_j) \otimes (b_i-a_j)$}{$1 \leq i,j \leq g$ distinct, either $i,j \geq 3$ or $(i,j) = (2,3)$} \\
&\cup \Set{$(a_i+b_i) \otimes (a_i + b_i)$}{$3 \leq i \leq g$} \\
&\cup \{(a_1-a_2) \otimes (a_1+b_1+b_2)\}.
\end{align*}
From this, we see that $\Fix{a_1+b_1,a_2}$ is generated by the following elements:

\begin{case}{1}
\label{case:twovsone1.1}
$\PresA{(a_1+b_1) \wedge x,a_2 \wedge y}$ for $x \in A_V$ and $y \in A_W$ with $\omega(x,y) = 0$.
\end{case}

\noindent
These equal $-\PPresA{x \wedge (a_1+b_1),a_2 \wedge y} \in S_{12}$.

\begin{case}{2}
\label{case:twovsone1.2}
$\PresA{(a_1+b_1) \wedge (a_i-b_j),a_2 \wedge (b_i-a_j)}$ with $1 \leq i,j \leq g$ distinct and either
$i,j \geq 3$ or $(i,j) = (2,3)$.
\end{case}

\noindent
These equal $\OPresA{(a_i-b_j) \wedge (a_1+b_1),(b_i-a_j) \wedge a_2} \in S_3$.

\begin{case}{3}
\label{case:twovsone1.3}
$\PresA{(a_1+b_1) \wedge (a_i+b_i),a_2 \wedge (a_i+b_i)}$ with $3 \leq i \leq g$.
\end{case}

\noindent
These equal $\PresA{(a_i+b_i) \wedge (a_1+b_1),(a_i+b_i) \wedge a_2}$, which lie
in $\Span{S}$ by Lemma \ref{lemma:s12twoa1weak}.

\begin{case}{4}
\label{case:twovsone1.4}
$\PresA{(a_1+b_1) \wedge (a_2-a_1),a_2 \wedge (a_1+b_1+b_2)}$.
\end{case}

\noindent
Since $\{a_1+b_1,a_2-a_1,a_2,a_1+b_1+b_2,a_3,b_3,\ldots,a_g,b_g\}$ is a symplectic basis for $H_{\Z}$,
in $(\wedge^2 H)/\Q$ we have
\[(a_1+b_1) \wedge (a_2+a_1) + a_2 \wedge (a_1+b_1+b_2) + a_3 \wedge b_3 + \cdots + a_g \wedge b_g = 0.\]
Solving for $a_2 \wedge (a_1+b_1+b_2)$ and plugging the result into the second
entry of our element, we see that $\PresA{(a_1+b_1) \wedge (a_2-a_1),a_2 \wedge (a_1+b_1+b_2)}$ equals
\begin{align*}
 &-\PresA{(a_1+b_1) \wedge (a_2-a_1),(a_1+b_1) \wedge (a_2-a_1)} - \sum\nolimits_{i=3}^g \PresA{(a_1+b_1) \wedge (a_2-a_1), a_i \wedge b_i} \\
=&-\sum\nolimits_{i=3}^g \left(\PPresA{a_1 \wedge a_2,a_i \wedge b_i} + \PPresA{b_1 \wedge a_2,a_i \wedge b_i} - \PPresA{b_1 \wedge a_1,a_i \wedge b_i}\right) \in \Span{S_{12}}.\qedhere
\end{align*}
\end{proof}

\begin{lemma}[Two vs one, II]
\label{lemma:twovsone2}
Both $\Fix{a_1+b_2,b_2}$ and $\Fix{a_2+b_1,b_1}$ are subsets of $\Span{S}$.
\end{lemma}
\begin{proof}
The subsets differ by an element of $\SymSp_g$, so by
Lemma \ref{lemma:altsymspacts} it is enough to prove that $\Fix{a_1+b_2,b_2} \subset \Span{S}$.
Following the notation in \S \ref{section:identificationalt2}, define
\begin{alignat*}{8}
&&\cV &= &\Span{b_2}_{\Q}^{\perp}/\Span{a_1+b_2} &\cong &\Span{A_V}_{\Q} &\quad \text{with} \quad &&A_V = \{b_1,b_2,a_3,b_3,\ldots,a_g,b_g\},&&\\
&&\cW &= &\Span{a_1+b_2}_{\Q}^{\perp}/\Span{b_2} &\cong &\Span{A_W}_{\Q} &\quad \text{with} \quad &&A_W = \{a_1,b_1+a_2,a_3,b_3,\ldots,a_g,b_g\}.&&
\end{alignat*}
We proved in Lemma \ref{lemma:mapbilinear2} that $\Fix{a_1+b_2,b_2}$ is isomorphic
to a quotient of the kernel of the map $\cV \otimes \cW \rightarrow \Q$
induced by the symplectic form $\omega$.  Under this isomorphism, a generator
$\PresA{(a_1+b_2) \wedge x,b_2 \wedge y}$ of $\Fix{a_1+b_2,b_2}$ maps to
$x \otimes y \in \cV \otimes \cW$.

The kernel of $\cV \otimes \cW \rightarrow \Q$ is spanned by $X \cup Y$ where
\begin{align*}
X = &\Set{$x \otimes y$}{$x \in A_V$, $y \in A_W$, $\omega(x,y)=0$},\\
Y = &\Set{$a_i \otimes b_i + b_j \otimes a_j$}{$i,j \geq 3$} \\
    &\cup \{a_3 \otimes b_3 + b_1 \otimes a_1,b_1 \otimes a_1 - b_2 \otimes (b_1+a_2)\}.
\end{align*}
Just like in the proof of Lemma \ref{lemma:twovsone1}, we can replace $Y$ by the set
\begin{align*}
&\Set{$(a_i-b_j) \otimes (b_i-a_j)$}{$1 \leq i,j \leq g$ distinct, either $i,j \geq 3$ or $(i,j) = (3,1)$} \\
&\cup \Set{$(a_i+b_i) \otimes (a_i + b_i)$}{$3 \leq i \leq g$}\\
&\cup \{(b_1-b_2) \otimes (a_1+b_1+a_2)\}.
\end{align*}
From this, we see that $\Fix{a_1+b_2,b_2}$ is generated by the following elements:

\begin{case}{1}
$\PresA{(a_1+b_2) \wedge x,b_2 \wedge y}$ for $x \in A_V$ and $y \in A_W$ with $\omega(x,y) = 0$. 
\end{case}

\noindent
These equal $-\PPresA{x \wedge (a_1+b_2),b_2 \wedge y} \in S_{12}$.

\begin{case}{2}
$\PresA{(a_1+b_2) \wedge (a_i-b_j),b_2 \wedge (b_i-a_j)}$ for $1 \leq i,j \leq g$ distinct
with either $i,j \geq 3$ or $(i,j) = (3,1)$.
\end{case}

\noindent
These equal $\OPresA{(a_i-b_j) \wedge (a_1+b_2),(b_i-a_j) \wedge b_2} \in S_3$.

\begin{case}{3}
$\PresA{(a_1+b_2) \wedge (a_i+b_i),b_2 \wedge (a_i+b_i)}$ for $3 \leq i \leq g$.
\end{case}

\noindent
These equal $\PresA{(a_i+b_i) \wedge (a_1+b_2),(a_i+b_i) \wedge b_2}$, which lie
in $\Span{S}$ by Lemma \ref{lemma:s12twoa1weak}.

\begin{case}{4}
$\PresA{(a_1+b_2) \wedge (b_1-b_2),b_2 \wedge (a_1+b_1+a_2)}$.
\end{case}

\noindent
This equals $\PresA{(a_1+b_1) \wedge (b_1-b_2),b_2 \wedge (a_1+b_1+a_2)}$, which
lies in $\Span{S}$ by Lemma \ref{lemma:twovsone1}.
\end{proof}

\begin{lemma}[Two vs one, III]
\label{lemma:twovsone3}
The following hold:
\begin{itemize}
\item For $a' \in \cB \setminus \{a_1,b_1,a_2,b_2\}$, we have $\Fix{a_1+b_2,a'} \subset \Span{S}$.
\item For $a' \in \cB \setminus \{a_1,a_2,a_2,b_2\}$, we have $\Fix{a_2+b_1,a'} \subset \Span{S}$.
\end{itemize}
\end{lemma}
\begin{proof}
The two sets differ by an element of $\SymSp_g$, so by Lemma \ref{lemma:altsymspacts} it
is enough to deal with $\Fix{a_1+b_2,a'}$.  Applying a further element of $\SymSp_g$, we can
reduce to the case $a' = a_3$, i.e., to $\Fix{a_1+b_2,a_3}$.  In fact, to simplify our notation
we will apply yet another element of $\SymSp_g$ and transform our goal into 
proving that $\Fix{a_1+a_3,a_2} \subset \Span{S}$.
Following the notation in \S \ref{section:identificationalt2}, define
\begin{alignat*}{8}
&&\cV &= &\Span{a_2}_{\Q}^{\perp}/\Span{a_1+a_3} &\cong &\Span{A_V}_{\Q} &\quad \text{with} \quad &&A_V = \{b_1,a_1,    a_2,b_3,a_4,b_4,\ldots,a_g,b_g\},&&\\
&&\cW &= &\Span{a_1+a_3}_{\Q}^{\perp}/\Span{a_2} &\cong &\Span{A_W}_{\Q} &\quad \text{with} \quad &&A_W = \{a_1,b_1-b_3,b_2,a_3,a_4,b_4,\ldots,a_g,b_g\}.&&
\end{alignat*}
We proved in Lemma \ref{lemma:mapbilinear2} that $\Fix{a_1+a_3,a_2}$ is isomorphic
to a quotient of the kernel of the map $\cV \otimes \cW \rightarrow \Q$
induced by the symplectic form $\omega$.  Under this isomorphism, a generator
$\PresA{(a_1+a_3) \wedge x,a_2 \wedge y}$ of $\Fix{a_1+a_3,a_2}$ maps to
$x \otimes y \in \cV \otimes \cW$.

The kernel of $\cV \otimes \cW \rightarrow \Q$ is spanned by $X \cup Y$ where
\begin{align*}
X = &\Set{$x \otimes y$}{$x \in A_V$, $y \in A_W$, $\omega(x,y)=0$},\\
Y = &\Set{$a_i \otimes b_i + b_j \otimes a_j$}{$4 \leq i,j \leq g$} \\ 
    &\cup \{a_4 \otimes b_4 + b_1 \otimes a_1,b_1 \otimes a_1 + a_1 \otimes (b_1-b_3),a_2 \otimes b_2+b_4 \otimes a_4,a_4 \otimes b_4 + b_3 \otimes a_3\}.
\end{align*}
Just like in the proof of Lemma \ref{lemma:twovsone1}, we can replace $Y$ by
\begin{align*}
&\Set{$(a_i-b_j) \otimes (b_i-a_j)$}{$1 \leq i,j \leq g$ distinct, either $i,j \geq 4$ or $(i,j) \in \{(4,1),(2,4),(4,3)\}$} \\
&\cup \Set{$(a_i+b_i) \otimes (a_i + b_i)$}{$4 \leq i \leq g$} \\
&\cup \{(a_1+b_1) \otimes (a_1+b_1-b_3)\}.
\end{align*}
From this, we see that $\Fix{a_1+a_3,a_2}$ is generated by the following elements:

\begin{case}{1}
$\PresA{(a_1+a_3) \wedge x,a_2 \wedge y}$ for $x \in A_V$ and $y \in A_W$ with $\omega(x,y) = 0$.
\end{case}

\noindent
These equal $-\PPresA{x \wedge (a_1+a_3),a_2 \wedge y} \in S_{12}$.

\begin{case}{2}
$\PresA{(a_1+a_3) \wedge (a_i-b_j),a_2 \wedge (b_i-a_j)}$ for $1 \leq i,j \leq g$ distinct
with either $i,j \geq 4$ or $(i,j) \in \{(4,1),(2,4),(4,3)\}$.
\end{case}

\noindent
These equal $\OPresA{(a_i-b_j) \wedge (a_1+a_3),(b_i-a_j) \wedge a_2} \in S_3$.

\begin{case}{3}
$\PresA{(a_1+a_3) \wedge (a_i+b_i),a_2 \wedge (a_i+b_i)}$ for $4 \leq i \leq g$.
\end{case}

\noindent
These equal $\PresA{(a_i+b_i) \wedge (a_1+a_3),(a_i+b_i) \wedge a_2}$, which lie
in $\Span{S}$ by Lemma \ref{lemma:s12twoa1weak}.

\begin{case}{4}
$\PresA{(a_1+a_3) \wedge (a_1+b_1),a_2 \wedge (a_1+b_1-b_3)}$.
\end{case}

\noindent
This equals $-\PresA{(a_1+b_1) \wedge (a_1+a_3),a_2 \wedge (a_1+b_1-b_3)}$, which 
lies in $\Span{S}$ by Lemma \ref{lemma:twovsone1}.
\end{proof}

\subsection{Two vs two}

Our next three lemmas are about $\Fix{a,a'}$ where both $a$ and $a'$ involve
two generators.  Note that we already proved many such results
in Lemma \ref{lemma:s3symmetricalt}.

\begin{lemma}[Two vs two, I]
\label{lemma:twovstwo1}
Both $\Fix{a_1+b_2,a_1+b_2}$ and $\Fix{a_2+b_1,a_2+b_1}$ are subsets of $\Span{S}$.
\end{lemma}
\begin{proof}
The pairs $(a_1+b_2,a_1+b_2)$ and $(a_2+b_1,a_2+b_1)$ differ by an element
of $\SymSp_g$, so by Lemma \ref{lemma:altsymspacts} it is enough to prove that
$\Fix{a_1+b_2,a_1+b_2} \subset \Span{S}$.  Following the notation in
\S \ref{section:identificationalt1}, define
\[\cU = \Span{a_1+b_2}_{\Q}^{\perp}/\Span{a_1+b_2} \cong \Span{A}_{\Q} \quad \text{with} \quad A = \{b_1+a_2,b_2,a_3,b_3,\ldots,a_g,b_g\}.\]
We proved in Lemma \ref{lemma:mapbilinear1} that $\Fix{a_1+b_2,a_1+b_2}$ is isomorphic
to the kernel of the map $\wedge^2 \cU \rightarrow \Q$
induced by the symplectic form $\omega$.  Under this isomorphism, a generator
$\PresA{(a_1+b_2) \wedge x,(a_1+b_2) \wedge y}$ of $\Fix{a_1+b_2,a_1+b_2}$ maps to
$x \wedge y \in \wedge^2 \cU$.

The kernel of $\wedge^2 \cU \rightarrow \Q$ is spanned by $X \cup Y$ with
\begin{align*}
X = &\Set{$x \wedge y$}{$x,y \in A$ distinct, $\omega(x,y) = 0$},\\
Y = &\Set{$a_i \wedge b_i - a_j \wedge b_j$}{$3 \leq i < j \leq g$} 
    \cup \{(b_1+a_2) \wedge b_2 - a_3 \wedge b_3\}.
\end{align*}
Just like in the proof of Lemma \ref{lemma:twovsone1}, we can replace $Y$ by
\begin{align*}
&\Set{$(a_i - b_j) \wedge (b_i - a_j)$}{$3 \leq i < j \leq g$} 
\cup \{(b_1+a_2+b_3) \wedge (b_2+a_3)\}.
\end{align*}
It follows that $\Fix{a_1+b_2,a_1+b_2}$ is generated by the following elements:

\begin{case}{1}
\label{case:twovstwo1.1}
$\PresA{(a_1+b_2) \wedge x,(a_1+b_2) \wedge y}$ for $x,y \in A$ with $\omega(x,y) = 0$.
\end{case}

\noindent
If $x,y \in \{b_2,a_3,b_3,\ldots,a_g,b_g\}$, then
\[\PresA{(a_1+b_2) \wedge x,(a_1+b_2) \wedge y} = \PPresA{x \wedge (a_1+b_2),y \wedge (a_1+b_2)} \in S_{12}.\]
If instead one of $x$ and $y$ equals $b_1+a_2$, then swapping $x$ and $y$ if necessary we can assume that
$y = b_1+a_2$.  In this case, using Lemma \ref{lemma:s3symmetricalt} we have
\[\PresA{(a_1+b_2) \wedge x,(a_1+b_2) \wedge (b_1+a_2)} = -\PresA{(a_1+b_2) \wedge x,(b_1+a_2) \wedge (a_1+b_2)} \in \Span{S}.\]

\begin{case}{2}
\label{case:twovstwo1.2}
$\PresA{(a_1+b_2) \wedge (a_i-b_j),(a_1+b_2) \wedge (b_i-a_j)}$ for $3 \leq i < j \leq g$.
\end{case}

\noindent
These equal $\OPresA{(a_i-b_j) \wedge (a_1+b_2),(b_i-a_j) \wedge (a_1+b_2)} \in S_3$.

\begin{case}{3}
\label{case:twovstwo1.3}
$\PresA{(a_1+b_2) \wedge (b_1+a_2+b_3),(a_1+b_2) \wedge (b_2+a_3)}$.
\end{case}

\noindent
This element equals
\begin{equation}
\label{eqn:twovstwo1.3toprove}
\PresA{a_1 \wedge (b_1+a_2+b_3),(a_1+b_2) \wedge (b_2+a_3)} + \PresA{b_2 \wedge (b_1+a_2+b_3),(a_1+b_2) \wedge (b_2+a_3)}.
\end{equation}
This is the sum of an element of $\Fix{a_1,a_1+b_2} = \Fix{a_1+b_2,a_1}$ and an element of $\Fix{b_2,a_1+b_2} = \Fix{a_1+b_2,b_2}$.
Lemma \ref{lemma:twovsone2} says that $\Fix{a_1+b_2,b_2} \subset \Span{S}$.  The set $\Fix{a_1+b_2,a_1}$ differs from
$\Fix{a_1+b_2,b_2}$ by an element of $\SymSp_g$, so by Lemma \ref{lemma:altsymspacts} it also lies in $\Span{S}$.  We conclude
that \eqref{eqn:twovstwo1.3toprove} also lies in $\Span{S}$, as desired.
\end{proof}

The next two lemmas are not needed in the next section, but will be invoked during proofs
later in this section.

\begin{lemma}[Two vs two, II]
\label{lemma:twovstwo2}
We have $\Fix{a_1+a_2,a_1+a_3} \subset \Span{S}$.
\end{lemma}
\begin{proof}
Following the notation in \S \ref{section:identificationalt2}, define
\begin{alignat*}{8}
&&\cV &= &\Span{a_1+a_3}_{\Q}^{\perp}/\Span{a_1+a_2} &\cong &\Span{A_V}_{\Q} &\quad \text{with} \quad &&A_V = \{a_1,    b_1-b_3,b_2,a_3,a_4,b_4,\ldots,a_g,b_g\},&&\\
&&\cW &= &\Span{a_1+a_2}_{\Q}^{\perp}/\Span{a_1+a_3} &\cong &\Span{A_W}_{\Q} &\quad \text{with} \quad &&A_W = \{b_1-b_2,a_1,    a_2,b_3,a_4,b_4,\ldots,a_g,b_g\}.&&
\end{alignat*}
We proved in Lemma \ref{lemma:mapbilinear2} that $\Fix{a_1+a_2,a_1+a_3}$ is isomorphic
to a quotient of the kernel of the map $\cV \otimes \cW \rightarrow \Q$
induced by the symplectic form $\omega$.  Under this isomorphism, a generator
$\PresA{(a_1+a_2) \wedge x,(a_1+a_3) \wedge y}$ of $\Fix{a_1+a_2,a_1+a_3}$ maps to
$x \otimes y \in \cV \otimes \cW$.

The kernel of $\cV \otimes \cW \rightarrow \Q$ is spanned by $X \cup Y$ where
\begin{align*}
X = &\Set{$x \otimes y$}{$x \in A_V$, $y \in A_W$, $\omega(x,y)=0$},\\
Y = &\Set{$a_i \otimes b_i + b_j \otimes a_j$}{$4 \leq i,j \leq g$}\\ 
    &\cup \{a_1 \otimes (b_1-b_2)+b_2 \otimes a_2, a_3 \otimes b_3 + (b_1-b_3) \otimes a_1, \\
&\quad\quad a_4 \otimes b_4 + b_2 \otimes a_2,a_3 \otimes b_3 + b_4 \otimes a_4\}.
\end{align*}
Just like in the proof of Lemma \ref{lemma:twovsone1}, we can replace $Y$ by
\begin{align*}
&\Set{$(a_i-b_j) \otimes (b_i-a_j)$}{$1 \leq i,j \leq g$ distinct, either $i,j \geq 4$ or $(i,j) \in \{(4,2),(3,4))$} \\
&\cup \Set{$(a_i+b_i) \otimes (a_i + b_i)$}{$4 \leq i \leq g$} \\
&\cup \{(a_1+b_2) \otimes (b_1+a_2-b_2), (b_1+a_3-b_3) \otimes (a_1+b_3)\}.
\end{align*}
From this, we see that $\Fix{a_1+a_2,a_1+a_3}$ is generated by the following elements:

\begin{case}{1}
$\PresA{(a_1+a_2) \wedge x,(a_1+a_3) \wedge y}$ for $x \in A_V$ and $y \in A_W$ with $\omega(x,y) = 0$.
\end{case}

\noindent
There are three cases:
\begin{itemize}
\item If $y = b_1-b_2$, then this equals $-\PresA{(a_1+a_2) \wedge x,(b_1-b_2) \wedge (a_1+a_3)}$, which
lies in $\Span{S}$ by Lemma \ref{lemma:s3symmetricalt}.
\item If $x = b_1-b_3$, then this equals $\PresA{(a_1+a_3) \wedge y,(b_1-b_3) \wedge (a_1+a_2)}$, which
lies in $\Span{S}$ by Lemma \ref{lemma:s3symmetricalt}.
\item If neither equality holds, then
this equals $\PPresA{x \wedge (a_1+a_2),y \wedge (a_1+a_3)} \in S_{12}$.
\end{itemize}

\begin{case}{2}
$\PresA{(a_1+a_2) \wedge (a_i-b_j),(a_1+a_3) \wedge (b_i-a_j)}$ for $1 \leq i,j \leq g$ distinct
with either $i,j \geq 4$ or $(i,j) \in \{(4,2),(3,4))$.
\end{case}

\noindent
These equal $\OPresA{(a_i-b_j) \wedge (a_1+a_2),(b_i-a_j) \wedge (a_1+a_3)} \in S_3$.

\begin{case}{3}
$\PresA{(a_1+a_2) \wedge (a_i+b_i),(a_1+a_3) \wedge (a_i+b_i)}$ for $4 \leq i \leq g$.
\end{case}

\noindent
These equal $\PresA{(a_i+b_i) \wedge (a_1+a_2),(a_i+b_i) \wedge (a_1+a_3)}$, which lie
in $\Span{S}$ by Lemma \ref{lemma:s12twoa1weak}.

\begin{case}{4}
The following elements:
\begin{itemize}
\item[(a)] $\PresA{(a_1+a_2) \wedge (a_1+b_2),(a_1+a_3) \wedge (b_1+a_2-b_2)}$
\item[(a)] $\PresA{(a_1+a_2) \wedge (b_1+a_3-b_3),(a_1+a_3) \wedge (a_1+b_3)}$
\end{itemize}
\end{case}

\noindent
The element in (a) equals
\begin{equation}
\label{eqn:twotwo.2}
\PresA{(a_1+a_2) \wedge (a_1+b_2),a_1 \wedge (b_1+a_2-b_2)} + \PresA{(a_1+a_2) \wedge (a_1+b_2),a_3 \wedge (b_1+a_2-b_2)}
\end{equation}
The first term lies in $\Fix{a_1+a_2,a_1}$, which differs from $\Fix{a_1+b_2,b_2}$ by an element of
$\SymSp_g$.  Lemma \ref{lemma:twovsone2} says that $\Fix{a_1+b_2,b_2} \subset \Span{S}$, so by
Lemma \ref{lemma:altsymspacts} the set $\Fix{a_1+a_2,a_1}$ lies in $\Span{S}$ too.  The same argument (but
using Lemma \ref{lemma:twovsone3} instead of Lemma \ref{lemma:twovsone2})
shows that the second term in \eqref{eqn:twotwo.2} also lies in $\Span{S}$, so \eqref{eqn:twotwo.2} lies in
$\Span{S}$.

The element in (b) equals
\[\PresA{a_1 \wedge (b_1+a_3-b_3),(a_1+a_3) \wedge (a_1+b_3)} + \PresA{a_2 \wedge (b_1+a_3-b_3),(a_1+a_3) \wedge (a_1+b_3)}.\]
Using the fact that $\Fix{-,-}$ is symmetric in its inputs, this can be shown to lie in $\Span{S}$ just like in the previous paragraph.
\end{proof}

\begin{lemma}[Two vs two, III]
\label{lemma:twovstwo3}
We have $\Fix{a_1+a_2,a_3+a_4} \subset \Span{S}$.
\end{lemma}
\begin{proof}
Following the notation in \S \ref{section:identificationalt2}, define
\begin{alignat*}{8}
&&\cV &= &\Span{a_3+a_4}_{\Q}^{\perp}/\Span{a_1+a_2} &\cong &\Span{A_V}_{\Q} &\text{ with } &&A_V = \{a_1,b_1,b_2,a_3,a_4,b_3-b_4,a_5,b_5,\ldots,a_g,b_g\},&&\\
&&\cW &= &\Span{a_1+a_2}_{\Q}^{\perp}/\Span{a_3+a_4} &\cong &\Span{A_W}_{\Q} &\text{ with } &&A_W = \{b_1-b_2,a_1,a_2,b_3,b_4,a_3,a_5,b_5,\ldots,a_g,b_g\}.&&
\end{alignat*}
We proved in Lemma \ref{lemma:mapbilinear2} that $\Fix{a_1+a_2,a_3+a_4}$ is isomorphic
to a quotient of the kernel of the map $\cV \otimes \cW \rightarrow \Q$
induced by the symplectic form $\omega$.  Under this isomorphism, a generator
$\PresA{(a_1+a_2) \wedge x,(a_3+a_4) \wedge y}$ of $\Fix{a_1+a_2,a_3+a_4}$ maps to
$x \otimes y \in \cV \otimes \cW$.

Recall that $g \geq 4$ (Assumption \ref{assumption:genusalt}).
The argument is slightly different when $g = 4$ and when $g \geq 5$.  Assume
first that $g = 4$.
The kernel of $\cV \otimes \cW \rightarrow \Q$ is spanned by $X \cup Y_4$ where
\begin{align*}
X   = &\Set{$x \otimes y$}{$x \in A_V$, $y \in A_W$, $\omega(x,y)=0$},\\
Y_4 = &\{a_1 \otimes (b_1-b_2) +b_1 \otimes a_1,b_1 \otimes a_1 - b_2 \otimes a_2,a_3 \otimes b_3 + b_2 \otimes a_2, \\
      &\quad a_3 \otimes b_3 - a_4 \otimes b_4,(b_3-b_4) \otimes a_3 + a_4 \otimes b_4\}.
\end{align*}
When $g \geq 5$, you instead take $X \cup Y_4 \cup Y_5$ with
\[Y_5 = \Set{$a_i \otimes b_i + b_j \otimes a_j$}{$5 \leq i,j \leq g$} \cup \{a_4 \otimes b_4 + b_5 \otimes a_5\}.\]
Just like in the proof of Lemma \ref{lemma:twovsone1}, you can replace $Y_4$ with
\begin{align*}
&\{(a_1+b_1) \otimes (a_1+b_1-b_2),(b_1-b_2) \otimes (a_1+a_2),(b_2+a_3) \otimes (a_2+b_3),\\
&\quad (a_3+a_4) \otimes (b_3 - b_4),(b_3+a_4-b_4) \otimes (a_3+b_4)\}
\end{align*}
and $Y_5$ by
\begin{align*}
&\Set{$(a_i-b_j) \otimes (b_i-a_j)$}{$4 \leq i,j \leq g$ distinct, either $i,j \geq 5$ or $(i,j) = (4,5)$} \\
&\cup \Set{$(a_i+b_i) \otimes (a_i + b_i)$}{$5 \leq i \leq g$}.
\end{align*}
From this, we see that $\Fix{a_1+a_2,a_3+a_4}$ is generated by the following elements:

\begin{case}{1}
$\PresA{(a_1+a_2) \wedge x,(a_3+a_4) \wedge y}$ for $x \in A_V$ and $Y \in A_W$ with $\omega(x,y) = 0$.
\end{case}

\noindent
There are three cases:
\begin{itemize}
\item If $y = b_1-b_2$, then this equals $-\PresA{(a_1+a_2) \wedge x,(b_1-b_2) \wedge (a_3+a_4)}$,
which lies in $\Span{S}$ by Lemma \ref{lemma:s3symmetricalt}.
\item If $x = b_3-b_4$, then this equals $\PresA{(a_3+a_4) \wedge y,(b_3-b_4) \wedge (a_1+a_2)}$,
which lies in $\Span{S}$ by Lemma \ref{lemma:s3symmetricalt}.
\item If neither equality holds, then 
this equals $\PPresA{x \wedge (a_1+a_2),y \wedge (a_3+a_4)} \in S_{12}$.
\end{itemize}

\begin{case}{2}
When $g \geq 5$, elements $\PresA{(a_1+a_2) \wedge (a_i-b_j),(a_3+a_4) \wedge (b_i-a_j)}$ 
for either $5 \leq i,j \leq g$ distinct or $(i,j) = (4,5)$.
\end{case}

\noindent
These equal $\OPresA{(a_i-b_j) \wedge (a_1+a_2),(b_i-a_j) \wedge (a_3+a_4)} \in S_3$.

\begin{case}{3}
When $g \geq 5$, elements $\PresA{(a_1+a_2) \wedge (a_i+b_i),(a_3+a_4) \wedge (a_i+b_i)}$ for $5 \leq i \leq g$.
\end{case}

\noindent
These equal $\PresA{(a_i+b_i) \wedge (a_1+a_2),(a_i+b_i) \wedge (a_3+a_4)}$, which lie
in $\Span{S}$ by Lemma \ref{lemma:s12twoa1weak}.

\begin{case}{4}
The following elements:
\begin{itemize}
\item[(a)] $\PresA{(a_1+a_2) \wedge (a_1+b_1),(a_3+a_4) \wedge (a_1+b_1-b_2)}$
\item[(b)] $\PresA{(a_1+a_2) \wedge (b_3+a_4-b_4),(a_3+a_4) \wedge (a_3+b_4)}$
\item[(c)] $\PresA{(a_1+a_2) \wedge (b_1-b_2),(a_3+a_4) \wedge (a_1+a_2)}$
\item[(d)] $\PresA{(a_1+a_2) \wedge (a_3+a_4),(a_3+a_4) \wedge (b_3-b_4)}$
\item[(e)] $\PresA{(a_1+a_2) \wedge (b_2+a_3),(a_3+a_4) \wedge (a_2+b_3)}$
\end{itemize}
\end{case}

The element in (a) equals
\begin{equation}
\label{eqn:twotwo.1}
\PresA{(a_1+a_2) \wedge (a_1+b_1),a_3 \wedge (a_1+b_1-b_2)} + \PresA{(a_1+a_2) \wedge (a_1+b_1),a_4 \wedge (a_1+b_1-b_2)}.
\end{equation}
The first term of \eqref{eqn:twotwo.1} lies in $\Fix{a_1+a_2,a_3}$, which differs from $\Fix{a_1+b_2,a_3}$ by an element
of $\SymSp_g$.  Lemma \ref{lemma:twovsone3} says that $\Fix{a_1+b_2,a_3} \subset \Span{S}$, so by
Lemma \ref{lemma:altsymspacts} the set $\Fix{a_1+a_2,a_3}$ lies in $\Span{S}$ too.  The same argument
shows that the second term in \eqref{eqn:twotwo.1} also lies in $\Span{S}$, so \eqref{eqn:twotwo.1} lies in
$\Span{S}$.

The element in (b) equals
\[\PresA{a_1 \wedge (b_3+a_4-b_4),(a_3+a_4) \wedge (a_3+b_4)} + \PresA{a_2 \wedge (b_3+a_4-b_4),(a_3+a_4) \wedge (a_3+b_4)}.\]
Since $\PresA{-,-}$ is anti-symmetric, the same argument from the previous paragraph applies here too.

For the elements in (c) and (d), note that they equal
\begin{align*}
-&\PresA{(a_1+a_2) \wedge (a_3+a_4),(b_1-b_2) \wedge (a_1+a_2)} \quad \text{and}\\
 &\PresA{(a_3+a_4) \wedge (a_1+a_2),(b_3-b_4) \wedge (a_3+a_4)}
\end{align*}
Lemma \ref{lemma:s3symmetricalt} says that both of these lie in $\Span{S}$.

Finally, the element in (e) equals
$-\PresA{(a_2+b_3) \wedge (a_3+a_4), (b_2+a_3) \wedge (a_1+a_2)}$.
Lemma \ref{lemma:s3symmetricalt} says that this lies in $\Span{S}$.
\end{proof}

\subsection{Three vs two}

Our final two results are about $\Fix{a,a'}$ where $a$ uses three generators
and $a'$ uses two:

\begin{lemma}[Three vs two, I]
\label{lemma:threevstwo1}
For all $j \geq 2$, both $\Fix{a_1+b_1-b_j,b_1-a_j}$ and $\Fix{a_1-b_1-b_j,b_1-a_j}$ are
subsets of $\Span{S}$.
\end{lemma}
\begin{proof}
These two subsets do not differ by an element of $\SymSp_g$, but the proofs that they are
contained in $\Span{S}$ are almost identical.  We will therefore give the details
for $\Fix{a_1+b_1-b_j,b_1-a_j}$ and leave the other case to the reader.
By Lemma \ref{lemma:altsymspacts}, we can apply an appropriate element of $\SymSp_g$ and
reduce ourselves to proving that $\Fix{a_1+b_1-b_2,b_1-a_2} \subset \Span{S}$.
Following the notation in \S \ref{section:identificationalt2}, define
\begin{alignat*}{8}
&&\cV &= &\Span{b_1-a_2}_{\Q}^{\perp}/\Span{a_1+b_1-b_2} &\cong &\Span{A_V}_{\Q} &\quad \text{with} \quad &&A_V = \{b_1,    a_2,a_3,b_3,\ldots,a_g,b_g\},&&\\
&&\cW &= &\Span{a_1+b_1-b_2}_{\Q}^{\perp}/\Span{b_1-a_2} &\cong &\Span{A_W}_{\Q} &\quad \text{with} \quad &&A_W = \{a_1+b_1,b_2,a_3,b_3,\ldots,a_g,b_g\}.&&
\end{alignat*}
We proved in Lemma \ref{lemma:mapbilinear2} that $\Fix{a_1+b_1-b_2,b_1-a_2}$ is isomorphic
to a quotient of the kernel of the map $\cV \otimes \cW \rightarrow \Q$
induced by the symplectic form $\omega$.  Under this isomorphism, a generator
$\PresA{(a_1+b_1-b_2) \wedge x,(b_1-a_2) \wedge y}$ of $\Fix{a_1+b_1-b_2,b_1-a_2}$ maps to
$x \otimes y \in \cV \otimes \cW$.

The kernel of $\cV \otimes \cW \rightarrow \Q$ is spanned by $X \cup Y$ where
\begin{align*}
X = &\Set{$x \otimes y$}{$x \in A_V$, $y \in A_W$, $\omega(x,y)=0$},\\
Y = &\Set{$a_i \otimes b_i + b_j \otimes a_j$}{$3 \leq i,j \leq g$}\\ 
    &\cup \{a_3 \otimes b_3 + b_1 \otimes (a_1+b_1), a_2 \otimes b_2+b_3 \otimes a_3\}.
\end{align*}
Just like in the proof of Lemma \ref{lemma:twovsone1}, we can replace $Y$ by
\begin{align*}
&\Set{$(a_i-b_j) \otimes (b_i-a_j)$}{$1 \leq i,j \leq g$ distinct, either $i,j \geq 3$ or $(i,j) = (2,3)$} \\
&\cup \Set{$(a_i+b_i) \otimes (a_i + b_i)$}{$3 \leq i \leq g$} \\
&\cup \{(b_1+a_3) \otimes (a_1+b_1+b_3)\}.
\end{align*}
From this, we see that $\Fix{a_1+b_1-b_2,b_1-a_2}$ is generated by the following elements:

\begin{case}{1}
$\PresA{(a_1+b_1-b_2) \wedge x,(b_1-a_2) \wedge y}$ for $x \in A_V$ and $y \in A_W$ with $\omega(x,y) = 0$.
\end{case}

\noindent
These equal $\PresA{(b_1-a_2) \wedge y,x \wedge (a_1+b_1-b_2)} \in \Fix{b_1-a_2,x}$.  There are
three cases:
\begin{itemize}
\item $x=b_1$.  In this case, $\Fix{b_1-a_2,b_1}$ differs from $\Fix{a_1+b_2,b_2}$ by an element of
$\SymSp_g$.  Lemma \ref{lemma:twovsone2} says that $\Fix{a_1+b_2,b_2} \subset \Span{S}$, so
by Lemma \ref{lemma:altsymspacts} so does $\Fix{b_1-a_2,b_1}$.
\item $x = a_2$.  In this case, $\Fix{b_1-a_2,a_2} = \Fix{b_1-a_2,-a_2}$ by Lemma \ref{lemma:isotropicbilineareasy}.
The set $\Fix{b_1-a_2,-a_2}$ differs from $\Fix{a_1+b_2,b_2}$ by an element of
$\SymSp_g$, so just like in the case $x=b_1$ it follows that $\Fix{b_1-a_2,-a_2}$ lies
in $\Span{S}$.
\item $x \in \{a_3,b_3,\ldots,a_g,b_g\}$.  In this case, $\Fix{b_1-a_2,x}$ differs from $\Fix{a_1+b_2,a_3}$
by an element of $\SymSp_g$.  Lemma \ref{lemma:twovsone3} says that $\Fix{a_1+b_2,a_3} \subset \Span{S}$, so
by Lemma \ref{lemma:altsymspacts} so does $\Fix{b_1-a_2,x}$.
\end{itemize}

\begin{case}{2}
$\PresA{(a_1+b_1-b_2) \wedge (a_i-b_j),(b_1-a_2) \wedge (b_i-a_j)}$ for $1 \leq i,j \leq g$ distinct
with either $i,j \geq 3$ or $(i,j) = (2,3)$.
\end{case}

\noindent
These equal $\OPresA{(a_i-b_j) \wedge (a_1+b_1-b_2),(b_i-a_j) \wedge (b_1-a_2)} \in S_3$.

\begin{case}{3}
$\PresA{(a_1+b_1-b_2) \wedge (a_i+b_i),(b_1-a_2) \wedge (a_i+b_i)}$ for $3 \leq i \leq g$.
\end{case}

\noindent
These equal $\PresA{(a_i+b_i) \wedge (a_1+b_1-b_2),(a_i+b_i) \wedge (b_1-a_2)}$, which lie 
in $\Span{S}$ by Lemma \ref{lemma:s12twoa1weak}.

\begin{case}{4}
$\PresA{(a_1+b_1-b_2) \wedge (b_1+a_3),(b_1-a_2) \wedge (a_1+b_1+b_3)}$.
\end{case}

\noindent
This equals $\PresA{(b_1+a_3) \wedge (a_1+b_1-b_2), (b_1-a_2) \wedge (a_1+b_1+b_3)} \in \Fix{b_1+a_3,b_1-a_2}$
The set $\Fix{b_1+a_3,b_1-a_2}$ differs from $\Fix{a_1+a_2,a_1+a_3}$ by an element of $\SymSp_g$.
Lemma \ref{lemma:twovstwo2} says that $\Fix{a_1+a_2,a_1+a_3}$ is contained in $\Span{S}$, so by
Lemma \ref{lemma:altsymspacts} the set $\Fix{b_1+a_3,b_1-a_2}$ is as well.
\end{proof}

\begin{lemma}[Three vs two, II]
\label{lemma:threevstwo2}
For $3 \leq j \leq g$, both $\Fix{b_1+a_2-b_j,b_2-a_j}$ and
$\Fix{a_1+b_2-b_j,b_1-a_j}$ are subsets of $\Span{S}$.
\end{lemma}
\begin{proof}
The two sets differ by an element of $\SymSp_g$, so by Lemma \ref{lemma:altsymspacts} it
is enough to deal with $\Fix{b_1+a_2-b_j,b_2-a_j}$.  Applying a further element of $\SymSp_g$, we can
reduce to the case $j=3$, i.e., to $\Fix{b_1+a_2-b_3,b_2-a_3}$.  In fact, to simplify our notation
we will apply yet another element of $\SymSp_g$ and transform our goal into 
proving that $\Fix{a_1+a_2+b_3,b_1+a_3} \subset \Span{S}$.
Following the notation in \S \ref{section:identificationalt2}, define
\begin{alignat*}{8}
&&\cV &= &\Span{b_1+a_3}_{\Q}^{\perp}/\Span{a_1+a_2+b_3} &\cong &\Span{A_V}_{\Q} &\quad \text{with} \quad &&A_V = \{b_1,a_2,    b_2,a_3,a_4,b_4,\ldots,a_g,b_g\},&&\\
&&\cW &= &\Span{a_1+a_2+b_3}_{\Q}^{\perp}/\Span{b_1+a_3} &\cong &\Span{A_W}_{\Q} &\quad \text{with} \quad &&A_W = \{a_1,b_2-b_1,a_2,b_3,a_4,b_4,\ldots,a_g,b_g\}.&&
\end{alignat*}
We proved in Lemma \ref{lemma:mapbilinear2} that $\Fix{a_1+a_2+b_3,b_1+a_3}$ is isomorphic
to a quotient of the kernel of the map $\cV \otimes \cW \rightarrow \Q$
induced by the symplectic form $\omega$.  Under this isomorphism, a generator
$\PresA{(a_1+a_2+b_3) \wedge x,(b_1+a_3) \wedge y}$ of $\Fix{a_1+a_2+b_3,b_1+a_3}$ maps to
$x \otimes y \in \cV \otimes \cW$.

The kernel of $\cV \otimes \cW \rightarrow \Q$ is spanned by $X \cup Y$ where
\begin{align*}
X = &\Set{$x \otimes y$}{$x \in A_V$, $y \in A_W$, $\omega(x,y)=0$},\\
Y = &\Set{$a_i \otimes b_i + b_j \otimes a_j$}{$4 \leq i,j \leq g$} \\
    &\cup \{a_4 \otimes b_4 + b_1 \otimes a_1,a_2 \otimes (b_2-b_1) + b_4 \otimes a_4,a_4 \otimes b_4 + b_2 \otimes a_2,a_3 \otimes b_3 + b_4 \otimes a_4\}.
\end{align*}
Just like in the proof of Lemma \ref{lemma:twovsone1}, we can replace $Y$ by
\begin{align*}
&\Set{$(a_i-b_j) \otimes (b_i-a_j)$}{$1 \leq i,j \leq g$ distinct, either $i,j \geq 4$ or $(i,j) \in \{(4,1),(4,2),(3,4)\}$} \\
&\cup \Set{$(a_i+b_i) \otimes (a_i + b_i)$}{$4 \leq i \leq g$} \\
&\cup \{(a_2+b_4) \otimes (b_2-b_1+a_4)\}.
\end{align*}
From this, we see that $\Fix{a_1+a_2+b_3,b_1+a_3}$ is generated by the following elements:

\begin{case}{1}
$\PresA{(a_1+a_2+b_3) \wedge x,(b_1+a_3) \wedge y}$ for $x \in A_V$ and $y \in A_W$ with $\omega(x,y) = 0$.
\end{case}

\noindent
These equal $\PresA{(b_1+a_3) \wedge y,x \wedge (a_1+a_2+b_3)} \in \Fix{b_1+a_3,x}$.  There are
two cases:
\begin{itemize}
\item $x=b_1$ or $x=a_3$.  In this case, $\Fix{b_1+a_3,x}$ differs from $\Fix{a_1+b_2,b_2}$ by an element of
$\SymSp_g$.  Lemma \ref{lemma:twovsone2} says that $\Fix{a_1+b_2,b_2} \subset \Span{S}$, so
by Lemma \ref{lemma:altsymspacts} so does $\Fix{b_1+a_3,x}$.
\item $x \in \{a_2,b_2,a_4,b_4,\ldots,a_g,b_g\}$.  In this case, $\Fix{b_1+a_3,x}$ differs from $\Fix{a_1+b_2,a_3}$
by an element of $\SymSp_g$.  Lemma \ref{lemma:twovsone3} says that $\Fix{a_1+b_2,a_3} \subset \Span{S}$, so
by Lemma \ref{lemma:altsymspacts} so does $\Fix{b_1+a_3,x}$.
\end{itemize}

\begin{case}{2}
$\PresA{(a_1+a_2+b_3) \wedge (a_i-b_j),(b_1+a_3) \wedge (b_i-a_j)}$ for $1 \leq i,j \leq g$ distinct
with either $i,j \geq 4$ or $(i,j) \in \{(4,1),(4,2),(3,4)\}$.
\end{case}

\noindent
These equal $\OPresA{(a_i-b_j) \wedge (a_1+a_2+b_3),(b_i-a_j) \wedge (b_1+a_3)} \in S_3$.

\begin{case}{3}
$\PresA{(a_1+a_2+b_3) \wedge (a_i+b_i),(b_1+a_3) \wedge (a_i+b_i)}$ for $4 \leq i \leq g$.
\end{case}

\noindent
These equal $\PresA{(a_i+b_i) \wedge (a_1+a_2+b_3),(a_i+b_i) \wedge (b_1+a_3)}$, which lie
in $\Span{S}$ by Lemma \ref{lemma:s12twoa1weak}.

\begin{case}{4}
$\PresA{(a_1+a_2+b_3) \wedge (a_2+b_4),(b_1+a_3) \wedge (b_2-b_1+a_4)}$.
\end{case}

\noindent
This equals $-\PresA{(a_2+b_4) \wedge (a_1+a_2+b_3), (b_1+a_3) \wedge (b_2-b_1+a_4)} \in \Fix{a_2+b_4,b_1+a_3}$
The set $\Fix{a_2+b_4,b_1+a_3}$ differs from $\Fix{a_1+a_2,a_3+a_4}$ by an element of $\SymSp_g$.
Lemma \ref{lemma:twovstwo3} says that $\Fix{a_1+a_2,a_3+a_4}$ is contained in $\Span{S}$, so by 
Lemma \ref{lemma:altsymspacts} the set $\Fix{a_2+b_4,b_1+a_3}$ is as well.
\end{proof}

\section{Symmetric kernel, alternating version IV: the set \texorpdfstring{$S$}{S} spans \texorpdfstring{$\fK_g^a$}{Kga}}
\label{section:presentationalt4}

We continue using all the notation from \S \ref{section:presentationalt2} -- \S \ref{section:presentationalt3}.
Recall that $S = S_{12} \cup S_3$, where
\[S_{12} = \bigcup_{\substack{a,a' \in \cB \\ \omega(a,a') = 0}} \pFix{a,a'} \quad \text{and} \quad
S_3    = \bigcup_{\substack{1 \leq i,j \leq g \\ i \neq j}} \oFix{a_i-b_j,b_i-a_j}.\]
Our goal in this section is to prove the following lemma:

\begin{lemma}
\label{lemma:presentationaltgenset}
The set $S$ spans $\fK_g^a$.
\end{lemma}
\begin{proof}
We first claim that the $\Sp_{2g}(\Z)$-orbit of $S$ spans $\fK_g^a$.
Indeed, using our original generating set for $\fK_g^a$ from Definition \ref{definition:kgalt} together
with Lemma \ref{lemma:generationomega}, we see that $\fK_g^a$ is generated by elements of the form
$\PresA{a \wedge b,a' \wedge b'}$, where $a \wedge b$ and $a' \wedge b'$ are symplectic pairs
such that $\Span{a,b}$ and $\Span{a',b'}$ are orthogonal.  The group $\Sp_{2g}(\Z)$
acts transitively on such elements.  The set $S$ contains many elements of this form; for instance,
it contains $\PPresA{a_1 \wedge b_1,a_2 \wedge b_2} \in \pFix{a_1,a_2}$.  It follows that $\Sp_{2g}(\Z)$-orbit of $S$
spans $\fK_g^a$, as claimed.

To prove the lemma, therefore, we must prove that the action of $\Sp_{2g}(\Z)$ on $\fK_g^a$ takes
$\Span{S}$ to itself.  By Corollary \ref{corollary:gensp}, the group $\Sp_{2g}(\Z)$ is generated as a monoid by
the set $\Lambda = \SymSp_g \cup \{X_1,X_1^{-1},Y_{12}\}$.  We will recall the definitions of these
elements as we use them.  We must prove that for $f \in \Lambda$ and $s \in S = S_{12} \cup S_3$ we have
$f(s) \in \Span{S}$.  We already did this for $\SymSp_g$ in Lemma \ref{lemma:altsymspacts}, so
we must handle the other generators.  We divide this into four claims. 

\begin{claim}{1}
Recall that $X_1 \in \Sp_{2g}(\Z)$ takes $a_1$ to $a_1+b_1$ and fixes all other generators in $\cB$.
For $\epsilon \in \{\pm 1\}$ and
\[s \in S_{12} = \bigcup_{\substack{a,a' \in \cB \\ \omega(a,a') = 0}} \pFix{a,a'},\]
we have $X_1^{\epsilon}(s) \in \Span{S}$.
\end{claim}

We have $s \in \pFix{a,a'}$ for some $a,a' \in \cB$ with $\omega(a,a') = 0$, so $X_1^{\epsilon}(s) \in \Fix{X_1^{\epsilon}(a),X_1^{\epsilon}(a')}$.  
We must show that $\Fix{X_1^{\epsilon}(a),X_1^{\epsilon}(a')} \subset \Span{S}$.  There are three cases:
\begin{itemize}
\item If $a,a' \neq a_1$, then
\[\Fix{X_1^{\epsilon}(a),X_1^{\epsilon}(a')} = \pFix{a,a'} \subset \Span{S}.\]
\item If one of $a$ and $a'$ equals $a_1$ and the other is not $a_1$, then 
since $\Fix{-,-}$ is symmetric in its entries (Lemma \ref{lemma:isotropicbilineareasy})
we can assume without loss of generality that $a = a_1$ and $a' \neq a_1$.  Since $\omega(a_1,a') = 0$, we also have $a' \neq b_1$. 
By Lemma \ref{lemma:twovsone1},
\[\Fix{X_1^{\epsilon}(a_1),X_1^{\epsilon}(a')} = \Fix{a_1 + \epsilon b_1,a'} \subset \Span{S}.\]
\item If both $a$ and $a'$ equal $a_1$ then we can apply Lemma \ref{lemma:s12twoa1} to see that
\[\Fix{X_1^{\epsilon}(a_1),X_1^{\epsilon}(a_1)} = \Fix{a_1 + \epsilon b_1,a_1 + \epsilon b_1} \subset \Span{S}.\]
\end{itemize}
The claim follows.

\begin{claim}{2}
For $\epsilon \in \{\pm 1\}$ and
\[s \in S_{3} = \bigcup_{1 \leq i < j \leq g} \oFix{a_i-b_j,b_i-a_j},\]
we have $X_1^{\epsilon}(s) \in \Span{S}$.
\end{claim}

We have $s \in \oFix{a_i-b_j,b_i-a_j}$ for some $1 \leq i < j \leq g$, 
so $X_1^{\epsilon}(s) \in \Fix{X_1^{\epsilon}(a_i-b_j),X_1^{\epsilon}(b_i-a_j)}$.
We must show that $\Fix{X_1^{\epsilon}(a_i-b_j),X_1^{\epsilon}(b_i-a_j)} \subset \Span{S}$.  There are two cases:
\begin{itemize}
\item If $i \geq 2$, then $j \geq 2$ as well, so $X_1^{\epsilon}$ fixes both $a_i-b_j$ and $b_i-a_j$.  It follows
that
\[\Fix{X_1^{\epsilon}(a_i-b_j),X_1^{\epsilon}(b_i-a_j)} = \oFix{a_i-b_j,b_i-a_j} \subset \Span{S}.\]
\item If $i = 1$, then $j \geq 2$, so $X_1^{\epsilon}$ fixes both $a_j$ and $b_j$.  By Lemma \ref{lemma:threevstwo1},
\[\Fix{X_1^{\epsilon}(a_1-b_j),X_1^{\epsilon}(b_1-a_j)} = \Fix{a_1 + \epsilon b_1 - b_j,b_1-a_j} \subset \Span{S}.\]
\end{itemize}
The claim follows.

\begin{claim}{3}
Recall that $Y_{12} \in \Sp_{2g}(\Z)$ takes $a_1$ to $a_1+b_2$ and $a_2$ to $a_2+b_1$ and fixes all other generators in $\cB$.
For 
\[s \in S_{12} = \bigcup_{\substack{a,a' \in \cB \\ \omega(a,a') = 0}} \pFix{a,a'},\]
we have $Y_{12}(s) \in \Span{S}$.
\end{claim}

We have $s \in \pFix{a,a'}$ for some $a,a' \in \cB$ with $\omega(a,a') = 0$, so $Y_{12}(s) \in \Fix{Y_{12}(a),Y_{12}(a')}$.
We must show that $\Fix{Y_{12}(a),Y_{12}(a')} \subset \Span{S}$.  There are six cases:
\begin{itemize}
\item If $a,a' \in \cB \setminus \{a_1,a_2\}$, then both $a$ and $a'$ are fixed by $Y_{12}$, so
\[\Fix{Y_{12}(a),Y_{12}(a')} = \pFix{a,a'} \subset \Span{S}.\]
\item If one of $a$ and $a'$ is $a_1$ and the other lies in $\cB \setminus \{a_1,a_2\}$, then
since $\Fix{-,-}$ is symmetric in its entries (Lemma \ref{lemma:isotropicbilineareasy})
we can assume without loss of generality that $a = a_1$ and $a' \in \cB \setminus \{a_1,a_2\}$.  
Since $\omega(a_1,a') = 0$, we also have $a' \neq b_1$.  By Lemmas \ref{lemma:twovsone2} and \ref{lemma:twovsone3},
\[\Fix{Y_{12}(a_1),Y_{12}(a')} = \Fix{a_1+b_2,a'} \subset \Span{S}.\]
\item If one of $a$ and $a'$ is $a_2$ and the other lies in $\cB \setminus \{a_1,a_2\}$, then
since $\Fix{-,-}$ is symmetric in its entries (Lemma \ref{lemma:isotropicbilineareasy})
we can assume without loss of generality that $a = a_2$ and $a' \in \cB \setminus \{a_1,a_2\}$.
Since $\omega(a_1,a') = 0$, we also have $a' \neq b_2$.  By Lemmas \ref{lemma:twovsone2} and \ref{lemma:twovsone3},
\[\Fix{Y_{12}(a_2),Y_{12}(a')} = \Fix{a_2+b_1,a'} \subset \Span{S}.\]
\item If $a = a' = a_1$, then by Lemma \ref{lemma:twovstwo1}
\[\Fix{Y_{12}(a_1),Y_{12}(a_1)} = \Fix{a_1+b_2,a_1+b_2} \subset \Span{S}.\]
\item If $a = a' = a_2$, then by Lemma \ref{lemma:twovstwo1} 
\[\Fix{Y_{12}(a_2),Y_{12}(a_2)} = \Fix{a_2+b_1,a_2+b_1} \subset \Span{S}.\]
\item If one of $a$ and $a'$ is $a_1$ and the other is $a_2$, then
since $\Fix{-,-}$ is symmetric in its entries (Lemma \ref{lemma:isotropicbilineareasy})
we can assume without loss of generality that $a = a_1$ and $a' = a_2$.  By Lemma \ref{lemma:s3symmetricalt}, 
\[\Fix{Y_{12}(a_1),Y_{12}(a_2)} = \Fix{a_1+b_2,b_1+a_2} \subset \Span{S}.\]
\end{itemize}
The claim follows.

\begin{claim}{4}
For 
\[s \in S_{3} = \bigcup_{1 \leq i < j \leq g} \oFix{a_i-b_j,b_i-a_j},\]
we have $Y_{12}(s) \in \Span{S}$.
\end{claim}

We have $s \in \pFix{a_i-b_j,b_i-a_j}$ for some $1 \leq i < j \leq g$, so 
$Y_{12}(s) \in \Fix{Y_{12}(a_i-b_j),Y_{12}(b_i-a_j)}$.
We must show that $\Fix{Y_{12}(a_i-b_j),Y_{12}(b_i-a_j)} \subset \Span{S}$.  There are four cases:
\begin{itemize}
\item If $i \geq 3$, then $j \geq 3$ as well and thus $Y_{12}$ fixes $a_i-b_j$ and $b_i-a_j$.  Therefore,
\[\Fix{Y_{12}(a_i-b_j),Y_{12}(b_i-a_j)} = \oFix{a_i-b_j,b_i-a_j} \subset \Span{S}.\]
\item If $i = 2$, then $j \geq 3$ and $Y_{12}$ fixes $a_j$ and $b_j$.  By Lemma \ref{lemma:threevstwo2}, we have
\[\Fix{Y_{12}(a_2-b_j),Y_{12}(b_2-a_j)} = \Fix{b_1+a_2-b_j,b_2-a_j} \subset \Span{S}.\]
\item If $i = 1$ and $j \geq 3$, then $Y_{12}$ fixes $a_j$ and $b_j$.  By Lemma \ref{lemma:threevstwo2}, we have
\[\Fix{Y_{12}(a_1-b_j),Y_{12}(b_1-a_j)} = \Fix{a_1+b_2-b_j,b_1-a_j} \subset \Span{S}.\]
\item If $i = 1$ and $j = 2$, then\footnote{This is the one easy case!}
\[\Fix{Y_{12}(a_1-b_2),Y_{12}(b_1-a_2)} = \pFix{a_1,a_2} \subset \Span{S_{12}}.\]
\end{itemize}
The claim follows.
\end{proof}

\section{Symmetric kernel, alternating version V: \texorpdfstring{$S_1$}{S1} and structure of target} 
\label{section:presentationalt5}

We will use the notation from \S \ref{section:presentationalt1} - \S \ref{section:presentationalt4}, and in
particular the set $S = S_{12} \cup S_3$.  We now turn to 
Theorem \ref{maintheorem:presentationalt}, which says that the 
linearization map $\Phi\colon \fK_g^a \rightarrow \cK_g^a$ is an isomorphism.
We will prove this by first handling $S_{12}$ in the next two sections (\S \ref{section:presentationalt5} -- \S \ref{section:presentationalt6}), and then in the final two sections (\S \ref{section:presentationalt7} -- \S \ref{section:presentationalt8})
extending this to $S_3$ and hence to all of $\Span{S_{12},S_3} = \fK_g^a$.

The vector space
$\cK_g^a$ is the kernel of the symmetric contraction $\fc\colon \wedge^2((\wedge^2 H)/\Q) \rightarrow \Sym^2(H)$,
and this section studies $\wedge^2((\wedge^2 H)/\Q)$ and constructs a subset $S_1$ of $S_{12}$
such that the restriction of $\Phi$ to $\Span{S_1}$ is an isomorphism onto its image.

\subsection{Generators and relations for target}
\label{section:genrelalt2}

Let $\prec$ be the following total order on $\cB$:
\[a_1 \prec b_1 \prec a_2 \prec b_2 \prec \cdots \prec a_g \prec b_g.\]
Using this ordering, order $\Set{$x \wedge y$}{$x,y \in \cB$, $x \prec y$}$
lexicographically, so $(x_1 \wedge y_1) \prec (x_2 \wedge y_2)$ if either $x_1 \prec x_2$ or if $x_1 = x_2$ and $y_1 \prec y_2$.  
Define $T = T_1 \cup T_2 \cup T_3$, where:
\begin{align*}
T_1 &= \SetLong{$(x \wedge y) \wedge (z \wedge w)$}{$x,y,z,w \in \cB$, $x \prec y$, $z \prec w$, $(x \wedge y) \prec (z \wedge w)$, \\ and $\omega(x,z) = \omega(x,w) = \omega(y,z) = \omega(y,w) = 0$}{,}{$x,y,z,w \in \cB$, $x \prec y$, $z \prec w$, $(x \wedge y) \prec (z \wedge w)$,} \\
T_2 &= \SetLong{$(a \wedge a_i) \wedge (a' \wedge b_i)$}{$1 \leq i \leq g$, $a \in \cB \setminus \{a_i\}$, $a' \in \cB \setminus \{b_i\}$, \\ and $\omega(a,a')=\omega(a_i,a')=\omega(b_i,a)=0$}{,}{and $\omega(a,a')=\omega(a_i,a')=\omega(b_i,a)=0$}\\
T_3 &= \Set{$(a_i \wedge a_j) \wedge (b_i \wedge b_j)$, $(a_i \wedge b_j) \wedge (b_i \wedge a_j)$}{$1 \leq i < j \leq g$}.
\end{align*}
The set $T$ generates $\wedge^2((\wedge^2 H)/\Q)$, so every element of $\wedge^2((\wedge^2 H)/\Q)$ can be written
as a linear combination of elements of $T$.  If we were working with $\wedge^2(\wedge^2 H)$, then $T$ would be
a basis; however, since we are working with $\wedge^2((\wedge^2 H)/\Q)$ there are relations between elements of $T$.
These relations are linear combinations of elements of $T$.
It is often awkward to write these linear combinations while maintaining
the orderings on the terms of $T$, so we introduce the following convention:

\begin{convention}
\label{convention:alttconvention}
Consider an expression
\[\sum_{i=1}^n \lambda_i (x_i \wedge y_i) \wedge (z_i \wedge w_i) \quad \text{with each $\lambda_i \in \Q$ and $x_i,y_i,z_i,w_i \in \cB$}.\]
We regard this as a linear combination of elements of $T$ in the following way:
\begin{itemize}
\item First, delete all terms where either $x_i=y_i$ or where $z_i = w_i$.  These terms vanish
in $\wedge^2((\wedge^2 H)/\Q)$.
\item Second, delete all terms where $\{x_i,y_i\} = \{z_i,w_i\}$ as unordered $2$-element sets.
These terms also vanish in $\wedge^2((\wedge^2 H)/\Q)$.
\item Finally, replace each term $(x_i \wedge y_i) \wedge (z_i \wedge w_i)$ with $\epsilon t$ for some
$\epsilon \in \{\pm 1\}$ and $t \in T$.  This involves possibly flipping $x_i$ and $y_i$, flipping $z_i$ and $w_i$, and
flipping $x_i \wedge y_i$ and $z_i \wedge w_i$.  Each flip introduces a sign.\qedhere
\end{itemize}
\end{convention}

Using this convention, the relations between elements of $T$ are generated by the set
\[R = \Set{$\sum\nolimits_{i=1}^g (a_i \wedge b_i) \wedge (x \wedge y)$}{$x,y \in \cB$, $x \prec y$}\]
of linear combinations of elements of $T$.

\subsection{Lifting elements of \texorpdfstring{$T_1$}{T1}}
\label{section:alts1}

For $(x \wedge y) \wedge (z \wedge w) \in T_1$, we have
\[\fc(x \wedge y,z \wedge w) = \omega(x,z) y \Cdot w - \omega(x,w) y \Cdot z - \omega(y,z) x \Cdot w + \omega(y,w) x \Cdot z = 0.\]
It follows that $(x \wedge y) \wedge (z \wedge w) \in \cK_g^a$, and
thus can be lifted to $\fK_g^a$.  Indeed, define
\[S_1 = \SetLong{$\BPresA{x \wedge y,z \wedge w}$}{$x,y,z,w \in \cB$, $x \prec y$, $z \prec w$, $(x \wedge y) \prec (z \wedge w)$, \\ and $\omega(x,z) = \omega(x,w) = \omega(y,z) = \omega(y,w) = 0$}{.}{$x,y,z,w \in \cB$, $x \prec y$, $z \prec w$, $(x \wedge y) \prec (z \wedge w)$,}\]
Like we did here, we will write elements of $\Span{S_1}$ in blue.  A generator $\BPresA{x \wedge y, z \wedge w}$ of
$S_1$ lies in $\pFix{x,z} \subset S_{12}$, so $S_1 \subset S_{12}$.
For $\BPresA{x \wedge y,z \wedge w} \in S_1$, we have
\[\Phi(\BPresA{x \wedge y,z \wedge w}) = (x \wedge y) \wedge (z \wedge w) \in T_1.\]
The map $\Phi$ restricts to a bijection between $S_1$ and $T_1$.

\subsection{Restricting linearization to \texorpdfstring{$S_1$}{S1}}

Recall that our goal is to prove Theorem \ref{maintheorem:presentationalt}, which
says that the linearization map $\Phi\colon \fK_g^a \rightarrow \cK_g^a$ 
is an isomorphism.  We
now prove the following partial result in this direction:

\begin{lemma}
\label{lemma:altspans1}
The linearization map $\Phi$ takes $\Span{S_1}$ isomorphically to $\Span{T_1}$.
\end{lemma}
\begin{proof}
Let $R_1$ be the subset of 
the relations $R$ consisting of relations between elements of $T_1$.  The set
$R_1$ consists of relations of the form
\[\text{$\sum\nolimits_{i=1}^g (a_i \wedge b_i) \wedge (a_k \wedge b_k)$ with $1 \leq k \leq g$}.\]
Set $R_2 = R \setminus R_1$.  Each element of $R_2$ involves an element of $T \setminus T_1$ that appears
in no other relations in $R$.  For instance, for $1 \leq k < \ell \leq g$ the set $R_2$ contains
the relation
\[\sum_{i=1}^g (a_i \wedge b_i) \wedge (a_k \wedge a_{\ell}),\]
and no other relation in $R$ uses the generator\footnote{Though $a_k$ appears twice in
$(a_k \wedge b_k) \wedge (a_k \wedge a_{\ell})$, this element is not $0$.  See Warning \ref{warning:alternating}.}
$(a_k \wedge b_k) \wedge (a_k \wedge a_{\ell})$.
This implies that the subspace of $\wedge^2((\wedge^2 H)/\Q)$ spanned by $T_1$ is generated
by $T_1$ subject to only the relations in $R_1$.
The map $\Phi$ takes $S_1$ bijectively to $T_1$.  The relations in $R_1$ lift 
to relations between the elements of $S_1$ due to the bilinearity relations in $\fK_g^a$:
\[\sum_{i=1}^g \BPresA{a_i \wedge b_i,a_k \wedge b_k} = \BPresA{\sum_{i=1}^g a_i \wedge b_i,a_k \wedge b_k} = \BPresA{0,a_k \wedge b_k} = 0.\]
Combining all of this, we conclude that $\Phi$ takes $\Span{S_1}$ isomorphically to $\Span{T_1}$, as desired.
\end{proof}

\section{Symmetric kernel, alternating version VI: \texorpdfstring{$S_2$}{S2} and \texorpdfstring{$S_{12}$}{S12}}
\label{section:presentationalt6}

We continue using all the notation from \S \ref{section:presentationalt1} -- \S \ref{section:presentationalt5}.
Recall that our goal is to prove Theorem \ref{maintheorem:presentationalt}, which says that
the linearization map $\Phi\colon \fK_g^a \rightarrow \cK_g^a$ is an isomorphism.  In the last 
section, we constructed a set $S_1 \subset S_{12}$ and proved Lemma \ref{lemma:altspans1}, which
says that $\Phi$ restricts to an isomorphism between $\Span{S_1}$ and $\Span{T_1}$.  In this section,
we prove that $\Phi$ restricts to an isomorphism between $\Span{S_{12}}$ and $\cK_g^a \cap \Span{T_1,T_2}$.

\subsection{The set \texorpdfstring{$T_2$}{T2}}
\label{section:analyzet2}

Consider an element $(a \wedge a_i) \wedge (a' \wedge b_i)$ in
\[T_2 = \SetLong{$(a \wedge a_i) \wedge (a' \wedge b_i)$}{$1 \leq i \leq g$, $a \in \cB \setminus \{a_i\}$, $a' \in \cB \setminus \{b_i\}$, \\ and $\omega(a,a')=\omega(a_i,a')=\omega(b_i,a)=0$}{,}{and $\omega(a,a')=\omega(a_i,a')=\omega(b_i,a)=0\quad\quad$}\]
The image of $(a \wedge a_i) \wedge (a' \wedge b_i)$ in $\Sym^2(H)$ under the symmetric contraction is
\[\fc(a \wedge a_i,a' \wedge b_i) = \omega(a,a') a_i \Cdot b_i - \omega(a,b_i) a_i \Cdot a' - \omega(a_i,a') a \Cdot b_i + \omega(a_i,b_i) a \Cdot a' = a \Cdot a'.\]
The set $\Set{$a \Cdot a'$}{$a,a' \in \cB$}$ is a basis for $\Sym^2(H)$, and this calculation
suggests dividing $T_2$ into subsets mapping to different basis elements.  Define
$S^2_0(\cB) = \Set{$a \Cdot a'$}{$a,a' \in \cB$, $\omega(a,a') = 0$}$.  For $a \Cdot a' \in S^2_0(\cB)$,
define
\begin{small}
\begin{align}
\label{eqn:definet2}
T_2(a \Cdot a') = &\Set{$(a \wedge a_i) \wedge (a' \wedge b_i)$}{$1 \leq i \leq g$, $a_i \neq a$, $b_i \neq a'$, $\omega(a_i,a')=\omega(b_i,a)=0$} \\
                  &\cup\Set{$(a' \wedge a_i) \wedge (a \wedge b_i)$}{$1 \leq i \leq g$, $a_i \neq a'$, $b_i \neq a$, $\omega(a_i,a)=\omega(b_i,a')=0$}.\notag
\end{align}
\end{small}%
Note that this is symmetric in $a$ and $a'$.  We have
\[T_2 = \bigcup_{a \Cdot a' \in S^2_0(\cB)} T_2(a \Cdot a').\]
By our calculation above, $\fc$ takes each element of $T_2(a \Cdot a')$ to $a \Cdot a'$.  Since we will need
it later, we record the following consequence the above calculations:

\begin{lemma}
\label{lemma:imagect2}
The symmetric contraction $\fc\colon \wedge^2((\wedge^2 H)/\Q) \rightarrow \Sym^2(H)$ takes $\Span{T_1,T_2}$ surjectively
onto $\Span{S^2_0(\cB)}$.
\end{lemma}
\begin{proof}
Immediate from the above calculation of $\fc$ on generators for $T_2(a \Cdot a')$ along with
the fact that $\fc$ vanishes on $T_1$.
\end{proof}

\subsection{Lifting \texorpdfstring{$T_2(a \Cdot a')$}{T2(a a')}}

Consider $a \Cdot a' \in S^2_0(\cB)$.
Recall from \S \ref{section:presentationalt1image} that
$\FixBigIm{a,a'}$ is the subspace of $\wedge^2((\wedge^2 H)/\Q)$ spanned
by elements of the form $(a \wedge x) \wedge (a' \wedge y)$ with $x,y \in H_{\Z}$
satisfying $\omega(x,y) = 0$.  Each element of $T_2(a \Cdot a')$ is a generator
of $\FixBigIm{a,a'}$.  It follows
that
\[\Span{T_2(a \Cdot a')} \subset \FixBigIm{a,a'}.\]
Since $\fc$ takes every element of $T_2(a \Cdot a')$ to $a \Cdot a' \in \Sym^2(H)$, 
the intersection of the symmetric kernel $\cK_g^a = \ker(\fc)$ with $\Span{T_2(a \Cdot a')}$
has codimension $1$ in $\Span{T_2(a \Cdot a')}$.  Indeed, it is spanned by  
elements of the form $t_1 - t_2$ with $t_1,t_2 \in T_2(a \Cdot a')$.
Letting\footnote{The set $\pFix{a,a'}$ is purple since it lies in $S_{12}$.} 
$\pFix{a,a'}$ be as in \S \ref{section:presentationalt1setup}, define
\[S_2(a \Cdot a') = \Set{$\eta \in \pFix{a,a'}$}{$\Phi(\eta) \in \Span{T_2(a \Cdot a')}$}.\]
By construction, this is a subspace of $\pFix{a,a'}$.  We will prove
below that $\Phi$ takes $S_2(a \Cdot a')$ isomorphically onto $\cK_g^a \cap \Span{T_2(a \Cdot a')}$.
In fact, we will do more than this.  Define
\[S_1(a \Cdot a') = \SetLong{$\BPresA{a \wedge x,a' \wedge y}$}{$x \in \cB \setminus \{a\}$, $y \in \cB \setminus \{a'\}$, $\{a,x\} \neq \{a',y\}$, \\ \text{and $\omega(a,y) = \omega(x,a') = \omega(x,y) = 0$}}{.}{$x \in \cB \setminus \{a\}$, $y \in \cB \setminus \{a'\}$, $\{a,x\} \neq \{a',y\}$,}\]
For each $\eta \in S_1(a \Cdot a')$, either $\eta$ or $-\eta$ lies in $S_1$.  Just like for $S_2(a \Cdot a')$,
we have $S_1(a \Cdot a') \subset \pFix{a,a'}$.  We will prove:

\begin{lemma}
\label{lemma:altlists2}
Consider $a \Cdot a' \in S^2_0(\cB)$.  Then:
\begin{itemize}
\item[(a)] The linearization map $\Phi$ restricts to an isomorphism between $S_2(a \Cdot a')$ and
$\cK_g^a \cap \Span{T_2(a \Cdot a')}$.
\item[(b)] We have $\Span{S_1(a \Cdot a'),S_2(a \Cdot a')} = \pFix{a,a'}$.
\end{itemize}
\end{lemma}
\begin{proof}
Conclusion (a) is immediate from Lemmas \ref{lemma:mapbilinear1} and \ref{lemma:mapbilinear2}, which 
together imply that $\Phi$ restricts to an isomorphism between $\pFix{a,a'}$
and $\cK_g^a \cap \FixBigIm{a,a'}$.  We must prove (b).  There are two cases:
 
\begin{case}{1}
\label{case:altlists2.1}
$a = a'$.
\end{case}

To simplify our notation, we will assume that $a \in \cB = \{a_1,b_1,\ldots,a_g,b_g\}$ equals $a_1$.  The other
cases are identical up to changes in indices.  Following the notation in
\S \ref{section:identificationalt1}, define
\[\cU = \Span{a_1}_{\Q}^{\perp}/\Span{a_1} \cong \Span{A}_{\Q} \quad \text{with} \quad A = \{a_2,b_2,\ldots,a_g,b_g\}.\]
We proved in Lemma \ref{lemma:mapbilinear1} that $\Fix{a_1,a_1}$ is isomorphic
to the kernel of the map $\wedge^2 \cU \rightarrow \Q$
induced by the symplectic form $\omega$.  Under this isomorphism, a generator
$\PresA{a_1 \wedge x,a_1 \wedge y}$ of $\Fix{a_1,a_1}$ maps
to $x \wedge y \in \wedge^2 \cU$.

The kernel of $\wedge^2 \cU \rightarrow \Q$ is spanned by $X \cup Y$ where
\begin{align*}
X &= \Set{$x \wedge y$}{$x,y \in A$, $\omega(x,y) = 0$}, \\
Y &= \Set{$a_i \wedge b_i - a_j \wedge b_j$}{$2 \leq i < j \leq g$}.
\end{align*}
Since for $2 \leq i < j \leq g$ we have
\[(a_i - b_j) \wedge (b_i - a_j) = a_i \wedge b_i - a_j \wedge b_j + \text{an element of $\Span{X}$},\]
we can replace $Y$ by 
\[\Set{$(a_i - b_j) \wedge (b_i - a_j)$}{$2 \leq i < j \leq g$}.\]
It follows that $\Fix{a_1,a_1}$ is generated by the following elements:
\begin{itemize}
\item $\BPresA{a_1 \wedge x,a_1 \wedge y}$ for $x,y \in A$ with $\omega(x,y) = 0$.  These
are elements of $S_1(a_1 \Cdot a_1)$.
\item $\PPresA{a_1 \wedge (a_i-b_j), a_1 \wedge (b_i - a_j)}$ for $2 \leq i < j \leq g$.
\end{itemize}
It is thus enough to prove that for $2 \leq i < j \leq g$ the element
$\PPresA{a_1 \wedge (a_i - b_j),a_1 \wedge (b_i - a_j)}$ lies in 
the span of $S_1(a_1 \Cdot a_1)$ and $S_2(a_1 \Cdot a_1)$.  For this, note that
\begin{align*}
&\Phi(\PPresA{a_1 \wedge (a_i - b_j),a_1 \wedge (b_i - a_j)} + \BPresA{a_1 \wedge a_i,a_1 \wedge a_j} + \BPresA{a_1 \wedge b_j,a_1 \wedge b_i}) \\
=&(a_1 \wedge (a_i - b_j)) \wedge (a_1 \wedge (b_i - a_j)) + (a_1 \wedge a_i) \wedge (a_1 \wedge a_j) + (a_1 \wedge b_j) \wedge (a_1 \wedge b_i) \\
=&(a_1 \wedge a_i) \wedge (a_1 \wedge b_i) - (a_1 \wedge a_j) \wedge (a_1 \wedge b_j) \in \Span{T_2(a_1 \Cdot a_1)}.
\end{align*}
It follows that
\[\PPresA{a_1 \wedge (a_i - b_j),a_1 \wedge (b_i - a_j)} + \BPresA{a_1 \wedge a_i,a_1 \wedge a_j} + \BPresA{a_1 \wedge b_j,a_1 \wedge b_i} \in S_2(a_1 \Cdot a_1).\]
Since
\[\BPresA{a_1 \wedge a_i,a_1 \wedge a_j} + \BPresA{a_1 \wedge b_j,a_1 \wedge b_i} \in \Span{S_1(a_1 \Cdot a_1)},\]
the case follows.

\begin{case}{2}
\label{case:altlists2.2}
$a \neq a'$.
\end{case}

To simplify our notation, we will assume that $a,a' \in \cB = \{a_1,b_1,\ldots,a_g,b_g\}$ are $a = a_1$ and $a' = a_2$.
The other cases are identical up to changes in indices.  Following the notation in \S \ref{section:identificationalt2}, define
\begin{alignat*}{8}
&&\cV &= &\Span{a_2}_{\Q}^{\perp}/\Span{a_1} &\cong &\Span{A_V}_{\Q} &\quad \text{with} \quad &A_V = \{b_1,a_2,a_3,b_3,\ldots,a_g,b_g\},&\\
&&\cW &= &\Span{a_1}_{\Q}^{\perp}/\Span{a_2} &\cong &\Span{A_W}_{\Q} &\quad \text{with} \quad &A_W = \{a_1,b_2,a_3,b_3,\ldots,a_g,b_g\}.&
\end{alignat*}
We proved in Lemma \ref{lemma:mapbilinear2} that $\Fix{a_1,a_2}$ is isomorphic
to a quotient of the kernel of the map $\cV \otimes \cW \rightarrow \Q$
induced by the symplectic form $\omega$.  Under this isomorphism, a generator
$\PresA{a_1 \wedge x,a_2 \wedge y}$ of $\Fix{a_1,a_2}$ maps
to $x \otimes y \in \cV \otimes \cW$.

The kernel of $\cV \otimes \cW \rightarrow \Q$ is spanned by $X \cup Y$ where\footnote{After reading all the proofs
in \S \ref{section:presentationalt2} -- \S \ref{section:presentationalt3}, the reader might expect the word ``distinct'' to not appear in $Y$.  To make
some of the proofs in \S \ref{section:presentationalt2} -- \S \ref{section:presentationalt3} work (e.g., the proof
of Lemma \ref{lemma:s12twoa1weak}), the set $Y$ needs to contain elements
of the form $a_i \otimes b_i + b_i \otimes a_i$.
Here, however, we can require $i$ and $j$ be distinct.  Indeed, consider $3 \leq i \leq g$.
We want to prove that $a_i \otimes b_i + b_i \otimes a_i$ is in the span of $Y$.  For this, note that
\[a_i \otimes b_i + b_i \otimes a_i = (a_i \otimes b_i + b_1 \otimes a_1) - (a_2 \otimes b_2 + b_1 \otimes a_1) + (a_2 \otimes b_2 + a_i \otimes b_i).\]}
\begin{align*}
X = &\Set{$x \otimes y$}{$x \in A_V$, $y \in A_W$, $\omega(x,y)=0$},\\
Y = &\Set{$a_i \otimes b_i + b_j \otimes a_j$}{$1 \leq i,j \leq g$ distinct, $i \neq 1$, $j \neq 2$}.
\end{align*}
Since for $1 \leq i,j \leq g$ distinct with $i \neq 1$ and $j \neq 2$ we have 
\begin{align*}
(a_i-b_j) \otimes (b_i - a_j) &= a_i \otimes b_i + b_j \otimes a_j + \text{an element of $\Span{X}$},
\end{align*}
we can replace $Y$ by the set
\begin{align*}
&\Set{$(a_i-b_j) \otimes (b_i-a_j)$}{$1 \leq i,j \leq g$ distinct, $i \neq 1$, $j \neq 2$}.
\end{align*}
From this, we see that $\Fix{a_1,a_2}$ is generated by the following elements:
\begin{itemize}
\item $\BPresA{a_1 \wedge x,a_2 \wedge y}$ for $x \in A_V$ and $y \in A_W$ with $\omega(x,y) = 0$.  These
are elements of $S_1(a_1 \Cdot a_2)$.
\item $\PPresA{a_1 \wedge (a_i-b_j), a_2 \wedge (b_i - a_j)}$ for $1 \leq i < j \leq g$ distinct
with $i \neq 1$ and $j \neq 2$. 
\end{itemize}
It is thus enough to prove that for $1 \leq i < j \leq g$ distinct with $i \neq 1$ and $j \neq 2$, the element 
$\PPresA{a_1 \wedge (a_i - b_j),a_2 \wedge (b_i - a_j)}$ lies in
the span of $S_1(a_1 \Cdot a_2)$ and $S_2(a_1 \Cdot a_2)$.  The proof
of this is similar the argument we gave in Case \ref{case:altlists2.1}, so we omit it.
\end{proof}

\subsection{The set \texorpdfstring{$S_2$}{S2}} 
Define
\[S_2 = \bigcup_{a,a' \in \cB} S_2(a \Cdot a').\]
Recall that we proved in Lemma \ref{lemma:altspans1} that the linearization
map $\Phi$ takes $\Span{S_1}$ isomorphically to $\Span{T_1}$.  We have
$\Span{T_1} \subset \cK_g^a$, but $\Span{T_1,T_2}$ is not contained in $\cK_g^a$.
We now prove:

\begin{lemma}
\label{lemma:altspans2}
The linearization map $\Phi$ takes $\Span{S_1,S_2}$ isomorphically to $\cK_g^a \cap \Span{T_1,T_2}$.
\end{lemma}
\begin{proof}

Recall from \S \ref{section:genrelalt2} that $T = T_1 \cup T_2 \cup T_3$ where 
\begin{align*} 
T_1 &= \SetLong{$(x \wedge y) \wedge (z \wedge w)$}{$x,y,z,w \in \cB$, $x \prec y$, $z \prec w$, $x \wedge y \prec z \wedge w$,\\ and $\omega(x,z)=\omega(x,w)=\omega(y,z)=\omega(y,w)=0$}{,}{and $\omega(x,z)=\omega(x,w)=\omega(y,z)=\omega(y,w)=0$}\\
T_2 &= \SetLong{$(a \wedge a_i) \wedge (a' \wedge b_i)$}{$1 \leq i \leq g$, $a \in \cB \setminus \{a_i\}$, $a' \in \cB \setminus \{b_i\}$, \\ and $\omega(a,a')=\omega(a_i,a')=\omega(b_i,a)=0$}{,}{and $\omega(a,a')=\omega(a_i,a')=\omega(b_i,a)=0$}\\
T_3 &= \Set{$(a_i \wedge a_j) \wedge (b_i \wedge b_j)$, $(a_i \wedge b_j) \wedge (b_i \wedge a_j)$}{$1 \leq i < j \leq g$}.
\end{align*}
Moreover, $\wedge^2((\wedge^2 H)/\Q)$ is the $\Q$-vector space with generators $T$ subject to the relations
\[R = \Set{$\sum\nolimits_{i=1}^g (a_i \wedge b_i) \wedge (x \wedge y)$}{$x,y \in \cB$, $x \prec y$}.\]
Here each element of $R$ should be interpreted as a linear combination of elements of $T$ 
using Convention \ref{convention:alttconvention}.
Another way of stating this is that $\wedge^2((\wedge^2 H)/\Q)$ is the quotient of\footnote{Here
the $\Span{T_i}$ are subspaces of $\wedge^2((\wedge^2 H)/\Q)$, so possibly some relations in $R$
already hold in $\Span{T_1} \oplus \Span{T_2} \oplus \Span{T_3}$; indeed, as we will see this is in
fact the case.} $\Span{T_1} \oplus \Span{T_2} \oplus \Span{T_3}$
by the span of elements corresponding to $R$.  We have:

\begin{claim}{1}
\label{claim:allrelationst12}
Each $r \in R$ corresponds to an element of $\Span{T_1} \oplus \Span{T_2}$.
\end{claim}
\begin{proof}[Proof of claim]
Consider $x,y \in \cB$ with $x \prec y$, so we have an element
\[r = \sum\nolimits_{i=1}^g (a_i \wedge b_i) \wedge (x \wedge y) \in R.\]
There are two cases.  The first is that $\omega(x,y) \neq 0$, so $(x,y) = (a_k,b_k)$ for some $1 \leq k \leq g$.
In this case, $r$ is actually an element of $\Span{T_1}$, or rather a linear combination of elements of $T_1$ that
vanishes in $\Span{T_1}$.  The point
here is that our convention is that the term $(a_k \wedge a_k) \wedge (a_k \wedge b_k)$ is deleted, and
the rest of the terms clearly lie in $T_1$.  

The second case is that $\omega(x,y) = 0$.  There are a number of cases, so we will explain how
to deal with $x = a_k$ and $y = a_{\ell}$ for some $1 \leq k < \ell \leq g$.  The other cases are similar (but
with slightly different notation).  We have
\begin{align*}
r &= \sum\nolimits_{i=1}^g (a_i \wedge b_i) \wedge (a_k \wedge a_{\ell}) \\ 
  &= (a_k \wedge b_k) \wedge (a_k \wedge a_{\ell}) + (a_{\ell} \wedge b_{\ell}) \wedge (a_k \wedge a_{\ell}) + \sum_{\substack{1 \leq i \leq g \\ i \neq k,\ell}} \blue{(a_i \wedge b_i) \wedge (a_k \wedge a_{\ell})}.
\end{align*}
The blue terms lie in $\Span{T_1}$, while the remaining terms lie in $\Span{T_2}$ since
\begin{align*}
&(a_k \wedge b_k) \wedge (a_k \wedge a_{\ell}) + (a_{\ell} \wedge b_{\ell}) \wedge (a_k \wedge a_{\ell}) \\
=&(a_{\ell} \wedge a_k) \wedge (a_k \wedge b_k) - (a_k \wedge a_{\ell}) \wedge (a_{\ell} \wedge b_{\ell}).
\end{align*}
The claim follows.
\end{proof}

The set $T_2$ is the disjoint
union of the $T_2(a \Cdot a')$ as $a \Cdot a'$ ranges over elements of 
$S^2_0(\cB) = \Set{$a \Cdot a' \in \Sym^2(H)$}{$a,a' \in \cB$, $\omega(a,a') = 0$}$.  
Using the above claim, we deduce that
$\Span{T_1,T_2}$ is the quotient of the direct sum
\[\Span{T_1} \oplus \bigoplus_{s \in S^2_0(\cB)} \Span{T_2(s)}\]
by the subspace generated by the relations in $R$.  The subspace $\cK_g^a$ is the kernel of the symmetric
contraction $\fc\colon \wedge^2((\wedge^2 H)/\Q) \rightarrow \Sym^2(H)$.  
The symmetric contraction $\fc$ vanishes on $T_1$, and for $s \in S^2_0(\cB)$ it
takes elements of $T_2(s)$ to $s$ (see \S \ref{section:analyzet2}).
Since $S^2_0(\cB)$ is a linearly independent subset of $\Sym^2(H)$, 
we deduce that $\cK_g^a \cap \Span{T_1,T_2}$ is the quotient of 
\begin{equation}
\label{eqn:decomposet1t2}
\Span{T_1} \oplus \bigoplus_{s \in S^2_0(\cB)} \cK_g^a \cap \Span{T_2(s)}
\end{equation}
by the relations in $R$.  We remark that the relations in $R$ must lie in the above
direct sum since otherwise they would map to nontrivial elements of $\Sym^2(H)$ under $\fc$.

We proved in Lemma \ref{lemma:altspans1} that $\Phi$ restricts to an isomorphism
between $\Span{S_1}$ and $\Span{T_1}$.  For $s \in S^2_0(\cB)$, recall that $S_2(s)$ is a
vector space.  We proved in Lemma \ref{lemma:altlists2} that $\Phi$ restricts to an
isomorphism between $S_2(s)$ and $\cK_g^a \cap \Span{T_2(s)}$.  From this and in light of the
previous paragraph,\footnote{In particular, the decomposition \eqref{eqn:decomposet1t2}.} to prove that $\Phi$
restricts to an isomorphism between $\Span{S_1,S_2}$ and $\Span{T_1,T_2}$, it is 
enough to prove that each relation in $R$ lifts to a relation in $\Span{S_1,S_2}$.

The relations of the form
\[r = \sum\nolimits_{i=1}^g (a_i \wedge b_i) \wedge (a_k \wedge b_k)\]
for some $1 \leq k \leq g$ are relations between elements of $T_1$, and since
$\Phi$ restricts to an isomorphism between $\Span{S_1}$ and $\Span{T_1}$ these
relations lift\footnote{See the proof of Lemma \ref{lemma:altspans1} for explicit lifts.} 
to relations in $\Span{S_1}$.  We must therefore only deal with the relations in the following claim:

\begin{claim}{2}
Consider $x,y \in \cB$ with $x \prec y$.  Then the relation
\[r = \sum\nolimits_{i=1}^g (a_i \wedge b_i) \wedge (x \wedge y)\]
lifts to a relation in $\fK_g^a$.
\end{claim}
\begin{proof}[Proof of claim]
We will give the details for $(x,y) = (a_1,a_2)$.
The other cases are similar but require worse notation.  Write our relation as
\begin{align*}
r &= \sum\nolimits_{i=1}^g (a_i \wedge b_i) \wedge (a_1 \wedge a_2) \\
  &= -(a_1 \wedge b_1) \wedge (a_2 \wedge a_1) - (a_1 \wedge a_2) \wedge (a_2 \wedge b_2)+ \sum\nolimits_{i=3}^g \blue{(a_i \wedge b_i) \wedge (a_1 \wedge a_2)}.
\end{align*}
The blue sum lifts to the following element of $\Span{S_1}$:
\[\sum\nolimits_{i=3}^g \BPresA{a_i \wedge b_i,a_1 \wedge a_2}.\]
We claim that the remaining part lifts to the following element of $\Span{S_1,S_2}$:
\begin{equation}
\label{eqn:partoflift}
-\PresA{a_1 \wedge (b_1+a_2),a_2 \wedge (a_1+b_2)} + \BPresA{a_1 \wedge b_1,a_2 \wedge b_2}.
\end{equation}
To see this, note that $\Phi$ maps \eqref{eqn:partoflift} to\footnote{Part of this calculation is that $(a_1 \wedge a_2) \wedge (a_2 \wedge a_1) = 0$.}
\begin{align*}
&-(a_1 \wedge (b_1+a_2)) \wedge (a_2 \wedge (a_1+b_2)) + (a_1 \wedge b_1) \wedge (a_2 \wedge b_2) \\
=&-(a_1 \wedge b_1) \wedge (a_2 \wedge a_1) - (a_1 \wedge a_2) \wedge (a_2 \wedge a_1) -(a_1 \wedge b_1) \wedge (a_2 \wedge b_2) \\
 &- (a_1 \wedge a_2) \wedge (a_2 \wedge b_2) + (a_1 \wedge b_1) \wedge (a_2 \wedge b_2) \\
=&-(a_1 \wedge b_1) \wedge (a_2 \wedge a_1) - (a_1 \wedge a_2) \wedge (a_2 \wedge b_2),
\end{align*}
as desired.

Combining the above lifts, we see that the relation in $\fK_g^a$ we must verify is
\begin{equation}
\label{eqn:r2liftalt}
0 = -\PresA{a_1 \wedge (b_1+a_2),a_2 \wedge (a_1+b_2)} + \BPresA{a_1 \wedge b_1,a_2 \wedge b_2} + \sum\nolimits_{i=3}^g \BPresA{a_i \wedge b_i,a_1 \wedge a_2}.
\end{equation}
For this, note that $\{a_1,b_1+a_2,a_2,a_1+b_2,a_3,b_3,\ldots,a_g,b_g\}$ is a symplectic basis.  In $(\wedge^2 H)/\Q$, we therefore
have
\[a_1 \wedge (b_1+a_2) + a_2 \wedge (a_1+b_2) + a_3 \wedge b_3 + \cdots + a_g \wedge b_g = 0.\]
By plugging this into its first term, we calculate that $\PresA{a_1 \wedge (b_1+a_2),a_2 \wedge (a_1+b_2)}$ equals
\begin{align*}
&-\PresA{a_2 \wedge (a_1+b_2),a_2 \wedge (a_1+b_2)} - \sum\nolimits_{i=3}^g \PresA{a_i \wedge b_i,a_2 \wedge (a_1+b_2)} \\
=&-\sum\nolimits_{i=3}^g \left(\BPresA{a_i \wedge b_i,a_2 \wedge a_1} + \BPresA{a_i \wedge b_i,a_2 \wedge b_2}\right) \\
=&- \left(\sum\nolimits_{i=3}^g \BPresA{a_i \wedge b_i,a_2 \wedge b_2}\right) + \left(\sum\nolimits_{i=3}^g \BPresA{a_i \wedge b_i,a_1 \wedge a_2}\right).
\end{align*}
In $(\wedge^2 H)/\Q$, we have $\sum\nolimits_{i=1}^g a_i \wedge b_i = 0$.  Plugging this into the first term of our formula, the formula becomes
\begin{align*}
&\BPresA{a_1 \wedge b_1,a_2 \wedge b_2} + \BPresA{a_2 \wedge b_2,a_2 \wedge b_2} + \sum\nolimits_{i=3}^g \BPresA{a_i \wedge b_i,a_1 \wedge a_2} \\
=&\BPresA{a_1 \wedge b_1,a_2 \wedge b_2} + \sum\nolimits_{i=3}^g \BPresA{a_i \wedge b_i,a_1 \wedge a_2}.
\end{align*}
Since this equals $\PresA{a_1 \wedge (b_1+a_2),a_2 \wedge (a_1+b_2)}$, the relation \eqref{eqn:r2liftalt} follows.
\end{proof}

This completes the proof of the lemma.
\end{proof}

\subsection{Relations in \texorpdfstring{$T_3$}{T3}}

We extract a useful consequence of the above proof:

\begin{lemma}
\label{lemma:quotientdim}
The quotient of $\wedge^2((\wedge^2 H)/\Q)$ by $\Span{T_1,T_2}$ has dimension $g(g-1)$.
\end{lemma}
\begin{proof}
Claim \ref{claim:allrelationst12} of the proof of Lemma \ref{lemma:altspans2} says
that all relations between $T = T_1 \cup T_2 \cup T_3$ are actually relations between $T_1 \cup T_2$.
Since $T$ generates $\wedge^2((\wedge^2 H)/\Q)$, it follows that the indicated quotient
has dimension $|T_3|$.  Since
\[T_3 = \Set{$(a_i \wedge a_j) \wedge (b_i \wedge b_j)$, $(a_i \wedge b_j) \wedge (b_i \wedge a_j)$}{$1 \leq i < j \leq g$},\]
the set $T_3$ has cardinality $2 \binom{g}{2} = g(g-1)$.  The lemma follows.
\end{proof}

\subsection{The set \texorpdfstring{$S_{12}$}{S12}}

Recall that
\[S_{12} = \bigcup_{a \Cdot a' \in S^2_0(\cB)} \pFix{a,a'}.\]
Our goal in the rest of Part \ref{part:alt} is to prove Theorem \ref{maintheorem:presentationalt}, which says that
$\Phi$ is an isomorphism from $\fK_g^a = \Span{S} = \Span{S_{12},S_3}$ to
\[\cK_g^a \subset \wedge^2((\wedge^2 H)/\Q) = \Span{T_1,T_2,T_3}.\]
We close this section by proving the following partial result in this direction:

\begin{lemma}
\label{lemma:altspans12}
The linearization map $\Phi$ takes $\Span{S_{12}}$ isomorphically to $\cK_g^a \cap \Span{T_1,T_2}$.
\end{lemma}
\begin{proof}
Lemma \ref{lemma:altspans2} says that $\Phi$ takes $\Span{S_1,S_2}$ isomorphically to
$\cK_g^a \cap \Span{T_1,T_2}$, so it is enough to prove that $\Span{S_{12}} = \Span{S_1,S_2}$.
For $a \Cdot a' \in S^2_0(\cB)$,
we proved in Lemma \ref{lemma:altlists2} that $S_1(a \Cdot a') \cup S_2(a \Cdot a')$ spans
$\pFix{a,a'}$.  For $i=1,2$, we have
\[S_i = \bigcup_{a \Cdot a' \in S^2_0(\cB)} S_i(a \Cdot a').\]
This is a disjoint union for $i=2$, but the $S_1(a \Cdot a')$ for different $a \Cdot a' \in S^2_0(\cB)$
overlap.  Combining these two facts, we see that
\[\Span{S_1,S_2} = \Span{\bigcup\nolimits_{a \Cdot a' \in S^2_0(\cB)} S_1(a \Cdot a') \cup S_2(a \Cdot a')} = \Span{\bigcup\nolimits_{a \Cdot a' \in S^2_0(\cB)} \pFix{a,a'}} = \Span{S_{12}}.\qedhere\]
\end{proof}

\section{Symmetric kernel, alternating version VII: structure of \texorpdfstring{$S_3$}{S3}}
\label{section:presentationalt7}

We will continue using all the notation from \S \ref{section:presentationalt1} -- \S \ref{section:presentationalt6}.
Having proved in Lemma \ref{lemma:altspans12} that $\Phi$ takes $\Span{S_{12}}$ isomorphically
to $\cK_g^a \cap \Span{T_1,T_2}$, our remaining task in Part \ref{part:alt} is to extend this
to $\Span{S} = \Span{S_{12},S_3}$ and prove Theorem \ref{maintheorem:presentationalt}.  We
will do this in \S \ref{section:presentationalt9}.  This section and the next one contain
some preliminary results about $S_3$.

\subsection{Quotients}

Define 
\begin{align*}
\fT_g &= \fK_g^a / \Span{S_{12}}, \\
\cT_g &= \cK_g^a / (\cK_g^a \cap \Span{T_1,T_2}).
\end{align*}
Lemma \ref{lemma:altspans12} says that the linearization map $\Phi\colon \fK_g^a \rightarrow \cK_g^a$ 
takes $\Span{S_{12}}$ isomorphically
to $\cK_g^a \cap \Span{T_1,T_2}$.  It follows that $\Phi$ descends to a map $\oPhi\colon \fT_g \rightarrow \cT_g$.
Our goal is to prove that $\Phi$ is an isomorphism (Theorem \ref{maintheorem:presentationalt}).
Since $\Phi$ restricts to an isomorphism from $\Span{S_{12}}$ to $\cK_g^a \cap \Span{T_1,T_2}$,
this is equivalent to proving that $\oPhi$ is an isomorphism.  That $\oPhi$ is surjective is easy:

\begin{lemma}
\label{lemma:quotientsurjective}
The map $\oPhi\colon \fT_g \rightarrow \cT_g$ is surjective.
\end{lemma}
\begin{proof}
This follows from the fact that $\Phi\colon \fK_g^a \rightarrow \cK_g^a$ is surjective.  This could be proved directly, but
another approach is to note that
\[\cK_g^a = \ker(\wedge^2((\wedge^2 H)/\Q) \stackrel{\fc}{\longrightarrow} \Sym^2(H))\]
is an irreducible algebraic representation of $\Sp_{2g}(\Z)$.  Such representations are indexed by
partitions with at most $g$ parts (see \cite[\S 17]{FultonHarris}), and $\cK_g^a$ is the one
corresponding to the partition $2+1+1$.  Since $\Phi$ is not the zero map, its image is a nonzero
subrepresentation of the irreducible representation $\cK_g^a$, and hence its image must
be $\cK_g^a$.
\end{proof}

\subsection{Dimension of target}

We now prove:

\begin{lemma}
\label{lemma:ctgdim}
The vector space $\cT_g$ is $g(g-2)$-dimensional.
\end{lemma}
\begin{proof}
The vector space $\cK_g^a$ is the kernel of the symmetric contraction
$\fc\colon \wedge^2((\wedge^2 H)/\Q) \rightarrow \Sym^2(H)$.  Let:
\begin{itemize}
\item $V$ be be the quotient of $\wedge^2((\wedge^2 H)/\Q)$ by $\Span{T_1,T_2}$; and
\item $W$ be the quotient of $\Sym^2(H)$ by $\fc(\Span{T_1,T_2})$.
\end{itemize}
The symmetric contraction induces a map $\ofc\colon V \rightarrow W$, and
$\cT_g$ is isomorphic to $\ker(\ofc)$.

Lemma \ref{lemma:quotientdim} says that $\dim(V) = g(g-1)$.  To calculate $\dim(W)$, note
that Lemma \ref{lemma:imagect2} implies that $\fc(\Span{T_1,T_2})$ is the subspace of
$\Sym^2(H)$ spanned by $\Set{$a \Cdot a'$}{$a,a' \in \cB$, $\omega(a,a')=0$}$.  This
implies that the set $\{a_1 \Cdot b_1,\ldots,a_g \Cdot b_g\}$ is a basis for a complement
to $\fc(\Span{T_1,T_2})$, so $\dim(W) = g$.

The map $\fc$ is surjective: this could be proved directly, but just like in the proof of Lemma \ref{lemma:quotientsurjective}
it also follows from the fact that $\Sym^2(H)$ is an irreducible algebraic representation
of $\Sp_{2g}(\Z)$.  The corresponding partition is simply $2$.  This implies that $\ofc$ is also
surjective.  Consequently,
\[\dim(\cT_g) = \dim(V) - \dim(W) = g(g-1) - g = g(g-2).\qedhere\]
\end{proof}

\subsection{Proof strategy}

Recall that we want to prove that $\oPhi\colon \fT_g \rightarrow \cT_g$ is an isomorphism.
Lemmas \ref{lemma:quotientsurjective} and \ref{lemma:ctgdim} say that $\oPhi$ is a surjective
map to a $g(g-2)$-dimensional vector space.  To prove that $\oPhi$ is an isomorphism, it is enough
to prove that $\fT_g$ is at most $g(g-2)$-dimensional.  We will do this via a calculation
involving generators and relations.  The rest of this section is devoted to constructing
a generating set for $\fT_g$.  We will then give some relations in $\fT_g$ in \S \ref{section:presentationalt8},
and complete the proof in \S \ref{section:presentationalt9}.

\subsection{Basic elements}

Recall that
\[S_3    = \bigcup_{\substack{1 \leq i,j \leq g \\ i \neq j}} \oFix{a_i-b_j,b_i-a_j}.\]
Elements of $S_3$ are written in orange.
For $\eta \in \fK_g^a$, let $\oeta$ be its image in $\fT_g = \fK_g^a/\Span{S_{12}}$.  For
$1 \leq i,j,k \leq g$ distinct, define
\[\Delta^i_{jk} = \overline{\OPresA{(a_i-b_j) \wedge (a_k - b_i),(b_i-a_j) \wedge (b_k - a_i)}} \in \fT_g.\]
We call $\Delta^i_{jk}$ a {\em basic element} of $\fT_g$.  These satisfy:

\begin{lemma}
\label{lemma:basiceasy}
For $1 \leq i,j,k \leq g$ distinct, we have 
$\Delta^i_{kj} = -\Delta^i_{jk}$.
\end{lemma}
\begin{proof}
Immediate from the fact that
\begin{small}
\[\OPresA{(a_i-b_j) \wedge (a_k - b_i),(b_i-a_j) \wedge (b_k - a_i)}
= -\OPresA{(a_i-b_k) \wedge (b_i - a_j),(b_i-a_k) \wedge (a_i - b_j)}.\qedhere\]
\end{small}%
\end{proof}

\subsection{Generation by basic elements}

We now prove:

\begin{lemma}
\label{lemma:basicgen}
The vector space $\fT_g$ is spanned by
$\Set{$\Delta^i_{jk}$}{$1 \leq i,j,k \leq g$ distinct}$.
\end{lemma}
\begin{proof}
Lemma \ref{lemma:presentationaltgenset} says that $\fK_g^a$ is spanned by $S = S_{12} \cup S_3$.
It follows that $\fT_g = \fK_g^a / \Span{S_{12}}$ is spanned by the image of $S_3$.  Fixing
some $1 \leq i,j \leq g$ distinct, it is therefore enough to prove that the image of
$\oFix{a_i-b_j,b_i-a_j}$ in $\fT_g$ is contained in the span of the indicated generating set.

In Lemma \ref{lemma:symmetrics}, we proved that the action of the symmetric group $\fS_g$ on
$\fK_g^a$ takes $\Span{S}$ to itself.  It follows from the proof of that lemma that this action
also takes $\Span{S_{12}}$ to itself, so we get an induced action of $\fS_g$ on $\fT_g$.  Applying
an appropriate of $\fS_g$, we reduce ourselves to proving that the image of $\oFix{a_1-b_2,b_1-a_2}$ 
is contained in the span of the indicated generating set.

We construct generators for $\oFix{a_1-b_2,b_1-a_2}$ in the now-familiar way and then show
that their images in $\fT_g$ are in the span of the indicated generating set.
Following the notation in \S \ref{section:identificationalt2}, define
\begin{alignat*}{8}
&&\cV &= &\Span{b_1-a_2}_{\Q}^{\perp}/\Span{a_1-b_2} &\cong &\Span{A_V}_{\Q} &\quad \text{with} \quad &A_V = \{b_1,a_2,a_3,b_3,\ldots,a_g,b_g\},&\\
&&\cW &= &\Span{a_1-b_2}_{\Q}^{\perp}/\Span{b_1-a_2} &\cong &\Span{A_W}_{\Q} &\quad \text{with} \quad &A_W = \{a_1,b_2,a_3,b_3,\ldots,a_g,b_g\}.&
\end{alignat*}
We proved in Lemma \ref{lemma:mapbilinear2} that $\oFix{a_1-b_2,b_1-a_2}$ is isomorphic
to a quotient of the kernel of the map $\cV \otimes \cW \rightarrow \Q$
induced by the symplectic form $\omega$.  Under this isomorphism, a generator
$\OPresA{(a_1-b_2) \wedge x,(b_1-a_2) \wedge y}$ of $\oFix{a_1-b_2,b_1-a_2}$ maps
to $x \otimes y \in \cV \otimes \cW$.

The kernel of $\cV \otimes \cW \rightarrow \Q$ is spanned by $X \cup Y$ where
\begin{align*}
X = &\Set{$x \otimes y$}{$x \in A_V$, $y \in A_W$, $\omega(x,y)=0$},\\
Y = &\Set{$a_k \otimes b_k + b_1 \otimes a_1$, $a_2 \otimes b_2 + b_k \otimes a_k$}{$3 \leq k \leq g$} \\
    &\cup \{a_2 \otimes b_2 + b_1 \otimes a_1\}.
\end{align*}
Since for $3 \leq k \leq g$ we have
\begin{align*}
(a_k-b_1) \otimes (b_k-a_1) &= a_k \otimes b_k + b_1 \otimes a_1 + \text{an element of $\Span{X}$},\\
(a_2-b_k) \otimes (b_2-a_k) &= a_2 \otimes b_2 + b_k \otimes a_k + \text{an element of $\Span{X}$},\\
(a_1-b_2) \otimes (b_1-a_2) &= a_1 \otimes b_1 + b_2 \otimes a_2 + \text{an element of $\Span{X}$},
\end{align*}
we can replace $Y$ by the set
\begin{align*}
&\Set{$(a_k-b_1) \otimes (b_k-a_1)$, $(a_2-b_k) \otimes (b_2-a_k)$}{$3 \leq k \leq g$},\\
&\quad\quad\cup \{(a_2-b_1) \otimes (b_2-a_1)\}.
\end{align*}
From this, we see that $\oFix{a_1-b_2,b_1-a_2}$ is generated by the elements listed in the following
cases.  To prove the lemma, we must prove that each of these generators maps to something
in the span of the indicated generators of $\fT_g$.

\begin{case}{1}
$\OPresA{(a_1-b_2) \wedge x,(b_1-a_2) \wedge y}$ for $x \in A_V$ and $y \in A_W$ with $\omega(x,y) = 0$.
\end{case}

\noindent
These equal $\PPresA{x \wedge (a_1-b_2),y \wedge (b_1-a_2)} \in S_{12}$, and thus
go to zero in $\fT_g$.

\begin{case}{2}
$\OPresA{(a_1-b_2) \wedge (a_k-b_1),(b_1-a_2) \wedge (b_k-a_1)}$ with $3 \leq k \leq g$.
\end{case}

\noindent
This maps to $\Delta^1_{2k} \in \fT_g$.

\begin{case}{3}
$\OPresA{(a_1-b_2) \wedge (a_2-b_k),(b_1-a_2) \wedge (b_2-a_k)}$ with $3 \leq k \leq g$.
\end{case}

\noindent
This equals $-\OPresA{(a_2-b_1) \wedge (a_k-b_2),(b_2-a_1) \wedge (b_k-a_2)}$, which
maps to $-\Delta^2_{1k} \in \fT_g$.

\begin{case}{4}
$\OPresA{(a_1-b_2) \wedge (a_2-b_1),(b_1-a_2) \wedge (b_2-a_1)}$.
\end{case}

Since $\OPres{-,-}$ is alternating, this equals $-\PresA{(a_1-b_2) \wedge (a_2-b_1),(a_1-b_2) \wedge (a_2-b_1)} = 0$.
We remark that this is why we always insist that $\oFix{a_1-b_2,b_1-a_2}$ is a {\em quotient}
of the kernel of the map $\cV \otimes \cW \rightarrow \Q$; see Lemma \ref{lemma:mapbilinear2}.
\end{proof}

\section{Symmetric kernel, alternating version VIII: relations between basic elements}
\label{section:presentationalt8}

We will continue using all the notation from \S \ref{section:presentationalt1} -- \S \ref{section:presentationalt7}.
This section contains three relations involving our basic elements $\Delta^i_{jk}$.
We will use these relations to prove Theorem \ref{maintheorem:presentationalt} in \S \ref{section:presentationalt9}.

\subsection{Relation I}

The first is:

\begin{lemma}
\label{lemma:altrelation1}
Let $1 \leq i,j,k,\ell \leq g$ be distinct.  Then $\Delta^k_{ij} + \Delta^i_{\ell k} = \Delta^{\ell}_{ij} + \Delta^j_{\ell k}$.
\end{lemma}
\begin{proof}
As in the proof of Lemma \ref{lemma:basicgen}, we can apply an appropriate element of the symmetric
group and reduce ourselves to proving that $\Delta^3_{12} + \Delta^1_{4 3} = \Delta^{4}_{12} + \Delta^2_{4 3}$.
What we will prove is that both sides of this identity equal
$\overline{\OPres{(a_3-b_1) \wedge (a_2 - b_4),(b_3-a_1) \wedge (b_2 - a_4)}}$:

\begin{claim}{1}
\label{claim:basicrelations1.1}
$\Delta^3_{12} + \Delta^1_{4 3} = \overline{\OPres{(a_3-b_1) \wedge (a_2 - b_4),(b_3-a_1) \wedge (b_2 - a_4)}}$.
\end{claim}

We have
\begin{align*}
\Delta^3_{12} &= \overline{\OPres{(a_3-b_1) \wedge (a_2-b_3),(b_3-a_1) \wedge (b_2-a_3)}}, \\
\Delta^1_{43} &= \overline{\OPres{(a_1-b_4) \wedge (a_3-b_1),(b_1-a_4) \wedge (b_3-a_1)}} \\
              &= \overline{\OPres{(a_3-b_1) \wedge (a_1-b_4),(b_3-a_1) \wedge (b_1-a_4)}}.
\end{align*}
We must therefore prove that
\begin{small}
\begin{equation}
\label{eqn:basicrelations1.1a}
\OPres{(a_3-b_1) \wedge (a_2-b_3),(b_3-a_1) \wedge (b_2-a_3)} + \OPres{(a_3-b_1) \wedge (a_1-b_4),(b_3-a_1) \wedge (b_1-a_4)}
\end{equation}
\end{small}%
equals
\begin{equation}
\label{eqn:basicrelations1.1b}
\OPres{(a_3-b_1) \wedge (a_2 - b_4),(b_3-a_1) \wedge (b_2 - a_4)}
\end{equation}
modulo elements of $\Span{S_{12}}$.  These both live in $\oFix{a_3-b_1,b_3-a_1}$, so we work there.

Following the notation in \S \ref{section:identificationalt2}, define 
\begin{alignat*}{8}
&&\cV &= &\Span{b_3-a_1}_{\Q}^{\perp}/\Span{a_3-b_1} &\cong &\Span{A_V}_{\Q} &\quad \text{with} \quad &A_V = \{b_3,a_1,a_2,b_2,a_3,b_3,\ldots,a_g,b_g\},&\\
&&\cW &= &\Span{a_3-b_1}_{\Q}^{\perp}/\Span{b_3-a_1} &\cong &\Span{A_W}_{\Q} &\quad \text{with} \quad &A_W = \{a_3,b_1,a_2,b_2,a_3,b_3,\ldots,a_g,b_g\}.&
\end{alignat*}
We proved in Lemma \ref{lemma:mapbilinear2} that $\oFix{a_3-b_1,b_3-a_1}$ is isomorphic
to a quotient of the kernel of the map $\cV \otimes \cW \rightarrow \Q$
induced by the symplectic form $\omega$.  Under this isomorphism, a generator
$\OPresA{(a_3-b_1) \wedge x,(b_3-a_1) \wedge y}$ of $\oFix{a_3-b_1,b_3-a_1}$ maps
to $x \otimes y \in \cV \otimes \cW$.

The elements of $\cV \otimes \cW$ corresponding to elements of $S_{12}$ are
\[X = \Set{$x \otimes y$}{$x \in A_V$, $y \in A_W$, $\omega(x,y) = 0$}.\]
We can therefore work modulo $\Span{X}$.  Let $\equiv$ denote equality modulo $\Span{X}$.
The element of $\cV \otimes \cW$ corresponding to \eqref{eqn:basicrelations1.1a} is
\begin{align*}
(a_2-b_3) \otimes (b_2-a_3) + (a_1-b_4) \otimes (b_1-a_4) &\equiv (a_2 \otimes b_2 + b_3 \otimes a_3) + (a_1 \otimes b_1 + b_4 \otimes a_4) \\
                                                          &=      (a_1 \otimes b_1 + b_3 \otimes a_3) + (a_2 \otimes b_2 + b_4 \otimes a_4) \\
                                                          &\equiv (a_1-b_3) \otimes (b_1-a_3) + (a_2-b_4) \otimes (b_2-a_4).
\end{align*}
The element $(a_2-b_4) \otimes (b_2-a_4)$ corresponds to \eqref{eqn:basicrelations1.1b}, so what we must
prove is that $(a_1-b_3) \otimes (b_1-a_3)$ corresponds to $0$.  In fact, this corresponds to
\[\OPresA{(a_3-b_1) \wedge (a_1-b_3),(b_3-a_1) \wedge (b_1-a_3)}.\]
Since $\OPresA{-,-}$ is alternating, this vanishes.  The claim follows.

\begin{claim}{2}
$\Delta^{4}_{12} + \Delta^2_{4 3} = \overline{\OPres{(a_3-b_1) \wedge (a_2 - b_4),(b_3-a_1) \wedge (b_2 - a_4)}}$.
\end{claim}

We have
\begin{align*}
\Delta^4_{12} &= \overline{\OPres{(a_4-b_1) \wedge (a_2-b_4),(b_4-a_1) \wedge (b_2-a_4)}}\\
              &= \overline{\OPres{(a_2-b_4) \wedge (a_4-b_1),(b_2-a_4) \wedge (b_4-a_1)}},\\
\Delta^2_{43} &= \overline{\OPres{(a_2-b_4) \wedge (a_3-b_2),(b_2-a_4) \wedge (b_3-a_2)}}. \\
\end{align*}
and
\[\overline{\OPres{(a_3-b_1) \wedge (a_2 - b_4),(b_3-a_1) \wedge (b_2 - a_4)}} = \overline{\OPres{(a_2 - b_4) \wedge (a_3-b_1),(b_2 - a_4) \wedge (b_3-a_1)}}.\]
We must therefore prove that
\begin{small}
\begin{equation}
\label{eqn:basicrelations1.2a}
\OPres{(a_2-b_4) \wedge (a_4-b_1),(b_2-a_4) \wedge (b_4-a_1)} + \OPres{(a_2-b_4) \wedge (a_3-b_2),(b_2-a_4) \wedge (b_3-a_2)}
\end{equation}
\end{small}%
equals
\begin{equation}
\label{eqn:basicrelations1.2b}
\OPres{(a_2 - b_4) \wedge (a_3-b_1),(b_2 - a_4) \wedge (b_3-a_1)}
\end{equation}
modulo elements of $\Span{S_{12}}$.  These both live in $\oFix{a_2-b_4,b_2-a_4}$.  The calculation
is similar to the one from Claim \ref{claim:basicrelations1.1}, so we omit it.
\end{proof}

\subsection{Relation II}

The second relation is:

\begin{lemma}
\label{lemma:altrelation2}
For $1 \leq i,j,k \leq g$ distinct, we have
\[\Delta^i_{jk} = \overline{\OPresA{(a_i-b_j) \wedge (a_k+a_j),(b_i-a_j) \wedge (b_k-b_j)}}.\]
\end{lemma}
\begin{proof}
As in the proof of Lemma \ref{lemma:basicgen}, we can apply an appropriate element
of the symmetric group and reduce ourselves to proving that
\[\Delta^1_{23} = \overline{\OPresA{(a_1-b_2) \wedge (a_3+a_2),(b_1-a_2) \wedge (b_3-b_2)}}.\]
By the definition of $\Delta^1_{23}$ and $\fT_g$, this lemma is asserting that
\begin{align}
\label{eqn:trianglealttoprove}
&\OPresA{(a_1-b_2) \wedge (a_3 - b_1),(b_1-a_2) \wedge (b_3 - a_1)}\\
&\quad-\OPresA{(a_1-b_2) \wedge (a_3+a_2),(b_1-a_2) \wedge (b_3-b_2)} \notag
\end{align}
lies in the span of $S_{12}$.  This difference lies in $\oFix{a_1-b_2,b_1-a_2}$, so we work there.

Following the notation in \S \ref{section:identificationalt2}, define
\begin{alignat*}{8}
&&\cV &= &\Span{b_1-a_2}_{\Q}^{\perp}/\Span{a_1-b_2} &\cong &\Span{A_V}_{\Q} &\quad \text{with} \quad &A_V = \{b_1,a_2,a_3,b_3,\ldots,a_g,b_g\},&\\
&&\cW &= &\Span{a_1-b_2}_{\Q}^{\perp}/\Span{b_1-a_2} &\cong &\Span{A_W}_{\Q} &\quad \text{with} \quad &A_W = \{a_1,b_2,a_3,b_3,\ldots,a_g,b_g\}.&
\end{alignat*}
We proved in Lemma \ref{lemma:mapbilinear2} that $\oFix{a_1-b_2,b_1-a_2}$ is isomorphic
to a quotient of the kernel of the map $\cV \otimes \cW \rightarrow \Q$
induced by the symplectic form $\omega$.  Under this isomorphism, a generator
$\OPresA{(a_1-b_2) \wedge x,(b_1-a_2) \wedge y}$ of $\oFix{a_1-b_2,b_1-a_2}$ maps
to $x \otimes y \in \cV \otimes \cW$.

The elements of $\cV \otimes \cW$ corresponding to elements of $S_{12}$ are
\[X = \Set{$x \otimes y$}{$x \in A_V$, $y \in A_W$, $\omega(x,y) = 0$}.\]  
We can therefore work modulo $\Span{X}$.  Let $\equiv$ denote equality modulo $\Span{X}$.  The 
difference \eqref{eqn:trianglealttoprove} corresponds to the following element of
$\cV \otimes \cW$:
\begin{align*}
&(a_3 - b_1) \otimes (b_3 - a_1) - (a_3+a_2) \otimes (b_3-b_2) \equiv (a_3 \otimes b_3 + b_1 \otimes a_1) - (a_3 \otimes b_3 -a_2 \otimes b_2) \\
=&a_2 \otimes b_2 + b_1 \otimes a_1 \equiv (a_2-b_1) \otimes (b_2-a_1).
\end{align*}
This corresponds to the element
\begin{small}
\[\OPresA{(a_1-b_2) \wedge (a_2-b_1),(b_1-a_2) \wedge (b_2-a_1)} = -\OPresA{(a_1-b_2) \wedge (a_2-b_1),(a_1-b_2) \wedge (a_2-b_1)}.\]
\end{small}%
Since $\OPresA{-,-}$ is alternating, this is $0$.  The lemma follows.
\end{proof}

\subsection{Relation III}

Our third and final relation is:

\begin{lemma} 
\label{lemma:altrelation3}
Let $1 \leq i,j,k,\ell \leq g$ be distinct.  Then 
\[\Delta^i_{jk} + \Delta^i_{k \ell} = \Delta^{i}_{j\ell} + \overline{\PresA{(a_{\ell}+a_j) \wedge (b_k-b_j),(b_{\ell}-b_j) \wedge (a_k+a_j)}}.\]
\end{lemma}
\begin{proof}
As in the proof of Lemma \ref{lemma:basicgen}, we can apply an appropriate element of the symmetric
group and reduce ourselves to proving that 
\[\Delta^1_{23} + \Delta^1_{34} = \Delta^{1}_{24} + \overline{\PresA{(a_4+a_2) \wedge (b_3-b_2),(b_4-b_2) \wedge (a_3+a_2)}}.\]
We will prove that both sides of this equal
$\overline{\OPres{(a_1-b_3) \wedge (a_4+a_2),(b_1-a_3) \wedge (b_4-b_2)}}$.

\begin{claim}{1}
\label{claim:altrelations3.1}
$\Delta^1_{23} + \Delta^1_{34} = \overline{\OPres{(a_1-b_3) \wedge (a_4+a_2),(b_1-a_3) \wedge (b_4-b_2)}}$.
\end{claim}

We have
\begin{align*}
\Delta^1_{23} &= \overline{\OPres{(a_1-b_2) \wedge (a_3-b_1),(b_1-a_2) \wedge (b_3-a_1)}}\\
              &= -\overline{\OPres{(a_1-b_3) \wedge (b_1-a_2),(b_1-a_3) \wedge (a_1-b_2)}},\\
\Delta^1_{34} &= \overline{\OPres{(a_1-b_3) \wedge (a_4-b_1),(b_1-a_3) \wedge (b_4-a_1)}}. \\
\end{align*}
We must therefore prove that
\begin{small}
\[-\OPres{(a_1-b_3) \wedge (b_1-a_2),(b_1-a_3) \wedge (a_1-b_2)} + \OPres{(a_1-b_3) \wedge (a_4-b_1),(b_1-a_3) \wedge (b_4-a_1)}\]
\end{small}%
equals
\[\OPres{(a_1-b_3) \wedge (a_4+a_2),(b_1-a_3) \wedge (b_4-b_2)}\]
modulo elements of $\Span{S_{12}}$.  These both live in $\oFix{a_1-b_3,b_1-a_3}$.  The calculation
is similar to the one from Claim \ref{claim:basicrelations1.1} of the proof of Lemma \ref{lemma:altrelation1}, so we omit it.

\begin{claim}{2}
\label{claim:altrelations3.2}
The elements
\[\Delta^{1}_{24} + \overline{\PresA{(a_4+a_2) \wedge (b_3-b_2),(b_4-b_2) \wedge (a_3+a_2)}}\]
and
\[\overline{\OPres{(a_1-b_3) \wedge (a_4+a_2),(b_1-a_3) \wedge (b_4-b_2)}}\]
are equal.
\end{claim}

Lemma \ref{lemma:altrelation2} says that
\begin{align*}
\Delta^1_{24} &= \overline{\OPresA{(a_1-b_2) \wedge (a_4+a_2),(b_1-a_2) \wedge (b_4-b_2)}} \\
              &= \overline{\OPresA{(a_4+a_2) \wedge (a_1-b_2),(b_4-b_2) \wedge (b_1-a_2)}}.
\end{align*}
Also,
\begin{small}
\[\overline{\OPres{(a_1-b_3) \wedge (a_4+a_2),(b_1-a_3) \wedge (b_4-b_2)}} = \overline{\OPres{(a_4+a_2) \wedge (a_1-b_3),(b_4-b_2) \wedge (b_1-a_3)}}.\]
\end{small}%
We must therefore prove that
\begin{small}
\[\OPresA{(a_4+a_2) \wedge (a_1-b_2),(b_4-b_2) \wedge (b_1-a_2)} + \PresA{(a_4+a_2) \wedge (b_3-b_2),(b_4-b_2) \wedge (a_3+a_2)}\]
\end{small}%
equals
\[\OPres{(a_4+a_2) \wedge (a_1-b_3),(b_4-b_2) \wedge (b_1-a_3)}\]
modulo elements of $\Span{S_{12}}$.  These both live in $\Fix{a_4+a_2,b_4-b_2}$.  The calculation
is similar to the one from Claim \ref{claim:basicrelations1.1} of the proof of Lemma \ref{lemma:altrelation1}, so we omit it.
\end{proof}

\section{Symmetric kernel, alternating version IX: the proof of Theorem \ref{maintheorem:presentationalt}}
\label{section:presentationalt9}

We will continue using all the notation from \S \ref{section:presentationalt1} -- \S \ref{section:presentationalt8}.
We finally prove Theorem \ref{maintheorem:presentationalt}, whose statement we recall.

\newtheorem*{maintheorem:presentationalt}{Theorem \ref{maintheorem:presentationalt}}
\begin{maintheorem:presentationalt}
For $g \geq 4$, the linearization map $\Phi\colon \fK_g^a \rightarrow \cK_g^a$ is an isomorphism.
\end{maintheorem:presentationalt}
\begin{proof}
Recall that Lemma \ref{lemma:presentationaltgenset} says that $\fK_g^a$ is generated
by $S = S_{12} \cup S_3$.  Lemma \ref{lemma:altspans12} says that $\Phi$ takes $\Span{S_{12}}$ isomorphically
onto its image.  Letting $\fT_g = \fK_g^a/\Span{S_{12}}$ and letting
$\cT_g$ be the quotient of $\cK_g^a$ by $\Span{\Phi(S_{12})}$, it is therefore
enough to prove that the induced map $\oPhi\colon \fT_g \rightarrow \cT_g$
is an isomorphism.
Lemmas \ref{lemma:quotientsurjective} and \ref{lemma:ctgdim} say that $\oPhi$ is a surjective
map to a $g(g-2)$-dimensional vector space.  It is therefore enough to prove
that $\fT_g$ is at most $g(g-2)$-dimensional.

Define
\[R_g = \Set{$\Delta^i_{jk} \in \fT_g$}{$1 \leq i,j,k \leq g$ distinct}.\]
By Lemma \ref{lemma:basicgen}, the set $R_g$ spans $\fT_g$ for $g \geq 4$.
The key to the proof is the following smaller generating set:

\begin{unnumberedclaim}
For all $g \geq 4$, the vector space $\fT_g$ is spanned by $R'_g \cup R_g[\leq g-1]$, where:
\begin{align*}
R'_g = &\Set{$\Delta^g_{1i}$}{$2 \leq i \leq g-1$} 
       \cup\Set{$\Delta^1_{gi}$}{$2 \leq i \leq g-1$} \cup \{\Delta^2_{g1}\}, \\
R_g[\leq g-1] = &\Set{$\Delta^i_{jk}$}{$1 \leq i,j,k \leq g-1$ distinct}.
\end{align*}
\end{unnumberedclaim}

We will prove this claim in a moment, but let us first see why it implies that
$\fT_g$ is at most $g(g-2)$-dimensional.  The proof is by induction on $g$.
The base case is $g=4$, where we have to prove that $\fT_4$ is at most
$4(4-2) = 8$ dimensional.  For this, note that 
\[R'_4 = \{\Delta^4_{12},\Delta^4_{13},\Delta^1_{42},\Delta^1_{43},\Delta^2_{41}\}\]
has $5$ elements.  The set 
\[R_4[\leq 3] = \{\Delta^1_{23},\Delta^1_{32},\Delta^2_{13},\Delta^2_{31},\Delta^3_{12},\Delta^3_{21}\}\]
has $6$ elements, but since $\Delta^i_{kj} = - \Delta^i_{jk}$ (see Lemma \ref{lemma:basiceasy}) only
$3$ are needed to span $\fT_4$.  We conclude that $\fT_4$ is at most $5+3=8$ dimensional, as desired.

For the inductive step, assume that $g \geq 5$ and that $\fT_{g-1}$ is at
most $(g-1)((g-1)-2)$-dimensional.  We must prove that
$\fT_g$ is at most $g(g-2)$-dimensional.  The natural map $\fK_{g-1}^a \rightarrow \fK_g^a$ taking
a generator $\PresA{\kappa_1,\kappa_2}$ of $\fK_{g-1}^a$ to the same generator
of $\fK_g^a$ induces a map $\iota\colon \fT_{g-1} \rightarrow \fT_g$.  The map
$\iota$ takes a basic element $\Delta^i_{jk} \in \fT_{g-1}$ to the same
basic element of $\fT_g$.  It follows that the image of $\iota$ is the span
of $R_g[\leq g-1]$.  In particular, $\Span{R_g[\leq g-1]}$ is at most
$(g-1)((g-1)-2)$-dimensional.  The set $R'_g$ from the above claim has
\[(g-2)+(g-2)+1 = 2g-3\]
elements, so $\Span{R'_g}$ is at most $2g-3$ dimensional.  Since
$\fT_g$ is spanned by $R'_g$ and $R_g[\leq g-1]$, we conclude that
$\fT_g$ is at most
\[(2g-3)+(g-1)((g-1)-2) = (2g-3) + (g^2-4g+3) = g^2-2g=g(g-2)\]
dimensional, as desired.

It remains to prove the above claim:

\begin{proof}[Proof of claim]
It is enough to prove that every element $\Delta$ of the generating set
\[R_g = \Set{$\Delta^i_{jk} \in \fT_g$}{$1 \leq i,j,k \leq g$ distinct}.\]
that does not lie in $R_g[\leq g-1]$ can be written as a linear combination
of elements of $R'_g$ and $R_g[\leq g-1]$.  There are
two families of elements of $R_g$ that do not lie in $R_g[\leq g-1]$.

The first are elements of the form $\Delta^g_{ij}$ with $1 \leq i,j \leq g-1$ distinct.
Since $\Delta^g_{ji} = - \Delta^g_{ij}$ (Lemma \ref{lemma:basiceasy}), we can
assume that $i < j$.  If $i=1$, then $\Delta^g_{ij} \in R'_g$, so we can assume that
$2 \leq i < j \leq g-1$.  Lemma \ref{lemma:altrelation3} gives a relation
\begin{equation}
\label{eqn:userelation3}
\Delta^g_{i1} + \Delta^g_{1j} = \Delta^{g}_{ij} + \overline{\PresA{(a_{j}+a_i) \wedge (b_1-b_i),(b_{j}-b_i) \wedge (a_1+a_i)}}.
\end{equation}
Set
\[\kappa = \PresA{(a_{j}+a_i) \wedge (b_1-b_i),(b_{j}-b_i) \wedge (a_1+a_i)}.\]
Since $1,i,j \leq g-1$, we can view $\kappa$ as an element of $\fK_{g-1}^a$.  
If $g \geq 5$, then Lemma \ref{lemma:basicgen} says that $\fT_{g-1}$ is generated
by basic elements, so $\okappa$ is in the span of $R_g[\leq g-1]$.  This
argument does not work if $g=4$; however, in this case  
Lemma \ref{lemma:s3symmetricalt} (which does work in genus $3$) says that $\kappa \in \fK_{g-1}^a$
can be written\footnote{It is enlightening to go through the proof and work this out
explicitly.} as a linear combination of elements of our basis $S$.  By the proof
of Lemma \ref{lemma:basicgen} these map to linear combinations of basic
elements in $\fT_{g-1}$, so again we deduce that $\okappa$ is in the span of $R_g[\leq g-1]$.

In either case, since $\Delta^g_{i1} = -\Delta^g_{1i}$ (Lemma \ref{lemma:basiceasy})
we can rearrange \eqref{eqn:userelation3} and see that
\[\Delta^g_{ij} = \Delta^g_{1j} - \Delta^g_{1i} - \okappa\]
lies in the span of $R_g[\leq g-1]$ and $R'_g$, as desired.

The other family of elements of $R_g$ that do not lie in $R_g[\leq g-1]$ are
those of the form $\Delta^i_{jg}$ and $\Delta^i_{gj}$ for $1 \leq i,j \leq g-1$ distinct.
We must show that these lie in the span of $R_g[\leq g-1]$ and $R'_g$.
Since $\Delta^i_{jg} = - \Delta^i_{gj}$ (Lemma \ref{lemma:basiceasy}), it is enough
to deal with $\Delta^i_{gj}$.  If $i=1$ then $\Delta^i_{gj} \in R'_g$, so
we can assume that $i \neq 1$.  Since we also have $\Delta^2_{g1} \in R'_g$,
we can assume that if $i=2$ then $j \neq 1$.  In other words, we can assume
that $2 \leq i,j \leq g-1$.  Lemma \ref{lemma:altrelation1} gives a relation
\[\Delta^1_{gj} + \Delta^g_{i1} = \Delta^{i}_{gj} + \Delta^j_{i1}.\]
We have $\Delta^1_{gj} \in R'_g$ and $\Delta^j_{i1} \in R_g[\leq g-1]$, and we
already proved that $\Delta^g_{i1}$ is in the span of $R'_g$ and $R_g[\leq g-1]$.
We conclude that $\Delta^{i}_{gj}$ is also in the span of $R'_g$ and $R_g[\leq g-1]$,
as desired.
\end{proof}

This completes the proof of the theorem.
\end{proof}

\part{Verifying the presentation for the symmetric kernel, symmetric version}
\label{part:sym}

We now turn to Theorem \ref{maintheorem:presentationsym}.  See
the introductory \S \ref{section:presentationsym0} for an outline of what we do in this part.
Throughout, we make the following genus assumption:

\begin{assumption}
\label{assumption:genussym}
Throughout Part \ref{part:sym}, we assume that $g \geq 4$.
\end{assumption}

\section{Symmetric kernel, symmetric version: introduction}
\label{section:presentationsym0}

We start by recalling some results and definitions from earlier in the paper, and
then outline what we prove in this part.

\subsection{Symmetric contraction}
Recall that $\omega$ is the symplectic form on $H$.
The {\em symmetric contraction} is the alternating $\Sym^2(H)$-valued alternating
form $\fc$ on $(\wedge^2 H)/\Q$ defined via the formula
\[\text{$\fc(x \wedge y,z \wedge w) = \omega(x,z) y \Cdot w - \omega(x,w) y \Cdot z - \omega(y,z) x \Cdot w + \omega(y,w) x \Cdot z$ for $x,y,z,w \in H$}.\]
Elements $\kappa_1,\kappa_2 \in (\wedge^2 H)/\Q$ are
{\em sym-orthogonal} if $\fc(\kappa_1,\kappa_2) = 0$.  The
{\em sym-orthogonal complement} of $\kappa \in (\wedge^2 H)/\Q$ is the subspace $\kappa^{\perp}$
consisting of all elements that are sym-orthogonal to $\kappa$.

\subsection{Special pairs}
A {\em special pair} in $(\wedge^2 H_{\Z})/\Z$ is an element of the form $x \wedge y$ with
$\omega(x,y) \in \{-1,0,1\}$.  Examples include symplectic pairs and isotropic pairs.
Lemmas \ref{lemma:symplecticorthogonal} and \ref{lemma:isotropicorthogonal} say that
the sym-orthogonal complements in $(\wedge^2 H)/\Q$ of these are:
\begin{itemize}
\item for a symplectic pair $a \wedge b$,  we have $(a \wedge b)^{\perp} = \overline{\wedge^2 \Span{a,b}_{\Q}^{\perp}}$; and
\item for an isotropic pair $a \wedge a'$, we have $(a \wedge a')^{\perp} = \overline{\wedge^2 \Span{a,a'}_{\Q}^{\perp}}$.
\end{itemize}

\subsection{Non-symmetric presentation}
We will use the generators and relations for $\fK_g$ from Theorem \ref{theorem:summarypresentation}, whose statement
we recall:

\newtheorem*{theorem:summarypresentation2}{Theorem \ref{theorem:summarypresentation}}
\begin{theorem:summarypresentation2}
For $g \geq 4$, the vector space $\fK_g$ has the following presentation:
\begin{itemize}
\item {\bf Generators}.
A generator $\Pres{\kappa_1,\kappa_2}$ for all sym-orthogonal
$\kappa_1,\kappa_2 \in (\wedge^2 H)/\Q$ such that either
$\kappa_1$ or $\kappa_2$ (or both) is a special pair.
\item {\bf Relations}.  The following two families of relations:
\begin{itemize}
\item For special pairs $\zeta \in (\wedge^2 H)/\Q$
and all $\kappa_1,\kappa_2 \in (\wedge^2 H)/\Q$ that are sym-orthogonal to $\zeta$
and all $\lambda_1,\lambda_2 \in \Q$, the linearity relations
\begin{align*}
\Pres{\zeta,\lambda_1 \kappa_1 + \lambda_2 \kappa_2} &= \lambda_1 \Pres{\zeta,\kappa_1} + \lambda_2 \Pres{\zeta,\kappa_2} \quad \text{and} \\
\Pres{\lambda_1 \kappa_1 + \lambda_2 \kappa_2,\zeta} &= \lambda_1 \Pres{\kappa_1,\zeta} + \lambda_2 \Pres{\kappa_2,\zeta}.
\end{align*}
\item For all special pairs $\zeta \in (\wedge^2 H)/\Q$ and all $\kappa \in (\wedge^2 H)/\Q$ that are sym-orthogonal to $\zeta$
and all $n \in \Z$ such that $n \zeta$ is a special pair, the relations
\begin{align*}
\Pres{n \zeta,\kappa} &= n \Pres{\zeta,\kappa} \quad \text{and} \\
\Pres{\kappa,n \zeta} &= n \Pres{\kappa,\zeta}.\qedhere
\end{align*}
\end{itemize}
\end{itemize}
\end{theorem:summarypresentation2}

\subsection{Symmetrizing}

Recall from Lemma \ref{lemma:decomposekg} that $\fK_g^s$ is the $+1$-eigenspace of the involution of
$\fK_g$ that takes a generator $\Pres{\kappa_1,\kappa_2}$ to $\Pres{\kappa_2,\kappa_1}$.  We symmetrize
a generator $\Pres{\kappa_1,\kappa_2}$ of $\fK_g$ to
\[\PresS{\kappa_1,\kappa_2} = \frac{1}{2} \left(\Pres{\kappa_1,\kappa_2} + \Pres{\kappa_2,\kappa_1}\right) \in \fK_g^s.\]
The symmetrized generators generate $\fK_g^s$.  They satisfy the same relations as the generators
of $\fK_g$, and also the symmetry relation
$\PresS{\kappa_2,\kappa_1} = \PresS{\kappa_1,\kappa_2}$.

\subsection{Goal and outline}

We have a linearization map $\Phi\colon \fK_g^s \rightarrow \Sym^2((\wedge^2 H)/\Q)$.  On generators,
it satisfies
\[\Phi(\PresS{\kappa_1,\kappa_2}) = \kappa_1 \Cdot \kappa_2 \in \Sym^2((\wedge^2 H)/\Q).\]
Our goal in this part of the paper is to prove Theorem \ref{maintheorem:presentationsym}, which says that $\Phi$ is an isomorphism
from $\fK_g^s$ to $\Sym^2((\wedge^2 H)/\Q)$.
The proof uses the proof technique described in \S \ref{section:prooftechnique}, and is modeled
on the proofs of Theorems \ref{maintheorem:slstd}--\ref{maintheorem:spsym}.  However, since
the calculations are lengthy we spread them out over nine sections:
\begin{itemize}
\item In \S \ref{section:presentationsym1} -- \S \ref{section:presentationsym4}, we construct
a subset $S$ of $\fK_g^s$ such that $\Phi$ restricted to $\Span{S}$ is an isomorphism (Step 1).  This calculation
is lengthy since it depends on the construction of three important families of elements of $\fK_g^s$
(the $\Theta$-, the $\Lambda$-, and the $\Omega$-elements), and it takes work to prove their basic properties.
\item In \S \ref{section:presentationsym5}, we prove that the
$\Sp_{2g}(\Z)$-orbit of $S$ spans $\fK_g^s$ (Step 2).  We also outline the proof that $\Sp_{2g}(\Z)$ takes
takes $\Span{S}$ to itself (Step 3).  
\item Finally, in \S \ref{section:presentationsym6} -- \ref{section:presentationsym9}
we complete the outlined proof that $\Sp_{2g}(\Z)$ takes $\Span{S}$ to itself.
Together with Step 2 this implies that $\Span{S} = \fK_g^s$, so by
Step 1 we conclude that $\Phi$ is an isomorphism.
\end{itemize}
Throughout the following nine sections, $\Phi$ will always mean the linearization map
$\Phi\colon \fK_g^s \rightarrow \Sym^2((\wedge^2 H)/\Q)$.  Also, $\fc$ will always mean
the symmetric contraction.
Finally, we will fix a symplectic basis $\cB = \{a_1,b_1,\ldots,a_g,b_g\}$ for $H_{\Z}$.

\section{Symmetric kernel, symmetric version I: \texorpdfstring{$S_1$}{S1} and structure of target}
\label{section:presentationsym1}

We discuss the structure of $\Sym^2((\wedge^2 H)/\Q)$, and then begin our construction of $S$.

\subsection{Generators and relations for target}
\label{section:genrelsym2}

Let $\prec$ be the following total order on $\cB$:
\[a_1 \prec b_1 \prec a_2 \prec b_2 \prec \cdots \prec a_g \prec b_g.\]
Set
\[T = \Set{$(x \wedge y) \Cdot (z \wedge w)$}{$x,y,z,w \in \cB$, $x \prec y$, $z \prec w$} \subset \Sym^2((\wedge^2 H)/\Q).\]
The set $T$ generates $\Sym^2((\wedge^2 H)/\Q)$, and the relations between elements of $T$ are generated by
the set
\[R = \Set{$\sum\nolimits_{i=1}^g (a_i \wedge b_i) \Cdot (x \wedge y)$}{$x,y \in \cB$, $x \prec y$}.\]

\subsection{Lifting easy generators}

Some elements of $T$ are easily lifted to $\fK_g^s$.  Define
\[S_1 = \Set{$\SPresS{x \wedge y,z \wedge w}$}{$x,y,z,w \in \cB$, $x \prec y$, $z \prec w$, $\fc(x \wedge y,z \wedge w) = 0$}.\]
For $\SPresS{x \wedge y,z \wedge w} \in S_1$, we have
\[\Phi(\SPresS{x \wedge y,z \wedge w}) = (x \wedge y) \Cdot (z \wedge w) \in T.\]
Like we did above, we will write elements of $S_1 \subset \fK_g^s$ in blue.  More generally, we will use blue to write
elements of $\fK_g^s$ that lie in $\Span{S_1}$.  The set $S_1$ consists of two kinds of elements:
\begin{itemize}
\item those of the form $\SPresS{x \wedge y,z \wedge w}$ for $x,y,z,w \in \cB$ with $x \prec y$ and $z \prec w$ and
$\omega(x,z) = \omega(x,w) = \omega(y,z) = \omega(y,w) = 0$; and
\item those of the form $\SPresS{a_i \wedge b_i,a_i \wedge b_i}$ for some $1 \leq i \leq g$.  These lie
in $S_1$ since $\fc$ is alternating, or more concretely due to the calculation
\[\fc(a_i \wedge b_i,a_i \wedge b_i) = -\omega(a_i,b_i) (b_i \Cdot a_i) - \omega(b_i,a_i) (a_i \Cdot b_i) = -(b_i \Cdot a_i) + (a_i \Cdot b_i) = 0.\]
\end{itemize}
Let $T_1 \subset T$ be the image of $S_1$.

\subsection{Lifting easy relations}
Let $R_1$ be the subset of $R$ consisting of relations
between elements of $T_1$.  Thus $R_1$ consists of relations of the form
\[\text{$\sum\nolimits_{i=1}^g (a_i \wedge b_i) \Cdot (a_k \wedge b_k)$ with $1 \leq k \leq g$}.\]
These lift to relations between the elements of $S_1$ due to the bilinearity relations in $\fK_g^s$:
\[\sum_{i=1}^g \SPresS{a_i \wedge b_i,a_k \wedge b_k} = \SPresS{\sum_{i=1}^g a_i \wedge b_i,a_k \wedge b_k} = \SPresS{0,a_k \wedge b_k} = 0.\]

\subsection{Other relations do not affect \texorpdfstring{$T_1$}{T1}}

Set $R_2 = R \setminus R_1$.  Each element of $R_2$ involves an element of $T$ that appears
in no other relations in $R$.  For instance, for $1 \leq k < \ell \leq g$ the set $R_2$ contains
the relation
\[\sum_{i=1}^g (a_i \wedge b_i) \Cdot (a_k \wedge a_{\ell}),\] 
and no other relation in $R$ uses the generator $(a_k \wedge b_k) \Cdot (a_k \wedge a_{\ell})$.
This implies that the subspace of $\Sym^2((\wedge^2 H)/\Q)$ spanned by $T_1$ is generated
by $T_1$ subject to only the relations in $R_1$.  This implies:

\begin{lemma}
\label{lemma:symspans1}
The linearization map $\Phi$ takes $\Span{S_1}$ isomorphically to $\Span{T_1}$.
\end{lemma}
\begin{proof}
Immediate from the fact $\Phi$ takes $S_1$ bijectively to $T_1$ and each relation in $R_1$ lifts
to a relation between the elements of $S_1$.
\end{proof}

\subsection{Remaining generators}

Define
\begin{align*}
T_2 &= \Set{$(a_i \wedge b_i) \Cdot (x \wedge b_i)$, $(a_i \wedge b_i) \Cdot (a_i \wedge y)$}{$1 \leq i \leq g$, $x,y \in \cB \setminus \{a_i,b_i\}$},\\
T_3 &= \Set{$(a_i \wedge y) \Cdot (x \wedge b_i)$}{$1 \leq i \leq g$, $x,y \in \cB \setminus \{a_i,b_i\}$, $\omega(x,y)=0$},\\
T_4 &= \Set{$(a_i \wedge a_j) \Cdot (b_i \wedge b_j)$, $(a_i \wedge b_j) \Cdot (b_i \wedge a_j)$}{$1 \leq i < j \leq g$}.
\end{align*}
The set $T_2 \cup T_3 \cup T_4$ is almost equal to $T \setminus T_1$.  The only difference is that
for some $\eta \in T \setminus T_1$ we have $-\eta \in T_2 \cup T_3 \cup T_4$.  For instance, we have
$(a_1 \wedge b_1) \Cdot (b_1 \wedge a_2) \in T \setminus T_1$ but
\[(a_1 \wedge b_1) \Cdot (a_2 \wedge b_1) = -(a_1 \wedge b_1) \Cdot (b_1 \wedge a_2) \in T_2.\]
In any case, we have that $\Sym^2((\wedge^2 H)/\Q)$ is generated by $T_1 \cup T_2 \cup T_3 \cup T_4$ subject
to appropriate versions of the relations in $R$.  In the next three sections, we will construct
sets $S_2,S_3,S_4 \subset \fK_g^s$ such that $\Phi$ takes $S_i$ bijectively to $T_i$, and we will
prove that all relations in $R$ lift to relations between elements of $S = S_1 \cup \cdots \cup S_4$.
The elements of $S_2$ and $S_3$ and $S_4$ are called $\Theta$-elements, $\Lambda$-elements, and $\Omega$-elements.

\subsection{Obvious blue elements}
\label{section:obviousbluesym}

Before we continue, we make a useful observation.
Recall that we write elements of $\Span{S_1}$ in blue.  One easy way to recognize these is as follows.
Consider a generator $\PresS{\kappa_1,\kappa_2}$ such that there exists subsets $\cB_1,\cB_2 \subset \cB$ such that:
\begin{itemize}
\item $\kappa_1 \in \overline{\wedge^2 \Span{\cB_1}}$ and $\kappa_2 \in \overline{\wedge^2 \Span{\cB_2}}$; and
\item $\omega(x,y) = 0$ for all $x \in \cB_1$ and $y \in \cB_2$.
\end{itemize}
Note that $\cB_1$ and $\cB_2$ need not be disjoint.  We then have that $\PresS{\kappa_1,\kappa_2} \in \Span{S_1}$,
so we can write $\SPresS{\kappa_1,\kappa_2}$.  This is most easily seen by example:
\begin{align*}
\PresS{(a_1+3b_1) \wedge a_2, a_2 \wedge (a_3 - 2b_3)} = &\PresS{(a_1 + 3b_1) \wedge a_2, a_2 \wedge a_3} - 2 \PresS{(a_1+3b_1) \wedge a_2,a_2 \wedge b_3} \\
                                                       = &\SPresS{a_1 \wedge a_2,a_2 \wedge a_3} + 3 \SPresS{b_1 \wedge a_2,a_2 \wedge a_3} \\
                                                         &- 2\SPresS{a_1 \wedge a_2,a_2 \wedge b_3} -6 \SPresS{b_1 \wedge a_2, a_2 \wedge b_3}.
\end{align*}
We thus have $\SPresS{(a_1+3b_1) \wedge a_2, a_2 \wedge (a_3 - 2b_3)} \in \Span{S_1}$.

\section{Symmetric kernel, symmetric version II: \texorpdfstring{$S_2$}{S2} and the \texorpdfstring{$\Theta$}{Theta}-elements}
\label{section:presentationsym2}

We continue using all the notation from \S \ref{section:presentationsym1}.
This section constructs the set $S_2$ that lifts $T_2$.  It consists of what are called
$\Theta$-elements of $\fK_g^s$, and the first part of this section constructs
these in more generality than is needed for $S_2$ alone.

\subsection{Definition}
Let $a \wedge b$ be a symplectic pair in $\wedge^2 H_{\Z}$ and let $x,y \in \Span{a,b}^{\perp}$.  The 
$\Theta$-elements $\ThetaS{a \wedge b,x \wedge b}$ and $\ThetaS{a \wedge b,a \wedge y}$ are elements
of $\fK_g^s$ that are taken by $\Phi$ to
\[(a \wedge b) \Cdot (x \wedge b) \in \Sym^2((\wedge^2 H)/\Q) \quad \text{and} \quad
  (a \wedge b) \Cdot (a \wedge y) \in \Sym^2((\wedge^2 H)/\Q),\]
respectively.
To find these elements, note that
\begin{align*}
\left((a+x) \wedge b\right) \Cdot \left((a+x) \wedge b\right) &= (a \wedge b) \Cdot (a \wedge b) + 2 (a \wedge b) \Cdot (x \wedge b) + (x \wedge b) \Cdot (x \wedge b),\\
\left(a \wedge (b+y)\right) \Cdot \left(a \wedge (b+y)\right) &= (a \wedge b) \Cdot (a \wedge b) + 2 (a \wedge b) \Cdot (a \wedge y) + (a \wedge y) \Cdot (a \wedge y).
\end{align*}
Since $\fc$ is alternating, for any $z \in (\wedge^2 H)/\Q$ we have $\fc(z,z) = 0$ and thus there exists
a generator $\PresS{z,z}$ of $\fK_g^s$.  This suggests:

\begin{definition}
\label{definition:theta}
For a symplectic pair $a \wedge b$ in $\wedge^2 H_{\Z}$ and $x,y \in \Span{a,b}^{\perp}$, define
\begin{align*}
\ThetaS{a \wedge b,x \wedge b} &= \frac{1}{2}\left(\PresS{(a+x) \wedge b,(a+x) \wedge b} - \PresS{a \wedge b,a \wedge b} - \PresS{x \wedge b,x \wedge b}\right),\\
\ThetaS{a \wedge b,a \wedge y} &= \frac{1}{2}\left(\PresS{a \wedge (b+y),a \wedge (b+y)} - \PresS{a \wedge b,a \wedge b} - \PresS{a \wedge y,a \wedge y}\right).\qedhere
\end{align*}
\end{definition}

By construction, we have
\begin{align*}
\Phi(\ThetaS{a \wedge b,x \wedge b}) &= (a \wedge b) \Cdot (x \wedge b), \\
\Phi(\ThetaS{a \wedge b,a \wedge y}) &= (a \wedge b) \Cdot (a \wedge y).
\end{align*}

\begin{remark}
\label{remark:thetaabuse}
Despite our notation $\ThetaS{a \wedge b,x \wedge b}$, this depends on the ordered
tuple $(a,b,x)$, not on $a \wedge b$ and $x \wedge b$.  A similar remark applies to $\ThetaS{a \wedge b,a \wedge y}$.
We chose to abuse notation like
this to emphasize that $\ThetaS{a \wedge b,x \wedge b}$ should be regarded as the ``missing''
element $\PresS{a \wedge b,x \wedge b}$ of $\fK_g^s$ that should exist if $\fK_g^s$ is isomorphic
to $\Sym^2((\wedge^2 H)/\Q)$.  Later we will prove that $\ThetaS{a \wedge b,x \wedge b}$ behaves
as if it only depends on $a \wedge b$ and $x \wedge b$, e.g., Lemma \ref{lemma:thetabilinear2}
below says that $\ThetaS{(a+nb) \wedge b,x \wedge b} = \ThetaS{a \wedge b,x \wedge b}$ for
all $n \in \Z$.  
\end{remark}

\subsection{\texorpdfstring{$\Theta$}{Theta}-expansion I}

If $x,y,z \in H_{\Z}$ are pairwise orthogonal elements such that $x \wedge z$ and $y \wedge z$ and $(x+y) \wedge z$ are isotropic pairs,
then the relations in $\fK_g^s$ imply that
\[\PresS{(x+y) \wedge z,(x+y) \wedge z} = \PresS{x \wedge z,x \wedge z} + 2 \PresS{x \wedge z,y \wedge z} + \PresS{y \wedge z,y \wedge z}.\]
Using $\Theta$-elements, we can similarly expand out some other elements:

\begin{lemma}[$\Theta$-expansion I]
\label{lemma:thetaexpansion1}
Let $a \wedge b$ be a symplectic pair and let $x,y \in \Span{a,b}^{\perp}$.  Then
\begin{align*}
\PresS{(a+x) \wedge b,(a+x) \wedge b} &= \PresS{a \wedge b,a \wedge b} + 2\ThetaS{a \wedge b,x \wedge b} + \PresS{x \wedge b,x \wedge b},\\
\PresS{a \wedge (b+y),a \wedge (b+y)} &= \PresS{a \wedge b,a \wedge b} + 2\ThetaS{a \wedge b,y \wedge a} + \PresS{y \wedge a,y \wedge a}.
\end{align*}
\end{lemma}
\begin{proof}
Immediate from Definition \ref{definition:theta}.
\end{proof}

\subsection{\texorpdfstring{$\Theta$}{Theta}-linearity}

The following is a key property of the $\Theta$-elements:

\begin{lemma}[$\Theta$-linearity]
\label{lemma:thetalinear}
Let $a \wedge b$ be a symplectic pair.  Then for all $z_1,z_2 \in (a \wedge b)^{\perp}$ and $\lambda_1,\lambda_2 \in \Z$ we have
\begin{align*}
\ThetaS{a \wedge b,(\lambda_1 z_1 + \lambda_2 z_2) \wedge b} &= \lambda_1 \ThetaS{a \wedge b, z_1 \wedge b} + \lambda_2 \ThetaS{a \wedge b,z_2 \wedge b}, \\
\ThetaS{a \wedge b,a \wedge (\lambda_1 z_1 + \lambda_2 z_2)} &= \lambda_1 \ThetaS{a \wedge b, a \wedge z_1} + \lambda_2 \ThetaS{a \wedge b,a \wedge z_2}.
\end{align*}
\end{lemma}
\begin{proof}
Both formulas are proved similarly, so we will prove the first.  
The key calculation is the following special case of the lemma:

\begin{unnumberedclaim}
For a partial basis $\{z_1,z_2\}$ of $(a \wedge b)^{\perp}$ with $\omega(z_1,z_2) = 0$ and $n_1,n_2 \in \Z$, we have
\[\ThetaS{a \wedge b,(n_1 z_1+n_2 z_2) \wedge b} = \ThetaS{a \wedge b,n_1 z_1 \wedge b} + \ThetaS{a \wedge b,n_2 z_2 \wedge b}.\]
\end{unnumberedclaim}
\begin{proof}[Proof of claim]
Whether or not the claim holds is invariant under
the action of $\Sp_{2g}(\Z)$ on $\fK_g^s$.  Recall that we have our fixed symplectic basis
$\cB = \{a_1,b_1,\ldots,a_g,b_g\}$ for $H_{\Z}$.  Applying an appropriate element of $\Sp_{2g}(\Z)$,
we can assume that
\[\text{$a_1 = a$,\ \ $b_1 = b$,\ \ $a_2 = z_1$,\ \ $a_3 = z_2$}.\]
To make our notation easier to digest, replace $n_1,n_2 \in \Z$ by $n_2,n_3 \in \Z$. 
Since the restriction of $\Phi$ to $\Span{S_1}$
is injective (Lemma \ref{lemma:symspans1}) and $\Phi$ takes both
\begin{equation}
\label{eqn:thetalineartoprove1}
\ThetaS{a_1 \wedge b_1,(n_2 a_2 + n_3 a_3) \wedge b_1}
\end{equation}
and
\begin{equation}
\label{eqn:thetalineartoprove2}
\ThetaS{a_1 \wedge b_1,n_2 a_2 \wedge b_1} + \ThetaS{a_1 \wedge b_1,n_3 a_3 \wedge b_1}
\end{equation}
to the same element of $\Sym^2((\wedge^2 H)/\Q)$, it is enough to prove that
\eqref{eqn:thetalineartoprove1} and \eqref{eqn:thetalineartoprove2} are equal
modulo $\Span{S_1}$.  Let $\equiv$ denote equality in $\fK_g^s$ modulo $\Span{S_1}$.

By definition, $2\ThetaS{a_1 \wedge b_1,(n_2 a_2 + n_3 a_3) \wedge b_1}$
equals\footnote{Here and in future calculations we use \S \ref{section:obviousbluesym} to identify and
then delete blue terms lying in $\Span{S_1}$.}
\begin{align*}
&\PresS{(a_1+n_2 a_2+n_3 a_3) \wedge b_1,(a_1+n_2 a_2+n_3 a_3) \wedge b_1} \\
&- \SPresS{a_1 \wedge b_1,a_1 \wedge b_1} - \SPresS{(n_2 a_2+n_3 a_3) \wedge b_1,(n_2 a_2+n_3 a_3) \wedge b_1} \\
\equiv &\PresS{(a_1+n_2 a_2+n_3 a_3) \wedge b_1,(a_1+n_2 a_2+n_3 a_3) \wedge b_1}
\end{align*}
and $2\ThetaS{a_1 \wedge b_1,n_2 a_2 \wedge b_1} + 2\ThetaS{a_1 \wedge b_1,n_3 a_3 \wedge b_1}$ equals
\begin{align*}
&\PresS{(a_1+n_2 a_2) \wedge b_1,(a_1+n_2 a_2) \wedge b_1} - \SPresS{a_1 \wedge b_1,a_1 \wedge b_1} - \SPresS{n_2 a_2 \wedge b_1,n_2 a_2 \wedge b_1}, \\
&+\PresS{(a_1+n_3 a_3) \wedge b_1,(a_1+n_3 a_3) \wedge b_1} - \SPresS{a_1 \wedge b_1,a_1 \wedge b_1} - \SPresS{n_3 a_3 \wedge b_1,n_3 a_3 \wedge b_1} \\
\equiv &\PresS{(a_1+n_2 a_2) \wedge b_1,(a_1+n_2 a_2) \wedge b_1} + \PresS{(a_1+n_3 a_3) \wedge b_1,(a_1+n_3 a_3) \wedge b_1}.
\end{align*}
Our goal, therefore, is to prove that
\begin{small}
\begin{align}
\label{eqn:thetalineartoprove3}
       &\PresS{(a_1+n_2 a_2+n_3 a_3) \wedge b_1,(a_1+n_2 a_2+n_3 a_3) \wedge b_1} \\
\equiv &\PresS{(a_1+n_2 a_2) \wedge b_1,(a_1+n_2 a_2) \wedge b_1} + \PresS{(a_1+n_3 a_3) \wedge b_1,(a_1+n_3 a_3) \wedge b_1} \nonumber
\end{align}
\end{small}%
In $(\wedge^2 H)/\Q$, we have the following identities.  The colors are there to help the reader
match up terms with later formulas.\footnote{The only color which has a definite meaning right now is blue,
which is used to indicate elements of $\Span{S_1}$.  In later sections we will give meanings to purple
and orange and green terms, but currently these colors have no meaning and we are free to use them.}
\begin{alignat*}{5}
&\purple{(a_1+n_2 a_2+n_3 a_3) \wedge b_1} &+& a_2 \wedge (b_2 - n_2 b_1) &+& a_3 \wedge (b_3 - n_3 b_1)   &+& \sum\nolimits_{i=4}^g a_i \wedge b_i &= 0, \\
&\orange{(a_1+n_2 a_2) \wedge b_1}         &+& a_2 \wedge (b_2 - n_2 b_1) &+& a_3 \wedge b_3           &+& \sum\nolimits_{i=4}^g a_i \wedge b_i &= 0, \\
&\green{(a_1+n_3 a_3) \wedge b_1}         &+& a_2 \wedge b_2             &+& a_3 \wedge (b_3-n_3 a_1) &+& \sum\nolimits_{i=4}^g a_i \wedge b_i &= 0.
\end{alignat*}
We plug these into the terms of \eqref{eqn:thetalineartoprove3}.
Matching up terms of the same color,
we see that $\PresS{(a_1+n_2 a_2+n_3 a_3) \wedge b_1,\purple{(a_1+n_2 a_2+n_3 a_3) \wedge b_1}}$ equals
\begin{small}
\begin{align*}
&-\PresS{(a_1+n_2 a_2+n_3 a_3) \wedge b_1,a_2 \wedge (b_2 - n_2 b_1)} 
- \PresS{(a_1+n_2 a_2+n_3 a_3) \wedge b_1,a_3 \wedge (b_3 - n_3 b_1)} \\
&- \sum\nolimits_{i=4}^g \SPresS{(a_1+n_2 a_2+n_3 a_3) \wedge b_1,a_i \wedge b_i} \\
\equiv &-\PresS{(a_1+n_2 a_2) \wedge b_1,a_2 \wedge (b_2 - n_2 b_1)} -n_3 \SPresS{a_3 \wedge b_1,a_2 \wedge (b_2 - n_2 b_1)} \\
&\quad  -\PresS{(a_1+n_3 a_3) \wedge b_1,a_3 \wedge (b_3 - n_3 b_1)} - n_2 \SPresS{a_2 \wedge b_1,a_3 \wedge (b_3 - n_3 b_1)} \\
\equiv &-\PresS{(a_1+n_2 a_2) \wedge b_1,a_2 \wedge (b_2 - n_2 b_1)} - \PresS{(a_1+n_3 a_3) \wedge b_1,a_3 \wedge (b_3 - n_3 b_1)}
\end{align*}
\end{small}%
and $\PresS{(a_1+n_2 a_2) \wedge b_1,\orange{(a_1+n_2 a_2) \wedge b_1}}$ equals
\begin{small}
\begin{align*}
&-\PresS{(a_1+n_2 a_2) \wedge b_1,a_2 \wedge (b_2 - n_2 b_1)}-\SPresS{(a_1+n_2 a_2) \wedge b_1,a_3 \wedge b_3} \\
&-\sum\nolimits_{i=4}^g \SPresS{(a_1+n_2 a_2) \wedge b_1,a_i,b_i} \\
\equiv &-\PresS{(a_1+n_2 a_2) \wedge b_1,a_2 \wedge (b_2 - n_2 b_1)}
\end{align*}
\end{small}%
and $\PresS{(a_1+n_3 a_3) \wedge b_1,\green{(a_1+n_3 a_3) \wedge b_1}}$ equals
\begin{small}
\begin{align*}
&-\SPresS{(a_1+n_3 a_3) \wedge b_1,a_2 \wedge b_2} - \PresS{(a_1+n_3 a_3) \wedge b_1,a_3 \wedge (b_3-n_3 a_1)} \\
&-\sum\nolimits_{i=4}^g \SPresS{(a_1+n_3 a_3) \wedge b_1,a_i \wedge b_i} \\
\equiv &-\PresS{(a_1+n_3 a_3) \wedge b_1,a_3 \wedge (b_3-n_3 a_1)}.
\end{align*}
\end{small}%
Combining these three equalities gives \eqref{eqn:thetalineartoprove3}.
\end{proof}

We now return to the proof of the lemma.  Define
$\fK_g^s[a \wedge b,- \wedge b]$ to be the subspace of $\fK_g^s$ spanned by the $\ThetaS{a \wedge b,x \wedge b}$ as $x$
ranges over all elements of $\Span{a,b}^{\perp}$.  The linearization map $\Phi\colon \fK_g^s \rightarrow \Sym^2((\wedge^2 H)/\Q)$
takes $\fK_g^s[a \wedge b,- \wedge b]$ to the subspace
\begin{equation}
\label{eqn:thetalinearitytarget}
\Set{$(a \wedge b) \Cdot (h \wedge b)$}{$h \in \Span{a,b}^{\perp}_{\Q}$} \cong \Span{a,b}^{\perp}_{\Q}.
\end{equation}
To prove the lemma, it is enough to prove that 
the restriction of $\Phi$ to $\fK_g^s[a \wedge b,- \wedge b]$ is an isomorphism.

Let $\fK_g^s[a \wedge b,- \wedge b]_{\prim}$ be the subspace of
$\fK_g^s[a \wedge b,- \wedge b]$ spanned by $\ThetaS{a \wedge b,x \wedge b}$ with $x$ a primitive element of
$\Span{a,b}^{\perp}$.  The case $n_1 = n_2 = 1$ of the above claim implies that we can
use Theorem \ref{maintheorem:spstd} to see that $\Phi$ takes $\fK_g^s[a \wedge b,- \wedge b]_{\prim}$ isomorphically
to \eqref{eqn:thetalinearitytarget}.  

To complete the proof, we must prove that every element of
$\fK_g^s[a \wedge b,- \wedge b]$ equals an element of $\fK_g^s[a \wedge b,- \wedge b]_{\prim}$.
For this, consider a general $x \in \Span{a,b}^{\perp}$.  Write $x = n x'$ with $x'$ primitive.  Let
$y$ be such that $\{x',y\}$ is a partial basis of $(a \wedge b)^{\perp}$ with $\omega(x',y) = 0$.
The above claim then implies that
\[\ThetaS{a \wedge b,(x+y) \wedge b} = \ThetaS{a \wedge b,x \wedge b} + \ThetaS{a \wedge b,y \wedge b}.\]
On the other hand, since $x+y$ is primitive the fact that $\Phi$ takes $\fK_g^s[a \wedge b,- \wedge b]_{\prim}$ isomorphically
to \eqref{eqn:thetalinearitytarget} implies that
\[\ThetaS{a \wedge b,(x+y) \wedge b} = \ThetaS{a \wedge b,(n x'+y) \wedge b} = n \ThetaS{a \wedge b,x' \wedge b} + \ThetaS{a \wedge b,y \wedge b}.\]
Combining these two identities, we conclude that
\[\ThetaS{a \wedge b,x \wedge b} = n \ThetaS{a \wedge b,x' \wedge b} \in \fK_g^s[a \wedge b,- \wedge b]_{\prim},\]
as desired.
\end{proof}

\subsection{\texorpdfstring{$\Theta$}{Theta}-symmetry}
\label{section:thetasymmetry}

It is inconvenient to require the entries of
$\ThetaS{a \wedge b,x \wedge b}$ and $\ThetaS{a \wedge b,a \wedge y}$ to appear
in a definite order.  We therefore define that each of the following terms
equals $\ThetaS{a \wedge b,x \wedge b}$:
\begin{alignat*}{6}
&&\ThetaS{a \wedge b,x \wedge b},\ \ &-&\ThetaS{b \wedge a,x \wedge b},\ \ &-&\ThetaS{a \wedge b,b \wedge x},\ \ &\ThetaS{b \wedge a,b \wedge x},\\
&&\ThetaS{x \wedge b,a \wedge b},\ \ &-&\ThetaS{x \wedge b,b \wedge a},\ \ &-&\ThetaS{b \wedge x,a \wedge b},\ \ &\ThetaS{b \wedge x,b \wedge a}.
\end{alignat*}
Similarly, we define that each of the following terms equals $\ThetaS{a \wedge b,a \wedge y}$:
\begin{alignat*}{6}
&&\ThetaS{a \wedge b,a \wedge y},\ \ &-&\ThetaS{b \wedge a,a \wedge y},\ \ &-&\ThetaS{a \wedge b,y \wedge a},\ \ &\ThetaS{b \wedge a,y \wedge a},\\
&&\ThetaS{a \wedge y,a \wedge b},\ \ &-&\ThetaS{a \wedge y,b \wedge a},\ \ &-&\ThetaS{y \wedge a,a \wedge b},\ \ &\ThetaS{y \wedge a,b \wedge a}.
\end{alignat*}

\subsection{\texorpdfstring{$\Theta$}{Theta}-signs}
\label{section:thetasigns}

Lemma \ref{lemma:thetalinear} ($\Theta$-linearity) implies that
\[\ThetaS{a \wedge b,(-x) \wedge b} = -\ThetaS{a \wedge b,x \wedge b} \quad \text{and} \quad
\ThetaS{a \wedge b,a \wedge (-y)} = -\ThetaS{a \wedge b,a \wedge y}.\]
For a symplectic pair $a \wedge b$, there are three other symplectic pairs obtained by swapping $a$ and $b$ while
multiplying them by $\pm 1$, namely $(-b) \wedge a$ and $b \wedge (-a)$ and $(-a) \wedge (-b)$.  The following
shows that changing $a \wedge b$ to one of these changes $\ThetaS{a \wedge b,x \wedge b}$ and $\ThetaS{a \wedge b,y \wedge a}$
in the obvious way.  The statement uses the notation from \S \ref{section:thetasymmetry}:

\begin{lemma}
\label{lemma:thetasign}
Let $a \wedge b$ be a symplectic pair in $H_{\Z}$ and let $x,y \in \Span{a,b}^{\perp}$.  Then
\begin{small}
\begin{alignat*}{6}
&&\ThetaS{a \wedge (-b), x \wedge (-b)}    &= & \ThetaS{a \wedge b,x \wedge b}, \quad &&\ThetaS{a \wedge (-b),a \wedge y}       &= &-\ThetaS{a \wedge b,a \wedge y}, \\
&&\ThetaS{(-a) \wedge b,x \wedge b}        &= &-\ThetaS{a \wedge b,x \wedge b}, \quad &&\ThetaS{(-a) \wedge b,(-a) \wedge y}    &= & \ThetaS{a \wedge b,a \wedge y}, \\
&&\ThetaS{(-a) \wedge (-b),x \wedge (-b)}  &= &-\ThetaS{a \wedge b,x \wedge b}, \quad &&\ThetaS{(-a) \wedge (-b),(-a) \wedge y} &= &-\ThetaS{a \wedge b,a \wedge y}.
\end{alignat*}
\end{small}%
\end{lemma}
\begin{proof}
These are all proved the same way, so we will give the details for $\ThetaS{(-a) \wedge b,x \wedge b} =-\ThetaS{a \wedge b,x \wedge b}$
and leave the others to the reader.\footnote{We chose this one because it is slightly harder than the other cases.}  
Here $(-a) \wedge b$ is not a symplectic pair, but $b \wedge (-a)$ is a symplectic pair.  Using
the notation from \S \ref{section:thetasymmetry}, we interpret $\ThetaS{(-a) \wedge b,x \wedge b}$ as\footnote{There is no
sign change here since both $(-a) \wedge b$ and $x \wedge b$ are flipped and each flip causes a sign change.}
$\ThetaS{b \wedge (-a),b \wedge x}$, so our goal is to prove that $\ThetaS{b \wedge (-a),b \wedge x} = -\ThetaS{a \wedge b,x \wedge b}$.
By definition, $\ThetaS{b \wedge (-a),b \wedge x}$ equals
\begin{align*}
&\PresS{b \wedge (-a+x),b \wedge (-a+x)} - \PresS{b \wedge (-a),b \wedge (-a)} - \PresS{b \wedge x,b \wedge x} \\
&=\PresS{(a-x) \wedge b,(a-x) \wedge b} - \PresS{a \wedge b,a \wedge b} - \PresS{(-x) \wedge b,(-x) \wedge b}.
\end{align*}
This last expression equals $\ThetaS{a \wedge b,-x \wedge b}$, which by Lemma \ref{lemma:thetalinear} ($\Theta$-linearity) equals
$-\ThetaS{a \wedge b,x \wedge b}$.
\end{proof}

\subsection{The set \texorpdfstring{$S_2$}{S2}}

We now return to constructing $S_2$.  Recall that
\[T_2 = \Set{$(a_i \wedge b_i) \Cdot (x \wedge b_i)$, $(a_i \wedge b_i) \Cdot (a_i \wedge y)$}{$1 \leq i \leq g$, $x,y \in \cB \setminus \{a_i,b_i\}$}.\]
Define
\[S_2 = \Set{$\SThetaS{a_i \wedge b_i,x \wedge b_i}$, $\SThetaS{a_i \wedge b_i,a_i \wedge y}$}{$1 \leq i \leq g$, $x,y \in \cB \setminus \{a_i,b_i\}$}.\]
Like we did here, we will write elements of $\Span{S_2}$ in purple.  For example, using
Lemma \ref{lemma:thetalinear} ($\Theta$-linearity), for $1 \leq i \leq g$ and $z \in \Span{a_i,b_i}^{\perp}$ we have
elements $\SThetaS{a_i \wedge b_i,z \in b_i}$ and $\SThetaS{a_i \wedge b_i,a_i \wedge z}$ in $\Span{S_2}$.
By construction, the linearization map $\Phi$ takes $S_2$ bijectively to $T_2$.  Even better:

\begin{lemma}
\label{lemma:symspans2}
The linearization map $\Phi$ takes $\Span{S_1,S_2}$ isomorphically to $\Span{T_1,T_2}$.
\end{lemma}
\begin{proof}
Recall that in Lemma \ref{lemma:symspans1} we proved that $\Phi$ takes $\Span{S_1}$ isomorphically
onto $\Span{T_1}$.  Part of the proof of that lemma was that $\Phi$ takes $S_1$ bijectively to
$T_1$.  It follows that $\Phi$ takes $S_1 \cup S_2$ bijectively to $T_1 \cup T_2$.  What we must
prove is that all relations between elements of $T_1 \cup T_2$ lift to relations between $S_1 \cup S_2$.

We constructed all the relations between elements of $T = T_1 \cup \cdots \cup T_4$ in
\S \ref{section:genrelsym2}, and in fact all of them only involve elements of $T_1 \cup T_2$.  Some
only involve elements of $T_1$, and as we observed in the proof of Lemma \ref{lemma:symspans1} these all lift
to relations between elements of $S_1$.  The remaining relations are of the form
\[\text{$\sum\nolimits_{i=1}^g (a_i \wedge b_i) \Cdot (x \wedge y)$ with $x,y \in \cB$ with $x \prec y$ and $\omega(x,y)=0$}.\]
Lemma \ref{lemma:liftsymr2} below proves that these do indeed lift to relations between elements of
$S_1 \cup S_2$.
\end{proof}

The above proof used the following, which for later use we state in more generality than we need at the moment:

\begin{lemma}[$\Theta$-symplectic basis]
\label{lemma:liftsymr2}
Let $\{x_1,y_1,\ldots,x_g,y_g\}$ be a symplectic basis for $H_{\Z}$, let $1 \leq n < m \leq g$, and
let $z \in \{x_n,y_n\}$ and $w \in \{x_m,y_m\}$.  Then
\[\ThetaS{x_n \wedge y_n,z \wedge w} + \ThetaS{x_m \wedge y_m,z \wedge w} + \sum_{\substack{1 \leq i \leq g \\ i \neq n,m}} \PresS{x_i \wedge y_i,z \wedge w} = 0.\]
\end{lemma}
\begin{proof}
Whether this holds is invariant under the action of $\Sp_{2g}(\Z)$ on $\fK_g^s$, so applying an appropriate
element of $\Sp_{2g}(\Z)$ we can assume that the given symplectic basis is our fixed symplectic basis
$\cB = \{a_1,b_1,\ldots,a_g,b_g\}$.  Moreover, recalling the subgroup $\SymSp_{g}$ from \S \ref{section:spgen}
we can apply an appropriate element of $\SymSp_{g}$ and ensure that $n=1$ and $m=2$.  Our desired relation
is thus
\begin{equation}
\label{eqn:liftsymr2toprove1}
\SThetaS{a_1 \wedge b_1,z \wedge w} + \SThetaS{a_2 \wedge b_2,z \wedge w} + \sum\nolimits_{i=3}^g \SPresS{a_i \wedge b_i,z \wedge 2} = 0.
\end{equation}
Finally, since $\Phi$ restricted to $\Span{S_1}$ is injective (Lemma \ref{lemma:symspans1}) and
$\Phi$ takes \eqref{eqn:liftsymr2toprove1} to a true relation in $\Sym^2((\wedge^2 H)/\Q)$, it is enough
to prove that \eqref{eqn:liftsymr2toprove1} holds modulo $\Span{S_1}$.  In other words, letting $\equiv$
denote equality modulo $\Span{S_1}$ we must prove that
\begin{equation}
\label{eqn:liftsymr2toprove2}
\SThetaS{a_1 \wedge b_1,z \wedge w} + \SThetaS{a_2 \wedge b_2,z \wedge w} \equiv 0.
\end{equation}
For concreteness, we will prove this for $z = a_1$ and $w = b_2$.  The other cases are similar.
Using Lemma \ref{lemma:thetalinear} ($\Theta$-linearity), the relation \eqref{eqn:liftsymr2toprove2} is equivalent to
\begin{equation}
\label{eqn:liftsymrel1}
\SThetaS{a_1 \wedge b_1, a_1 \wedge b_2} - \SThetaS{a_2 \wedge b_2, -a_1 \wedge b_2} \equiv 0.
\end{equation}
By definition, we have
\begin{align*}
2 \SThetaS{a_1 \wedge b_1, a_1 \wedge b_2}  = &\PresS{a_1 \wedge (b_1+b_2),a_1 \wedge (b_1+b_2)} - \SPresS{a_1 \wedge b_1,a_1 \wedge b_1} \\
                                              &- \SPresS{a_1 \wedge b_2,a_1 \wedge b_2} \equiv \PresS{a_1 \wedge (b_1+b_2),a_1 \wedge (b_1+b_2)},\\
2 \SThetaS{a_2 \wedge b_2, -a_1 \wedge b_2} = &\PresS{(a_2-a_1) \wedge b_2,(a_2-a_1) \wedge b_2} - \SPresS{a_2 \wedge b_2,a_2 \wedge b_2} \\
                                              &- \SPresS{a_1 \wedge b_2,a_1 \wedge b_2} \equiv \PresS{(a_2-a_1) \wedge b_2,(a_2-a_1) \wedge b_2}.
\end{align*}
The relation \eqref{eqn:liftsymrel1} is thus equivalent to
\begin{equation}
\label{eqn:liftsymrel2} 
\PresS{a_1 \wedge (b_1+b_2),a_1 \wedge (b_1+b_2)} - \PresS{(a_2-a_1) \wedge b_2,(a_2-a_1) \wedge b_2} \equiv 0.
\end{equation}
In $(\wedge^2 H)/\Q$, we have
\[a_1 \wedge (b_1+b_2) + (a_2-a_1) \wedge b_2 + \sum\nolimits_{i=3}^g a_i \wedge b_i = 0.\]
This implies that
\begin{align*}
\PresS{a_1 \wedge (b_1+b_2),a_1 \wedge (b_1+b_2)} = &-\PresS{a_1 \wedge (b_1+b_2),(a_2-a_1) \wedge b_2} \\
                                                    &- \sum\nolimits_{i=3}^g \SPresS{a_1 \wedge (b_1+b_2),a_i \wedge b_i} \\
                                                  \equiv &-\PresS{a_1 \wedge (b_1+b_2),(a_2-a_1) \wedge b_2}.
\end{align*}
Plugging this into \eqref{eqn:liftsymrel2}, we see that our desired relation is equivalent to showing
that the following is equivalent to $0$:
\begin{align*}
 &\PresS{a_1 \wedge (b_1+b_2),(a_2-a_1) \wedge b_2} + \PresS{(a_2-a_1) \wedge b_2,(a_2-a_1) \wedge b_2} \\
=&\PresS{a_1 \wedge (b_1+b_2) + (a_2-a_1) \wedge b_2, (a_2-a_1) \wedge b_2} \\
=&\PresS{a_1 \wedge b_1+a_2 \wedge b_2,(a_2-a_1) \wedge b_2} =-\sum\nolimits_{i=3}^g \SPresS{a_i \wedge b_i,(a_2-a_1) \wedge b_2} \equiv 0.\qedhere
\end{align*}
\end{proof}

\subsection{Additional bilinearity relations}

We close this section by proving some additional relations between the $\Theta$-elements.

\begin{lemma}[$\Theta$-bilinearity I]
\label{lemma:thetabilinear1}
Let $a \wedge b$ be a symplectic pair in $H_{\Z}$ and $z \in \Span{a,b}^{\perp}$.  Then:
\begin{itemize}
\item for $x \in \Span{a,b,z}^{\perp}$ we have
$\ThetaS{(a+z) \wedge b,x \wedge b} = \ThetaS{a \wedge b,x \wedge b} + \PresS{z \wedge b,x \wedge b}$.
\item for $y \in \Span{a,b,z}^{\perp}$ we have
$\ThetaS{a \wedge (b+z),a \wedge y} = \ThetaS{a \wedge b,a \wedge y} + \PresS{a \wedge z,a \wedge y}$.
\end{itemize}
\end{lemma}
\begin{proof}
Both are proved the same way, so we prove the first.  By Lemma \ref{lemma:thetaexpansion1} ($\Theta$-expansion I), 
the element
$\PresS{(a+z+x) \wedge b,(a+z+x) \wedge b}$ equals
\begin{align*}
&\PresS{(a+z) \wedge b,(a+z) \wedge b} + 2 \ThetaS{(a+z) \wedge b,x \wedge b} + \PresS{x \wedge b,x \wedge b} \\
&=\PresS{a \wedge b,a \wedge b} + 2\ThetaS{a \wedge b,z \wedge b} + \PresS{z \wedge b,z \wedge b} \\
&\quad+2 \ThetaS{(a+z) \wedge b,x \wedge b} + \PresS{x \wedge b,x \wedge b}.
\end{align*}
On the other hand, it also equals
\begin{align*}
&\PresS{a \wedge b,a \wedge b} + 2 \ThetaS{a \wedge b,(z+x) \wedge b} + \PresS{(z+x) \wedge b,(z+x) \wedge b} \\
&=\PresS{a \wedge b,a \wedge b} + 2 \ThetaS{a \wedge b,z \wedge b} + 2 \ThetaS{a \wedge b,x \wedge b} \\
&\quad+\PresS{z \wedge b,z \wedge b} + 2\PresS{z \wedge b,x \wedge b} + \PresS{x \wedge b,x \wedge b}.
\end{align*}
Here the equality uses Lemma \ref{lemma:thetalinear} ($\Theta$-linearity).  The above
two displays are thus equal, and the result follows.
\end{proof}

Lemma \ref{lemma:thetabilinear1} allows many standard generators of $\fK_g^s$ to be written as the sum
of two $\Theta$-elements:  

\begin{lemma}[$\Theta$-expansion II]
\label{lemma:thetaexpansion2}
Let $a \wedge b$ and $a' \wedge b'$ be symplectic pairs in $H_{\Z}$ such that 
$\Span{a,b}$ and $\Span{a',b'}$ are orthogonal.  Then
\begin{small}
\begin{alignat*}{5}
&\PresS{(a+a') \wedge (b-b'),x \wedge (b-b')}\  &=\ &\ThetaS{a \wedge (b-b'),x \wedge (b-b')}\  &+\ &\ThetaS{a' \wedge (b-b'),x \wedge (b-b')},\\
&\PresS{(a+a') \wedge (b-b'),(a+a') \wedge y}\  &=\ &\ThetaS{(a+a') \wedge b,(a+a') \wedge y}\  &-\ &\ThetaS{(a+a') \wedge b',(a+a') \wedge y},\\
&\PresS{(a+b') \wedge (b+a'),x' \wedge (b+a')}\ &=\ &\ThetaS{a \wedge (b+a'),x' \wedge (b+a')}\ &+\ &\ThetaS{b' \wedge (b+a'),x' \wedge (b+a')},\\
&\PresS{(a+b') \wedge (b+a'),(a+b') \wedge y'}\ &=\ &\ThetaS{(a+b') \wedge b,(a+b') \wedge y'}\ &+\ &\ThetaS{(a+b') \wedge a',(a+b') \wedge y'}.
\end{alignat*}
\end{small}%
for $x \in \Span{a,a',b-b'}^{\perp}$ and $y \in \Span{b,b',a+a'}^{\perp}$ and $x' \in \Span{a,b',b+a'}$
and $y' \in \Span{b,a',a+b'}$.
\end{lemma}
\begin{proof}
All are proved the same way, so we will prove the first. 
Lemma \ref{lemma:thetabilinear1} ($\Theta$-bilinearity I) implies that
\[\ThetaS{(-a') \wedge (b-b'),x \wedge (b-b')} + \PresS{(a+a') \wedge (b-b'), x \wedge (b-b')}\]
equals
\[\ThetaS{(-a' + (a+a')) \wedge (b-b'),x \wedge (b-b')} = \ThetaS{a \wedge (b-b'),x \wedge (b-b')}.\]
Rearranging this, we see that
\[\PresS{(a+a') \wedge (b-b'), x \wedge (b-b')} = \ThetaS{a \wedge (b-b'),x \wedge (b-b')} + \ThetaS{a' \wedge (b-b'),x \wedge (b-b')}.\qedhere\]
\end{proof}

\begin{lemma}[$\Theta$-bilinearity II]
\label{lemma:thetabilinear2}
Let $a \wedge b$ be a symplectic pair in $H_{\Z}$ and $n \in \Z$.  Then:
\begin{itemize}
\item for $x \in \Span{a,b}^{\perp}$ we have
\begin{align*}
\ThetaS{(a+nb) \wedge b,x \wedge b} &= \ThetaS{a \wedge b,x \wedge b},\\
\ThetaS{a \wedge (b+n a),x \wedge (b+na)} &= \ThetaS{a \wedge b, x \wedge b} + n \ThetaS{a \wedge b,x \wedge a}.
\end{align*}
\item for $y \in \Span{a,b}^{\perp}$ we have
\begin{align*}
\ThetaS{a \wedge (b+na),a \wedge y} &= \ThetaS{a \wedge b,a \wedge y},\\
\ThetaS{(a+nb) \wedge b,(a+nb) \wedge y} &= \ThetaS{a \wedge b, a \wedge y} + n \ThetaS{a \wedge b,b \wedge y}.
\end{align*}
\end{itemize}
\end{lemma}
\begin{proof}
The two bullet points are proved the same way, so we will prove the second.  Observe first that
$2\ThetaS{a \wedge (b+na),a \wedge y}$ equals
\begin{align*}
&\PresS{a \wedge (b+na+y),a \wedge (b+na+y)} - \PresS{a \wedge (b+na),a \wedge (b+na)} - \PresS{a \wedge y,a \wedge y}\\
&=\PresS{a \wedge (b+y),a \wedge (b+y)} - \PresS{a \wedge b,a \wedge b} - \PresS{a \wedge y,a \wedge y},
\end{align*}
which equals $2\ThetaS{a \wedge b,a \wedge y}$.  This gives the first equation.  For the second, by 
Lemma \ref{lemma:thetalinear} ($\Theta$-linearity) it is enough to prove it for $y$ primitive.
We can then find a symplectic basis $\{x_1,y_1,\ldots,x_g,y_g\}$ for $H_{\Z}$ such that $x_1 = a$ and $y_1 = b$
and $y_2 = y$.  Our goal is to prove that
\[\ThetaS{(x_1+ny_1) \wedge y_1,(x_1+ny_1) \wedge y_2} = \ThetaS{x_1 \wedge y_1, x_1 \wedge y_2} + n \ThetaS{x_1 \wedge y_1,y_1 \wedge y_2}.\]
Applying Lemma \ref{lemma:liftsymr2} ($\Theta$-symplectic basis) to the alternate symplectic basis $\cB' = \{x_1+ny_1,y_1,x_2,y_2,\ldots,x_g,y_g\}$, we see that
\begin{align*}
0=&\ThetaS{(x_1+ny_1) \wedge y_1, (x_1+ny_1) \wedge y_2} \\
  &\quad+ \ThetaS{x_2 \wedge y_2, (x_1+ny_1) \wedge y_2} + \sum\nolimits_{i=3}^g \PresS{x_i \wedge y_i,(x_1+ny_1) \wedge y_2}
\end{align*}
From this, we see that
$\ThetaS{(x_1+ny_1) \wedge y_1, (x_1+ny_1) \wedge y_2}$ equals
\begin{align*}
&-\ThetaS{x_2 \wedge y_2, (x_1+ny_1) \wedge y_2} - \sum\nolimits_{i=3}^g \PresS{x_i \wedge y_i,(x_1+ny_1) \wedge y_2} \\
&=-\left(\ThetaS{x_2 \wedge y_2, x_1 \wedge y_2} - \sum\nolimits_{i=3}^g \PresS{x_i \wedge y_i,x_1 \wedge y_2}\right)\\
&\quad-n\left(\ThetaS{x_2 \wedge y_2, y_1 \wedge y_2} - \sum\nolimits_{i=3}^g \PresS{x_i \wedge y_i,y_1 \wedge y_2}\right).
\end{align*}
The equality here uses Lemma \ref{lemma:thetalinear} ($\Theta$-linearity).  Again using Lemma \ref{lemma:liftsymr2} ($\Theta$-symplectic basis),
we recognize this as being $\ThetaS{x_1 \wedge y_1, x_1 \wedge y_2} + n \ThetaS{x_1 \wedge y_1,y_1 \wedge y_2}$, as desired.
\end{proof}

\section{Symmetric kernel, symmetric version III: \texorpdfstring{$S_3$}{S3} and the \texorpdfstring{$\Lambda$}{Lambda}-elements}
\label{section:presentationsym3}

We continue using all the notation from \S \ref{section:presentationsym1} -- \S \ref{section:presentationsym2}.
This section constructs the set $S_3$ that lifts $T_3$.  It consists of what are called
$\Lambda$-elements of $\fK_g^s$, and the first part of this section constructs
these in more generality than is needed for $S_3$ alone.

\subsection{Definition}
Let $a \wedge b$ be a symplectic pair in $\wedge^2 H_{\Z}$ and let $x,y \in \Span{a,b}^{\perp}$ satisfy $\omega(x,y)=0$.  The
$\Lambda$-element $\LambdaS{a \wedge y,x \wedge b}$ is an element
of $\fK_g^s$ that is taken by $\Phi$ to
\[(a \wedge y) \Cdot (x \wedge b) \in \Sym^2((\wedge^2 H)/\Q).\]
To find it, note that
\begin{align*}
\left(a \wedge (b+y)\right) \Cdot \left(x \wedge (b+y)\right) &= \left(a \wedge (b+y)\right) \Cdot (x \wedge y) 
+ (a \wedge b) \Cdot (x \wedge b) + (a \wedge y) \Cdot (x \wedge b),\\
\left((a+x) \wedge b\right) \Cdot \left((a+x) \wedge y\right) &= \left((a+x) \wedge b\right) \Cdot (x \wedge y)
+ (a \wedge b) \Cdot (a \wedge y) + (x \wedge b) \Cdot (a \wedge y).
\end{align*}
This suggests two possible elements of $\fK_g^s$ projecting to $(a \wedge y) \Cdot (x \wedge b)$:

\begin{definition}
\label{definition:lambda}
For a symplectic pair $a \wedge b$ in $H_{\Z}$ and $x,y \in \Span{a,b}^{\perp}$ with $\omega(x,y)=0$, define:
\begin{align*}
\LambdaIS{a \wedge y, x \wedge b}  &= \ThetaS{a \wedge (b+y),x \wedge (b+y)} - \ThetaS{a \wedge b,x \wedge b} - \PresS{a \wedge (b+y),x \wedge y},\\
\LambdaIIS{a \wedge y,x \wedge b}  &= \ThetaS{(a+x) \wedge b,(a+x) \wedge y} - \ThetaS{a \wedge b,a \wedge y} - \PresS{(a+x) \wedge b,x \wedge y}.\qedhere
\end{align*}
\end{definition}

See Remark \ref{remark:lambdaabuse} below for a caveat about this notation.  By construction, we have
\[\Phi(\LambdaIS{a \wedge y,x \wedge b}) = \Phi(\LambdaIIS{a \wedge y,x \wedge b}) = (a \wedge y) \Cdot (x \wedge b).\]
Since we are trying to prove that $\Phi$ is an isomorphism, we must prove that these are actually the same element:

\begin{lemma}
\label{lemma:lambdawelldefined}
Let $a \wedge b$ be a symplectic pair in $H_{\Z}$ and let $x,y \in \Span{a,b}^{\perp}$ satisfy $\omega(x,y) = 0$.
Then $\LambdaIS{a \wedge y,x \wedge b} = \LambdaIIS{a \wedge y,x \wedge b}$.
\end{lemma}
\begin{proof}
We must prove that
\begin{align*}
&\ThetaS{a \wedge (b+y),x \wedge (b+y)} - \ThetaS{a \wedge b,x \wedge b} - \PresS{a \wedge (b+y),x \wedge y},\\
=& \ThetaS{(a+x) \wedge b,(a+x) \wedge y} - \ThetaS{a \wedge b,a \wedge y} - \PresS{(a+x) \wedge b,x \wedge y}.
\end{align*}
Since
\begin{align*}
\PresS{a \wedge (b+y),x \wedge y} &= \PresS{a \wedge b,x \wedge y} + \PresS{a \wedge y,x \wedge y},\\
\PresS{(a+x) \wedge b,x \wedge y} &= \PresS{a \wedge b,x \wedge y} + \PresS{x \wedge b,x \wedge y},
\end{align*}
we can rearrange this and see that it is equivalent to prove that
\begin{equation}
\label{eqn:lambdadefined1}
\ThetaS{a \wedge (b+y),x \wedge (b+y)} - \ThetaS{(a+x) \wedge b,(a+x) \wedge y}
\end{equation}
equals
\begin{equation}
\label{eqn:lambdadefined2}
\ThetaS{a \wedge b,x \wedge b} - \ThetaS{a \wedge b,a \wedge y} + \PresS{a \wedge y,x \wedge y} - \PresS{x \wedge b,x \wedge y}.
\end{equation}
Applying Lemma \ref{lemma:thetaexpansion1} ($\Theta$-expansion I) twice, the element $\PresS{(a+x) \wedge (b+y),(a+x) \wedge (b+y)}$ equals
\begin{align*}
&\PresS{a \wedge (b+y),a \wedge (b+y)} + 2\ThetaS{a \wedge (b+y),x \wedge (b+y)} + \PresS{x \wedge (b+y),x \wedge (b+y)} \\
&=\PresS{a \wedge b,a \wedge b} +2 \ThetaS{a \wedge b,a \wedge y} + \PresS{a \wedge y,a \wedge y} \\
&\quad+ 2\ThetaS{a \wedge (b+y),x \wedge (b+y)} + \PresS{x \wedge (b+y),x \wedge (b+y)}.
\end{align*}
Applying Lemma \ref{lemma:thetaexpansion1} ($\Theta$-expansion I) twice again but in a different order,
the same element $\PresS{(a+x) \wedge (b+y),(a+x) \wedge (b+y)}$ also equals
\begin{align*}
&\PresS{(a+x) \wedge b,(a+x) \wedge b} + 2\ThetaS{(a+x) \wedge b, (a+x) \wedge y} + \PresS{(a+x) \wedge y,(a+x) \wedge y)} \\
&=\PresS{a \wedge b,a \wedge b} + 2\ThetaS{a \wedge b,x \wedge b} + \PresS{x \wedge b,x \wedge b} \\
&\quad+ 2\ThetaS{(a+x) \wedge b, (a+x) \wedge y} + \PresS{(a+x) \wedge y,(a+x) \wedge y)}.
\end{align*}
Equating the previous two displays and rearranging terms, we deduce that $2$ times \eqref{eqn:lambdadefined1} equals
\begin{align*}
&2\ThetaS{a \wedge b,x \wedge b} + \PresS{x \wedge b,x \wedge b} + \PresS{(a+x) \wedge y,(a+x) \wedge y)} \\
&\quad - 2 \ThetaS{a \wedge b,a \wedge y} - \PresS{a \wedge y,a \wedge y} - \PresS{x \wedge (b+y),x \wedge (b+y)},
\end{align*}
which using the usual bilinearity relations in $\fK_g^s$ equals
\begin{align*}
&2\ThetaS{a \wedge b,x \wedge b} - 2 \ThetaS{a \wedge b,a \wedge y} + \PresS{x \wedge b,x \wedge b} - \PresS{a \wedge y,a \wedge y} \\
&\quad +(\PresS{a \wedge y,a \wedge y} + 2\PresS{a \wedge y,x \wedge y} + \PresS{x \wedge y,x \wedge y}) \\
&\quad -(\PresS{x \wedge b,x \wedge b} + 2\PresS{x \wedge b,x \wedge y} + \PresS{x \wedge y,x \wedge y}),
\end{align*}
which after canceling terms equals $2$ times \eqref{eqn:lambdadefined2}.
\end{proof}

In light of this lemma, we will denote the common value of $\LambdaIS{a \wedge y,x \wedge b}$ and
$\LambdaIIS{a \wedge y,x \wedge b}$ by $\LambdaS{a \wedge y,x \wedge b}$.

\begin{remark}
\label{remark:lambdaabuse}
Just like for the $\Theta$-elements (cf.\ Remark \ref{remark:thetaabuse}), the elements
$\LambdaIS{a \wedge y,x \wedge b}$ and $\LambdaIIS{a \wedge y,x \wedge b}$ and $\LambdaS{a \wedge y,x \wedge b}$
depend on the ordered tuple $(a,y,x,b)$, not
on $a \wedge y$ and $x \wedge b$.
\end{remark}

\subsection{\texorpdfstring{$\Lambda$}{Lambda}-expansion I}

The following is an important way that $\Lambda$-elements appear in our calculations:

\begin{lemma}[$\Lambda$-expansion I]
\label{lemma:lambdaexpansion1}
Let $a \wedge b$ be a symplectic pair and let $x,y \in \Span{a,b}^{\perp}$ satisfy $\omega(x,y) = 0$.  Then
\begin{align*}
\ThetaS{a \wedge (b+y),x \wedge (b+y)} &= \LambdaS{a \wedge y, x \wedge b} + \ThetaS{a \wedge b,x \wedge b} + \PresS{a \wedge (b+y),x \wedge y},\\
\ThetaS{(a+x) \wedge b,(a+x) \wedge y} &= \LambdaS{a \wedge y,x \wedge b} + \ThetaS{a \wedge b,a \wedge y} + \PresS{(a+x) \wedge b,x \wedge y}.
\end{align*}
\end{lemma}
\begin{proof}
Immediate from Definition \ref{definition:lambda}.
\end{proof}

\subsection{\texorpdfstring{$\Lambda$}{Lambda}-linearity}

The following says that $\LambdaS{a \wedge y,x \wedge b}$ is linear in both $x$ and $y$:

\begin{lemma}[$\Lambda$-linearity]
\label{lemma:lambdalinear}
Let $a \wedge b$ be a symplectic pair in $H_{\Z}$.  Then:
\begin{itemize}
\item For $x,y_1,y_2 \in \Span{a,b}^{\perp}$ with $\omega(x,y_1) = \omega(x,y_2) = 0$ and $\lambda_1,\lambda_2 \in \Z$, we have
\[\LambdaS{a \wedge (\lambda_1 y_1 + \lambda_2 y_2),x \wedge b} = \lambda_1 \LambdaS{a \wedge y_1, x \wedge b} + \lambda_2 \LambdaS{a \wedge y_2,x \wedge b}.\]
\item For $x_1,x_2,y \in \Span{a,b}^{\perp}$ with $\omega(x_1,y) = \omega(x_2,y) = 0$ and $\lambda_1,\lambda_2 \in \Z$, we have
\[\LambdaS{a \wedge y,(\lambda_1 x_1 + \lambda_2 x_2) \wedge b} = \lambda_1 \LambdaS{a \wedge y, x_1 \wedge b} + \lambda_2 \LambdaS{a \wedge y,x_2 \wedge b}.\]
\end{itemize}
\end{lemma}

In fact, we will prove something more general.  Let $a \wedge b$ be a symplectic pair in $H_{\Z}$.  Define
$\fK_g^{s,\Lambda}[a \wedge -,- \wedge b]$ to be the subspace of $\fK_g^s$ spanned by $\LambdaS{a \wedge y,x \wedge b}$ as $x$ and $y$ range
over elements of $\Span{a,b}^{\perp}$ satisfying $\omega(x,y) = 0$.  The linearization map
$\Phi\colon \fK_g^s \rightarrow \Sym^2((\wedge^2 H)/\Q)$ takes $\fK_g^{s,\Lambda}[a \wedge -,- \wedge b]$ into
\[\Span{\text{$(a \wedge y) \Cdot (x \wedge b)$ $|$ $x,y \in \Span{a,b}^{\perp}$}} \cong \left(\Span{a,b}^{\perp}_{\Q}\right)^{\otimes 2}.\]
This isomorphism takes $(a \wedge y) \Cdot (x \wedge b)$ to $y \otimes x$.  The image is in the kernel of map
\[\left(\Span{a,b}^{\perp}_{\Q}\right)^{\otimes 2} \longrightarrow \Q\]
induced by $\omega$, which we denote $\cZ(\Span{a,b}^{\perp}_{\Q})$ (c.f.\ \S \ref{section:nonsymmetriczero}).  We
will prove the following, which strengthens Lemma \ref{lemma:lambdalinear}:

\begin{lemma}[strong $\Lambda$-linearity]
\label{lemma:lambdalinear2}
Let $a \wedge b$ and $\cZ(\Span{a,b}^{\perp}_{\Q})$ be as above.  Then the linearization map
\[\Phi\colon \fK_g^{s,\Lambda}[a \wedge -,- \wedge b] \longrightarrow \cZ(\Span{a,b}^{\perp}_{\Q}).\]
is an isomorphism.
\end{lemma}
\begin{proof}
Theorem \ref{theorem:spkernel} gives a presentation for $\cZ(\Span{a,b}^{\perp}_{\Q})$.  In light of this presentation,
it is enough to prove the following two special cases of
Lemma \ref{lemma:lambdalinear}:
\begin{itemize}
\item For $x \in \Span{a,b}^{\perp}$ and a partial basis $\{y_1,y_2\}$ of $\Span{a,b,x}^{\perp}$, we have
\[\LambdaS{a \wedge (y_1 + y_2) ,x \wedge b} = \LambdaS{a \wedge y_1, x \wedge b} + \LambdaS{a \wedge y_2,x \wedge b}.\]
\item For $y \in \Span{a,b}^{\perp}$ and a partial basis $\{x_1,x_2\}$ of $\Span{a,b,y}^{\perp}$, we have
\[\LambdaS{a \wedge y,(x_1 + x_2) \wedge b} = \LambdaS{a \wedge y, x_1 \wedge b} + \LambdaS{a \wedge y,x_2 \wedge b}.\]
\end{itemize}
For the first bullet point, $\LambdaS{a \wedge (y_1+y_2),x \wedge b} = \LambdaIIS{a \wedge (y_1+y_2),x \wedge b}$ equals
\[\ThetaS{(a+x) \wedge b,(a+x) \wedge (y_1+y_2)} - \ThetaS{a \wedge b,a \wedge (y_1+y_2)} - \PresS{(a+x) \wedge b,x \wedge (y_1+y_2)}.\]
Using Lemma \ref{lemma:thetalinear} ($\Theta$-linearity), all three terms are linear in the $y_i$:
\begin{align*}
\ThetaS{(a+x) \wedge b,(a+x) \wedge (y_1+y_2)} = &\ThetaS{(a+x) \wedge b,(a+x) \wedge y_1} \\
                                                 &+ \ThetaS{(a+x) \wedge b,(a+x) \wedge y_2},\\
\ThetaS{a \wedge b,a \wedge (y_1+y_2)}         = &\ThetaS{a \wedge b,a \wedge y_1} + \ThetaS{a \wedge b,a \wedge y_2},\\
\PresS{(a+x) \wedge b,x \wedge (y_1+y_2)}      = &\PresS{(a+x) \wedge b,x \wedge y_1} + \PresS{(a+x) \wedge b,x \wedge y_2}.
\end{align*}
The first bullet point follows.  The second bullet point is proved the same way, but using $\LambdaIS{a \wedge -,- \wedge b}$
instead of $\LambdaIIS{a \wedge -,- \wedge b}$.
\end{proof}

\subsection{\texorpdfstring{$\Lambda$}{Lambda}-symmetry}
\label{section:lambdasymmetry}

It is inconvenient to require the entries of
$\LambdaS{a \wedge y, x \wedge b}$ to appear
in a definite order.  We therefore define that each of the following terms
equals $\LambdaS{a \wedge y, x \wedge b}$:
\begin{alignat*}{6}
&&\LambdaS{a \wedge y,x \wedge b},\ \ &-&\LambdaS{y \wedge a,x \wedge b},\ \ &-&\LambdaS{a \wedge y,b \wedge x},\ \ &\LambdaS{y \wedge a,b \wedge x},\\
&&\LambdaS{x \wedge b,a \wedge y},\ \ &-&\LambdaS{x \wedge b,y \wedge a},\ \ &-&\LambdaS{b \wedge x,a \wedge y},\ \ &\LambdaS{b \wedge x,y \wedge a}.
\end{alignat*}

\subsection{\texorpdfstring{$\Lambda$}{Lambda}-signs}
\label{section:lambdasigns}

Lemma \ref{lemma:lambdalinear} ($\Lambda$-linearity) implies that
\[\LambdaS{a \wedge (-y),x \wedge b} = -\LambdaS{a \wedge y,x \wedge b} \quad \text{and} \quad
\LambdaS{a \wedge y,(-x) \wedge b} = -\LambdaS{a \wedge y,x \wedge b}.\]
For a symplectic pair $a \wedge b$, there are three other symplectic pairs obtained by swapping $a$ and $b$ while
multiplying them by $\pm 1$, namely $(-b) \wedge a$ and $b \wedge (-a)$ and $(-a) \wedge (-b)$.  The following
shows that changing $a \wedge b$ to one of these changes $\LambdaS{a \wedge y,x \wedge b}$
in the obvious way.  The statement uses the notation from \S \ref{section:lambdasymmetry}:

\begin{lemma}
\label{lemma:lambdasign}
Let $a \wedge b$ be a symplectic pair in $H_{\Z}$ and let $x,y \in \Span{a,b}^{\perp}$ satisfy $\omega(x,y) = 0$.  Then
\begin{alignat*}{6}
&&\LambdaS{a \wedge y,x \wedge (-b)}    &= &-\LambdaS{a \wedge y,x \wedge b}, \\
&&\LambdaS{(-a) \wedge y,x \wedge b}    &= &-\LambdaS{a \wedge y,x \wedge b}, \\
&&\LambdaS{(-a) \wedge y,x \wedge (-b)} &= & \LambdaS{a \wedge y,x \wedge b}.
\end{alignat*}
\end{lemma}
\begin{proof}
These are all proved the same way, so we will give the details for $\LambdaS{a \wedge y,x \wedge (-b)} = -\LambdaS{a \wedge y,x \wedge b}$
and leave the others to the reader.  Using
the notation from \S \ref{section:lambdasymmetry}, we interpret $\LambdaS{a \wedge y,x \wedge (-b)}$ as
$\LambdaS{(-b) \wedge x,y \wedge a}$.  Our goal is to prove that $\LambdaS{(-b) \wedge x,y \wedge a} = - \LambdaS{a \wedge y,x \wedge b}$
By the definition of $\Lambda$-elements (Definition \ref{definition:lambda}), $\LambdaS{(-b) \wedge x,y \wedge a} = \LambdaIS{(-b) \wedge x,y \wedge a}$ equals
\begin{align*}
&\ThetaS{(-b) \wedge (a+x),y \wedge (a+x)} - \ThetaS{(-b) \wedge a,y \wedge a} - \PresS{(-b) \wedge (a+x),y \wedge x}\\
&=-\ThetaS{b \wedge (a+x),y \wedge (a+x)} + \ThetaS{b \wedge a,y \wedge a} + \PresS{b \wedge (a+x),y \wedge x}\\
&=-\ThetaS{(a+x) \wedge b,(a+x) \wedge y} + \ThetaS{a \wedge b,a \wedge y} + \PresS{(a+x) \wedge b,x \wedge y}.
\end{align*}
This last expression equals $-\LambdaIIS{a \wedge y,x \wedge b} = -\LambdaS{a \wedge y,x \wedge b}$.
\end{proof}

\subsection{The set \texorpdfstring{$S_3$}{S3}}

We now return to constructing $S_3$.  Recall that 
\[T_3 = \Set{$(a_i \wedge y) \Cdot (x \wedge b_i)$}{$1 \leq i \leq g$, $x,y \in \cB \setminus \{a_i,b_i\}$, $\omega(x,y)=0$}.\]
Define
\[S_3 = \Set{$\SLambdaS{a_i \wedge y,x \wedge b_i}$}{$1 \leq i \leq g$, $x,y \in \cB \setminus \{a_i,b_i\}$, $\omega(x,y)=0$}.\]
Like we did here, we will write elements of $\Span{S_3}$ in orange.  For example, using 
Lemma \ref{lemma:lambdalinear} ($\Lambda$-linearity) the following holds for $1 \leq i \leq g$.  Consider
$x,y \in \Span{a_i,b_i}^{\perp}$ with $\omega(x,y) = 0$.  Assume there exist
$\cB_1,\cB_2 \subset \cB \setminus \{a_i,b_i\}$ (not necessarily disjoint) with $\omega(z_1,z_2)=0$ for all $z_1 \in \cB_1$ and $z_2 \in \cB_2$
such that $x \in \Span{\cB_1}$ and $y \in \Span{\cB_2}$.  Then
$\SLambdaS{a_i \wedge y,x \wedge b_i} \in \Span{S_3}$.

\begin{remark}
It is not true, however, that in general elements of the form $\LambdaS{a_i \wedge y,x \wedge b_i}$ with
$x,y \in \Span{a_i,b_i}^{\perp}$ lie in $\Span{S_3}$.  See \S \ref{section:lambdarelation} below.
\end{remark}

By construction,
the linearization map $\Phi$ takes $S_3$ bijectively to $T_3$.  Even better:

\begin{lemma}
\label{lemma:symspans3}
The linearization map $\Phi$ takes $\Span{S_1,S_2,S_3}$ isomorphically to $\Span{T_1,T_2,T_3}$.
\end{lemma}
\begin{proof}
Recall that in Lemma \ref{lemma:symspans2} we proved that $\Phi$ takes $\Span{S_1,S_2}$ isomorphically
to $\Span{T_1,T_2}$.  Part of the proof of that lemma was that $\Phi$ takes $S_1 \cup S_2$ bijectively to 
$T_1 \cup T_2$.  It follows that $\Phi$ takes $S_1 \cup S_2 \cup S_3$ bijectively to $T_1 \cup T_2 \cup T_3$.  
What is more, in the proof of Lemma \ref{lemma:symspans2} we proved that all relations between elements of 
$T_1 \cup \cdots \cup T_4$ are actually relations between elements of $T_1 \cup T_2$ and lift to relations
between elements of $S_1 \cup S_2$.  The lemma follows.
\end{proof}

\subsection{Additional bilinearity relations}

We close this section by proving some additional relations between the $\Lambda$-elements.

\begin{lemma}[$\Lambda$-bilinearity I]
\label{lemma:lambdabilinear1}
Let $a \wedge b$ be a symplectic pair in $H_{\Z}$, let $x,y \in \Span{a,b}^{\perp}$ satisfy $\omega(x,y)=0$,
and let $z \in \Span{a,b,x,y}^{\perp}$.  Then:
\begin{align*}
\LambdaS{(a+z) \wedge y,x \wedge b} &= \LambdaS{a \wedge y,x \wedge b} + \PresS{z \wedge y,x \wedge b},\\
\LambdaS{a \wedge y,x \wedge (b+z)} &= \LambdaS{a \wedge y,x \wedge b} + \PresS{a \wedge y,x \wedge z}.
\end{align*}
\end{lemma}
\begin{proof}
Both formulas are proved the same way, so we will prove the first.  By definition,
$\LambdaS{(a+z) \wedge y,x \wedge b} = \LambdaIS{(a+z) \wedge y,x \wedge b}$ equals
\[\ThetaS{(a+z) \wedge (b+y),x \wedge (b+y)} - \ThetaS{(a+z) \wedge b,x \wedge b} - \PresS{(a+z) \wedge (b+y), x \wedge y}.\]
By Lemma \ref{lemma:thetabilinear1} ($\Theta$-bilinearity I), this equals
\begin{align*}
&\ThetaS{a \wedge (b+y),x \wedge (b+y)} + \PresS{z \wedge (b+y),x \wedge (b+y)} - \ThetaS{a \wedge b,x \wedge b} - \PresS{z \wedge b,x \wedge b} \\
&- \PresS{(a+z) \wedge (b+y), x \wedge y} \\
=&\LambdaS{a \wedge y,x \wedge b} + \PresS{a \wedge (b+y),x \wedge y} + \PresS{z \wedge (b+y),x \wedge (b+y)} - \PresS{z \wedge b,x \wedge b}\\ 
&- \PresS{a \wedge (b+y), x \wedge y} - \PresS{z \wedge (b+y), x \wedge y}.
\end{align*}
All the terms after the first in this can be expanded out and many of the resulting terms cancel, leaving
$\LambdaS{a \wedge y,x \wedge b} + \PresS{z \wedge y,x \wedge b}$.
\end{proof}

Lemma \ref{lemma:lambdabilinear1} allows some standard generators of $\fK_g^s$ to be written as the sum
of two $\Lambda$-elements:\footnote{There are many variants of this in the style of Lemma \ref{lemma:thetaexpansion2}, but
we give the only one we use.}

\begin{lemma}[$\Lambda$-expansion II]
\label{lemma:lambdaexpansion2}
Let $a \wedge b$ and $a' \wedge b'$ be
symplectic pairs in $H_{\Z}$ such that $\Span{a,b}$ and $\Span{a',b'}$ are orthogonal and let $y,w \in \Span{a,b',b+a'}^{\perp}$
satisfy $\omega(y,w)=0$.  Then
\[\PresS{(a+b') \wedge y,(b+a') \wedge w} = \LambdaS{a \wedge y,(b+a') \wedge w} + \LambdaS{b' \wedge y,(b+a') \wedge w}.\]
\end{lemma}
\begin{proof}
Lemma \ref{lemma:lambdabilinear1} implies that
\[\LambdaS{(-b') \wedge y,w \wedge (b+a')} + \PresS{(a+b') \wedge y, w \wedge (b+a')}\]
equals
\[\LambdaS{(-b' + (a+b')) \wedge y,w \wedge (b+a')} = \LambdaS{a \wedge y,w \wedge (b+a')}.\]
Rearranging this, we see that
\[\PresS{(a+b') \wedge y,(b+a') \wedge w} = \LambdaS{a \wedge y,(b+a') \wedge w} + \LambdaS{b' \wedge y,(b+a') \wedge w}.\qedhere\]
\end{proof}

\begin{lemma}[$\Lambda$-bilinearity II]
\label{lemma:lambdabilinear2}
For a symplectic pair $a \wedge b$ in $H_{\Z}$ and $x,y \in \Span{a,b}^{\perp}$ with $\omega(x,y) = 0$ and $n \in \Z$,
we have
\begin{align*}
\LambdaS{(a+nb) \wedge y,x \wedge b} &= \LambdaS{a \wedge y,x \wedge b} + n\PresS{b \wedge y,x \wedge b},\\
\LambdaS{a \wedge y,x \wedge (b+na)} &= \LambdaS{a \wedge y,x \wedge b} + n\PresS{a \wedge y,x \wedge a}.
\end{align*}
\end{lemma}
\begin{proof}
These two relations are proved in the same way, so we will give the details for the first.
The element $\LambdaS{(a+nb) \wedge y,x \wedge b} = \LambdaIS{(a+nb) \wedge y,x \wedge b}$ equals
\begin{align}
\label{eqn:lambdalinear2.1}
&\ThetaS{(a+nb) \wedge (b+y),x \wedge (b+y)} - \ThetaS{(a+nb) \wedge b,x \wedge b} \\
&\quad- \PresS{(a+nb) \wedge b,x \wedge y} - \PresS{(a+nb) \wedge y,x \wedge y} \nonumber\\
&=\ThetaS{(a+nb) \wedge (b+y),x \wedge (b+y)} - \ThetaS{a \wedge b,x \wedge b} \nonumber \\
&\quad - \PresS{a \wedge b,x \wedge y} - \PresS{a \wedge y,x \wedge y} - n\PresS{b \wedge y,x \wedge y}.\nonumber
\end{align}
Here we use Lemma \ref{lemma:thetabilinear2} ($\Theta$-bilinearity II).  That lemma also implies that $\ThetaS{(a+nb) \wedge (b+y),x \wedge (b+y)}$ equals
\begin{align*}
&\ThetaS{(a-ny) \wedge (b+y),x \wedge (b+y)} \\
                                            &= \ThetaS{a \wedge (b+y),x \wedge (b+y)} - n\PresS{y \wedge (b+y),x \wedge (b+y)} \\
                                            &= \ThetaS{a \wedge (b+y),x \wedge (b+y)} - n\PresS{y \wedge b,x \wedge b} - n \PresS{y \wedge b,x \wedge y},
\end{align*}
where we are also using Lemma \ref{lemma:thetabilinear1} ($\Theta$-bilinearity I).
Plugging this into \eqref{eqn:lambdalinear2.1} and canceling terms gives
\[\ThetaS{a \wedge (b+y),x \wedge (b+y)} - \ThetaS{a \wedge b,x \wedge b} - \PresS{a \wedge b,x \wedge y} - \PresS{a \wedge y,x \wedge y} - n\PresS{y \wedge b,x \wedge b},\]
which equals $\LambdaS{a \wedge y,x \wedge b} + n\PresS{b \wedge y,x \wedge b}$.
\end{proof}

Again, Lemma \ref{lemma:lambdabilinear2} allows some standard generators of $\fK_g^s$ to be written as the sum
of two $\Lambda$-elements:

\begin{lemma}[$\Lambda$-expansion III]
\label{lemma:lambdaexpansion3}
Let $a \wedge b$ be a symplectic pair in $H_{\Z}$ and let $y,w \in \Span{a,b}^{\perp}$
satisfy $\omega(y,w)=0$.  Then for $\epsilon \in \{\pm 1\}$ we have
\[\PresS{(a+\epsilon b) \wedge y,(a+\epsilon b) \wedge w} = \LambdaS{a \wedge y,(a+\epsilon b) \wedge w} + \epsilon \LambdaS{b \wedge y,(a+\epsilon b) \wedge w}.\]
\end{lemma}
\begin{proof}
Lemma \ref{lemma:lambdabilinear1} implies that
\[\LambdaS{(-\epsilon b) \wedge y,w \wedge (a+\epsilon b)} + \PresS{(a+\epsilon b) \wedge y, w \wedge (a+\epsilon b)}\]
equals
\[\LambdaS{(-\epsilon b + (a+\epsilon b)) \wedge y,w \wedge (a+\epsilon b)} = \LambdaS{a \wedge y,w \wedge (a+\epsilon b)}.\]
Rearranging this, we see that
\[\PresS{(a+\epsilon b) \wedge y,(a+\epsilon b) \wedge w} = \LambdaS{a \wedge y,(a+\epsilon b) \wedge w} + \epsilon \LambdaS{b \wedge y,(a+\epsilon b) \wedge w}\qedhere.\]
\end{proof}

\section{Symmetric kernel, symmetric version IV: \texorpdfstring{$S_4$}{S4} and the \texorpdfstring{$\Omega$}{Omega}-elements}
\label{section:presentationsym4}

We continue using all the notation from \S \ref{section:presentationsym1} -- \S \ref{section:presentationsym3}.
This section constructs the set $S_4$ that lifts $T_4$.  It consists of what are called
$\Omega$-elements of $\fK_g^s$, and the first part of this section constructs
these in more generality than is needed for $S_4$ alone.

\subsection{Definition}
Let $a \wedge b$ and $a' \wedge b'$ be symplectic pairs in $\wedge^2 H_{\Z}$ such that $\Span{a,b}$ and $\Span{a',b'}$
are orthogonal.  The $\Omega$-element $\OmegaS{a \wedge a',b' \wedge b}$ is an element of 
$\fK_g^s$ that is taken by $\Phi$ to
\[(a \wedge a') \Cdot (b' \wedge b) \in \Sym^2((\wedge^2 H)/\Q).\]
To find it, note that $\left(a \wedge (b+a')\right) \Cdot \left((a+b') \wedge (b+a')\right)$ equals
\begin{align*}
& \left(a \wedge (b+a')\right) \Cdot \left(a \wedge (b+a')\right)
+ (a \wedge b) \Cdot (b' \wedge b)\\
&\quad + (a \wedge b) \Cdot (b' \wedge a')
+ (a \wedge a') \Cdot (b' \wedge b)
+ (a \wedge a') \Cdot (b' \wedge a').
\end{align*}
There are similar formulas involving
\begin{align*}
&\left(a'     \wedge (a+b')\right) \Cdot \left((b+a') \wedge (a+b')\right),\ \text{and}\\
&\left((a+b') \wedge b \right)     \Cdot \left((a+b') \wedge (b+a')\right),\ \text{and}\\
&\left((b+a') \wedge b'\right)     \Cdot \left((b+a') \wedge (a+b')\right).
\end{align*}
This suggests four possible elements of $\fK_g^s$ projecting to $(a \wedge a') \Cdot (b' \wedge b)$:

\begin{definition}
\label{definition:omega}
For symplectic pairs $a \wedge b$ and $a' \wedge b'$ in $H_{\Z}$ with $\Span{a,b}$ and $\Span{a',b'}$ orthogonal, 
define 
\begin{align*}
\OmegaIS{a \wedge a',b' \wedge b} = &\ThetaS{a \wedge (b+a'), (a+b') \wedge (b+a')}
-\PresS{a \wedge (b+a'), a \wedge (b+a')} \\
&-\ThetaS{a \wedge b,b' \wedge b}
-\PresS{a \wedge b,b' \wedge a'}
-\ThetaS{a \wedge a',b' \wedge a'},\\
\OmegaIIS{a \wedge a',b' \wedge b} = &\ThetaS{a' \wedge (a+b'),(b+a') \wedge (a+b')}
-\PresS{a' \wedge (a+b'),a' \wedge (a+b')} \\
&-\ThetaS{a' \wedge b',b \wedge b'}
-\PresS{a' \wedge b',b \wedge a}
-\ThetaS{a' \wedge a,b \wedge a},\\
\OmegaIIIS{a \wedge a',b' \wedge b} = &\ThetaS{(a+b') \wedge b,(a+b') \wedge (b+a')}
-\PresS{(a+b') \wedge b,(a+b') \wedge b} \\
&-\ThetaS{a \wedge b,a \wedge a'}
-\PresS{a \wedge b,b' \wedge a'}
-\ThetaS{b' \wedge b,b' \wedge a'},\\
\OmegaIVS{a \wedge a',b' \wedge b} = &\ThetaS{(b+a') \wedge b',(b+a') \wedge (a+b')}
-\PresS{(b+a') \wedge b',(b+a') \wedge b'} \\
&-\ThetaS{a'\wedge b',a' \wedge a}
-\PresS{a'\wedge b',b \wedge a}
-\ThetaS{b \wedge b',b \wedge a}.\qedhere
\end{align*}
\end{definition}

See Remark \ref{remark:omegaabuse} below for a caveat about this notation.  By construction, we have
\[\Phi(\OmegaiS{a \wedge a',b' \wedge b}) = (a \wedge a') \Cdot (b' \wedge b) \quad \text{for $1 \leq i \leq 4$}.\]
Since we are trying to prove that $\Phi$ is an isomorphism, we must prove that these are actually the same element:

\begin{lemma}
\label{lemma:omegawelldefined}
Let $a \wedge b$ and $a' \wedge b'$ be symplectic pairs in $H_{\Z}$ such that $\Span{a,b}$ and $\Span{a',b'}$
are orthogonal.  Then the $\OmegaiS{a \wedge a',b' \wedge b}$ for $1 \leq i \leq 4$ are all equal.
\end{lemma}
\begin{proof}
Whether or not they are equal is invariant under the action of $\Sp_{2g}(\Z)$.
Recall that we have our fixed symplectic basis
$\cB = \{a_1,b_1,\ldots,a_g,b_g\}$ for $H_{\Z}$.  Applying an appropriate element of $\Sp_{2g}(\Z)$,
we can assume that
\[\text{$a_1 = a$,\ \ $b_1 = b$,\ \ $a_2 = a'$,\ \ $b_2 = b'$}.\]
Since $\Phi$ takes the $\OmegaiS{a_1 \wedge a_2,b_2 \wedge b_1}$ to the same thing and the restriction
of $\Phi$ to $\Span{S_1,S_2,S_3}$ is injective (Lemma \ref{lemma:symspans3}), it is enough
to prove that the $\OmegaiS{a_1 \wedge a_2,b_2 \wedge b_1}$ are equal modulo $\Span{S_1,S_2,S_3}$.  Let
$\equiv$ denote equality modulo $\Span{S_1,S_2,S_3}$.  We have
\begin{small}
\begin{align*}
\OmegaIS{a_1 \wedge a_2,b_2 \wedge b_1} = &\ThetaS{a_1 \wedge (b_1+a_2), (a_1+b_2) \wedge (b_1+a_2)}
-\PresS{a_1 \wedge (b_1+a_2), a_1 \wedge (b_1+a_2)} \\
&-\SThetaS{a_1 \wedge b_1,b_2 \wedge b_1}
-\SPresS{a_1 \wedge b_1,b_2 \wedge a_2}
-\SThetaS{a_1 \wedge a_2,b_2 \wedge a_2} \\
\equiv &\ThetaS{a_1 \wedge (b_1+a_2), (a_1+b_2) \wedge (b_1+a_2)} \\
&-\left(\SPresS{a_1 \wedge b_1,a_1 \wedge b_1} + 2\SThetaS{a_1 \wedge b_1,a_1 \wedge a_2} + \SPresS{a_1 \wedge a_2,a_1 \wedge a_2}\right) \\
\equiv &\ThetaS{a_1 \wedge (b_1+a_2), (a_1+b_2) \wedge (b_1+a_2)}.
\end{align*}
\end{small}%
Similarly,
\begin{align*}
\OmegaIIS{a_1 \wedge a_2,b_2 \wedge b_1} &\equiv \ThetaS{a_2 \wedge (a_1+b_2),(b_1+a_2) \wedge (a_1+b_2)},\\
\OmegaIIIS{a_1 \wedge a_2,b_2 \wedge b_1} &\equiv \ThetaS{(a_1+b_2) \wedge b_1,(a_1+b_2) \wedge (b_1+a_2)},\\
\OmegaIVS{a_1 \wedge a_2,b_2 \wedge b_1} &\equiv \ThetaS{(b_1+a_2) \wedge b_2,(b_1+a_2) \wedge (a_1+b_2)}.
\end{align*}
It is therefore enough to prove the following two claims:

\begin{claim}{1}
We have $\ThetaS{a_1 \wedge (b_1+a_2), (a_1+b_2) \wedge (b_1+a_2)} \equiv \ThetaS{a_2 \wedge (a_1+b_2),(b_1+a_2) \wedge (a_1+b_2)}$
and $\ThetaS{(a_1+b_2) \wedge b_1,(a_1+b_2) \wedge (b_1+a_2)} \equiv \ThetaS{(b_1+a_2) \wedge b_2,(b_1+a_2) \wedge (a_1+b_2)}$.
\end{claim}

\noindent
Both equalities are proved the same way, so we will give details for the first.  Consider the symplectic basis
$\{a_1,b_1+a_2,a_2,a_1+b_2,a_3,b_3,\ldots,a_g,b_g\}$ for $H_{\Z}$.
Lemma \ref{lemma:liftsymr2} ($\Theta$-symplectic basis) implies that
\begin{align*}
0=&\ThetaS{a_1 \wedge (b_1+a_2),(a_1+b_2) \wedge (b_1+a_2)} + \ThetaS{a_2 \wedge (a_1+b_2),(a_1+b_2) \wedge (b_1+a_2)} \\
&+ \sum\nolimits_{i=3}^g \SThetaS{a_i \wedge b_i,(a_1+b_2) \wedge (b_1+a_2)}.
\end{align*}
We conclude that
\begin{align*}
\ThetaS{a_1 \wedge (b_1+a_2),(a_1+b_2) \wedge (b_1+a_2)} &\equiv -\ThetaS{a_2 \wedge (a_1+b_2),(a_1+b_2) \wedge (b_1+a_2)} \\
                                                         &=      \ThetaS{a_2 \wedge (a_1+b_2),(b_1+a_2) \wedge (a_1+b_2)}.
\end{align*}

\begin{claim}{2}
We have $\ThetaS{a_1 \wedge (b_1+a_2), (a_1+b_2) \wedge (b_1+a_2)} \equiv \ThetaS{(b_1+a_2) \wedge b_2,(b_1+a_2) \wedge (a_1+b_2)}$.
\end{claim}

\noindent
Using Lemma \ref{lemma:thetalinear} ($\Theta$-linearity), we have
\[-\ThetaS{a_1 \wedge (b_1+a_2), (a_1+b_2) \wedge (b_1+a_2)} = \ThetaS{a_1 \wedge (b_1+a_2), (-a_1-b_2) \wedge (b_1+a_2)}.\]
By definition (Definition \ref{definition:theta}), this equals
\begin{align*}
\frac{1}{2}(&\PresS{(-b_2) \wedge (b_1+a_2),(-b_2) \wedge (b_1+a_2)} - \PresS{a_1 \wedge (b_1+a_2),a_1 \wedge (b_1+a_2)} \\
& - \PresS{(a_1+b_2) \wedge (b_1+a_2),(a_1+b_2) \wedge (b_1+a_2)}).
\end{align*}
The term $\PresS{(-b_2) \wedge (b_1+a_2),(-b_2) \wedge (b_1+a_2)}$ equals
\[\SPresS{b_2 \wedge b_1,b_2 \wedge b_1} + 2\SThetaS{b_2 \wedge b_1,b_2 \wedge a_2} + \SPresS{b_2 \wedge a_2,b_2 \wedge a_2} \equiv 0\]
and the term $\PresS{a_1 \wedge (b_1+a_2),a_1 \wedge (b_1+a_2)}$ equals
\[\SPresS{a_1 \wedge b_1,a_1 \wedge b_1} + 2 \SThetaS{a_1 \wedge b_1,a_1 \wedge a_2} + \SPresS{a_1 \wedge a_2,a_1 \wedge a_2} \equiv 0.\]
We deduce that
\[\ThetaS{a_1 \wedge (b_1+a_2), (a_1+b_2) \wedge (b_1+a_2)} \equiv \frac{1}{2} \PresS{(a_1+b_2) \wedge (b_1+a_2),(a_1+b_2) \wedge (b_1+a_2)}.\]
Lemma \ref{lemma:thetaexpansion2} ($\Theta$-expansion II) implies that
\begin{align*}
\PresS{(a_1+b_2) \wedge (b_1+a_2),(a_1+b_2) \wedge (b_1+a_2)} = &\ThetaS{a_1 \wedge (b_1+a_2),(a_1+b_2) \wedge (b_1+a_2)} \\
&+\ThetaS{b_2 \wedge (b_1+a_2),(a_1+b_2) \wedge (b_1+a_2)}.
\end{align*}
Combining this with our previous formula, we conclude that
\begin{align*}
\ThetaS{a_1 \wedge (b_1+a_2),(a_1+b_2) \wedge (b_1+a_2)} \equiv &\ThetaS{b_2 \wedge (b_1+a_2),(a_1+b_2) \wedge (b_1+a_2)} \\
                                                         =      &\ThetaS{(b_1+a_2) \wedge b_2,(b_1+a_2) \wedge (a_1+b_2)}.\qedhere
\end{align*}
\end{proof}

In light of this lemma, we will denote the common value of $\OmegaiS{a \wedge a',b' \wedge b}$
by $\OmegaS{a \wedge a',b' \wedge b}$.

\begin{remark}
\label{remark:omegaabuse}
Just like for the $\Theta$- and $\Lambda$-elements (cf.\ Remarks \ref{remark:thetaabuse} and \ref{remark:lambdaabuse}), the elements
$\OmegaiS{a \wedge a',b' \wedge b}$ and $\OmegaS{a \wedge a',b' \wedge b}$
depend on the ordered tuple $(a,a',b',b)$, not
on $a \wedge a'$ and $b \wedge b'$.
\end{remark}

\subsection{Relation to \texorpdfstring{$\Lambda$}{Lambda}-elements}
\label{section:lambdarelation}

The following lemma shows that the $\Omega$-elements can almost (but not quite) be written
in terms of the $\Lambda$-elements:

\begin{lemma}[$\Lambda$ to $\Omega$]
\label{lemma:omegalambda}
Let $a \wedge b$ and $a' \wedge b'$ and $a'' \wedge b''$ be symplectic pairs in $H_{\Z}$ such
that $\Span{a,b}$ and $\Span{a',b'}$ and $\Span{a'',b''}$ are all orthogonal.  Then
\[\LambdaS{a \wedge (a' + a''),(b' - b'') \wedge b} - \LambdaS{a \wedge a'',b' \wedge b} + \LambdaS{a \wedge a',b'' \wedge b}\]
equals $\OmegaS{a \wedge a',b' \wedge b} - \OmegaS{a \wedge a'',b'' \wedge b}$.
\end{lemma}
\begin{proof}
Whether or not they are equal is invariant under the action of $\Sp_{2g}(\Z)$.
Recall that we have our fixed symplectic basis 
$\cB = \{a_1,b_1,\ldots,a_g,b_g\}$ for $H_{\Z}$.  Applying an appropriate element of $\Sp_{2g}(\Z)$,
we can assume that
\[\text{$a_1 = a$,\ \ $b_1 = b$,\ \ $a_2 = a'$,\ \ $b_2 = b'$,\ \ $a_3 = a''$,\ \ $b_3 = b''$}.\]
Since $\Phi$ takes our two elements to the same thing and the restriction
of $\Phi$ to $\Span{S_1,S_2,S_3}$ is injective (Lemma \ref{lemma:symspans3}), it is enough
to prove that they are equal modulo $\Span{S_1,S_2,S_3}$.  Let $\equiv$ denote equality
modulo $\Span{S_1,S_2,S_3}$.  The element 
\[\LambdaS{a_1 \wedge (a_2 + a_3),(b_2 - b_3) \wedge b_1} - \SLambdaS{a_1 \wedge a_2,b_2 \wedge b_1} + \SLambdaS{a_1 \wedge a_3,b_3 \wedge b_1}\]
is equivalent to $\LambdaS{a_1 \wedge (a_2 + a_3),(b_2 - b_3) \wedge b_1}$, and we must prove that this
is equivalent to $\OmegaS{a_1 \wedge a_2,b_2 \wedge b_1} - \OmegaS{a_1 \wedge a_3,b_3 \wedge b_1}$.
Below we will prove the following three facts:
\begin{align*}
\LambdaS{a_1 \wedge (a_2 + a_3),(b_2 - b_3) \wedge b_1} &\equiv \ThetaS{a_1 \wedge (b_1+a_2+a_3),(b_2-b_3) \wedge (b_1+a_2+a_3)}, \\
\OmegaS{a_1 \wedge a_2,b_2 \wedge b_1} &\equiv \ThetaS{a_1 \wedge (b_1+a_2+a_3),(a_1+b_2) \wedge (b_1+a_2+a_3)},\\
\OmegaS{a_1 \wedge a_3,b_3 \wedge b_1} &\equiv \ThetaS{a_1 \wedge (b_1+a_2+a_3),(a_1+b_3) \wedge (b_1+a_2+a_3)}.
\end{align*}
These will imply the lemma; indeed, the linearity of $\Theta$-elements will imply that
\begin{align*}
&\OmegaS{a_1 \wedge a_2,b_2 \wedge b_1} - \OmegaS{a_1 \wedge a_3,b_3 \wedge b_1} \\
\equiv &\ThetaS{a_1 \wedge (b_1+a_2+a_3),(a_1+b_2) \wedge (b_1+a_2+a_3)} \\
       &- \ThetaS{a_1 \wedge (b_1+a_2+a_3),(a_1+b_3) \wedge (b_1+a_2+a_3)} \\
=      &\ThetaS{a_1 \wedge (b_1+a_2+a_3),(b_2-b_3) \wedge (b_1+a_2+a_3)} \\
\equiv &\LambdaS{a_1 \wedge (a_2 + a_3),(b_2 - b_3) \wedge b_1}.
\end{align*}
It remains to prove the above three facts:

\begin{claim}{1}
We have 
\[\LambdaS{a_1 \wedge (a_2 + a_3),(b_2 - b_3) \wedge b_1} \equiv \ThetaS{a_1 \wedge (b_1+a_2+a_3),(b_2-b_3) \wedge (b_1+a_2+a_3)}.\]
\end{claim}

\noindent
By definition, $\LambdaS{a_1 \wedge (a_2 + a_3),(b_2 - b_3) \wedge b_1} = \LambdaIS{a_1 \wedge (a_2 + a_3),(b_2 - b_3) \wedge b_1}$
equals
\begin{align*}
&\ThetaS{a_1 \wedge (b_1+a_2+a_3),(b_2 - b_3) \wedge (b_1+a_2+a_3)} \\
&- \SThetaS{a_1 \wedge b_1,(b_2 - b_3) \wedge b_1} - \PresS{a_1 \wedge (b_1+a_2+a_3),(b_2 - b_3) \wedge (a_2+a_3)}\\
\equiv &\ThetaS{a_1 \wedge (b_1+a_2+a_3),(b_2 - b_3) \wedge (b_1+a_2+a_3)} \\
       &-\SPresS{a_1 \wedge b_1,(b_2 - b_3) \wedge (a_2+a_3)} -\PresS{a_1 \wedge (a_2+a_3),(b_2 - b_3) \wedge (a_2+a_3)} \\
\equiv &\ThetaS{a_1 \wedge (b_1+a_2+a_3),(b_2 - b_3) \wedge (b_1+a_2+a_3)} \\
       &-\PresS{(a_2 + a_3) \wedge (b_2 - b_3),(a_2+a_3) \wedge a_1}
\end{align*}
To prove the claim, we must show that $\PresS{(a_2 + a_3) \wedge (b_2 - b_3),(a_2+a_3) \wedge a_1}$ is equivalent
to $0$.  By Lemma \ref{lemma:thetaexpansion2} ($\Theta$-expansion II), this element equals
\[\ThetaS{(a_2+a_3) \wedge b_2,(a_2+a_3) \wedge a_1} - \ThetaS{(a_2+a_3) \wedge b_3,(a_2+a_3) \wedge a_1}.\]
By Lemma \ref{lemma:lambdaexpansion1} ($\Lambda$-expansion I),
the element $\ThetaS{(a_2+a_3) \wedge b_2,(a_2+a_3) \wedge a_1}$ equals
\[\SLambdaS{a_2 \wedge a_1,a_3 \wedge b_2} + \SThetaS{a_2 \wedge b_2,a_2 \wedge a_1} + \SPresS{(a_2+a_3) \wedge b_2,a_3 \wedge a_1} \equiv 0.\]
Similarly, $\ThetaS{(a_2+a_3) \wedge b_3,(a_2+a_3) \wedge a_1} \equiv 0$.
The claim follows.

\begin{claim}{2}
We have
\begin{align*}
\OmegaS{a_1 \wedge a_2,b_2 \wedge b_1} &\equiv \ThetaS{a_1 \wedge (b_1+a_2+a_3),(a_1+b_2) \wedge (b_1+a_2+a_3)},\\
\OmegaS{a_1 \wedge a_3,b_3 \wedge b_1} &\equiv \ThetaS{a_1 \wedge (b_1+a_2+a_3),(a_1+b_3) \wedge (b_1+a_2+a_3)}.
\end{align*}
\end{claim}

\noindent
Both equalities are proved in the same way, so we will give the details for the first.  Consider
the symplectic basis
\[\{a_1,b_1+a_2+a_3,a_2,a_1+b_2,a_3,a_1+b_3,a_4,b_4,\ldots,a_g,b_g\}\]
for $H_{\Z}$.  Applying Lemma \ref{lemma:liftsymr2} ($\Theta$-symplectic basis), we deduce that
\begin{small}
\begin{align*}
&\ThetaS{a_1 \wedge (b_1+a_2+a_3),(a_1+b_2) \wedge (b_1+a_2+a_3)} - \ThetaS{a_2 \wedge (a_1+b_2),(b_1+a_2+a_3) \wedge (a_1+b_2)} \\
&+\PresS{a_3 \wedge (a_1+b_3),(a_1+b_2) \wedge (b_1+a_2+a_3)} + \sum\nolimits_{i=4}^g \SPresS{a_i \wedge b_i,(a_1+b_2) \wedge (b_1+a_2+a_3)}
\end{align*}
\end{small}%
equals $0$.  Throwing away terms that are equivalent modulo $\Span{S_1,S_2,S_3}$ to $0$, we deduce that
$\ThetaS{a_1 \wedge (b_1+a_2+a_3),(a_1+b_2) \wedge (b_1+a_2+a_3)}$ is equivalent to
\begin{align*}
&\ThetaS{a_2 \wedge (a_1+b_2),(b_1+a_2+a_3) \wedge (a_1+b_2)} - \PresS{a_3 \wedge (a_1+b_3),(a_1+b_2) \wedge (b_1+a_2+a_3)} \\
= &\ThetaS{a_2 \wedge (a_1+b_2),(b_1+a_2) \wedge (a_1+b_2)} + \ThetaS{a_2 \wedge (a_1+b_2),a_3 \wedge (a_1+b_2)}\\
  &-\PresS{a_3 \wedge (a_1+b_3),(a_1+b_2) \wedge (b_1+a_2+a_3)}.
\end{align*}
Just like at the beginning of the proof of Lemma \ref{lemma:omegawelldefined}, we have
\[\ThetaS{a_2 \wedge (a_1+b_2),(b_1+a_2) \wedge (a_1+b_2)} \equiv \OmegaIIS{a_1 \wedge a_2,b_2 \wedge b_1} = \OmegaS{a_1 \wedge a_2,b_2 \wedge b_1}.\]
To prove the claim, we must therefore prove that the other two terms in the above sum are equivalent to $0$.

For $\ThetaS{a_2 \wedge (a_1+b_2),a_3 \wedge (a_1+b_2)}$, it follows from
Lemma \ref{lemma:lambdaexpansion1} ($\Lambda$-expansion I) that
it equals
\[\SLambdaS{a_2 \wedge a_1, a_3 \wedge b_2} + \SThetaS{a_2 \wedge b_2,a_3 \wedge b_2} + \SPresS{a_2 \wedge (b_2+a_1),a_3 \wedge a_1} \equiv 0.\]
For $\PresS{a_3 \wedge (a_1+b_3),(a_1+b_2) \wedge (b_1+a_2+a_3)}$, it equals
\begin{align*}
&\SPresS{a_3 \wedge (a_1+b_3),(a_1+b_2) \wedge a_2} + \PresS{a_3 \wedge (a_1+b_3),(a_1+b_2) \wedge (b_1+a_3)} \\
\equiv &\PresS{a_3 \wedge (a_1+b_3),a_1 \wedge (b_1+a_3)} + \PresS{a_3 \wedge (a_1+b_3),b_2 \wedge (b_1+a_3)}.
\end{align*}
We must show that both of these terms vanish.  For $\PresS{a_3 \wedge (a_1+b_3),a_1 \wedge (b_1+a_3)}$, in
$(\wedge^2 H)/\Q$ we have
\[a_3 \wedge (a_1+b_3) + a_1 \wedge (b_1+a_3) + \sum_{\substack{2 \leq i \leq g \\ i \neq 3}} a_i \wedge b_i = 0.\]
This implies that $\PresS{a_3 \wedge (a_1+b_3),a_1 \wedge (b_1+a_3)}$ equals
\begin{align*}
&-\PresS{a_1 \wedge (b_1+a_3),a_1 \wedge (b_1+a_3)} - \sum_{\substack{2 \leq i \leq g \\ i \neq 3}} \SPresS{a_i \wedge b_i,a_1 \wedge (b_1+a_3)} \\
\equiv &-\PresS{a_1 \wedge b_1,a_1 \wedge b_1} - 2 \ThetaS{a_1 \wedge b_1,a_1 \wedge a_3} - \PresS{a_1 \wedge a_3,a_1 \wedge a_3} \equiv 0.
\end{align*}
For $\PresS{a_3 \wedge (a_1+b_3),b_2 \wedge (b_1+a_3)}$, we use the same symplectic basis
\[\{a_3,a_1+b_3,a_1,b_1+a_3,a_2,b_2,a_4,b_4,\ldots,a_g,b_g\}\]
for $H_{\Z}$, but this time we use Lemma \ref{lemma:liftsymr2} ($\Theta$-symplectic basis) to see that
\begin{align*}
0=&\PresS{a_3 \wedge (a_1+b_3),b_2 \wedge (b_1+a_3)} + \ThetaS{a_1 \wedge (b_1+a_3),b_2 \wedge (b_1+a_3)} \\
&- \SThetaS{a_2 \wedge b_2,(b_1+a_3) \wedge b_2} + \sum_{\substack{2 \leq i \leq g \\ i \neq 3}} \SPresS{a_i \wedge b_i,b_2 \wedge (b_1+a_3)}.
\end{align*}
Throwing away terms that are equivalent to $0$, we deduce that 
\[\PresS{a_3 \wedge (a_1+b_3),b_2 \wedge (b_1+a_3)} \equiv -\ThetaS{a_1 \wedge (b_1+a_3),b_2 \wedge (b_1+a_3)}.\]
By Lemma \ref{lemma:lambdaexpansion1} ($\Lambda$-expansion I), this equals $-1$ times
\[\SLambdaS{a_1 \wedge a_3, b_2 \wedge b_1}  + \SThetaS{a_1 \wedge b_1,b_2 \wedge b_1} + \SPresS{a_1 \wedge (b_1+a_3),b_2 \wedge a_3} \equiv 0,\]
as desired.
\end{proof}

\subsection{\texorpdfstring{$\Omega$}{Omega}-symmetry and signs}
\label{section:omegasigns}

It is inconvenient to require the entries of
$\OmegaS{a \wedge a',b' \wedge b}$ to appear in a definite order.
We therefore would like to define that each of the following terms
equals $\OmegaS{a \wedge a',b' \wedge b}$:
\begin{alignat*}{6}
&&\OmegaS{a \wedge a',b' \wedge b},\ \ &-&\OmegaS{a' \wedge a,b' \wedge b},\ \ &-&\OmegaS{a \wedge a',b \wedge b'},\ \ &\OmegaS{a' \wedge a,b \wedge b'},\\
&&\OmegaS{b' \wedge b,a \wedge a'},\ \ &-&\OmegaS{b' \wedge b,a' \wedge a},\ \ &-&\OmegaS{b \wedge b',a \wedge a'},\ \ &\OmegaS{b \wedge b',a' \wedge a}.
\end{alignat*}
Similarly, we would like to by able to multiply terms by $-1$ in the usual way
and define that each of the following terms equals $\OmegaS{a \wedge a',b' \wedge b}$:
\begin{alignat*}{4}
&& &\OmegaS{a \wedge a',b' \wedge b},\ \       &-&\OmegaS{(-a) \wedge a',b' \wedge b},\\
&&-&\OmegaS{a \wedge (-a'),b' \wedge b},\ \    & &\OmegaS{(-a) \wedge (-a'),b' \wedge b},\\
&&-&\OmegaS{a \wedge a',(-b') \wedge b},\ \    & &\OmegaS{(-a) \wedge a',(-b') \wedge b},\\
&& &\OmegaS{a \wedge (-a'),(-b') \wedge b},\ \ &-&\OmegaS{(-a) \wedge (-a'),(-b') \wedge b},\\
&&-&\OmegaS{a \wedge a',b' \wedge (-b)},\ \    & &\OmegaS{(-a) \wedge a',b' \wedge (-b)},\\
&& &\OmegaS{a \wedge (-a'),b' \wedge (-b)},\ \ &-&\OmegaS{(-a) \wedge (-a'),b' \wedge (-b)},\\
&& &\OmegaS{a \wedge a',(-b') \wedge (-b)},\ \ &-&\OmegaS{(-a) \wedge a',(-b') \wedge (-b)},\\
&&-&\OmegaS{a \wedge (-a'),(-b') \wedge (-b)},\ \ & &\OmegaS{(-a) \wedge (-a'),(-b') \wedge (-b)}.
\end{alignat*}
The problem is that these definitions introduce ambiguity into our notation since some of these
are other $\Omega$-elements.  For instance, $\OmegaS{a \wedge (-a'),(-b') \wedge b}$ is
another $\Omega$-element associated to the symplectic pairs $a \wedge b$ and $(-a') \wedge (-b')$.
The following lemma says that all the possible other $\Omega$-elements obtained in this
way are actually the same:

\begin{lemma}
\label{lemma:omegaambiguity}
Let $a \wedge b$ and $a' \wedge b'$ be symplectic pairs in $H_{\Z}$ such
that $\Span{a,b}$ and $\Span{a',b'}$ are orthogonal.  Then the following are
all equal to $\OmegaS{a \wedge a',b' \wedge b}$:
\begin{small}
\[\begin{tabular}{lll}
\OmegaS{(-a)  \wedge a'   ,b'    \wedge (-b)}  & \OmegaS{a'    \wedge a   ,b     \wedge b'}    & \OmegaS{(-b) \wedge (-b'),a'    \wedge a} \\
\OmegaS{a     \wedge (-a'),(-b') \wedge b}     & \OmegaS{(-a') \wedge a   ,b     \wedge (-b')} & \OmegaS{b    \wedge (-b'),a'    \wedge (-a)} \\
\OmegaS{(-a)  \wedge (-a'),(-b') \wedge (-b)}  & \OmegaS{a'    \wedge (-a),(-b)  \wedge b'}    & \OmegaS{(-b) \wedge b'   ,(-a') \wedge a} \\
\OmegaS{(-b') \wedge (-b) , a    \wedge a'}    & \OmegaS{(-a') \wedge (-a),(-b)  \wedge (-b')} & \OmegaS{b    \wedge b'   ,(-a') \wedge (-a)} \\
\OmegaS{b'    \wedge (-b) , a    \wedge (-a')} & \OmegaS{(-b') \wedge b   ,(-a) \wedge a'}     & \OmegaS{b'   \wedge b    ,(-a)  \wedge (-a')}
\end{tabular}\]
\end{small}%
\end{lemma}
\begin{proof}
Below we will prove the following three special cases:
\begin{itemize}
\item[(i)]    $\OmegaS{a \wedge a',b' \wedge b} = \OmegaS{a' \wedge a,b \wedge b'}$
\item[(ii)]   $\OmegaS{a \wedge a',b' \wedge b} = \OmegaS{a \wedge (-a'),(-b') \wedge b}$
\item[(iiii)] $\OmegaS{a \wedge a',b' \wedge b} = \OmegaS{(-b') \wedge (-b),a \wedge a'}$
\end{itemize}
As is easily verified, all the other equalities we are trying to prove can be obtained by composing these
three.  For instance, to see that $\OmegaS{a \wedge a',b' \wedge b} = \OmegaS{(-a) \wedge a',b' \wedge (-b)}$ we compose
(i) and (ii) and (i):
\[\OmegaS{a \wedge a',b' \wedge b} = \OmegaS{a' \wedge a,b \wedge b'} = \OmegaS{a' \wedge (-a),(-b) \wedge b'} = \OmegaS{(-a) \wedge a',b' \wedge (-b)}.\]
We separate the proofs of (i) and (ii) and (iii) into the following three claims:

\begin{claim}{1}
$\OmegaS{a \wedge a',b' \wedge b} = \OmegaS{a' \wedge a,b \wedge b'}$.
\end{claim}

\noindent
We can calculate these using any of the $\Omega_i$-formulas from
Definition \ref{definition:omega}, so the claim follows from the fact that the
following are equal: $\OmegaIS{a \wedge a',b' \wedge b}$, which is
\begin{align*}
&\ThetaS{a \wedge (b+a'), (a+b') \wedge (b+a')}
-\PresS{a \wedge (b+a'), a \wedge (b+a')} \\
&-\ThetaS{a \wedge b,b' \wedge b}
-\PresS{a \wedge b,b' \wedge a'}
-\ThetaS{a \wedge a',b' \wedge a'}.
\end{align*}
and $\OmegaIIS{a' \wedge a,b \wedge b'}$, which is
\begin{align*}
&\ThetaS{a \wedge (a'+b),(b'+a) \wedge (a'+b)}
-\PresS{a \wedge (a'+b),a \wedge (a'+b)} \\
&-\ThetaS{a \wedge b,b' \wedge b}
-\PresS{a \wedge b,b' \wedge a'}
-\ThetaS{a \wedge a',b' \wedge a'}
\end{align*}

\begin{claim}{2}
\label{claim:omegasignclaim2}
$\OmegaS{a \wedge a',b' \wedge b} = \OmegaS{a \wedge (-a'),(-b') \wedge b}$.
\end{claim}

\noindent
Since $g \geq 4$ (Assumption \ref{assumption:genussym}), we can find a symplectic pair
$a'' \wedge b''$ such that $\Span{a'',b''}$ is orthogonal to both $\Span{a,b}$ and
$\Span{a',b'}$.  
Lemma \ref{lemma:omegalambda} ($\Lambda$ to $\Omega$) says that
\begin{equation}
\label{eqn:omegasignclaim2.1}
\LambdaS{a \wedge (a' + a''),(b' - b'') \wedge b} - \LambdaS{a \wedge a'',b' \wedge b} + \LambdaS{a \wedge a',b'' \wedge b}
\end{equation}
equals $\OmegaS{a \wedge a',b' \wedge b} - \OmegaS{a \wedge a'',b'' \wedge b}$
and that
\begin{equation}
\label{eqn:omegasignclaim2.2}
\LambdaS{a \wedge (-a' + a''),(-b' - b'') \wedge b} - \LambdaS{a \wedge a'',(-b') \wedge b} + \LambdaS{a \wedge (-a'),b'' \wedge b}
\end{equation}
equals $\OmegaS{a \wedge (-a'),(-b') \wedge b} - \OmegaS{a \wedge a'',b'' \wedge b}$.
To prove the claim, it is thus enough to prove that
\eqref{eqn:omegasignclaim2.1} equals \eqref{eqn:omegasignclaim2.2}.  By 
Lemma \ref{lemma:lambdalinear2} (strong $\Lambda$-linearity),
this is equivalent to the following identity in $(\Span{a,b}^{\perp})^{\otimes 2}$:
\[(a' + a'') \otimes (b' - b'') - a'' \otimes b' + a' \otimes b''
= (-a' + a'') \otimes (-b' - b'') - a'' \otimes (-b') + (-a') \otimes b''.\]
In fact, both sides of this equal $a' \otimes b' - a'' \otimes b''$.

\begin{claim}{3}
$\OmegaS{a \wedge a',b' \wedge b} = \OmegaS{(-b') \wedge (-b),a \wedge a'}$.
\end{claim}

\noindent
By Claim \ref{claim:omegasignclaim2}, it is enough to prove instead that
$\OmegaS{a \wedge (-a'),(-b') \wedge b} = \OmegaS{(-b') \wedge (-b),a \wedge a'}$.
We can calculate these using any of the $\Omega_i$-formulas from
Definition \ref{definition:omega}, so the claim follows from the fact that the
following are equal: $\OmegaIS{a \wedge (-a'),(-b') \wedge b}$, which is
\begin{align*}
&\ThetaS{a \wedge (b-a'), (a-b') \wedge (b-a')}
-\PresS{a \wedge (b-a'), a \wedge (b-a')} \\
&-\ThetaS{a \wedge b,(-b') \wedge b}
-\PresS{a \wedge b,(-b') \wedge (-a')}
-\ThetaS{a \wedge (-a'),(-b') \wedge (-a')},
\end{align*}
and $\OmegaIVS{(-b') \wedge (-b),a \wedge a'}$, which is
\begin{align*}
&\ThetaS{(a'-b) \wedge a,(a'-b) \wedge (-b'+a)}
-\PresS{(a'-b) \wedge a,(a'-b) \wedge a} \\
&-\ThetaS{(-b) \wedge a,(-b) \wedge (-b')}
-\PresS{(-b) \wedge a,a' \wedge (-b')}
-\ThetaS{a' \wedge a,a' \wedge (-b')}.\qedhere
\end{align*}
\end{proof}

In light of this lemma, we can permute terms and multiply them by $-1$ as described
before the lemma.  For instance, if $a \wedge b$ and $a' \wedge b'$ are symplectic
pairs such that $\Span{a,b}$ and $\Span{a',b'}$ are orthogonal, then we
can write things in the natural order and discuss $\OmegaS{a \wedge a',b \wedge b'}$, which
equals $-\OmegaS{a \wedge a',b' \wedge b}$.  We only used the unnatural order to avoid
signs in our formulas.  We can also now 
talk about $\OmegaS{a \wedge b',b \wedge a'}$, which equals
$\OmegaS{a \wedge (-b'),a' \wedge b}$ or $\OmegaS{a \wedge b',(-a') \wedge b}$.

\subsection{The set \texorpdfstring{$S_4$}{S4}}

We now return to constructing $S_4$.  Recall that
\[T_4 = \Set{$(a_i \wedge a_j) \Cdot (b_i \wedge b_j)$, $(a_i \wedge b_j) \Cdot (b_i \wedge a_j)$}{$1 \leq i < j \leq g$}.\]
Define
\[S_4 = \Set{$\SOmegaS{a_i \wedge a_j,b_i \wedge b_j}$, $\SOmegaS{a_i \wedge b_j,b_i \wedge a_j}$}{$1 \leq i < j \leq g$}.\]
Like we did here, we will write elements of $\Span{S_4}$ in green.  By construction,
the linearization map $\Phi$ takes $S_4$ bijectively to $T_4$.  Even better:

\begin{lemma}
\label{lemma:symspans4}
The linearization map $\Phi$ takes $\Span{S_1,\ldots,S_4}$ isomorphically to 
$\Sym^2((\wedge^2 H)/\Q)$.
\end{lemma}
\begin{proof}
Recall that in Lemma \ref{lemma:symspans3} we proved that $\Phi$ takes $\Span{S_1,S_2,S_3}$ isomorphically
to $\Span{T_1,T_2,T_3}$.  Part of the proof of that lemma was that $\Phi$ takes $S_1 \cup S_2 \cup S_3$ bijectively to
$T_1 \cup T_2 \cup T_3$.  It follows that $\Phi$ takes $S_1 \cup \cdots \cup S_4$ bijectively to $T_1 \cup \cdots \cup T_4$.
As we discussed in \S \ref{section:presentationsym1}, the set $T_1 \cup \cdots \cup T_4$ generates
$\Sym^2((\wedge^2 H)/\Q)$. 
In the proof of Lemma \ref{lemma:symspans3} we proved that all relations between elements of
$T_1 \cup \cdots \cup T_4$ are actually relations between elements of $T_1 \cup T_2$ and lift to relations
between elements of $S_1 \cup S_2$.  The lemma follows.
\end{proof}

\subsection{Additional relations}

For later use, we record the following.  Its statement uses our
fixed symplectic basis $\cB = \{a_1,b_1,\ldots,a_g,b_g\}$.

\begin{lemma}
\label{lemma:identifyomega}
Let $\equiv$ denote equality modulo $\Span{S_1,S_2,S_3}$.  For distinct
$1 \leq i,j \leq g$, we have
\begin{align*}
\PresS{(a_i+b_j) \wedge (b_i+a_j),(a_i+b_j) \wedge (b_i+a_j)} &\equiv \PresS{(a_i-b_j) \wedge (b_i-a_j),(a_i-b_j) \wedge (b_i-a_j)} \\
                                                              &\equiv -2 \SOmegaS{a_i \wedge a_j,b_i \wedge b_j}
\end{align*}
and
\begin{align*}
\PresS{(a_i+a_j) \wedge (b_i-b_j),(a_i+a_j) \wedge (b_i-b_j)} &\equiv \PresS{(a_i-a_j) \wedge (b_i+b_j),(a_i-a_j) \wedge (b_i+b_j)} \\
                                                              &\equiv 2 \SOmegaS{a_i \wedge b_j,b_i \wedge a_j}.
\end{align*}
\end{lemma}
\begin{proof}
Both equalities are proved the same way, so we will give the details for the first.  
Lemma \ref{lemma:thetaexpansion2} ($\Theta$-expansion II)
implies that $\PresS{(a_i+b_j) \wedge (b_i+a_j),(a_i+b_j) \wedge (b_i+a_j)}$ equals
\begin{equation}
\label{eqn:identifyomega1}
\ThetaS{a_i \wedge (b_i+a_j),(a_i+b_j) \wedge (b_i+a_j)} + \ThetaS{b_j \wedge (b_i+a_j),(a_i+b_j) \wedge (b_i+a_j)}
\end{equation}
and $\PresS{(a_i-b_j) \wedge (b_i-a_j),(a_i-b_j) \wedge (b_i-a_j)}$ equals
\begin{equation}
\label{eqn:identifyomega2}
\ThetaS{a_i \wedge (b_i-a_j),(a_i-b_j) \wedge (b_i-a_j)} + \PresS{(-b_j) \wedge (b_i-a_j),(a_i-b_j) \wedge (b_i-a_j)}.
\end{equation}
It is enough to prove that each term in \eqref{eqn:identifyomega1} and \eqref{eqn:identifyomega2} is equivalent to 
$-\SOmegaS{a_i \wedge a_j,b_i \wedge b_j} = \SOmegaS{a_i \wedge a_j,b_j \wedge b_i}$.  For
\eqref{eqn:identifyomega1}, we proved in the beginning of the proof of 
Lemma \ref{lemma:omegawelldefined} that\footnote{That lemma only dealt with $i=1$ and $j=2$, but the proof
works in general.}
\begin{align*}
\OmegaIS{a_i \wedge a_j,b_j \wedge b_i}  &\equiv \ThetaS{a_i \wedge (b_i+a_j), (a_i+b_j) \wedge (b_i+a_j)},\\
\OmegaIVS{a_i \wedge a_j,b_j \wedge b_i} &\equiv \ThetaS{(b_i+a_j) \wedge b_j,(b_i+a_j) \wedge (a_i+b_j)} \\
                                          &=      \ThetaS{b_j \wedge (b_i+a_j),(a_i+b_j) \wedge (b_i+a_j)}.
\end{align*}
For \eqref{eqn:identifyomega2}, that same argument shows that
\begin{align*}
\OmegaIS{a_i \wedge (-a_j),(-b_j) \wedge b_i}  &\equiv \ThetaS{a_i \wedge (b_i-a_j), (a_i-b_j) \wedge (b_i-a_j)},\\
\OmegaIVS{a_i \wedge (-a_j),(-b_j) \wedge b_i} &\equiv \ThetaS{(b_i-a_j) \wedge (-b_j),(b_i-a_j) \wedge (a_i-b_j)} \\
                                               &=      \ThetaS{(-b_j) \wedge (b_i-a_j),(a_i-b_j) \wedge (b_i-a_j)}.
\end{align*}
Lemma \ref{lemma:omegaambiguity} implies that $\OmegaS{a_i \wedge (-a_j),(-b_j) \wedge b_i} = \OmegaS{a_i \wedge a_j,b_j \wedge b_i}$,
so this is enough.
\end{proof}

\section{Symmetric kernel, symmetric version V: skeleton of rest of proof}
\label{section:presentationsym5}

We continue using all the notation from \S \ref{section:presentationsym1} -- \S \ref{section:presentationsym4}.
Recall that our goal in this part of the paper is to prove:

\newtheorem*{maintheorem:presentationsym}{Theorem \ref{maintheorem:presentationsym}}
\begin{maintheorem:presentationsym}
For\footnote{Note that this is our standing assumption in this part; see Assumption \ref{assumption:genussym}.} $g \geq 4$, the linearization map $\Phi\colon \fK_g^s \rightarrow \Sym^2((\wedge^2 H)/\Q)$ is an isomorphism.
\end{maintheorem:presentationsym}

We prove this using the three step proof technique outlined in \S \ref{section:prooftechnique}.

\subsection{Step 1}
In the previous four sections, we took the first step towards proving Theorem \ref{maintheorem:presentationsym}.  We
accomplished the following, which is a restatement with more details of Lemma \ref{lemma:symspans4}.
Recall that $\cB = \{a_1,b_1,\ldots,a_g,b_g\}$ is our fixed symplectic basis for $H_{\Z}$, which 
is endowed with the total order $\prec$ in which the indicated list is strictly increasing.

\begin{lemma}[Step 1]
\label{lemma:presentationsymstep1}
Let $S = S_1 \cup \cdots \cup S_4$, where the $S_i$ are:
\begin{align*}
S_1 &= \Set{$\SPresS{x \wedge y,z \wedge w}$}{$x,y,z,w \in \cB$, $x \prec y$, $z \prec w$, $\fc(x \wedge y,z \wedge w) = 0$},\\
S_2 &= \Set{$\SThetaS{a_i \wedge b_i,x \wedge b_i}$, $\SThetaS{a_i \wedge b_i,a_i \wedge y}$}{$1 \leq i \leq g$, $x,y \in \cB \setminus \{a_i,b_i\}$},\\
S_3 &= \Set{$\SLambdaS{a_i \wedge y,x \wedge b_i}$}{$1 \leq i \leq g$, $x,y \in \cB \setminus \{a_i,b_i\}$, $\omega(x,y)=0$},\\
S_4 &= \Set{$\SOmegaS{a_i \wedge a_j,b_i \wedge b_j}$, $\SOmegaS{a_i \wedge b_j,b_i \wedge a_j}$}{$1 \leq i < j \leq g$}.
\end{align*}
Then the restriction of $\Phi$ to $\Span{S}$ is an isomorphism.
\end{lemma}

\subsection{Step 2}
The next step is:

\begin{lemma}[Step 2]
\label{lemma:presentationsymstep2}
The $\Sp_{2g}(\Z)$-orbit of the set $S$ from Lemma \ref{lemma:presentationsymstep1} spans
$\fK_g^s$.
\end{lemma}
\begin{proof}
Using our original generating set for $\fK_g^s$ from Definition \ref{definition:kgsym} together
with Lemma \ref{lemma:generationomega}, we see that $\fK_g^s$ is generated by elements of the form
$\PresS{a \wedge b,a' \wedge b'}$, where $a \wedge b$ and $a' \wedge b'$ are symplectic pairs
such that $\Span{a,b}$ and $\Span{a',b'}$ are orthogonal.  The group $\Sp_{2g}(\Z)$
acts transitively on such elements.  The set $S$ contains many such elements; for instance,
it contains $\SPresS{a_1 \wedge b_1,a_2 \wedge b_2}$.  It follows that $\Sp_{2g}(\Z)$-orbit of $S$
spans $\fK_g^s$. 
\end{proof}

\subsection{Step 3}
The following lemma completes the proof of Theorem \ref{maintheorem:presentationsym}.

\begin{lemma}[Step 3]
\label{lemma:presentationsymstep3}
Let $S \subset \fK_g^s$ be the set from Lemma \ref{lemma:presentationsymstep1}.  Then
the action of $\Sp_{2g}(\Z)$ on $\fK_g^s$ takes $\Span{S}$ to itself.
By Lemma \ref{lemma:presentationsymstep2} this implies that $\Span{S} = \fK_g^s$,
and thus by Lemma \ref{lemma:presentationsymstep1} that $\Phi$ is an isomorphism.
\end{lemma}

\noindent
We begin the proof of Lemma \ref{lemma:presentationsymstep3}, with the main steps postponed
to the next 4 sections.

\begin{proof}[Beginning of proof of Lemma \ref{lemma:presentationsymstep3}]
Corollary \ref{corollary:gensp} says that $\Sp_{2g}(\Z)$ is generated as a monoid by
$\SymSp_g \cup \{X_1,X_1^{-1},Y_{12}\}$.
Let $f \in \SymSp_g \cup \{X_1,X_1^{-1},Y_{12}\}$ and let $s \in S$.  It is enough
to prove that $f(s) \in \Span{S}$.

The first case is $f \in \SymSp_g$.  Assume first that $s \in S_1$.  Write
$s = \SPresS{x \wedge y,z \wedge w}$ with $x,y,z,w \in \cB$ satisfying $x \prec y$ and $z \prec w$ and
$\fc(x \wedge y,z \wedge w) = 0$.  There exist $x',y',z',w' \in \cB$ and signs
$\epsilon_1,\ldots,\epsilon_4 \in \{\pm 1\}$ such that
\[f(x) = \epsilon_1 x',\ f(y) = \epsilon_2 y',\ f(z) = \epsilon_3 z',\ f(w) = \epsilon_4 w'.\]
We then have that $f(s)$ equals
\[\SPresS{(\epsilon_1 x') \wedge (\epsilon_2 y'),(\epsilon_3 z') \wedge (\epsilon_4 w')} = \epsilon_1 \cdots \epsilon_4 \SPresS{x' \wedge y',z' \wedge w'} \in \Span{S_1}.\]
We remark that it is possible that either $y' \prec x'$ or $w' \prec z'$, so $\SPresS{x' \wedge y',z' \wedge w'}$ itself might not lie in $S_1$; however,
either it or $(-1)$ times it lies in $S_1$.
The cases where $s \in S_2$ or $s \in S_3$ or $s \in S_4$ are handled the same way, using the sign
rules for $\Theta$- and $\Lambda$- and $\Omega$-elements discussed in \S \ref{section:thetasigns}
and \S \ref{section:lambdasigns} and \S \ref{section:omegasigns}.

We now must deal with the cases where $f \in \{X_1,X_1^{-1},Y_{12}\}$.  These calculations are lengthy, so
we postpone them.  We deal with $s \in S_1$ in \S \ref{section:presentationsym6}, with
$s \in S_2$ in \S \ref{section:presentationsym7}, with $s \in S_3$ in \S \ref{section:presentationsym8},
and finally with $s \in S_4$ in \S \ref{section:presentationsym9}.
\end{proof}

\section{Symmetric kernel, symmetric version VI: closure of \texorpdfstring{$S_1$}{S1}}
\label{section:presentationsym6}

We continue using all the notation from \S \ref{section:presentationsym1} -- \S \ref{section:presentationsym5}.
In this section, we continue the proof of Lemma \ref{lemma:presentationsymstep3} by proving that for
$f \in \{X_1,X_1^{-1},Y_{12}\}$ and $s \in S_1$, we have $f(s) \in \Span{S}$.  

\subsection{X-closure}
Recall that $X_1 \in \Sp_{2g}(\Z)$ takes $a_1$ to $a_1+b_1$ and fixes all other generators in $\cB$.
We start with:

\begin{lemma}
\label{lemma:presentationsymxs1}
Let 
\[s \in S_1 = \Set{$\SPresS{x \wedge y,z \wedge w}$}{$x,y,z,w \in \cB$, $x \prec y$, $z \prec w$, $\fc(x \wedge y,z \wedge w) = 0$}.\]
Then $X_1^{\epsilon}(s) \in \Span{S}$ for $\epsilon \in \{\pm 1\}$.
\end{lemma}
\begin{proof}
Write $s = \SPresS{x \wedge y,z \wedge w}$.
There is nothing to prove if $a_1 \notin \{x,y,z,w\}$.  What is more, the result is obvious if one of
$\{x,y\}$ and $\{z,w\}$ contains $a_1$ and the other does not.  For instance, we
have (c.f.\ \S \ref{section:obviousbluesym})
\[X_1^{\epsilon}(\SPresS{a_1 \wedge a_2,a_3 \wedge b_3}) = \SPresS{(a_1+\epsilon b_1) \wedge a_2,a_3 \wedge b_3} \in \Span{S_1}.\]
Since $x \prec y$ and $z \prec w$, we have reduced ourselves to 
$s = \SPresS{a_1 \wedge y,a_1 \wedge w}$ with $y,w \in \cB \setminus \{a_1\}$ satisfying $\fc(a_1 \wedge y,a_1 \wedge w) = 0$.  The condition $\fc(a_1 \wedge y,a_1 \wedge w) = 0$ implies that one
of the following holds:
\begin{itemize}
\item $y,w \in \cB \setminus \{a_1,b_1\}$ and $\omega(y,w) = 0$; or
\item $y = w = b_1$.
\end{itemize}
For instance, if $w \in \cB \setminus \{a_1,b_1\}$ then
\[\fc(a_1 \wedge b_1,a_1 \wedge w) = a_1 \Cdot w \neq 0.\]
We deal with the above cases separately.  Let $\equiv$ denote equality modulo 
$\Span{S}$, so our goal is to prove that $X_1^{\epsilon}(s) \equiv 0$.

\begin{case}{1}
$s = \SPresS{a_1 \wedge b_1,a_1 \wedge b_1}$.
\end{case}

\noindent
We have
\[X_1^{\epsilon}(\SPresS{a_1 \wedge b_1,a_1 \wedge b_1}) = \PresS{(a_1+\epsilon b_1) \wedge b_1,(a_1+\epsilon b_1) \wedge b_1} = \SPresS{a_1 \wedge b_1,a_1 \wedge b_1} \equiv 0.\]

\begin{case}{2}
$s = \SPresS{a_1 \wedge y,a_1 \wedge w}$ with $y,w \in \cB \setminus \{a_1,b_1\}$ satisfying $\omega(y,w)=0$.
\end{case}

\noindent
By Lemma \ref{lemma:lambdaexpansion3} ($\Lambda$-expansion III),
we have that $X_1^{\epsilon}(\SPresS{a_1 \wedge y,a_1 \wedge w})$ equals
\[\PresS{(a_1+\epsilon b_1) \wedge y,(a_1+\epsilon b_1) \wedge w} = \LambdaS{a_1 \wedge y,(a_1+\epsilon b_1) \wedge w} + \epsilon \LambdaS{b_1 \wedge y,(a_1+\epsilon b_1) \wedge w}.\]
By Lemma \ref{lemma:lambdabilinear2} ($\Lambda$-bilinearity II), this equals
\[\SPresS{a_1 \wedge y,a_1 \wedge w} + \epsilon \SLambdaS{a_1 \wedge y,b_1 \wedge w}
+\epsilon \SLambdaS{b_1 \wedge y,a_1 \wedge w} + \epsilon^2 \SPresS{b_1 \wedge y,b_1 \wedge w} \equiv 0.\qedhere\]
\end{proof}

\subsection{1-2 swaps}
Recall that $Y_{12} \in \Sp_{2g}(\Z)$ takes $a_1$ to $a_1+b_2$ and $a_2$ to $a_2+b_1$ and fixes all other generators in $\cB$.
For this element, the indices $1$ and $2$ are special.  The {\em $1$-$2$ swap} is the element
$\sigma \in \SymSp_g$ such that
\[\sigma(a_1) = a_2,\ \sigma(b_1) = b_2,\ \sigma(a_2) = a_1,\ \sigma(b_2) = b_1\]
and such that $\sigma$ fixes all other elements of $\cB$.  It satisfies the following:

\begin{lemma}
\label{lemma:12swapsym}
Let $\sigma \in \SymSp_g$ be the $1$-$2$ swap and let $z \in \fK_g^s$.  Then $Y_{12}(z) \in \Span{S}$ if and only
if $Y_{12}(\sigma(z)) \in \Span{S}$.
\end{lemma}
\begin{proof}
The element $\sigma$ commutes with $Y_{12}$, so
\begin{equation}
\label{eqn:12swapsym}
Y_{12}(\sigma(z)) = \sigma(Y_{12}(z)).
\end{equation}
We already proved at the end of \S \ref{section:presentationsym5} that
$\SymSp_g$ takes $\Span{S}$ to itself.  This holds in particular for $\sigma$.
In light of \eqref{eqn:12swapsym}, the lemma follows.
\end{proof}

\subsection{Y-closure}
We now prove:

\begin{lemma}
\label{lemma:presentationsymys1}
Let 
\[s \in S_1 = \Set{$\SPresS{x \wedge y,z \wedge w}$}{$x,y,z,w \in \cB$, $x \prec y$, $z \prec w$, $\fc(x \wedge y,z \wedge w) = 0$}.\]
Then $Y_{12}(s) \in \Span{S}$.
\end{lemma}
\begin{proof}
Write $s = \SPresS{x \wedge y,z \wedge w}$.  To avoid having the deal with even more special cases, we relax
the conditions $x \prec y$ and $z \prec w$ to $x \neq y$ and $z \neq w$.

There are a large number of cases to consider: each of $x$ and $y$ and $z$ and $w$ can either lie in
$\{a_1,b_1,a_2,b_2\}$, or they can be some other element of $\cB$.  The condition $\fc(x \wedge y,z \wedge w)=0$ eliminates
some of these, but there are still far too many cases.  To cut this down to something reasonable, we use
the following three symmetries, which we will call the {\em $Y_{12}$-symmetries}:
\begin{itemize}
\item Flipping $x \wedge y$ and $z \wedge w$ does not change $s$.
\item Flipping $x$ and $y$ multiplies $s$ by $-1$, and thus does not change the truth of the lemma.  Similarly, we
can flip $z$ and $w$.
\item Finally, by Lemma \ref{lemma:12swapsym} we can apply a $1$-$2$ swap to $s$ without
changing whether or not it lies in $\Span{S}$.
\end{itemize}
Using these, we will reduce ourselves to six cases as follows.

If $a_1,a_2 \notin \{x,y,z,w\}$, then there is nothing to prove.  Otherwise, after
performing a sequence of $Y_{12}$-symmetries we can assume that $x$ is either $a_1$ or $a_2$.  
Applying a $1$-$2$ swap if necessary, we can assume that $x = a_1$, so $s = \SPresS{a_1 \wedge y,z \wedge w}$.

If $a_2 \notin \{y,z,w\}$, then
in most cases we have $Y_{12}(\SPresS{a_1 \wedge y,z \wedge w}) \in \Span{S_1}$.  For instance, we have (c.f.\ \S \ref{section:obviousbluesym})
\[Y_{12}(\SPresS{a_1 \wedge a_3, a_1 \wedge b_2}) = \SPresS{(a_1+b_2) \wedge a_3,(a_1+b_2) \wedge b_2} \in \Span{S_1}.\]
Up to flipping $z$ and $w$, the only case where $a_2 \notin \{y,z,w\}$ and 
$Y_{12}(\SPresS{a_1 \wedge y,z \wedge w}) \notin \Span{S_1}$ is
$\SPresS{a_1 \wedge b_1,a_1 \wedge b_1}$, which is Case \ref{case:symys1.1} below.

It remains to enumerate the cases where $a_2 \in \{y,z,w\}$.  We start by enumerating
the cases where $y = a_2$, so $s = \SPresS{a_1 \wedge a_2,z \wedge w}$.
If $a_1,a_2 \notin \{z,w\}$, then $Y_{12}(\SPresS{a_1 \wedge a_2,z \wedge w}) \in \Span{S_1}$.  For instance,
\[Y_{12}(\SPresS{a_1 \wedge a_2, a_3 \wedge b_3}) = \SPresS{(a_1+b_2) \wedge (b_1+a_2),a_3 \wedge b_3} \in \Span{S_1}.\]
If instead either $a_1$ or $a_2$ lie in $\{z,w\}$, then up to $Y_{12}$-symmetries (including possibly a $1$-$2$ swap) 
we can assume that $z = a_1$.  The condition
$\fc(a_1 \wedge a_2,a_1 \wedge w) = 0$ implies that $\omega(a_1,w) = \omega(a_2,w) = 0$, so there are two cases:
$s=\SPresS{a_1 \wedge a_2,a_1 \wedge a_2}$, and $s=\SPresS{a_1 \wedge a_2,a_1 \wedge w}$ with $w \in \cB \setminus \{a_1,b_1,a_2,b_2\}$.
These are Cases \ref{case:symys1.2} and \ref{case:symys1.3} below.  This
completes our enumeration of the cases where $y = a_2$.

The remaining cases are $s = \SPresS{a_1 \wedge y,z \wedge w}$ with $y \neq a_2$ but 
$a_2 \in \{z,w\}$.  After possibly flipping $z$ and $w$,
we can assume that $z = a_2$.  In other words, we have reduced ourselves to enumerating
the cases where $s = \SPresS{a_1 \wedge y,a_2 \wedge w}$ with $y \neq a_2$.
Up to $Y_{12}$-symmetries, we have already handled the case where $w = a_1$, so we
can also assume that $w \neq a_1$.  The condition
$\fc(a_1 \wedge y,a_2 \wedge w) = 0$ implies that $y \neq b_2$ and $w \neq b_1$.  There are now four cases:
\begin{itemize}
\item $\SPresS{a_1 \wedge b_1,a_2 \wedge b_2}$, which is Case \ref{case:symys1.4} below.
\item $\SPresS{a_1 \wedge b_1,a_2 \wedge w}$ with $w \in \cB \setminus \{a_1,b_1,a_2,b_2\}$ and
$\SPresS{a_1 \wedge y,a_2 \wedge b_2}$ with $w \in \cB \setminus \{a_1,b_1,a_2,b_2\}$.  These differ
by $Y_{12}$-symmetries, so we only need to deal with the first.  This is Case \ref{case:symys1.5} below.
\item $\SPresS{a_1 \wedge y,a_2 \wedge w}$ with $y,w \in \cB \setminus \{a_1,b_1,a_2,b_2\}$ satisfying
$\omega(y,w) = 0$.  This is Case \ref{case:symys1.6} below.
\end{itemize}
It remains to deal with all these cases.  Let $\equiv$ denote equality modulo $\Span{S}$, so our goal is to 
prove that $Y_{12}(s) \equiv 0$.

\begin{case}{1}
\label{case:symys1.1}
$\SPresS{a_1 \wedge b_1,a_1 \wedge b_1}$.
\end{case}

\noindent
By Lemma \ref{lemma:thetaexpansion1} ($\Theta$-expansion I), the element
$Y_{12}(\SPresS{a_1 \wedge b_1,a_2 \wedge b_1}) = \PresS{(a_1 + b_2) \wedge b_1,(a_1+b_2) \wedge b_1}$ equals
\[\SPresS{a_1 \wedge b_1,a_1 \wedge b_1} + 2 \SThetaS{a_1 \wedge b_1,b_2 \wedge b_1} + \SPresS{b_2 \wedge b_1,b_2 \wedge b_1} \equiv 0.\]

\begin{case}{2}
\label{case:symys1.2}
$\SPresS{a_1 \wedge a_2,a_1 \wedge a_2}$.
\end{case}

\noindent
Lemma \ref{lemma:identifyomega} implies that
$Y_{12}(\SPresS{a_1 \wedge a_2,a_1 \wedge a_2}) = \PresS{(a_1+b_2) \wedge (b_1+a_2),(a_1+b_2) \wedge (b_1+a_2)}$
is equivalent to $-2\SOmegaS{a_1 \wedge a_2,b_1 \wedge b_2} \equiv 0$.

\begin{case}{3}
\label{case:symys1.3}
$\SPresS{a_1 \wedge a_2,a_1 \wedge w}$ with $w \in \cB \setminus \{a_1,b_1,a_2,b_2\}$.
\end{case}

\noindent
Lemma \ref{lemma:thetaexpansion2} ($\Theta$-expansion II) implies that
\[Y_{12}(\SPresS{a_1 \wedge a_2,a_1 \wedge w}) = \PresS{(a_1+b_2) \wedge (b_1+a_2),(a_1+b_2) \wedge w}\]
equals
\[\ThetaS{(a_1+b_2) \wedge b_1,(a_1+b_2) \wedge w} + \ThetaS{(a_1+b_2) \wedge a_2,(a_1+b_2) \wedge w}.\]
Lemma \ref{lemma:lambdaexpansion1} ($\Lambda$-expansion I) implies that
$\ThetaS{(a_1+b_2) \wedge b_1,(a_1+b_2) \wedge w}$ equals
\[\SThetaS{a_1 \wedge b_1,a_1 \wedge w} + \SPresS{a_1 \wedge b_1,b_2 \wedge w} + \SLambdaS{b_2 \wedge b_1,a_1 \wedge w} + \SPresS{b_2 \wedge b_1,b_2 \wedge w} \equiv 0.\]
Similarly, $\ThetaS{(a_1+b_2) \wedge a_2,(a_1+b_2) \wedge w} \equiv 0$.  The case follows.

\begin{case}{4}
\label{case:symys1.4}
$\SPresS{a_1 \wedge b_1,a_2 \wedge b_2}$.
\end{case}

\noindent
In $(\wedge^2 H)/\Q$, we have
\[a_2 \wedge b_2 = -a_1 \wedge b_1 -\sum\nolimits_{i=3}^g a_i \wedge b_i.\]
Plugging this into $\SPresS{a_1 \wedge b_1,a_2 \wedge b_2}$, we see that
\[\SPresS{a_1 \wedge b_1,a_2 \wedge b_2} = -\SPresS{a_1 \wedge b_1,a_1 \wedge b_1} -\sum\nolimits_{i=3}^g \SPresS{a_1 \wedge b_1,a_i \wedge b_i}.\]
It follows that $Y_{12}(\SPresS{a_1 \wedge b_1,a_2 \wedge b_2})$ equals
\[-Y_{12}(\SPresS{a_1 \wedge b_1,a_1 \wedge b_1}) - \sum\nolimits_{i=3}^g Y_{12}(\SPresS{a_1 \wedge b_1,a_i \wedge b_i}).\]
We proved that $Y_{12}(\SPresS{a_1 \wedge b_1,a_1 \wedge b_1}) \equiv 0$ in Case \ref{case:symys1.1}, and for
$3 \leq i \leq g$ we have
\[Y_{12}(\SPresS{a_1 \wedge b_1,a_i \wedge b_i}) = \SPresS{(a_1+b_2) \wedge b_1,a_i \wedge b_i} \equiv 0;\]
see \S \ref{section:obviousbluesym}.  The case follows.

\begin{case}{5}
\label{case:symys1.5}
$\SPresS{a_1 \wedge b_1,a_2 \wedge w}$ with $w \in \cB \setminus \{a_1,b_1,a_2,b_2\}$.
\end{case}

\noindent
To simplify our notation, we will explain how to deal with $w = a_3$.  The other cases are similar.  Lemma \ref{lemma:liftsymr2} ($\Theta$-symplectic basis)
implies that
\[\SPresS{a_1 \wedge b_1,a_2 \wedge a_3} + \SThetaS{a_2 \wedge b_2,a_2 \wedge a_3} + \SThetaS{a_3 \wedge b_3,a_2 \wedge a_3} + \sum\nolimits_{i=4}^g \SPresS{a_i \wedge b_i,a_2 \wedge a_3} = 0.\]
It follows that $Y_{12}(\SPresS{a_1 \wedge b_1,a_2 \wedge a_3})$ equals
\begin{align*}
&-Y_{12}(\SThetaS{a_2 \wedge b_2,a_2 \wedge a_3}) - Y_{12}(\SThetaS{a_3 \wedge b_3,a_2 \wedge a_3}) - \sum\nolimits_{i=4}^g Y_{12}(\SPresS{a_i \wedge b_i,a_2 \wedge a_3}) \\
=&-\ThetaS{(a_2+b_1) \wedge b_2,(a_2+b_1) \wedge a_3} \\
 &- \SThetaS{a_3 \wedge b_3,(a_2+b_1) \wedge a_3} - \sum\nolimits_{i=4}^g \SPresS{a_i \wedge b_i,(a_2+b_1) \wedge a_3} \\
\equiv &-\ThetaS{(a_2+b_1) \wedge b_2,(a_2+b_1) \wedge a_3}.
\end{align*}
The last $\equiv$ uses Lemma \ref{lemma:thetalinear} ($\Theta$-linearity) to show that
\[\SThetaS{a_3 \wedge b_3,(a_2+b_1) \wedge a_3} \in \Span{S_2}.\]
This reduces us to proving
that $\ThetaS{(a_2+b_1) \wedge b_2,(a_2+b_1) \wedge a_3} \equiv 0$.  For this,
Lemma \ref{lemma:lambdaexpansion1} ($\Lambda$-expansion I) implies that
$\ThetaS{(a_2+b_1) \wedge b_2,(a_2+b_1) \wedge a_3}$ equals
\[\SThetaS{a_2 \wedge b_2,a_2 \wedge a_3} + \SPresS{a_2 \wedge b_2,b_1 \wedge a_3} + \SLambdaS{b_1 \wedge b_2,a_2 \wedge a_3} + \SPresS{b_1 \wedge b_2,b_1 \wedge a_3} \equiv 0.\]

\begin{case}{6}
\label{case:symys1.6}
$\SPresS{a_1 \wedge y,a_2 \wedge w}$ with $y,w \in \cB \setminus \{a_1,b_1,a_2,b_2\}$ satisfying $\omega(y,w) = 0$.
\end{case}

\noindent
Lemma \ref{lemma:lambdaexpansion2} ($\Lambda$-expansion II) implies that
$Y_{12}(\SPresS{a_1 \wedge y,a_2 \wedge w}) = \PresS{(a_1+b_2) \wedge y,(b_1+a_2) \wedge w}$ equals
\[\LambdaS{a_1 \wedge y,(b_1+a_2) \wedge w} + \LambdaS{b_2 \wedge y,(b_1+a_2) \wedge w}.\]
Lemma \ref{lemma:lambdabilinear1} ($\Lambda$-bilinearity I) shows that
\[\LambdaS{a_1 \wedge y,(b_1+a_2) \wedge w} = \SLambdaS{a_1 \wedge y,b_1 \wedge w} + \SPresS{a_1 \wedge y,a_2 \wedge w} \equiv 0.\]
Similarly, $\LambdaS{b_2 \wedge y,(b_1+a_2) \wedge w} \equiv 0$.  The case follows.
\end{proof}

\section{Symmetric kernel, symmetric version VII: closure of \texorpdfstring{$S_2$}{S2}}
\label{section:presentationsym7}

We continue using all the notation from \S \ref{section:presentationsym1} -- \S \ref{section:presentationsym5}.
In this section, we continue the proof of Lemma \ref{lemma:presentationsymstep3} by proving that for
$f \in \{X_1,X_1^{-1},Y_{12}\}$ and $s \in S_2$, we have $f(s) \in \Span{S}$.  

\subsection{X-closure}
Recall that $X_1 \in \Sp_{2g}(\Z)$ takes $a_1$ to $a_1+b_1$ and fixes all other generators in $\cB$.  
We start with:

\begin{lemma}
\label{lemma:presentationsymxs2}
Let 
\[s \in S_2 = \Set{$\SThetaS{a_i \wedge b_i,x \wedge b_i}$, $\SThetaS{a_i \wedge b_i,a_i \wedge y}$}{$1 \leq i \leq g$, $x,y \in \cB \setminus \{a_i,b_i\}$}.\]
Then $X_1^{\epsilon}(s) \in \Span{S}$ for $\epsilon \in \{\pm 1\}$.
\end{lemma}
\begin{proof}
The lemma is trivial if $X_1$ fixes $s$.  The remaining cases are as follows.  Let $\equiv$ denote equality modulo
$\Span{S}$, so our goal is to prove that $X_1^{\epsilon}(s) \equiv 0$.

\begin{case}{1}
$s = \SThetaS{a_1 \wedge b_1,x \wedge b_1}$ with $x \in \cB \setminus \{a_1,b_1\}$.
\end{case}

\noindent
Lemma \ref{lemma:thetabilinear2} ($\Theta$-bilinearity II) implies that
\[X_1^{\epsilon}(\SThetaS{a_1 \wedge b_1,x \wedge b_1}) = \ThetaS{(a_1+\epsilon b_1) \wedge b_1,x \wedge b_1} = \SThetaS{a_1 \wedge b_1,x \wedge b_1} \equiv 0.\]

\begin{case}{2}
$s = \SThetaS{a_1 \wedge b_1,a_1 \wedge y}$ with $y \in \cB \setminus \{a_1,b_1\}$.
\end{case}

\noindent
Lemma \ref{lemma:thetabilinear2} ($\Theta$-bilinearity II) implies that
\begin{align*}
X_1^{\epsilon}(\SThetaS{a_1 \wedge b_1,a_1 \wedge y}) &= \ThetaS{(a_1+\epsilon b_1) \wedge b_1,(a_1+\epsilon b_1) \wedge y} \\
                                           &= \SThetaS{a_1 \wedge b_1,a_1 \wedge y} + \epsilon \SThetaS{a_1 \wedge b_1,b_1 \wedge y} \equiv 0.
\end{align*}

\begin{case}{3}
$s = \SThetaS{a_i \wedge b_i,a_1 \wedge b_i}$ or $s = \SThetaS{a_i \wedge b_i,a_i \wedge a_1}$ for some $2 \leq i \leq g$.
\end{case}

\noindent
Both cases are handled identically, so we will give the details for $s = \SThetaS{a_i \wedge b_i,a_1 \wedge b_i}$.
By Lemma \ref{lemma:thetalinear} ($\Theta$-linearity), we have that
$X_1^{\epsilon}(\SThetaS{a_i \wedge b_i,a_1 \wedge b_i}) = \ThetaS{a_i \wedge b_i,(a_1+\epsilon b_1) \wedge b_i}$ equals
\[\SThetaS{a_i \wedge b_i,a_1 \wedge b_i} + \epsilon \SThetaS{a_i \wedge b_i,b_1 \wedge b_i} \equiv 0.\qedhere\]
\end{proof}

\subsection{Y-closure}
Recall that $Y_{12} \in \Sp_{2g}(\Z)$ takes $a_1$ to $a_1+b_2$ and $a_2$ to $a_2+b_1$ and fixes all other generators in $\cB$.
We next prove:

\begin{lemma}
\label{lemma:presentationsymys2}
Let 
\[s \in S_2 = \Set{$\SThetaS{a_i \wedge b_i,x \wedge b_i}$, $\SThetaS{a_i \wedge b_i,a_i \wedge y}$}{$1 \leq i \leq g$, $x,y \in \cB \setminus \{a_i,b_i\}$}.\]
Then $Y_{12}(s) \in \Span{S}$. 
\end{lemma}
\begin{proof}
For $i \geq 3$, Lemma \ref{lemma:thetalinear} ($\Theta$-linearity) implies
that this holds for $s = \SThetaS{a_i \wedge b_i,x \wedge b_i}$ and $s = \SThetaS{a_i \wedge b_i,a_i \wedge y}$ with
$x,y \in \cB \setminus \{a_i,b_i\}$.  For instance,
\begin{align*}
Y_{12}(\SThetaS{a_3 \wedge b_3,a_1 \wedge b_3}) &= \ThetaS{a_3 \wedge b_3,(a_1+b_2) \wedge b_3} \\
                                                &= \SThetaS{a_3 \wedge b_3,a_1 \wedge b_3} + \SThetaS{a_3 \wedge b_3,b_2 \wedge b_3} \in \Span{S_2}.
\end{align*}
The remaining cases are when $i=1$ and $i=2$.  Applying a $1$-$2$ swap as described in Lemma \ref{lemma:12swapsym},
it is enough to deal with the case $i=1$.  We divide this into the following cases.  Let
$\equiv$ denote equality modulo $\Span{S}$, so our goal is to prove that
$Y_{12}(s) \equiv 0$. 

\begin{case}{1}
$s = \SThetaS{a_1 \wedge b_1,x \wedge b_1}$ with $x \in \cB \setminus \{a_1,b_1,a_2\}$.
\end{case}

\noindent
Lemma \ref{lemma:thetabilinear1} ($\Theta$-bilinearity I) implies that $Y_{12}(\SThetaS{a_1 \wedge b_1,x \wedge b_1}) = \ThetaS{(a_1+b_2) \wedge b_1,x \wedge b_1})$
equals
\[\SThetaS{a_1 \wedge b_1,x \wedge b_1} + \SPresS{b_2 \wedge b_1,x \wedge b_1} \equiv 0.\]

\begin{case}{2}
$s = \SThetaS{a_1 \wedge b_1,a_1 \wedge y}$ with $y \in \cB \setminus \{a_1,b_1,a_2\}$.
\end{case}

\noindent
By Lemma \ref{lemma:lambdaexpansion1} ($\Lambda$-expansion I), we have that
\[Y_{12}(\SThetaS{a_1 \wedge b_1,a_1 \wedge y}) = \ThetaS{(a_1+b_2) \wedge b_1,(a_1+b_2) \wedge y}\]
equals
\[\SThetaS{a_1 \wedge b_1,a_1 \wedge y} + \SPresS{a_1 \wedge b_1,b_2 \wedge y} + \SLambdaS{b_2 \wedge b_1,a_1 \wedge y} + \SPresS{b_2 \wedge b_1,b_2 \wedge y} \equiv 0.\]

\begin{case}{3}
$s = \SThetaS{a_1 \wedge b_1,a_2 \wedge b_1}$.
\end{case}

\noindent
By the definition of $\Theta$-elements (Definition \ref{definition:theta}), we have that
\[Y_{12}(\SThetaS{a_1 \wedge b_1,a_2 \wedge b_1}) = \ThetaS{(a_1+b_2) \wedge b_1,(b_1+a_2) \wedge b_1}\]
equals $1/2$ times
\begin{align*}
&\PresS{(a_1+b_2+b_1+a_2) \wedge b_1,(a_1+b_2+b_1+a_2) \wedge b_1} - \PresS{(a_1+b_2) \wedge b_1,(a_1+b_2) \wedge b_1} \\
&- \SPresS{(b_1+a_2) \wedge b_1,(b_1+a_2) \wedge b_1} \\
\equiv&\PresS{(a_1+b_2+a_2) \wedge b_1,(a_1+b_2+a_2) \wedge b_1} - Y_{12}(\SPresS{a_1 \wedge b_1,a_1 \wedge b_1}) \\
\equiv&\PresS{(a_1+b_2+a_2) \wedge b_1,(a_1+b_2+a_2) \wedge b_1}.
\end{align*}
The last $\equiv$ uses the fact that we have already proved that $Y_{12}(t) \equiv 0$ for
$t \in S_1$ (Lemma \ref{lemma:presentationsymys1}).
Lemma \ref{lemma:thetaexpansion1} ($\Theta$-expansion I) along with Lemma \ref{lemma:thetalinear} ($\Theta$-linearity)
implies that this equals
\begin{align*}
&\SPresS{a_1 \wedge b_1,a_1 \wedge b_1} + 2\ThetaS{a_1 \wedge b_1,(b_2+a_2) \wedge b_1} + \PresS{(b_2+a_2) \wedge b_1,(b_2+a_2) \wedge b_1} \\
\equiv& 2\SThetaS{a_1 \wedge b_1,b_2 \wedge b_1} + 2\SThetaS{a_1 \wedge b_1,a_2 \wedge b_1} + \PresS{(b_2+a_2) \wedge b_1,(b_2+a_2) \wedge b_1} \\
\equiv& \PresS{(a_2+b_2) \wedge b_1,(a_2+b_2) \wedge b_1}.
\end{align*}
Applying Lemma \ref{lemma:lambdaexpansion3} ($\Lambda$-expansion III) and then
Lemma \ref{lemma:lambdabilinear2} ($\Lambda$-bilinearity II), this equals
\begin{align*}
&\LambdaS{a_2 \wedge b_1,(a_2+b_2) \wedge b_1} + \LambdaS{b_2 \wedge b_1,(a_2+b_2) \wedge b_1} \\
=&\SPresS{a_2 \wedge b_1,a_2 \wedge b_1} + \SLambdaS{a_2 \wedge b_1,b_2 \wedge b_1} + \SLambdaS{b_2 \wedge b_1,a_2 \wedge b_1} + \SPresS{b_2 \wedge b_1,b_2 \wedge b_1} \equiv 0.
\end{align*}

\begin{case}{4}
$s = \SThetaS{a_1 \wedge b_1,a_1 \wedge a_2}$.
\end{case}

\noindent
We have
\[Y_{12}(\SThetaS{a_1 \wedge b_1,a_1 \wedge a_2}) = \ThetaS{(a_1+b_2) \wedge b_1,(a_1+b_2) \wedge (b_1+a_2)}.\]
By Lemma \ref{lemma:thetalinear} ($\Theta$-linearity), it is enough to prove that
\[\ThetaS{(a_1+b_2) \wedge b_1,(a_1+b_2) \wedge (-b_1-a_2)} \equiv 0.\]
By the definition of $\Theta$-elements (Definition \ref{definition:theta}), this equals
\begin{align*}
&\PresS{(a_1+b_2) \wedge (-a_2),(a_1+b_2) \wedge (-a_2)} - \PresS{(a_1+b_2) \wedge b_1,(a_1+b_2) \wedge b_1} \\
&- \PresS{(a_1+b_2) \wedge (-b_1-a_2),(a_1+b_2) \wedge (-b_1-a_2)} \\
=&\PresS{(a_1+b_2) \wedge a_2,(a_1+b_2) \wedge a_2} - Y_{12}(\SPresS{a_1 \wedge b_1,a_1 \wedge b_1}) - Y_{12}(\SPresS{a_1 \wedge a_2,a_1 \wedge b_2}) \\
\equiv &\PresS{(a_1+b_2) \wedge a_2,(a_1+b_2) \wedge a_2}.
\end{align*}
The last $\equiv$ uses the fact that we have already proved that $Y_{12}(t) \equiv 0$ for
$t \in S_1$ (Lemma \ref{lemma:presentationsymys1}).
Lemma \ref{lemma:thetaexpansion1} ($\Theta$-expansion I) says that this equals
\[\SPresS{a_1 \wedge a_2,a_1 \wedge a_2} + 2\SThetaS{a_1 \wedge a_2,b_2 \wedge a_2} + \SPresS{b_2 \wedge a_2,b_2 \wedge a_2} \equiv 0.\qedhere\]
\end{proof}

\section{Symmetric kernel, symmetric version VIII: closure of \texorpdfstring{$S_3$}{S3}}
\label{section:presentationsym8}

We continue using all the notation from \S \ref{section:presentationsym1} -- \S \ref{section:presentationsym5}.
In this section, we continue the proof of Lemma \ref{lemma:presentationsymstep3} by proving that for
$f \in \{X_1,X_1^{-1},Y_{12}\}$ and $s \in S_3$, we have $f(s) \in \Span{S}$.

\subsection{More general Lambda-elements}

Recall that
\[S_3 = \Set{$\SLambdaS{a_i \wedge y,x \wedge b_i}$}{$1 \leq i \leq g$, $x,y \in \cB \setminus \{a_i,b_i\}$, $\omega(x,y)=0$}.\]
Before we prove our main results, we prove:

\begin{lemma}
\label{lemma:alllambda}
Let $1 \leq i \leq g$ and let $x,y \in \Span{a_i,b_i}^{\perp}$ satisfy $\omega(x,y) = 0$.  Then
$\LambdaS{a_i \wedge y,x \wedge b_i} \in \Span{S}$.
\end{lemma}
\begin{proof}
Recall that $\fK_g^{s,\Lambda}[a_i \wedge -,- \wedge b_i]$ is the subspace of $\fK_g^s$ spanned by $\Lambda$-elements
as in the statement of the lemma.  It follows from Lemma \ref{lemma:lambdalinear2} (strong $\Lambda$-linearity) that 
$\fK_g^{s,\Lambda}[a_i \wedge -,- \wedge b_i]$ is spanned by three kinds of elements:
\begin{itemize}
\item Elements of the form $\SLambdaS{a_i \wedge y,x \wedge b_i}$ with $x,y \in \cB \setminus \{a_i,b_i\}$ satisfying
$\omega(x,y)=0$.  These are elements of $S_3$.
\item Elements of the form $\LambdaS{a_i \wedge (a_j + a_k),(b_j - b_k) \wedge b_i}$ for distinct $1 \leq j,k \leq g$
with $j,k \neq i$.  By Lemma \ref{lemma:omegalambda} ($\Lambda$ to $\Omega$),
these equal
\begin{align*}
&\SOmegaS{a_i \wedge a_j,b_j \wedge a_i} - \SOmegaS{a_i \wedge a_k,b_k \wedge b_i} + \SLambdaS{a_i \wedge a_k,b_j \wedge b_i} - \SLambdaS{a_i \wedge a_j,b_k \wedge b_k} \in \Span{S}.
\end{align*}
\item Elements of the form $\LambdaS{a_i \wedge (a_j + b_k),(b_j + a_k) \wedge b_i}$ for distinct $1 \leq j,k \leq g$  
with $j,k \neq i$.  Again, Lemma \ref{lemma:omegalambda} ($\Lambda$ to $\Omega$) implies that these lie in $\Span{S}$.\qedhere
\end{itemize}
\end{proof}

\subsection{X-closure}
Recall that $X_1 \in \Sp_{2g}(\Z)$ takes $a_1$ to $a_1+b_1$ and fixes all other generators in $\cB$.  
We now prove:

\begin{lemma}
\label{lemma:presentationsymxs3}
Let 
\[s \in S_3 = \Set{$\SLambdaS{a_i \wedge y,x \wedge b_i}$}{$1 \leq i \leq g$, $x,y \in \cB \setminus \{a_i,b_i\}$, $\omega(x,y)=0$}.\]
Then $X_1^{\epsilon}(s) \in \Span{S}$ for $\epsilon \in \{\pm 1\}$.
\end{lemma}
\begin{proof}
There are two cases.  Let $\equiv$ denote equality modulo $\Span{S}$, so our goal is to prove that $X_1^{\epsilon}(s) \equiv 0$.

\begin{case}{1}
$s = \SLambdaS{a_1 \wedge y,x \wedge b_1}$ with $x,y \in \cB \setminus \{a_1,b_1\}$ satisfying $\omega(x,y) = 0$.
\end{case}

\noindent
Lemma \ref{lemma:lambdabilinear2} ($\Lambda$-bilinearity II) says that $X_1^{\epsilon}(\SLambdaS{a_1 \wedge y,x \wedge b_1}) = \LambdaS{(a_1+\epsilon b_1) \wedge y,x \wedge b_1}$
equals
\[\SLambdaS{a_1 \wedge y,x \wedge b_1} + \epsilon \SPresS{b_1 \wedge y,x \wedge b_1} \equiv 0.\]

\begin{case}{2}
$s = \SLambdaS{a_i \wedge y,x \wedge b_i}$ with $x,y \in \cB \setminus \{a_i,b_i\}$ satisfying $\omega(x,y) = 0$.
\end{case}

\noindent
Lemma \ref{lemma:alllambda} implies that
\[X_1^{\epsilon}(\SLambdaS{a_i \wedge y,x \wedge b_i}) = \LambdaS{a_i \wedge X_1^{\epsilon}(y),X_1^{\epsilon}(x) \wedge b_i} \in \Span{S}.\qedhere\]
\end{proof}

\subsection{Y-closure}
Recall that $Y_{12} \in \Sp_{2g}(\Z)$ takes $a_1$ to $a_1+b_2$ and $a_2$ to $a_2+b_1$ and fixes all other generators in $\cB$.
We next prove:

\begin{lemma}
\label{lemma:presentationsymys3}
Let 
\[s \in S_3 = \Set{$\SLambdaS{a_i \wedge y,x \wedge b_i}$}{$1 \leq i \leq g$, $x,y \in \cB \setminus \{a_i,b_i\}$, $\omega(x,y)=0$}.\]
Then $Y_{12}(s) \in \Span{S}$. 
\end{lemma}
\begin{proof}
Recall that $Y_{12} \in \Sp_{2g}(\Z)$ takes $a_1$ to $a_1+b_2$ and $a_2$ to $a_2+b_1$ and fixes all other generators in $\cB$.
Write $s = \SLambdaS{a_i \wedge y,x \wedge b_i}$.  The proof is different when $i=1$, when $i=2$, and when $3 \leq i \leq g$.  However,
applying a $1$-$2$ swap as described in Lemma \ref{lemma:12swapsym} we can reduce the proof for $i=2$ to the proof
for $i=1$.  We divide the cases $i=1$ and $3 \leq i \leq g$ into the following cases.  Let $\equiv$ denote equality modulo $\Span{S}$,
so our goal is to prove that $Y_{12}(s) \equiv 0$.

\begin{case}{1}
$s = \SLambdaS{a_i \wedge y,x \wedge b_i}$ with $3 \leq i \leq g$ and $y,x \in \cB \setminus \{a_i,b_i\}$ satisfying $\omega(x,y) = 0$.
\end{case}

\noindent
Lemma \ref{lemma:alllambda} implies that 
\[Y_{12}(\SLambdaS{a_i \wedge y,x \wedge b_i}) = \LambdaS{a_i \wedge Y_{12}(y),Y_{12}(x) \wedge b_i} \in \Span{S}.\]

\begin{case}{2}
$s = \SLambdaS{a_1 \wedge y,x \wedge b_1}$ with $x,y \in \cB \setminus \{a_1,b_1,a_2\}$ such that $\omega(x,y)=0$.
\end{case}

\noindent
Lemma \ref{lemma:lambdabilinear1} ($\Lambda$-bilinearity I) implies that $Y_{12}(\SLambdaS{a_1 \wedge y,x \wedge b_1} = \LambdaS{(a_1+b_2) \wedge y,x \wedge b_1}$
equals
\[\SLambdaS{a_1 \wedge y,x \wedge b_1} + \SPresS{b_2 \wedge y,x \wedge b_1} \equiv 0.\]

\begin{case}{3}
$s = \SLambdaS{a_1 \wedge a_2,x \wedge b_1}$ with $x \in \cB \setminus \{a_1,b_1,a_2,b_2\}$.
\end{case}

\noindent
We have $Y_{12}(\SLambdaS{a_1 \wedge a_2,x \wedge b_1}) = \LambdaS{(a_1+b_2) \wedge (b_1+a_2),x \wedge  b_1}$.
By Lemma \ref{lemma:lambdalinear} ($\Lambda$-linearity), it is enough to prove that
\[\LambdaS{(a_1+b_2) \wedge (-b_1-a_2),x \wedge  b_1} \equiv 0.\]
By the definition of $\Lambda$-elements (Definition \ref{definition:lambda}), the element
$\LambdaS{(a_1+b_2) \wedge (-b_1-a_2),x \wedge  b_1} = \LambdaIS{(a_1+b_2) \wedge (-b_1-a_2),x \wedge  b_1}$ equals
\begin{align*}
&\ThetaS{(a_1+b_2) \wedge (-a_2),x \wedge (-a_2)} - \ThetaS{(a_1+b_2) \wedge b_1,x \wedge b_1} \\
&- \PresS{(a_1+b_2) \wedge (-a_2),x \wedge (-b_1-a_2)}\\
=&\ThetaS{(a_1+b_2) \wedge a_2,x \wedge a_2} - Y_{12}(\SThetaS{a_1 \wedge b_1,x \wedge b_1}) - \PresS{(a_1+b_2) \wedge a_2,x \wedge (b_1+a_2)} \\
\equiv &\ThetaS{(a_1+b_2) \wedge a_2,x \wedge a_2} - \PresS{(a_1+b_2) \wedge a_2,x \wedge (b_1+a_2)}.
\end{align*}
The last $\equiv$ uses the fact that we have already proved that $Y_{12}(t) \equiv 0$ for 
$t \in S_2$ (Lemma \ref{lemma:presentationsymys2}).
We must prove that both of these terms are equivalent to $0$.  Lemma \ref{lemma:thetabilinear1} ($\Theta$-bilinearity I) implies that the first
term $\ThetaS{(a_1+b_2) \wedge a_2,x \wedge a_2}$ equals
\[\SPresS{a_1 \wedge a_2,x \wedge a_2} + \SThetaS{b_2 \wedge a_2,x \wedge a_2} \equiv 0.\]
The second term $\PresS{(a_1+b_2) \wedge a_2,x \wedge (b_1+a_2)}$ equals
\begin{align*}
&\PresS{(a_1+b_2) \wedge (b_1+a_2),x \wedge (b_1+a_2)} - \PresS{(a_1+b_2) \wedge b_1,x \wedge (b_1+a_2)} \\
=&Y_{12}(\SPresS{a_1 \wedge a_2,x \wedge a_2}) - Y_{12}(\SPresS{a_1 \wedge b_1,x \wedge a_2}) \equiv 0.
\end{align*}
The $\equiv$ uses the fact that we have already proved that $Y_{12}(t) \equiv 0$ for all
$t \in S_1$ (Lemma \ref{lemma:presentationsymxs1}).

\begin{case}{4}
$s = \SLambdaS{a_1 \wedge y,a_2 \wedge b_1}$ with $y \in \cB \setminus \{a_1,b_1,a_2,b_2\}$.
\end{case}

\noindent
By the definition of $\Lambda$-elements (Definition \ref{definition:lambda}), 
$Y_{12}(\SLambdaS{a_1 \wedge y,a_2 \wedge b_1}) = \LambdaIIS{(a_1+b_2) \wedge y,(b_1+a_2) \wedge b_1}$ equals
\begin{align}
\label{eqn:symys3.4toprove}
&\ThetaS{(a_1+b_2+b_1+a_2) \wedge b_1,(a_1+b_2+b_1+a_2) \wedge y} \\
&- \ThetaS{(a_1+b_2) \wedge b_1,(a_1+b_2) \wedge y} - \PresS{(a_1+b_2+b_1+a_2) \wedge b_1,(b_1+a_2) \wedge y}.\nonumber
\end{align}
We must prove that each term in \eqref{eqn:symys3.4toprove} is equivalent to $0$.  The first is the most
difficult, so we save it for last.  The second term $\ThetaS{(a_1+b_2) \wedge b_1,(a_1+b_2) \wedge y}$ equals
\[Y_{12}(\SThetaS{a_1 \wedge b_1,a_1 \wedge y}) \equiv 0\]
since we have already proved that $Y_{12}(t) \equiv 0$ for all $t \in S_2$ (Lemma \ref{lemma:presentationsymxs2}).
The third term in \eqref{eqn:symys3.4toprove} is $\PresS{(a_1+b_2+b_1+a_2) \wedge b_1,(b_1+a_2) \wedge y}$,
which equals
\begin{align*}
&\PresS{(a_1+b_2) \wedge b_1,(b_1+a_2) \wedge y} + \SPresS{(b_1+a_2) \wedge b_1,(b_1+a_2) \wedge y} \\
\equiv &Y_{12}(\SPresS{a_1 \wedge b_1,a_2 \wedge y}) \equiv 0.
\end{align*}
Here again we used the fact that $Y_{12}(t) \equiv 0$ for all $t \in S_2$ (Lemma \ref{lemma:presentationsymxs2}).

It remains to deal with the first term $\ThetaS{(a_1+b_2+b_1+a_2) \wedge b_1,(a_1+b_2+b_1+a_2) \wedge y}$
in \eqref{eqn:symys3.4toprove}.  By Lemma \ref{lemma:thetabilinear2} ($\Theta$-bilinearity II), it equals
\begin{equation}
\label{eqn:symys3.4toprove2}
\ThetaS{(a_1+b_2+a_2) \wedge b_1,(a_1+b_2+a_2) \wedge y} + \ThetaS{(a_1+b_2+a_2) \wedge b_1,b_1 \wedge y}.
\end{equation}
We must show that both terms of \eqref{eqn:symys3.4toprove2} are equivalent to $0$.
Using Lemma \ref{lemma:lambdaexpansion1} ($\Lambda$-expansion I) along with
Lemma \ref{lemma:lambdalinear} ($\Lambda$-linearity), the first term of
\eqref{eqn:symys3.4toprove2} equals
\begin{align*}
&\LambdaS{a_1 \wedge y,(b_2+a_2) \wedge b_1}  + \SThetaS{a_1 \wedge b_1,a_1 \wedge y}
+ \SPresS{(a_1+b_2+a_2) \wedge b_1,(b_2+a_2) \wedge y} \\
\equiv &\SLambdaS{a_1 \wedge y,b_2 \wedge b_1} + \SLambdaS{a_1 \wedge y,a_2 \wedge b_1} \equiv 0.
\end{align*}
For the second term $\ThetaS{(a_1+b_2+a_2) \wedge b_1,b_1 \wedge y}$ of \eqref{eqn:symys3.4toprove2},
Lemma \ref{lemma:thetabilinear1} ($\Theta$-bilinearity I) implies that it equals
\[\SThetaS{a_1 \wedge b_1,b_1 \wedge y} + \SPresS{(b_2 + a_2) \wedge b_1,b_1 \wedge y} \equiv 0.\]

\begin{case}{5}
$s = \SLambdaS{a_1 \wedge a_2,a_2 \wedge b_1}$.
\end{case}

\noindent
By Lemma \ref{lemma:lambdalinear} ($\Lambda$-linearity), to prove
that 
\[Y_{12}(\SLambdaS{a_1 \wedge a_2,a_2 \wedge b_1}) = \LambdaS{(a_1+b_2) \wedge (b_1+a_2),(b_1+ a_2) \wedge b_1}\]
is equivalent to $0$ it is enough to prove that
\[\LambdaS{(a_1+b_2) \wedge (-b_1-a_2),(b_1+ a_2) \wedge b_1}\]
is equivalent to $0$.  By the definition of $\Lambda$-elements (Definition \ref{definition:lambda}), this
equals
\begin{align*}
  &\ThetaS{(a_1+b_2) \wedge (-a_2),(b_1+ a_2) \wedge (-a_2)} - \ThetaS{(a_1+b_2) \wedge b_1,(b_1+ a_2) \wedge b_1} \\
  &- \PresS{(a_1+b_2) \wedge (-a_2),(b_1+ a_2) \wedge (-b_1-a_2)}. \\
= &\ThetaS{(a_1+b_2) \wedge a_2,(b_1+ a_2) \wedge a_2} - Y_{12}(\SThetaS{a_1 \wedge b_1,a_2 \wedge b_1}) \\
\equiv &\ThetaS{a_2 \wedge (a_1+b_2),a_2 \wedge (b_1+a_2)}.
\end{align*}
Here we used the fact that $Y_{12}(t) \equiv 0$ for all $t \in S_2$ (Lemma \ref{lemma:presentationsymxs2}).
Expanding this using the definition of $\Theta$-elements (Definition \ref{definition:theta}), we get $1/2$ times
\begin{align*}
&\PresS{a_2 \wedge (a_1+b_2+b_1+a_2),a_2 \wedge (a_1+b_2+b_1+a_2)} - \PresS{a_2 \wedge (a_1+b_2),a_2 \wedge (a_1+b_2)}\\
&- \SPresS{a_2 \wedge (b_1+a_2),a_2 \wedge (b_1+a_2)}\\
\equiv &\PresS{a_2 \wedge (b_2+a_1+b_1),a_2 \wedge (b_2+a_1+b_1)} - \PresS{a_2 \wedge (b_2+a_1),a_2 \wedge (b_2+a_1)}
\end{align*}
We must prove that both of these terms are equivalent to $0$.  For the first term
$\PresS{a_2 \wedge (b_2+a_1+b_1),a_2 \wedge (b_2+a_1+b_1)}$, Lemma \ref{lemma:thetaexpansion1} ($\Theta$-expansion I)
along with
Lemma \ref{lemma:thetalinear} ($\Theta$-linearity) implies that it equals
\begin{align*}
&\SPresS{a_2 \wedge b_2,a_2 \wedge b_2} + 2\ThetaS{a_2 \wedge b_2, a_2 \wedge (a_1+b_1)} + \PresS{a_2 \wedge (a_1+b_1),a_2 \wedge (a_1+b_1)} \\
\equiv &2\SThetaS{a_2 \wedge b_2, a_2 \wedge a_1} + 2\ThetaS{a_2 \wedge b_2, a_2 \wedge b_1} + X_1(\SPresS{a_2 \wedge a_1,a_2 \wedge a_1}) \equiv 0.
\end{align*}
The last $\equiv$ uses the fact that we have already proved that $X_1(t) \equiv 0$ for
$t \in S_1$ (Lemma \ref{lemma:presentationsymxs1}).
For the second term $\PresS{a_2 \wedge (b_2+a_1),a_2 \wedge (b_2+a_1)}$, Lemma \ref{lemma:thetaexpansion1} ($\Theta$-expansion I) 
implies that it equals
\[\SPresS{a_2 \wedge b_2} + 2\SThetaS{a_2 \wedge b_2,a_2 \wedge a_1} + \SPresS{a_2 \wedge a_1,a_2 \wedge a_1} \equiv 0.\qedhere\]
\end{proof}

\section{Symmetric kernel, symmetric version IX: closure of \texorpdfstring{$S_4$}{S4}}
\label{section:presentationsym9}

We continue using all the notation from \S \ref{section:presentationsym1} -- \S \ref{section:presentationsym5}.
In this section, we complete the proof of Lemma \ref{lemma:presentationsymstep3} (and hence also of
Theorem \ref{maintheorem:presentationsym}) by proving that for
$f \in \{X_1,X_1^{-1},Y_{12}\}$ and $s \in S_4$, we have $f(s) \in \Span{S}$.

\subsection{X-closure}
Recall that $X_1 \in \Sp_{2g}(\Z)$ takes $a_1$ to $a_1+b_1$ and fixes all other generators in $\cB$.  
We start with:

\begin{lemma}
\label{lemma:presentationsymxs4}
Let 
\[s \in S_4 = \Set{$\SOmegaS{a_i \wedge a_j,b_i \wedge b_j}$, $\SOmegaS{a_i \wedge b_j,b_i \wedge a_j}$}{$1 \leq i < j \leq g$}\]
Then $X_1^{\epsilon}(s) \in \Span{S}$ for $\epsilon \in \{\pm 1\}$.
\end{lemma}
\begin{proof}
The lemma is trivial if $X_1$ fixes $s$.  The remaining cases are as follows.  Let $\equiv$ denote equality modulo $\Span{S}$,
so our goal is to prove that $X_1^{\epsilon}(s) \equiv 0$.

\begin{case}{1}
\label{case:presentationsymxs4.1}
$s = \SOmegaS{a_1 \wedge a_j,b_1 \wedge b_j}$ with $2 \leq j \leq g$.
\end{case}

\noindent
Since $g \geq 4$ (Assumption \ref{assumption:genussym}), we can pick $2 \leq k \leq g$ with $k \neq j$.  Lemma \ref{lemma:omegalambda} ($\Lambda$ to $\Omega$) implies that
$\SOmegaS{a_1 \wedge a_j,b_1 \wedge b_j} = -\SOmegaS{a_j \wedge a_1,b_1 \wedge b_j}$ equals
\begin{align*}
&\LambdaS{a_j \wedge (a_k + a_1),(b_k - b_1) \wedge b_j} - \SLambdaS{a_j \wedge a_1,b_k \wedge b_j}\\
&+ \SLambdaS{a_j \wedge a_k,b_1 \wedge b_j}-\SOmegaS{a_j \wedge a_k,b_k \wedge b_j}
\end{align*}
Applying $X_1^{\epsilon}$, since $X_1$ fixes $\{a_j,b_j,a_k,b_k\}$ we get
\begin{align*}
&\LambdaS{a_j \wedge (a_k + a_1+\epsilon b_1),(b_k - b_1) \wedge b_j} - X_1^{\epsilon}(\SLambdaS{a_j \wedge a_1,b_k \wedge b_j}) \\
&+X_1^{\epsilon}(\SLambdaS{a_j \wedge a_k,b_1 \wedge b_j})-\SOmegaS{a_j \wedge a_k,b_k \wedge b_j} \\
\equiv &\LambdaS{a_j \wedge (a_k + a_1+\epsilon b_1),(b_k - b_1) \wedge b_j}.
\end{align*}
Here we are using the fact that $X_1^{\epsilon}(t) \equiv 0$ for all $t \in S_3$ (Lemma \ref{lemma:presentationsymxs3}).
Lemma \ref{lemma:alllambda} implies that $\LambdaS{a_j \wedge (a_k + a_1+\epsilon b_1),(b_k - b_1) \wedge b_j} \equiv 0$,
and we are done. 

\begin{case}{2}
$s =\SOmegaS{a_1 \wedge b_j,b_1 \wedge a_j}$ with $2 \leq j \leq g$.
\end{case}

\noindent
Since $g \geq 4$ (Assumption \ref{assumption:genussym}), we can pick $2 \leq k \leq g$ with $k \neq j$.
Lemma \ref{lemma:omegalambda} ($\Lambda$ to $\Omega$) implies that
$\SOmegaS{a_1 \wedge b_j,b_1 \wedge a_j} = \SOmegaS{(-b_j) \wedge a_1,b_1 \wedge a_j}$ equals
\begin{align*}
 &\LambdaS{(-b_j) \wedge (a_1 + a_k),(b_1 - b_k) \wedge a_j} - \LambdaS{(-b_j) \wedge a_k,b_1 \wedge (a_j)} \\
 &+ \LambdaS{(-b_j) \wedge a_1,b_k \wedge a_j} +\OmegaS{(-b_j) \wedge a_k,b_k \wedge a_j} \\
=&-\LambdaS{b_j \wedge (a_1 + a_k),(b_1 - b_k) \wedge a_j} + \SLambdaS{b_j \wedge a_k,b_1 \wedge (a_j)} \\
 &- \SLambdaS{b_j \wedge a_1,b_k \wedge a_j} +\SOmegaS{a_k \wedge b_j,b_k \wedge a_j}.
\end{align*}
Now proceed as in Case \ref{case:presentationsymxs4.1}.
\end{proof}

\subsection{Y-closure}
Recall that $Y_{12} \in \Sp_{2g}(\Z)$ takes $a_1$ to $a_1+b_2$ and $a_2$ to $a_2+b_1$ and fixes all other generators in $\cB$.
We next prove the following, which completes the proof of Lemma \ref{lemma:presentationsymstep3} and hence
of Theorem \ref{maintheorem:presentationsym}:

\begin{lemma}
\label{lemma:presentationsymys4}
Let 
\[s \in S_4 = \Set{$\SOmegaS{a_i \wedge a_j,b_i \wedge b_j}$, $\SOmegaS{a_i \wedge b_j,b_i \wedge a_j}$}{$1 \leq i < j \leq g$}\]
Then $Y_{12}(s) \in \Span{S}$. 
\end{lemma}
\begin{proof}
Write $S_4 = S_4(1) \cup S_4(2) \cup S_4(3) \cup S_4(4)$ with
\begin{align*}
S_4(1) &= \Set{$\SOmegaS{a_i \wedge a_j,b_i \wedge b_j}$, $\SOmegaS{a_i \wedge b_j,b_i \wedge a_j}$}{$3 \leq i < j \leq g$},\\
S_4(2) &= \Set{$\SOmegaS{a_1 \wedge a_j,b_1 \wedge b_j}$, $\SOmegaS{a_1 \wedge b_j,b_1 \wedge a_j}$}{$3 \leq j \leq g$},\\
S_4(3) &= \Set{$\SOmegaS{a_2 \wedge a_j,b_2 \wedge b_j}$, $\SOmegaS{a_2 \wedge b_j,b_2 \wedge a_j}$}{$3 \leq j \leq g$},\\
S_4(4) &= \{\SOmegaS{a_1 \wedge a_2,b_1 \wedge b_2}, \SOmegaS{a_1 \wedge b_2,b_1 \wedge a_2}\}.
\end{align*}
The lemma is trivial for $s \in S_4(1)$ since in that case $Y_{12}(s) = s$.  For $s \in S_4(2)$, the lemma
can be proved exactly like Case \ref{case:presentationsymxs4.1} of the proof of 
Lemma \ref{lemma:presentationsymxs4}.  The only necessary change is that the $k$ in that proof
should be chosen such that $3 \leq k \leq g$ and $k \neq j$, which is possible
since $g \geq 4$ (Assumption \ref{assumption:genussym}; note that in 
Case \ref{case:presentationsymxs4.1} of the proof of 
Lemma \ref{lemma:presentationsymxs4} we really only used $g \geq 3$).
The same argument works for $s \in S_4(3)$.  

It remains
to deal with $S_4(4)$, which we divide into two cases.  Let $\equiv$ denote equality
modulo $\Span{S}$, so our goal is to prove that $Y_{12}(s) \equiv 0$.

\begin{case}{1}
$s = \SOmegaS{a_1 \wedge a_2,b_1 \wedge b_2}$.
\end{case}

\noindent
Set $s' = \PresS{(a_1-b_2) \wedge (b_1 -a_2),(a_1-b_2) \wedge (b_1 - a_2)}$.
In Lemma \ref{lemma:identifyomega}, we proved that $-2s$ equals $s'$ modulo
$\Span{S_1,S_2,S_3}$.
We have already proved that $Y_{12}(t) \equiv 0$ for $t \in S_1 \cup S_2 \cup S_3$; see
Lemmas \ref{lemma:presentationsymys1} and \ref{lemma:presentationsymys2} and \ref{lemma:presentationsymys3}.
It is therefore enough to prove that $Y_{12}(s') \equiv 0$:
\[Y_{12}(\PresS{(a_1-b_2) \wedge (b_1 -a_2),(a_1-b_2) \wedge (b_1 - a_2)}) = \SPresS{a_1 \wedge (-a_2),a_1 \wedge (-a_2)} \equiv 0.\]

\begin{case}{2}
$s = \SOmegaS{a_1 \wedge b_2,b_1 \wedge a_2}$.
\end{case}

\noindent
Set $s' = \PresS{(a_1+a_2) \wedge (b_1-b_2),(a_1+a_2) \wedge (b_1-b_2)}$.
In Lemma \ref{lemma:identifyomega}, we proved that $s$ equals $s'$ modulo
$\Span{S_1,S_2,S_3}$.  Just like in the previous case, this implies that it
is enough to prove that $Y_{12}(s') \equiv 0$.  We calculate:
\begin{align*}
&Y_{12}(\PresS{(a_1+a_2) \wedge (b_1-b_2),(a_1+a_2) \wedge (b_1-b_2)}) \\
=&\PresS{(a_1+b_2+a_2+b_1) \wedge (b_1-b_2),(a_1+b_2+a_2+b_1) \wedge (b_1-b_2)} \\
=&\PresS{(a_1+a_2+2b_1) \wedge (b_1-b_2),(a_1+a_2+2b_1) \wedge (b_1-b_2)} = X_{1}^2(s').
\end{align*}
In Lemmas \ref{lemma:presentationsymxs1} and \ref{lemma:presentationsymxs2} and \ref{lemma:presentationsymxs3}
and \ref{lemma:presentationsymxs4}, we proved that $X_1^{\epsilon}(t) \in \Span{S}$ for
$\epsilon \in \{\pm 1\}$ and $s \in S_1 \cup \cdots \cup S_4 = S$.  This implies 
that the cyclic group generated by $X_1$ takes $\Span{S}$ to $\Span{S}$.  
Lemma \ref{lemma:identifyomega} implies that $s' \in S$, so $X_{1}^2(s') \equiv 0$, as desired.
\end{proof}

\appendix
\part{Appendices}

\section{A modified presentation}
\label{appendix:presentation}

The goal of this appendix is to transform the presentation for $\fK_g$ from Definition \ref{definition:kg} to the one
needed for our work on the Torelli group in \cite{MinahanPutmanAbelian, MinahanPutmanH2}.
Recall from \S \ref{section:notation} that $H = \Q^{2g}$ and $H_{\Z} = \Z^{2g}$, and that
$\omega\colon H \times H \rightarrow \Q$ is the standard symplectic form on
$H$.

\subsection{Symmetric kernel and contraction}

Recall from the introduction that the symmetric contraction is the alternating bilinear map
\[\fc\colon ((\wedge^2 H)/\Q) \times ((\wedge^2 H)/\Q) \longrightarrow \Sym^2(H)\]
defined by the formula
\[\text{$\fc(x \wedge y,z \wedge w) = \omega(x,z) y \Cdot w - \omega(x,w) y \Cdot z - \omega(y,z) x \Cdot w + \omega(y,w) x \Cdot z$ for $x,y,z,w \in H$}.\]
This induces a map
$((\wedge^2 H)/\Q)^{\otimes 2} \longrightarrow \Sym^2(H)$
whose kernel $\cK_g$ is the symmetric kernel.
Recall that $\kappa_1,\kappa_2 \in (\wedge^2 H)/\Q$ are sym-orthogonal if
\[\fc(\kappa_1,\kappa_2)=-\fc(\kappa_2,\kappa_1) = 0,\]
or equivalently if $\kappa_1 \otimes \kappa_2$ and $\kappa_2 \otimes \kappa_1$ lie in $\cK_g$.

\subsection{Presentation}

We recall the definition of $\fK_g$:

\begin{definition}
Define $\fK_g$ to be the $\Q$-vector space with the following presentation:
\begin{itemize}
\item {\bf Generators}. 
A generator $\Pres{\kappa_1,\kappa_2}$ for all sym-orthogonal
$\kappa_1,\kappa_2 \in (\wedge^2 H)/\Q$ such that either
$\kappa_1$ or $\kappa_2$ (or both) is a symplectic pair in $(\wedge^2 H_{\Z})/\Z$.
\item {\bf Relations}.  For all symplectic pairs $a \wedge b \in (\wedge^2 H_{\Z})/\Z$ and all
$\kappa_1,\kappa_2 \in (\wedge^2 H)/\Q$ that are sym-orthogonal to $a \wedge b$
and all $\lambda_1,\lambda_2 \in \Q$, the relations
\begin{align*}
\Pres{a \wedge b,\lambda_1 \kappa_1 + \lambda_2 \kappa_2} &= \lambda_1 \Pres{a \wedge b,\kappa_1} + \lambda_2 \Pres{a \wedge b,\kappa_2} \quad \text{and} \\
\Pres{\lambda_1 \kappa_1 + \lambda_2 \kappa_2,a \wedge b} &= \lambda_1 \Pres{\kappa_1,a \wedge b} + \lambda_2 \Pres{\kappa_2,a \wedge b}.\qedhere
\end{align*}
\end{itemize}
\end{definition}

There is a linearization map $\Phi\colon \fK_g \rightarrow ((\wedge^2 H)/\Q)^{\otimes 2}$ defined
by $\Phi(\Pres{\kappa_1,\kappa_2}) = \kappa_1 \otimes \kappa_2$.  This takes relations to relations,
and thus gives a well-defined map.  Since $\kappa_1$ and $\kappa_2$ are sym-orthogonal, the image
of $\Phi$ is contained in $\cK_g$.  Theorem \ref{theorem:presentation} says that $\Phi$
is an isomorphism for $g \geq 4$.

\subsection{Symplectic summands}

As we said above, our goal is to modify the presentation of $\fK_g$ to the one
needed for our papers \cite{MinahanPutmanAbelian, MinahanPutmanH2}.  This requires some preliminaries.
A {\em symplectic summand} of $H_{\Z}$ is a subgroup $V < H_{\Z}$ such that
$H_{\Z} = V \oplus V^{\perp}$.
A symplectic summand $V$ of $H_{\Z}$ is isomorphic to $\Z^{2h}$ for some
$h$ called its {\em genus}.  If $V$ is a symplectic summand of $H_{\Z}$, then $V^{\perp}$ is too.

\subsection{Symplectic form}

Let $W$ be a symplectic
summand of $H_{\Z}$.  The symplectic form on $W$ identifies $W$ with its
dual.  This allows us to identify alternating bilinear forms on $W$ with
elements of $\wedge^2 W \subset \wedge^2 H_{\Z}$.  In particular, the symplectic form
on $W$ is an element $\omega_W$ of $\wedge^2 H_{\Z}$.  If $\{a_1,b_1,\ldots,a_h,b_h\}$
is a symplectic basis for $W$, then
$\omega_W = a_1 \wedge b_1 + \cdots + a_h \wedge b_h$.
With this notation, the symplectic form $\omega$ on $H_{\Z}$ is $\omega = \omega_{H_{\Z}}$.

\subsection{Symplectic pairs}

Recall that $(\wedge^2 H)/\Q$ is the quotient of $\wedge^2 H$ by the $\Q$-span of
$\omega \in \wedge^2 H_{\Z}$.  For a symplectic summand $W$ of $H_{\Z}$, let
$\oomega_W$ be the image of $\omega_W \in \wedge^2 H_{\Z}$ in $(\wedge^2 H)/\Q$.
Since $\omega = \omega_W + \omega_{W^{\perp}}$, we have $\oomega_{W^{\perp}} = -\oomega_W$.
The elements $\oomega_W$ with $W$ a genus-$1$ symplectic summand of $H_{\Z}$ are exactly
the symplectic pairs.  

\subsection{Generators and summands}

Recall that the generators for $\fK_g$ are as follows:
\begin{itemize}
\item Let $V$ be a genus-$1$ symplectic summand of $H_{\Z}$ and let $\kappa \in (\wedge^2 H)/\Q$ be
sym-orthogonal to $\oomega_V$.  We then have generators $\Pres{\oomega_V,\kappa}$ and $\Pres{\kappa,\oomega_V}$.
\end{itemize}
The following lemma says that the element $\oomega_V$ is determined by $V$:

\begin{lemma}
\label{lemma:identifyv}
Assume that $g \geq 3$.
Let $V$ and $W$ be genus-$1$ symplectic summands of $H_{\Z}$ such that $\oomega_V = \oomega_W$.  Then $V = W$.
\end{lemma}
\begin{proof}
Let $\omega \in \wedge^2 H$ be the element corresponding to the symplectic form.
Since $\oomega_V = \oomega_W$, there exists some $\lambda \in \Q$ such that $\omega_V-\omega_W = \lambda \omega$.
The orthogonal complements $V^{\perp}$ and $W^{\perp}$ in $H$ both have codimension $2$.
Since $g \geq 3$, it follows that $H$ has dimension at least $6$
and hence we can find some nonzero $x \in V^{\perp} \cap W^{\perp}$.
View elements of $\wedge^2 H$ as alternating bilinear forms on $H$.
Since $\omega$ is a nondegenerate pairing on $H$, we can find $y \in H$ such that $\omega(x,y) = 1$.
We then have
\[0 = \omega_V(x,y) - \omega_W(x,y) = \lambda \omega(x,y) = \lambda,\]
so $\omega_V = \omega_W$.  The alternating bilinear forms $\omega_V$ and $\omega_W$ determine $V$ and $W$; for instance,
the kernel of the form $\omega_V$ is $V^{\perp}$ and $V = (V^{\perp})^{\perp}$.  We conclude that $V = W$.
\end{proof}

\subsection{Lifting sym-orthogonal subspace}

Recall that for a subspace $U$ of $\wedge^2 H$, we denote by $\oU$ the image of $U$ in
$(\wedge^2 H)/\Q$.  Also, for a subgroup $V$ of $H_{\Z}$ we write $V_{\Q}$ for the subspace
$V \otimes \Q$ of $H$.  For a genus-$1$ symplectic summand $V$, Lemma \ref{lemma:symplecticorthogonal} says
that the elements of $(\wedge^2 H)/\Q$ that are sym-orthogonal to $\oomega_V$ are those
lying in $\overline{\wedge^2 V_{\Q}^{\perp}}$.  The following lets us lift these 
to elements of $\wedge^2 V_{\Q}^{\perp}$:

\begin{lemma}
\label{lemma:identifykappa}
Let $V$ be a genus-$1$ symplectic summand of $H_{\Z}$.  Then the map
$\wedge^2 V_{\Q}^{\perp} \rightarrow \overline{\wedge^2 V_{\Q}^{\perp}}$ obtained
by restricting the projection $\wedge^2 H \rightarrow (\wedge^2 H)/\Q$ is an isomorphism.
\end{lemma}
\begin{proof}
Let $\omega \in \wedge^2 H$ be the element corresponding to the symplectic form.
We must prove that $\omega \notin \wedge^2 V_{\Q}^{\perp}$.
Let
$\cB = \{a_1,b_1,\ldots,a_{g},b_{g}\}$ be a symplectic basis for $H_{\Z}$ such that $V = \Span{a_{g},b_{g}}$.
Let $\prec$ be the total order on $\cB$ indicated in the above list.  Then $\wedge^2 H$
has the basis $\Set{$x \wedge y$}{$x,y \in \cB$, $x \prec y$}$, the subspace
$\wedge^2 V_{\Q}^{\perp}$ has for a basis the subset $\Set{$x \wedge y$}{$x,y \in \cB \setminus \{a_{g},b_{g}\}$, $x \prec y$}$,
and
\[\omega = a_1 \wedge b_1 + \cdots + a_{g} \wedge b_{g}.\]
Since $a_{g} \wedge b_{g}$ is a basis element {\em not} included in the basis for
$\wedge^2 V_{\Q}^{\perp}$, the lemma follows.
\end{proof}

\subsection{Symplectic pairs in sym-orthogonal complement}

For a genus-$1$ symplectic summand $V$ of $H_{\Z}$, we have
$\oomega_V = -\oomega_{V^{\perp}} \in \overline{\wedge^2 V_{\Q}^{\perp}}$.
In particular, by Lemma \ref{lemma:symplecticorthogonal} the element
$\oomega_V$ is sym-orthogonal to itself.  The following lemma says that this
is the only non-obvious $\oomega_W$ contained in the sym-orthogonal complement
of $\oomega_V$:

\begin{lemma}
\label{lemma:ambiguity}
Let $V$ and $W$ be genus-$1$ symplectic summand of $H_{\Z}$ such $\oomega_V$ and $\oomega_W$
are sym-orthogonal.  Then either $W \subset V^{\perp}$ or $W = V$.
\end{lemma}
\begin{proof}
Assume for the sake of contradiction
that $W \neq V$ and $W \notsubset V^{\perp}$.  Since $V$ and $W$ are both genus-$1$ symplectic summands
of $H_{\Z}$, the assumption $W \neq V$ implies that $W \notsubset V$.  Since $H_{\Z} = V \oplus V^{\perp}$,
the assumptions that $W \notsubset V^{\perp}$ and $W \notsubset V$ imply that there exist
nonzero $x_1 \in V$ and $x_2 \in V^{\perp}$ such that $x_1 + x_2 \in W$.

Since $H = V_{\Q} \oplus V_{\Q}^{\perp}$, we have
\begin{equation}
\label{eqn:decomposewedge2}
\wedge^2 H = \left(\wedge^2 V_{\Q}\right) \oplus \left(\wedge^2 V_{\Q}^{\perp}\right) \oplus \left(V_{\Q} \wedge V_{\Q}^{\perp}\right).
\end{equation}
Let $\omega \in \wedge^2 H$ be the symplectic form.  Both $\omega_V$ and $\omega = \omega_V + \omega_{V^{\perp}}$ lie
in the subspace
\begin{equation}
\label{eqn:decomposewedge2.part}
\left(\wedge^2 V_{\Q}\right) \oplus \left(\wedge^2 V_{\Q}^{\perp}\right)
\end{equation}
of \eqref{eqn:decomposewedge2}.  Since $\oomega_W$ is sym-orthogonal to $\oomega_V$, Lemma \ref{lemma:symplecticorthogonal}
says that $\oomega_W \in \overline{\wedge^2 V_{\Q}^{\perp}}$.  Equivalently, modulo
$\Q \omega$ the element $\omega_W$ lies in $\wedge^2 V_{\Q}^{\perp}$, so $\omega_W$ lies in
\eqref{eqn:decomposewedge2.part} as well.

Recall that we have nonzero $x_1 \in V$ and $x_2 \in V^{\perp}$ with $x_1+x_2 \in W$.  Regard
$x_1+x_2$ as an element of $W_{\Q}$.  Pick $y_1 \in V$ and $y_2 \in V^{\perp}$ such that
$W_{\Q} = \Span{x_1+x_2,y_1+y_2}$.  We have
\[\omega_W = (x_1+x_2) \wedge (y_1 + y_2) = x_1 \wedge y_1 + x_2 \wedge y_2 + x_1 \wedge y_2 - y_1 \wedge x_2.\]
Since $\omega_W$ lies in \eqref{eqn:decomposewedge2.part}, we must have $x_1 \wedge y_2 = y_1 \wedge x_2$ in
$V_{\Q} \wedge V_{\Q}^{\perp}$.  Since $x_1$ and $x_2$ are nonzero, this implies that there exists
some $\lambda_1,\lambda_2 \in \Q$ such that $y_1 = \lambda_1 x_1$ and $y_2 = \lambda_2 x_2$.  Since $\omega(x_1,x_2) = 0$,
we conclude that
\[\omega(x_1+x_2,y_1+y_2) = \omega(x_1+x_2,\lambda_1 x_1+\lambda_2 x_2) = \lambda_1 \omega(x_1,x_1) + \lambda_2 \omega(x_2,x_2) = 0.\]
This implies that $\omega$ vanishes identically on $W_{\Q} = \Span{x_1+x_2,y_1+y_2}$, contradicting the fact
that it is a symplectic summand.
\end{proof}

\subsection{Modified presentation}

Define the following:

\begin{definition}
Define $\fK'_g$ to be the vector space with the following presentation:
\begin{itemize}
\item {\bf Generators}.  For all genus-$1$ symplectic summands $V$ of $H_{\Z}$ and all
$\kappa \in \wedge^2 V_{\Q}^{\perp}$, generators $\Pres{V,\kappa}$ and $\Pres{\kappa,V}$.
\item {\bf Relations}.  The following families of relations:
\begin{itemize}
\item For all genus-$1$ symplectic summands $V$ of $H_{\Z}$ and all $\kappa_1,\kappa_2 \in \wedge^2 V_{\Q}^{\perp}$
and all $\lambda_1,\lambda_2 \in \Q$, the relations
\begin{align*}
\Pres{V,\lambda_1 \kappa_1 + \lambda_2 \kappa_2} &= \lambda_1 \Pres{V,\kappa_1} + \lambda_2 \Pres{V,\kappa_2} \quad \text{and} \\ 
\Pres{\lambda_1 \kappa_1 + \lambda_2 \kappa_2,V} &= \lambda_1 \Pres{\kappa_1,V} + \lambda_2 \Pres{\kappa_2,V}.
\end{align*}
\item For all orthogonal genus-$1$ symplectic summands $V$ and $W$ of $H_{\Z}$, the relation
\[\Pres{V,\omega_W} = \Pres{\omega_V,W}\]
\item For all genus-$1$ symplectic summands $V$ of $H_{\Z}$, the relation
\[\Pres{V,\omega_{V^{\perp}}} = \Pres{\omega_{V^{\perp}},V}.\qedhere\]
\end{itemize}
\end{itemize}
\end{definition}

The actions of $\Sp_{2g}(\Z)$ on $H_{\Z}$ and $H$ induce an action of $\Sp_{2g}(\Z)$
on $\fK'_g$.  The main result of this appendix is:

\begin{theorem}
\label{theorem:modifiedpresentation}
For $g \geq 4$, there is an $\Sp_{2g}(\Z)$-equivariant isomorphism between $\fK'_g$ and the 
symmetric kernel $\cK_g$.  In particular, $\fK'_g$ is a finite-dimensional algebraic representation
of $\Sp_{2g}(\Z)$.
\end{theorem}
\begin{proof}
By Theorem \ref{theorem:presentation}, it is enough
to construct an $\Sp_{2g}(\Z)$-equivariant isomorphism from $\fK'_g$ to $\fK_g$.  For $\kappa \in \wedge^2 H$, let $\okappa$ be its
image in $(\wedge^2 H)/\Q$.  Define a map
$f\colon \fK'_g \rightarrow \fK_g$ on generators via the formulas
\[f(\Pres{V,\kappa}) = \Pres{\oomega_V,\okappa} \quad \text{and} \quad f(\Pres{\kappa,V}) = \Pres{\okappa,\oomega_V}.\]
This takes relations to relations; indeed, the linearity relations are obvious, and the other relations
can be checked as follows:
\begin{itemize}
\item Consider orthogonal genus-$1$ symplectic summands $V$ and $W$ of $H_{\Z}$.  We must prove
that
\[f(\Pres{V,\omega_W}) = \Pres{\oomega_V,\oomega_W} \quad \text{and} \quad f(\Pres{\omega_V,W}) = \Pres{\oomega_V,\oomega_W}\]
are equal, which is clear.
\item Consider a genus-$1$ symplectic summand $V$ of $H_{\Z}$.  We must prove that
\[f(\Pres{V,\omega_{V^{\perp}}}) = \Pres{\oomega_V,\oomega_{V^{\perp}}} \quad \text{and} \quad f(\Pres{\omega_{V^{\perp}},V}) 
= \Pres{\oomega_{V^{\perp}},\oomega_V}\]
are equal, which follows from the calculation
\[\Pres{\oomega_V,\oomega_{V^{\perp}}} = \Pres{\oomega_V,-\oomega_V} = -\Pres{\oomega_V,\oomega_V} = \Pres{-\oomega_V,\oomega_V} = \Pres{\oomega_{V^{\perp}},\oomega_V}.\]
\end{itemize}
This implies that $f$ is a well-defined map.

To prove that $f$ is an isomorphism, we must construct an inverse $h\colon \fK_g \rightarrow \fK'_g$.  Let
$V$ be a genus-$1$ symplectic summand of $H_{\Z}$ and let $\kappa \in (\wedge^2 H)/\Q$ be sym-orthogonal to
$\oomega_V$.  We must define $h$ on $\Pres{\oomega_V,\kappa}$ and $\Pres{\kappa,\oomega_V}$.  
Lemma \ref{lemma:symplecticorthogonal} implies that $\kappa \in \overline{\wedge^2 V_{\Q}^{\perp}}$, and
Lemma \ref{lemma:identifykappa} says that $\kappa$ can be uniquely lifted
to $\tkappa \in \wedge^2 V_{\Q}^{\perp}$.  Lemma \ref{lemma:identifyv} says that $\oomega_V$ determines
$V$, so we can define
\[h_1(\Pres{\oomega_V,\kappa}) = \Pres{V,\tkappa} \quad \text{and} \quad h_2(\Pres{\kappa,\oomega_V}) = \Pres{\tkappa,V}.\]
The reason for distinguishing between $h_1$ and $h_2$ is that it is possible for a generator
of $\fK_g$ to be of both of these forms.  To define $h$, we must check that:

\begin{unnumberedclaim}
Let $V$ and $W$ be genus-$1$ symplectic summands of $H_{\Z}$ such that $\oomega_V$ and $\oomega_W$ are
sym-orthogonal.  Then $h_1(\Pres{\oomega_V,\oomega_W}) = h_2(\Pres{\oomega_V,\oomega_W})$.
\end{unnumberedclaim}
\begin{proof}[Proof of claim]
Since $\oomega_V$ and $\oomega_W$ are sym-orthogonal, Lemma \ref{lemma:ambiguity} implies that
one of the following holds:
\begin{itemize}
\item $W \subset V^{\perp}$.  The unique lift of $\oomega_W \in \overline{\wedge^2 V_{\Q}^{\perp}}$ to
$\wedge^2 V_{\Q}^{\perp}$ is $\omega_W$, and the unique lift of $\oomega_V \in \overline{\wedge^2 W_{\Q}^{\perp}}$
is $\omega_V$.  We now calculate as follows, where the orange $\orange{=}$ are applications of relations
in $\fK'_g$:
\end{itemize}
\[h_1(\Pres{\oomega_V,\oomega_W}) = \Pres{V,\omega_W} \orange{=} \Pres{\omega_V,W} = h_2(\Pres{\oomega_V,\oomega_W}).\]
\begin{itemize}
\item $W = V$.  The unique lift of
\[\oomega_V = -\oomega_{V^{\perp}} \in \overline{\wedge^2 V_{\Q}^{\perp}}\]
to $\wedge^2 V_{\Q}^{\perp}$ is $-\omega_{V^{\perp}}$.
We now calculate as follows, where the orange $\orange{=}$ are applications of relations
in $\fK'_g$:
\end{itemize}
\begin{align*}
h_1(\Pres{\oomega_V,\oomega_V}) &= \Pres{V,-\omega_{V^{\perp}}} \orange{=} -\Pres{V,\omega_{V^{\perp}}}   \\
                                &\orange{=} -\Pres{\omega_{V^{\perp}},V} \orange{=} \Pres{-\omega_{V^{\perp}},V} = h_2(\Pres{\oomega_V,\oomega_V}).\qedhere
\end{align*}
\end{proof}

In light of this claim, we can unambiguously
define a map $h\colon \fK_g \rightarrow \fK'_g$ on generators $\Pres{\kappa_1,\kappa_2}$
by letting $h(\Pres{\kappa_1,\kappa_2})$ equal whichever one of
$h_1(\Pres{\kappa_1,\kappa_2})$ or $h_2(\Pres{\kappa_1,\kappa_2})$ is defined.
The map $h$ takes relations to relations, and thus gives a well-defined map.
By construction, $f$ and $h$ are inverses to each other.  The theorem follows.
\end{proof}

\begin{remark}
The isomorphism $\fK'_g \rightarrow \cK_g$ takes
$\Pres{V,\kappa}$ to $\oomega_V \otimes \okappa$ and $\Pres{\kappa,V}$ to $\okappa \otimes \oomega_V$.
\end{remark}

\end{document}